\documentclass[11pt,oneside,a4paper,mathscr]{amsart}

\usepackage{amscd,amsmath,amssymb,euscript}
\usepackage[frame,cmtip,curve,arrow,matrix,line,graph]{xy}
\usepackage{tikz-cd}
\usepackage{hyperref}
\usepackage{stmaryrd}
\SetSymbolFont{stmry}{bold}{U}{stmry}{m}{n}
\usepackage{bm}
\usepackage{here}
\usepackage{enumitem}
\hypersetup{hidelinks}
\title[The Dolbeault geometric Langlands beyond the elliptic locus]{The Dolbeault geometric Langlands correspondence for type \texorpdfstring{$A$}{A} groups
beyond the elliptic locus}
\author{Yukinobu Toda}

\newtheorem{thm}{Theorem}[section]
\newtheorem*{mainthm}{Main Theorem}
\newtheorem{cor}[thm]{Corollary}
\newtheorem{prop}[thm]{Proposition}

\newtheorem{conj}[thm]{Conjecture}
\newtheorem{lemma}[thm]{Lemma}

\theoremstyle{definition}
\newtheorem{defn}[thm]{Definition}

\newtheorem{thm*}[thm]{Theorem$^*$}
\newtheorem{remark}[thm]{Remark}

\newtheorem{example}[thm]{Example}

\newtheorem{step}{Step}
\newtheorem{sstep}{Step}
\newtheorem{ssstep}{Step}

\newcommand{\comment}[1]{}

\renewcommand{\leq}{\leqslant}
\renewcommand{\geq}{\geqslant}

\newcommand{\tr}{\operatorname{\mathrm{tr}}}

\newcommand{\Y}{\mathcal{Y}}
\newcommand{\cS}{\mathcal{S}}
\newcommand{\zZ}{\mathcal{Z}}

\newcommand{\LL}{\operatorname{L}}
\newcommand{\Pic}{\operatorname{Pic}}

\newcommand{\QCoh}{\operatorname{QCoh}}

\newcommand{\X}{\mathcal{X}}

\newcommand{\cH}{\mathcal{H}}

\newcommand{\cL}{\mathcal{L}}
\newcommand{\cB}{\mathcal{B}}
\newcommand{\cC}{\mathcal{C}}

\newcommand{\cP}{\mathcal{P}}
\newcommand{\cE}{\mathcal{E}}
\newcommand{\cF}{\mathcal{F}}
\newcommand{\cO}{\mathcal{O}}
\newcommand{\cD}{\mathcal{D}}
\newcommand{\ev}{\mathrm{ev}}
\newcommand{\bgm}{B\mathbb{G}_m}

\newcommand{\diasquare}{\ar@{}[rd]|\square}
\newcommand{\cHecke}{\mathcal{H}ecke}
\newcommand{\cHom}{\mathcal{H}om}
\newcommand{\rank}{\operatorname{rank}}

\newcommand{\rB}{\mathrm{B}}

\newcommand{\Coh}{\operatorname{Coh}}
\newcommand{\Aut}{\operatorname{Aut}}

\newcommand{\id}{\operatorname{id}}
\newcommand{\Ext}{\operatorname{Ext}}
\newcommand{\Ind}{\operatorname{Ind}}
\newcommand{\IndCoh}{\operatorname{IndCoh}}
\newcommand{\IndL}{\operatorname{IndL}}
\newcommand{\Hom}{\operatorname{Hom}}
\newcommand{\Map}{\operatorname{Map}}
\newcommand{\Spec}{\operatorname{Spec}}

\newcommand{\GL}{\operatorname{GL}}
\newcommand{\SL}{\operatorname{SL}}
\newcommand{\PGL}{\operatorname{PGL}}
\newcommand{\rS}{\mathrm{S}}
\newcommand{\rP}{\mathrm{P}}

\newcommand{\inclusion}{\ar@<-0.3ex>@{^{(}->}[r]}
\newcommand{\iinclusion}{\ar@<-0.3ex>@{^{(}->}[rr]}
\newcommand{\linclusion}{\ar@<0.3ex>@{_{(}->}[l]}
\newcommand{\dinclusion}{\ar@<-0.3ex>@{^{(}->}[d]}
\newcommand{\uinclusion}{\ar@<-0.3ex>@{^{(}->}[u]}

\newcommand{\wt}{\mathrm{wt}}
\newcommand{\Hig}{\mathrm{Higgs}}
\newcommand{\vd}{\varepsilon\dag}
\newcommand{\vdss}{\varepsilon\dag\text{-ss}}
\newcommand{\Bun}{\mathrm{Bun}}

\renewcommand{\labelenumi}{(\theenumi)}

\makeatletter
\newcommand{\colim@}[2]{%
  \vtop{\m@th\ialign{##\cr
    \hfil$#1\operator@font colim$\hfil\cr
    \noalign{\nointerlineskip\kern1.5\ex@}#2\cr
    \noalign{\nointerlineskip\kern-\ex@}\cr}}%
}
\newcommand{\colim}{%
  \mathop{\mathpalette\colim@{}}\nmlimits@
}
\makeatother

\usetikzlibrary{positioning,fit,calc}
\tikzstyle{block}=[draw=black, width=1cm, minimum height=2cm, align=center] 
\tikzstyle{block2}=[draw=black, text width=2cm, minimum height=1cm, align=center] 
\tikzstyle{block3}=[draw=black, text width=2cm, minimum height=1cm, align=center] 

\makeatletter

\@addtoreset{equation}{section}	
\makeatother

\begin{document}

\begin{abstract}
In this paper, we prove a Dolbeault geometric Langlands equivalence for $\GL_r$
and for the Langlands dual pair $\SL_r/\PGL_r$ over an open locus of the
Hitchin base which strictly contains the elliptic locus. This open locus contains the points corresponding to spectral curves with at worst 
type $A$ singularities, without any restriction on the number of irreducible components. 

The Dolbeault geometric Langlands equivalence considered here is the one
formulated in our previous work with Tudor P\u{a}durariu, which links categorical Donaldson--Thomas theory with the geometric Langlands correspondence. It relates coherent
sheaves on moduli stacks of semistable Higgs bundles to the limit category
associated with the full moduli stack of Higgs bundles. The use of limit
categories is essential beyond the elliptic locus, where the full Higgs moduli
stack is no longer quasi-compact and contains infinitely many
Harder--Narasimhan strata.

The key step is to prove the Whittaker normalization conjecture over the locus
of spectral curves with type $A$ singularities, following and extending the
strategy developed in the author's proof of the $\GL_2$ case over the reduced
spectral curve locus. As a consequence, we also obtain the Dolbeault geometric
Langlands conjecture for $\SL_2/\PGL_2$ over the reduced spectral curve locus.

\end{abstract}

\maketitle

 \setcounter{tocdepth}{2}
\tableofcontents

\section{Introduction}

\subsection{Main result}
Let $C$ be a smooth projective curve of genus $g\geq 2$, and let $G$ be a reductive group with Langlands
dual ${}^{L}G$.
Let $\Hig_G$ be the derived moduli stack of $G$-Higgs bundles on $C$, 
equipped with the Hitchin fibration
\begin{align*}
    h\colon \Hig_G \to \rB_G. 
\end{align*}
The main result of this paper is the following theorem, which proves the
\textit{Dolbeault geometric Langlands conjecture} for $\Hig_G$ and
$\Hig_{{}^{L}G}$ for groups of type $A$ beyond the elliptic locus.

\begin{mainthm}
For $G=\GL_r$, $\SL_r$ or $\PGL_r$ with $r\geq 2$, there is an open subset $\rB_G^{\circ} \subset \rB_G$, which strictly 
contains the elliptic locus, such that 
the Dolbeault geometric Langlands conjecture as formulated in~\cite{PTlim}
holds over $\rB_G^{\circ}$. 
\end{mainthm}
Here we give two remarks on the above result:
\begin{itemize}[leftmargin=1.5em]
\item For $G=\GL_r$, the open subset $\mathrm{B}^{\circ}_G$ contains the locus corresponding to spectral curves with at worst type $A$ singularities, without any restriction on the number of irreducible components.
\item Over the locus $\rB_G^{\circ}$, the stack $\Hig_G$ is singular and each component is not quasi-compact. The framework of~\cite{PTlim}, using \textit{limit categories}, is essential both for the formulation and for the proof.
\end{itemize}

Below, we recall some background on the Dolbeault geometric Langlands conjecture and give a more precise statement of the main theorem.

\subsection{The Dolbeault geometric Langlands conjecture}
The Dolbeault geometric Langlands conjecture is a conjectural equivalence 
of dg-categories associated with moduli stacks of Higgs bundles, which 
may be regarded as a classical limit of the geometric Langlands correspondence~\cite{BD0, GLC1, GLC2, GLC3, GLC4, GLC5}. 
It is based on the following observation of Donagi--Pantev~\cite{DoPa}:
there is an isomorphism $\rB_G\cong\rB_{{}^{L}G}$, and hence there is a diagram 
\begin{align*}
    \xymatrix{
\Hig_{{}^{L}G} \ar[rd]_-{{}^{L}h} & & \Hig_G \ar[ld]^-{h} \\
& \rB_G. &
    }
\end{align*}
It is proved in~\cite{DoPa} that, over a dense open subset of $\rB_G$, the above 
maps are dual abelian fibrations. In particular, the generic fibers 
are derived equivalent by classical Fourier--Mukai theory~\cite{Mu1}.

It is expected in~\cite{DoPa} that the above equivalence extends over 
the full Hitchin base. However, the stack $\Hig_G$ is too large: indeed, each 
component is not quasi-compact unless $G$ is a torus, 
and 
standard dg-categories associated with $\Hig_G$ have poor finiteness properties. For instance,
$\IndCoh(\Hig_G)$ is not compactly generated. So far, for $G=\GL_r$, the study of such a derived equivalence has been restricted to the elliptic locus in $\rB_G$, where the above issue is avoided; see the next subsection.

In our paper with Tudor P\u{a}durariu~\cite{PTlim}, we introduced a dg-category, called a
\emph{limit category},
\begin{align}\label{intro:limit}
    \IndL_{\mathcal{N}}(\Hig_G(\chi))_w \subset \IndCoh_{\mathcal{N}}(\Hig_G(\chi))_w
\end{align}
for $\chi \in \pi_1(G)$ and $w\in Z_G^{\vee}$, where $Z_G$ is the center of $G$ and
$\mathcal{N}$ denotes the nilpotent singular support condition~\cite{AG}.
The stack $\Hig_G(\chi)$ is a connected component of $\Hig_G$ corresponding to $\chi$, and $(-)_w$ refers to the subcategory of central weight $w$. 
The limit category is expected to be an effective version of the classical limit of D-modules on $\Bun_G$. The dg-category~\eqref{intro:limit} is 
better behaved for the non-quasi-compact stack $\Hig_G(\chi)$: for example, it is compactly generated, and it also behaves as if it were a category of \textit{algebro-geometric $A$-branes}. 
As we will recall in Definition~\ref{def:Lcat}, its locally compact objects are
characterized by weight conditions with respect to maps
\[
    \nu \colon \bgm \to \Hig_G(\chi).
\]
For a perfect complex $\cE$, the above weight condition is given by 
(see Lemma~\ref{lem:perfect})
\begin{align*}
    \wt(\nu^* \cE) \subset \left[ \frac{1}{2}c_1(\nu^* \mathbb{L}_{\Hig_G}^{<0}), \frac{1}{2}c_1(\nu^* \mathbb{L}_{\Hig_G}^{>0}) \right] +c_1(\nu^* \delta_{w}).
\end{align*}
These conditions are motivated by the construction of noncommutative
resolutions of GIT quotients by Spenko--Van den Bergh~\cite{SvdB},
magic windows by Halpern-Leistner--Sam~\cite{hls}, and quasi-BPS
categories in categorical Donaldson--Thomas theory~\cite{T, PT0, PT1, PTquiver, PThiggs, PThiggs2}. The above weight condition is also natural 
as a dual notion of semistability of Higgs bundles, see Proposition~\ref{prop:natural}.

We then proposed a version of the \emph{Dolbeault geometric Langlands conjecture}~\cite{DoPa}
with nilpotent singular support, which may be regarded as
a classical limit of the 
geometric Langlands correspondence~\cite{BD0, AG, GLC1, GLC2, GLC3, GLC4, GLC5}, as follows:
\begin{conj}\emph{(\cite[Conjecture~1.1]{PTlim}, Conjecture~\ref{conj:dl})}\label{conj:intro}
    There is a $\rB_G(=\rB_{{}^{L}G})$-linear equivalence
    \begin{align}\label{intro:IndL}
        \IndCoh_{\mathcal{N}}(\Hig_{{}^{L}G}(w)^{\mathrm{ss}})_{-\chi} \stackrel{\sim}{\to}
        \IndL_{\mathcal{N}}(\Hig_G(\chi))_w.
    \end{align}
\end{conj}
On the left-hand side, we consider \textit{semistable} Higgs moduli stacks, which are quasi-compact, while on the right-hand side we consider the \textit{limit category} over the full Higgs moduli stacks, which are not quasi-compact. 
The formulation of Conjecture~\ref{conj:intro} is also motivated by the categorification of Donaldson--Thomas theory and BPS invariants, in particular \textit{$\chi$-independence phenomena}; see~\cite[Section~1]{PTlim}.

\subsection{Elliptic locus and beyond}\label{subsec:introell}
When $G=\GL_r$, each point $b\in \rB_{\GL_r}$ corresponds to a 
\textit{spectral curve}
\begin{align}\label{intro:scurve}
    \cC_b \subset S:=\mathrm{Tot}_C(\Omega_C),
\end{align}
and the fiber $h^{-1}(b)$ is identified with the moduli stack of torsion-free sheaves on $\cC_b$ with fundamental one-cycle $[\cC_b]$. 
The \textit{elliptic locus} 
\begin{align*}
    \rB_{\GL_r}^{\mathrm{ell}} \subset \rB_{\GL_r}
\end{align*}
corresponds to those spectral curves~\eqref{intro:scurve} which are integral, i.e.\ irreducible and reduced. 
Over $\rB_{\GL_r}^{\mathrm{ell}}$, the stack $\Hig_{\GL_r}$ coincides with its semistable and stable loci; in particular, each connected component is quasi-compact:
\begin{align}\label{intro:Hell}
    \Hig_{\GL_r}\times_{\rB_{\GL_r}}\rB_{\GL_r}^{\mathrm{ell}}=
     \Hig_{\GL_r}^{\mathrm{ss}}\times_{\rB_{\GL_r}}\rB_{\GL_r}^{\mathrm{ell}}=
      \Hig_{\GL_r}^{\mathrm{st}}\times_{\rB_{\GL_r}}\rB_{\GL_r}^{\mathrm{ell}}.
\end{align}
Conjecture~\ref{conj:intro} over the elliptic locus then reduces to 
a derived equivalence for compactified Jacobians of irreducible 
planar curves, proved by Arinkin~\cite{Ardual}. In particular, the 
limit category is not needed over $\rB_{\GL_r}^{\mathrm{ell}}$, since 
we have 
\begin{align*}
    \IndL_{\mathcal{N}}(\Hig_{\GL_r}\times _{\rB_{\GL_r}}\rB_{\GL_r}^{\mathrm{ell}})
    =
    \IndCoh_{\mathcal{N}}(\Hig_{\GL_r}\times _{\rB_{\GL_r}}\rB_{\GL_r}^{\mathrm{ell}}).
\end{align*}
The $\SL_r/\PGL_r$ case over the elliptic locus was studied in~\cite[Theorem~4.7]{GS}.

However, the equality~\eqref{intro:Hell} fails outside $\rB_{\GL_r}^{\mathrm{ell}}$. 
More precisely, for any open subset $\rB_{\GL_r}^{\circ}\subset \rB_{\GL_r}$ which is not contained in 
$\rB_{\GL_r}^{\mathrm{ell}}$, we have strict inclusions
 \begin{align*}
       \Hig_{\GL_r}\times_{\rB_{\GL_r}}\rB_{\GL_r}^{\circ} \supsetneq
     \Hig_{\GL_r}^{\mathrm{ss}}\times_{\rB_{\GL_r}}\rB_{\GL_r}^{\circ} \supsetneq
      \Hig_{\GL_r}^{\mathrm{st}}\times_{\rB_{\GL_r}}\rB_{\GL_r}^{\circ}.
 \end{align*}
The limit category plays an essential role over loci outside $\rB_{\GL_r}^{\mathrm{ell}}$. 
Indeed, a generic point of $\rB_{\GL_r} \setminus \rB_{\GL_r}^{\mathrm{ell}}$ 
corresponds to a reduced spectral curve with smooth irreducible components intersecting transversely; there are infinitely many Harder--Narasimhan strata on the moduli stack of coherent sheaves on such a spectral curve. 

We now summarize the differences in the behavior of the stack $\Hig_{\GL_r}$ over 
$\rB_{\GL_r}^{\mathrm{ell}}$ and outside it, highlighting the substantial difficulties of Higgs moduli stacks outside the elliptic locus. 
\begin{itemize}[leftmargin=1.5em]
    \item Over the elliptic locus $\rB_{\GL_r}^{\mathrm{ell}} \subset \rB_{\GL_r}$,
    up to a trivial derived structure, the stack $\Hig_{\GL_r}$ is a $\mathbb{G}_m$-gerbe over a smooth symplectic variety. Over this locus, every Higgs bundle with integral spectral curve is automatically stable. 
    \item Over \textit{any} open subset $\rB_{\GL_r}^{\circ} \subset \rB_{\GL_r}$ not contained in the elliptic locus, the stack $\Hig_{\GL_r}$ is a singular derived stack whose connected components are not quasi-compact. 
    Each connected component of $\Hig_{\GL_r}^{\mathrm{ss}}$ is quasi-compact, but $\Hig_{\GL_r}^{\mathrm{ss}}$ contains strictly semistable Higgs bundles, and it, respectively its good moduli space, is a singular stack, respectively a singular symplectic variety. 
\end{itemize}

\subsection[Spectral curves with type A singularities]{Spectral curves with type \texorpdfstring{$A$}{A} singularities}
For $G=\GL_r$ and $b\in \rB_{\GL_r}$, let $\cC_b$ be the corresponding spectral 
curve. We say that a singular point $p\in \cC_b$, with $c=\pi(p) \in C$, 
is a \textit{type $A$} singularity if 
$(\widehat{\cC_b})_{p} \to \widehat{C}_c$
is isomorphic to 
\begin{align*}
    \Spec k[[x, y]]/(y^2-x^m) \to \Spec k[[x]]
\end{align*}
for some $m\in \mathbb{Z}_{\geq 2}$. 
Here, for a variety $X$ and a point $x\in X$, we denote by
$\widehat{X}_x:=\Spec \widehat{\cO}_{X, x}$ the formal completion of $X$ at $x$.
Note that $m=2$ corresponds to the nodal singularity. 

It is well-known that any small deformation of a type $A$ singularity is again of type $A$ (\cite[Theorem~1.2.48, Corollary~1.2.49]{GreuelLossenShustin}), and therefore there is a derived open subscheme 
\begin{align}\label{intro:BA}
\rB_{\GL_r}^A \subset \rB_{\GL_r}
\end{align}
corresponding to spectral curves
with at worst type $A$ singularities. 
For $G=\SL_r$ or $\PGL_r$, the Hitchin base
$\rB_{\SL_r}=\rB_{\PGL_r}$ is the traceless part of $\rB_{\GL_r}$, and the open subsets
$\rB_G^{\mathrm{ell}}, \rB_G^A \subset \rB_G$ are defined to be their intersections with $\rB_{\GL_r}^{\mathrm{ell}}, \rB_{\GL_r}^A$. 
A more precise statement of the main theorem is as follows: 
\begin{thm}\emph{(Theorem~\ref{thm:GL}, Theorem~\ref{thm:SL}, Theorem~\ref{thm:PGL})}\label{intro:mainthm}
    For $G=\GL_r$, $\SL_r$ or $\PGL_r$, the Dolbeault geometric Langlands conjecture holds over the open subset 
    \begin{align*}
        \rB_G^{\circ}=\rB_{G}^{\mathrm{ell}} \cup \rB_G^A \subset \rB_G.
    \end{align*}
\end{thm}

Here, for an open subset $\rB_G^{\circ} \subset \rB_G$, we say that the
\textit{Dolbeault geometric Langlands conjecture holds over $\rB_G^{\circ}$}
if there is a $\rB_G^{\circ}$-linear equivalence 
\begin{align}\label{intro:IndL2}
        \IndCoh_{\mathcal{N}}(\Hig_{{}^{L}G}(w)^{\mathrm{ss}}\times_{\rB_G}\rB_G^{\circ})_{-\chi}
        \stackrel{\sim}{\to}
        \IndL_{\mathcal{N}}(\Hig_G(\chi)\times_{\rB_G}\rB_G^{\circ})_w.
\end{align}
Here we explain why it is natural to consider the locus (\ref{intro:BA}):

\begin{itemize}[leftmargin=1.5em]
\item For $G=\GL_2$, we have $\rB_{\GL_2}^A=\rB_{\GL_2}^{\mathrm{red}}$, 
where $\rB_{\GL_r}^{\mathrm{red}}$ corresponds to the locus of reduced spectral curves. Thus $\rB_{\GL_2}^A$ gives 
    a natural class of spectral curves for $\GL_r$ which extends the class of 
    reduced spectral curves considered in~\cite{TodaGL2}.

    \item The locus $\rB_{\GL_r}^A$ contains the generic points of the boundary 
    components corresponding to reducible reduced spectral curves. More precisely, 
    for each partition $r=r_1+\cdots+r_k$, the locus $\rB_{\GL_r}^A$ contains the image 
    of the generic point of the addition map
    \begin{align*}
        \rB_{\GL_{r_1}}\times \cdots \times \rB_{\GL_{r_k}} \to \rB_{\GL_r},
    \end{align*}
    defined by adding spectral curves as divisors on $S$. The image of this 
    generic point corresponds to a spectral curve with $k$ smooth irreducible 
    components intersecting transversely. 

      \item Type $A$ singularities are precisely those for which the 
    zero-dimensional subschemes of $\cC_b$ defined by the conductor ideals are 
    curvilinear. This property is currently used in an essential way to prove 
    compact generation of the limit category under iterated Hecke operators 
    in~\cite{TodaGL2}, and it is also essential in the present paper for 
    constructing explicit resolutions of Arinkin sheaves. 
\end{itemize}

By the first remark above, as a corollary, we obtain 
the following result, which gives a version of the main result of~\cite[Theorem~1.2]{TodaGL2} for $\SL_2/\PGL_2$:
\begin{cor}\label{cor:SL2}
For $G=\SL_2$ or $\PGL_2$, the Dolbeault geometric Langlands conjecture holds over the locus $\rB_G^{\mathrm{red}} \subset \rB_G$.
\end{cor}

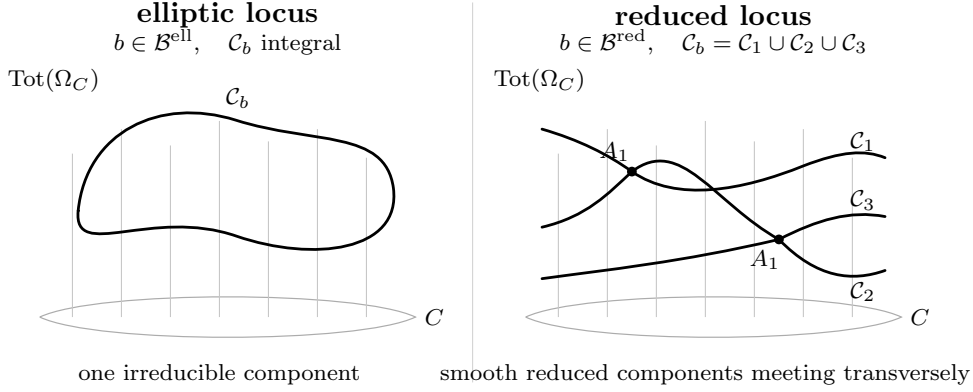
\begin{figure}[t]
\centering
\begin{tikzpicture}[
  x=1cm,y=1cm,
  scale=1.08,
  transform shape,
  font=\small,
  fiber/.style={gray!45, thin},
  base/.style={gray!65, line width=0.45pt},
  curve/.style={black, line width=1.05pt, line cap=round, line join=round},
  sing/.style={circle, fill=black, inner sep=1.2pt},
  lab/.style={font=\scriptsize, inner sep=1.2pt},
  title/.style={font=\small\bfseries, align=center},
  subtitle/.style={font=\scriptsize, align=center}
]

\begin{scope}[shift={(0.25,0)}]
  \node[title] at (2.45,3.95) {elliptic locus};
  \node[subtitle] at (2.45,3.62)
    {$b\in\mathcal B^{\mathrm{ell}},\quad \mathcal C_b\ \mathrm{integral}$};

  \draw[base]
    (0.15,0.25) .. controls (0.85,-0.08) and (4.05,-0.08) .. (4.75,0.25)
    .. controls (4.05,0.58) and (0.85,0.58) .. (0.15,0.25);
  \node[lab] at (4.98,0.25) {$C$};
  \node[lab] at (0.32,3.15) {$\mathrm{Tot}(\Omega_C)$};

  \foreach \x/\h in {0.55/2.25,1.15/2.55,1.75/2.35,2.35/2.65,2.95/2.40,3.55/2.55,4.15/2.20}
    \draw[fiber] (\x,0.25) -- (\x,\h);

  \draw[curve]
    (0.62,1.55)
    .. controls (0.72,2.52) and (1.62,2.95) .. (2.55,2.66)
    .. controls (3.50,2.36) and (4.38,2.55) .. (4.48,1.83)
    .. controls (4.58,1.05) and (3.50,0.92) .. (2.56,1.24)
    .. controls (1.52,1.60) and (0.56,0.92) .. (0.62,1.55);

  \node[lab] at (2.58,2.93) {$\mathcal C_b$};
  \node[lab, align=center] at (2.35,-0.45) {one irreducible component};
\end{scope}

\draw[gray!35] (5.70,-0.55) -- (5.70,4.10);

\begin{scope}[shift={(6.20,0)}]
  \node[title] at (2.45,3.95) {reduced locus};
  \node[subtitle] at (2.45,3.62)
    {$b\in\mathcal B^{\mathrm{red}},\quad
      \mathcal C_b=\mathcal C_1\cup\mathcal C_2\cup\mathcal C_3$};

  \draw[base]
    (0.15,0.25) .. controls (0.85,-0.08) and (4.05,-0.08) .. (4.75,0.25)
    .. controls (4.05,0.58) and (0.85,0.58) .. (0.15,0.25);
  \node[lab] at (4.98,0.25) {$C$};
  \node[lab] at (0.32,3.15) {$\mathrm{Tot}(\Omega_C)$};

  \foreach \x/\h in {0.55/2.25,1.15/2.55,1.75/2.35,2.35/2.65,2.95/2.40,3.55/2.55,4.15/2.20}
    \draw[fiber] (\x,0.25) -- (\x,\h);

  \coordinate (p12) at (1.45,2.03);
  \coordinate (p23) at (3.25,1.20);

  \draw[curve]
    (0.35,2.55)
    .. controls (0.80,2.42) and (1.16,2.24) .. (p12)
    .. controls (2.00,1.65) and (2.80,1.80) .. (3.55,2.10)
    .. controls (4.02,2.28) and (4.30,2.30) .. (4.55,2.20);

  \draw[curve]
    (0.35,1.35)
    .. controls (0.90,1.48) and (1.16,1.76) .. (p12)
    .. controls (2.05,2.48) and (2.45,1.65) .. (p23)
    .. controls (3.72,0.72) and (4.15,0.68) .. (4.55,0.82);

  \draw[curve]
    (0.35,0.72)
    .. controls (1.15,0.83) and (2.05,0.92) .. (p23)
    .. controls (3.72,1.44) and (4.12,1.56) .. (4.55,1.48);

  \node[sing] at (p12) {};
  \node[lab] at (1.24,2.27) {$A_1$};

  \node[sing] at (p23) {};
  \node[lab] at (3.08,0.95) {$A_1$};

  \node[lab] at (4.28,2.40) {$\mathcal C_1$};
  \node[lab] at (4.28,0.56) {$\mathcal C_2$};
  \node[lab] at (4.28,1.68) {$\mathcal C_3$};

  \node[lab, align=center] at (2.35,-0.45)
    {smooth reduced components meeting transversely};
\end{scope}

\end{tikzpicture}
\caption{A schematic comparison between the elliptic locus and the reduced locus.
On the elliptic locus, the spectral curve $\mathcal C_b$ is integral. On the reduced locus,
the spectral curve may be reducible; the right-hand picture illustrates a reduced spectral
curve whose smooth irreducible components meet transversely.}
\label{fig:elliptic-reduced-loci}
\end{figure}

\subsection{The key ingredients of this paper}
The main technical contribution of this paper is to prove the
\textit{Whittaker normalization}, formulated as follows.

Let $G=\GL_r$, and let
\[
    \cB \subset (\rB_{\GL_r}^{\mathrm{red}})^{\mathrm{cl}}
\]
be an open subset, where $(-)^{\mathrm{cl}}$ denotes the classical truncation.
We set
\begin{align*}
    \cH(\chi):=\Hig_{\GL_r}(\chi)\times_{\rB_{\GL_r}} \cB.
\end{align*}
There is a Cohen--Macaulay extension of the Poincar\'e line bundle, constructed by Arinkin~\cite{Ardual}
over the elliptic locus, and further extended by Li~\cite{MLi} to the locus
$(\rB_{\GL_r}^{\mathrm{red}})^{\mathrm{cl}}$,
\begin{align}\label{intro:Arsheaf}
    \cP \in \Coh(\cH(w)\times_{\cB}\cH(\chi)).
\end{align}
This object induces the Fourier--Mukai functor
\begin{align}\label{intro:Phi}
    \Phi \colon \Coh(\cH(w)^{\mathrm{ss}})_{-\chi'} \to
    \Coh(\cH(\chi))_{w'}.
\end{align}
Here $\chi'$ and $w'$ are given by
\begin{align*}
    \chi' := \chi+(r^2-r)(g-1), \qquad
    w' := w+(r^2-r)(g-1).
\end{align*}
Here we note that these shifts of weights are harmless for the formulation of the Dolbeault geometric Langlands conjecture, by the periodicity of the relevant dg-categories, see Remark~\ref{rmk:thm1}.

On the other hand, let $s$ be the Hitchin section
\begin{align*}
    s\colon \cB \to \cH(\chi_0), \ \chi_0:=(r-r^2)(g-1)
\end{align*}
given by the structure sheaf of the universal spectral curve.
The theory of limit categories allows us to construct the
\textit{left adjoint} of the functor $s^*$,
\begin{align*}
    s_{!} \colon \Coh(\cB) \to \LL_{\mathcal{N}}(\cH(\chi_0))_{w'}.
\end{align*}
The following conjecture was formulated in~\cite[Conjecture~4.4]{TodaGL2}.

\begin{conj}\emph{(\cite[Conjecture~4.4]{TodaGL2}, Conjecture~\ref{conj:Whit})}\label{intro:conj:whit}
There is an isomorphism
\begin{align}\label{intro:whit:isom}
    \Phi(\cO_{\cH(w)^{\mathrm{ss}}}) \cong s_{!}\cO_{\cB}.
\end{align}
\end{conj}

Such an isomorphism, if it exists, is a key property which connects the
Cohen--Macaulay extension of the Poincar\'e sheaf by Arinkin with the
weight condition in the definition of the limit category. This is the most
difficult part in proving the Dolbeault geometric Langlands conjecture
formulated in~\cite{PTlim}; a priori, a relation of 
Cohen--Macaulay extension and the bounds on $\mathbb{G}_m$-weights is not clear and still mysterious.

Moreover, it is proved in~\cite[Theorem~4.29]{TodaGL2} that
Conjecture~\ref{intro:conj:whit}, together with the generation of the limit
category by applying iterated Hecke operators to $s_!\cO_{\cB}$, yields the Dolbeault geometric
Langlands equivalence over $\cB$. 
Over the type $A$-locus $\cB^A:=(\rB_{\GL_r}^A)^{\mathrm{cl}}$, the above compact 
generation statement is also proved in~\cite[Corollary~4.28]{TodaGL2}. 
The Whittaker normalization in
Conjecture~\ref{intro:conj:whit} is the remaining problem, and we prove it
in this paper.

\begin{thm}\emph{(Theorem~\ref{thm:whit})}\label{thm:Whit:intro}
The Whittaker normalization Conjecture~\ref{intro:conj:whit} is true over
$\cB=\cB^A$.
\end{thm}

The argument proving Theorem~\ref{thm:Whit:intro} follows the idea of~\cite[Section~5]{TodaGL2},
but the required weight estimates are more complicated for general
$G=\GL_r$. The main contribution of this paper is to settle this problem.
An outline of the proof is as follows
(which is summarized in Figure~\ref{fig:flow}): 

\begin{enumerate}[leftmargin=1.5em]
\renewcommand{\labelenumi}{(\arabic{enumi})}
\item We first prove that, for any semistable Higgs bundle $E$ whose
spectral curve $\cC_b$ has at worst type $A$ singularities, the associated
object $\cP_E \in \Coh(\cH_b)$ 
for $\cH_b=\cH\times_{\cB}\{b\}$
admits a left resolution by an explicit vector bundle $\mathcal{V}_E$
(Proposition~\ref{prop:resol}). We then estimate the weights of
$\mathcal{V}_E$ with respect to maps $\bgm \to \cH$
(Proposition~\ref{prop:VE}). Combining this weight estimate with the
left resolution by $\mathcal{V}_E$, we conclude that $\cP_E$ lies in the
limit category locally on $\cH$ (Corollary~\ref{cor:PE}).

\item Using the result of (1), we show that
(Theorem~\ref{thm:PhiO})
\begin{align*}
    \Phi(\cO_{\cH(w)^{\mathrm{ss}}}) \in
    j_{!} \LL(\cH(\chi_0)^{\mathrm{ss}})_{w'}.
\end{align*}
A key idea is to perturb the stability condition after an \'etale base
change of $\cB$, thereby reducing the problem to the case where
$\cH(w)^{\mathrm{ss}}$ is ``almost'' a scheme proper over $\cB$. The
above property then follows by globalizing the pointwise weight
estimates obtained in (1).

\item By the result of (2), we are reduced to showing the isomorphism
\eqref{intro:whit:isom} after restriction to the semistable locus
$\cH(\chi_0)^{\mathrm{ss}}$. The point is that the object
$s_{!}\cO_{\cB}$ is much easier to understand over
$\cH(\chi_0)^{\mathrm{ss}}$: it is ``almost'' the usual $*$-pushforward of
the Hitchin section, up to shift. We then again use the perturbation of
stability after base change, together with the derived equivalences of
fine compactified Jacobians in~\cite{MRVF2}, to conclude
\eqref{intro:whit:isom}.
\end{enumerate}

Finally, we deduce the Dolbeault geometric Langlands equivalence for
$G=\SL_r$ or $\PGL_r$, by descending the functor~\eqref{intro:Phi}
(with the appropriate notion of degree and weight components, modified in the setting of $\SL_r/\PGL_r$)

\begin{align*}
    \Phi_{^{L}G/G} \colon
    \Coh(\cH_{^{L}G}(w)^{\mathrm{ss}})_{-\chi'}
    \to \Coh(\cH_G(\chi))_{w'}
\end{align*}
using the descent of the Arinkin sheaf~\eqref{intro:Arsheaf}. We then
show that, as a consequence of the Dolbeault geometric Langlands
equivalence over $\cB^A$ for $G=\GL_r$, this functor induces an
equivalence
\begin{align*}
    \Phi_{^{L}G/G} \colon
    \Coh_{\mathcal{N}}(\cH_{^{L}G}(w)^{\mathrm{ss}})_{-\chi'}
    \stackrel{\sim}{\to}
    \LL_{\mathcal{N}}(\cH_G(\chi))_{w'}.
\end{align*}

\begin{figure}[!htbp]
\centering
\newdimen\MainW
\MainW=.90\linewidth

\begin{tikzpicture}[
  node distance=7mm,
  scale=0.96,
  every node/.style={transform shape}
]

\tikzset{
  box/.style={
    rounded corners,
    draw,
    align=left,
    inner sep=5pt,
    text width=\MainW,
    minimum height=8mm
  }
}

\node[box] (step1) {
\textbf{Step 1. Resolution of the Arinkin sheaf}
(Proposition~\ref{prop:resol}).

For a rank-one torsion-free sheaf
\(E\in \Coh^{\heartsuit}(\cC_b)\), construct a left resolution on \(\cH_b\):
\[
    \cdots
    \longrightarrow \mathcal{V}_E^{\oplus k_1}
    \longrightarrow \mathcal{V}_E^{\oplus k_0}
    \longrightarrow \cP_E
    \longrightarrow 0,
\]
where \(\mathcal{V}_E\) is an explicitly constructed vector bundle.
};

\node[box, below=of step1] (step2) {
\textbf{Step 2. Weight estimates}
(Proposition~\ref{prop:VE} and Corollary~\ref{cor:PE}).

Compute the \(\mathbb{G}_m\)-weights of \(\mathcal{V}_E\) for maps
\(\nu\colon \bgm\to \cH\). Combining the estimate with the resolution from
Step~1 gives
\[
    \cP_E \in \widetilde{\LL}(\cH)_{w'}.
\]
};

\node[box, below=of step2] (step3) {
\textbf{Step 3. Globalization}
(Theorem~\ref{thm:PhiO}).

Approximate \(\cH(w)^{\mathrm{ss}}\) by schemes proper over \(\cB\).
Globalizing the pointwise estimate from Step~2 yields
\[
    \Phi(\cO_{\cH(w)^{\mathrm{ss}}})
    \in
    j_{!}\LL(\cH(\chi_0)^{\mathrm{ss}})_{w'}.
\]
};

\node[box, below=of step3] (step4) {
\textbf{Step 4. Whittaker normalization}
(Theorem~\ref{thm:Whit:intro}).

It remains to prove, after restriction to the semistable locus, that
\[
    \Phi(\cO_{\cH(w)^{\mathrm{ss}}})|_{\cH(\chi_0)^{\mathrm{ss}}}
    \cong
    s_{!}\cO_{\cB}|_{\cH(\chi_0)^{\mathrm{ss}}}.
\]
This is shown by reducing to schemes and applying the derived equivalence of
compactified Jacobians from~\cite{MRVF2}.
};

\draw[->] (step1) -- (step2);
\draw[->] (step2) -- (step3);
\draw[->] (step3) -- (step4);

\end{tikzpicture}

\caption{Flowchart of the proof of Theorem~\ref{thm:Whit:intro}.}
\label{fig:flow}
\end{figure}

\subsection{Acknowledgements}
The author would like to thank Tudor P{\u a}durariu and Tasuki Kinjo
for useful discussions, and Yongbin Ruan and Yaoxiong Wen for the discussions and comments during
the Mini-Workshop on Topological Geometric Langlands held on
April 25--26, 2026.

The author is supported by the World Premier International Research
Center Initiative (WPI Initiative), MEXT, Japan, the Inamori Research
Institute for Science, and JSPS KAKENHI Grant Number JP24H00180.

 \subsection{Notation and Conventions}\label{subsec:notation0}
    \subsubsection*{Conventions on stacks}
We work over an algebraically closed field $k$ of characteristic $0$.
Unless stated otherwise, all (derived) stacks are locally of finite presentation and quasi-separated over $k$.
For a stack $\X$ over $k$, we write $\X(k)$ for the set of $k$-valued points of $\X$.
We denote by $\X^{\mathrm{cl}}$ its classical truncation. 
A stack is called \textit{QCA} if it is quasi-compact with affine geometric stabilizer groups.

For a morphism of derived stacks $f\colon \X \to \Y$, let $\mathbb{L}_{f}$ and $\mathbb{T}_{f}$ denote the $f$-relative cotangent and tangent complexes. The morphism $f$ is called \textit{quasi-smooth} if $\mathbb{L}_{f}$ is perfect of cohomological amplitude $[-1, 1]$.
For a connected derived stack $\mathcal{X}$ whose cotangent complex $\mathbb{L}_{\mathcal{X}}$ is perfect, 
its virtual dimension $\operatorname{vdim} \mathcal{X}$ is defined to be the rank of $\mathbb{L}_{\mathcal{X}}$, and 
its dimension $\dim \mathcal{X}$ means the dimension of $\mathcal{X}^{\mathrm{cl}}$. 
If $\mathcal{X}$ is quasi-smooth, we have $\mathrm{vdim}\mathcal{X} \leq \dim \mathcal{X}$, 
and equality holds if and only if $\mathcal{X}$ is classical, see~\cite[Proposition~1]{KhanClassicality}. 

All the fiber products are derived fiber products, and we use the notation $\square$ for the 
Cartesian square.

\subsubsection*{Notations for (ind or quasi)coherent sheaves}\label{subsec:notation}
For a dg-category $\cC$, denote by $\cC^{\mathrm{cp}} \subset \cC$ the subcategory of compact objects. The dg-category $\cC$ is called \textit{compactly generated} if $\cC=\Ind(\cC^{\mathrm{cp}})$, where $\Ind(-)$ is the ind-completion. 
A dg-functor $F \colon \cC_1 \to \cC_2$ is called \textit{continuous} 
if it commutes with taking direct sums. 

We use the notation of~\cite{MR3701352} for (quasi)coherent sheaves and (ind)coherent sheaves. 
For a derived stack $\X$, we denote by 
$\Coh(\X)$ the dg-category of coherent sheaves,
by $\QCoh(\X)$ the dg-category of quasi-coherent sheaves,
and by $\IndCoh(\X)$ the dg-category of ind-coherent sheaves. 
The heart of a standard t-structure on $\Coh(\X)$ is denoted by 
$\Coh^{\heartsuit}(\X)$. 

For a morphism $f\colon \X \to \Y$, we use the standard $*$-pull-back and $*$-push-forward
\begin{align*}
    f^* \colon \QCoh(\Y) \to \QCoh(\X), \ f_{*} \colon \QCoh(\X) \to \QCoh(\Y).
\end{align*}
If $f$ is quasi-smooth (resp.\ proper), the above functors restrict to 
\begin{align}\label{intro:funct1}
    f^* \colon \Coh(\Y) \to \Coh(\X), \ f_{*} \colon \Coh(\X) \to \Coh(\Y).
\end{align}
In these cases, we also use the functors 
\begin{align*}
        f^* \colon \IndCoh(\Y) \to \IndCoh(\X), \ f_{*} \colon \IndCoh(\X) \to \IndCoh(\Y)
\end{align*}
which were denoted by $f^{\mathrm{IndCoh}*}$, $f_{*}^{\mathrm{IndCoh}}$ in~\cite{MR3701352}. 

If $f$ is quasi-smooth and proper, the functor $f_{*}$ in (\ref{intro:funct1}) admits a right adjoint $f^!$ (see~\cite[Proposition~7.3.8]{MR3136100})
\begin{align*}
    f^!(-)=f^*(-)\otimes \det \mathbb{L}_f [\dim f]\colon \Coh(\mathcal{Y}) \to \Coh(\mathcal{X}), 
    \end{align*}
and 
$f^*$ in (\ref{intro:funct1}) admits a left adjoint $f_{!}$
\begin{align*}
    f_{!}(-)=f_{*}(- \otimes \det \mathbb{L}_f[\dim f]) \colon \Coh(\mathcal{X}) \to \Coh(\mathcal{Y}).
\end{align*}
For $A\in \Coh(\Y)$, we often use $A|_{\X}$ for the $*$-pull-back $f^*A$.
Sometimes for $A\in \Coh^{\heartsuit}(\Y)$, we may also write $A|_{\X}$ for the \textit{classical} pull-back $\cH^0(f^*A)\in \Coh^{\heartsuit}(\X)$; we will remark so in this case. 

For a closed substack $\zZ \subset \X$, we denote by $\QCoh_{\zZ}(\X)$ 
the subcategory of $\QCoh(\X)$ consisting of objects $A$ such that 
$A|_{\mathcal{X}\setminus \mathcal{Z}}=0$. 
We have the \textit{local cohomology functor}
\begin{align}\label{loc:coh}
    \Gamma_{\mathcal{Z}}(-) \colon \QCoh(\X) \to \QCoh_{\mathcal{Z}}(\X)
\end{align}
which gives a right adjoint of the inclusion $\QCoh_{\mathcal{Z}}(\X) \subset \QCoh(\X)$. 

For derived stacks $\mathcal{X}$ and $\mathcal{Y}$ such that either one of them is QCA, there is an equivalence, 
see~\cite[Corollary~4.3.4]{MR3037900}:
\begin{align}\notag
    \QCoh(\mathcal{X})\otimes \QCoh(\mathcal{Y}) \stackrel{\sim}{\to}\QCoh(\mathcal{X}\times \mathcal{Y}).
\end{align}
We use the notation $\Coh(\X) \otimes \Coh(\Y)$ for the subcategory of the left 
hand side, which corresponds to $\Coh(\X\times \Y)$ 
under the above equivalence.

\subsubsection*{\texorpdfstring{Notations for $\mathbb{G}_m$-weights}{Notation on Gm-weights}}\label{subsec:Gmwt}
For a $\mathbb{G}_m$-gerbe $\X \to X$, there is an orthogonal weight 
decomposition 
\begin{align*}
    \Coh(\X)=\bigoplus_{w\in \mathbb{Z}}\Coh(\X)_w,
\end{align*}
where $\Coh(\X)_w$ is the subcategory of $\mathbb{G}_m$-weights $w$. 
For the Brauer class $\beta$ corresponding to the $\mathbb{G}_m$-gerbe $\X \to X$, 
we can identify $\Coh(\X)_w$ with the dg-category of $\beta^w$-twisted coherent sheaves 
$\Coh(X, \beta^w)$, see~\cite[Section~2.1.3]{MR2309155}.

For an object $A \in \Coh(\bgm)$,
let $A=\oplus_{w\in \mathbb{Z}} A_w$ 
be its weight decomposition. We write $A^{>0}=\oplus_{w>0} A_w$. 
We denote by $\wt(A)\subset \mathbb{Z}$ the set of $w\in \mathbb{Z}$ such that $A_w\neq 0$, 
and $\wt^{\mathrm{max}}(A)\in \mathbb{Z}$ is the maximal element in $\wt(A)$. 
We will use the following map 
\begin{align*}
c_1 \colon K(\bgm)_{\mathbb{Q}} \to H^2(\bgm, \mathbb{Q})=\mathbb{Q}, 
\end{align*}
which is given by $c_1(A)=\sum_w w \cdot \dim A_w$, where $\dim A_w$ denotes the Euler characteristic of the weight-$w$ complex

\subsubsection*{$\Theta$-stack and $\Theta$-stratum}
We use the terminology of $\Theta$-stratifications from~\cite{Halpinstab}. 
Let $\Theta=\mathbb{A}^1/\mathbb{G}_m$ where $\mathbb{G}_m$ acts on $\mathbb{A}^1$ by weight one. 
We have the natural maps 
\begin{align*}
    \bgm \rightleftarrows\Theta \leftarrow \mathrm{pt}.
\end{align*}
For a derived stack $\X$, the above maps induce the maps 
\begin{align*}
    \mathrm{Map}(\bgm, \X) \leftrightarrows \Map(\Theta, \X) \to \X.
\end{align*}
An open and closed substack $\cS \subset \Map(\Theta, \X)$ is called \textit{$\Theta$-stratum} if the induced map 
$\cS \to \X$ is a closed immersion. 
For a $\Theta$-stratum $\cS$, there is an associated 
open and closed substack $\mathcal{Z}\subset \Map(\bgm, \X)$, called 
\textit{center} of $\cS$ such that the above diagram restricts to 
\begin{align*}
    \mathcal{Z} \leftrightarrows \mathcal{S} \hookrightarrow \X. 
\end{align*}

A stratification 
\begin{align*}
    \X=\X' \sqcup \cS_1 \sqcup \cdots \sqcup \cS_N
\end{align*}
is called \textit{$\Theta$-stratification} if each 
\begin{align*}
    \cS_i \subset \X' \sqcup \cS_1 \sqcup \cdots \sqcup \cS_i
\end{align*}
is a $\Theta$-stratum (in particular $\X'\subset \X$ is open). We refer to~\cite{Halpinstab} for more details.

\section{Preliminaries}
In this section, we recall Higgs bundles, their moduli stacks, limit categories, 
and the formulation of Dolbeault geometric Langlands correspondence in~\cite{PTlim}. 
We also recall Arinkin's construction of Cohen--Macaulay extension of the Poincar\'e line bundle, 
and the associated Fourier--Mukai functor. We also refer to~\cite[Section~2]{TodaGL2} for the content 
of this section.

\subsection{Moduli stacks of Higgs bundles}\label{subsec:higgs}
Let $C$ be a smooth projective curve over $k$ of genus $g\geq 2$.
Let $G$ be a reductive group with Lie algebra $\mathfrak{g}$, and $^{L}G$ the Langlands dual group
of $G$. 
A \textit{$G$-Higgs bundle} consists of 
\begin{align*}
    (\mathcal{E}, \theta), \ \mathcal{E} \to C, \ \theta \in \Gamma((\mathcal{E}\times^G \mathfrak{g})\otimes \Omega_C)
\end{align*}
where $\mathcal{E} \to C$ is a $G$-bundle and $G$ acts on $\mathfrak g$ via the adjoint representation.
We denote by 
\begin{align*}
    \Hig_{G}=\coprod_{\chi\in \pi_1(G)} \Hig_{G}(\chi)
\end{align*}
the derived moduli stack of $G$-Higgs bundles, 
and each $\Hig_{G}(\chi)$ is the connected component 
corresponding to $\chi \in \pi_1(G)$. 
It is a zero-shifted symplectic stack~\cite{PTVV}; indeed we have 
\begin{align*}
    \Hig_G=\Omega_{\Bun_G}
\end{align*}
where $\Bun_G$ is the moduli stack of $G$-bundles, and $\Omega_{\Bun_G}$ denotes the cotangent stack of $\Bun_G$.

In the case of $G=\GL_r$, giving a $\GL_r$-Higgs bundle is 
equivalent to giving 
\begin{align}\label{def:higgs}
    (F, \theta), \ \theta \colon F\to F\otimes \Omega_C
\end{align}
where $F \to C$ is a rank $r$ vector bundle. 
In this case, $\pi_1(\GL_r)=\mathbb{Z}$ and $\Hig_{\GL_r}(\chi)$ 
corresponds to (\ref{def:higgs}) such that $\deg F=\chi$. 
We denote by 
\begin{align}\label{univ:F}
    (\cF, \vartheta), \ \cF \to C \times \Hig_{\GL_r}
\end{align}
the universal Higgs bundle. 

We have the (derived) Hitchin map 
\begin{align*}
   h \colon \Hig_{\GL_r} \to \mathrm{B}_{\GL_r}=\bigoplus_{i=1}^r \Gamma(\Omega_C^i)=\mathrm{B}_{\GL_r}^{\mathrm{cl}}\times k[-1].
\end{align*}
It sends $(F, \theta)$ to $(\tr(\wedge^i \theta))_{1\leq i\leq r}$.
Here 
$k[-1]:=\Spec k[\epsilon]$ with $\deg \epsilon=-1$, and the above 
$k[-1]$-factor comes from $H^1(\Omega_C)=k$. 
We have the following structure result of Hitchin map, 
which is mentioned in~\cite[Theorem~2.2.4]{BD0}, also see~\cite[Theorem~2.1]{TodaGL2}: 
\begin{thm}\emph{(\cite{Hitchin, Faltings1993, Ginzburg2001})}\label{thm:higgs:basic}
We have the Cartesian square
\begin{align}\label{dia:HigB}
    \xymatrix{
\Hig_{\GL_r}^{\mathrm{cl}} \inclusion \ar[d] \diasquare & \Hig_{\GL_r} \ar[d] \\
\rB_{\GL_r}^{\mathrm{cl}} \inclusion & \rB_{\GL_r}.
    }
\end{align}
Moreover the left vertical arrow is flat, surjective and an lci morphism. We have
\[
\dim \Hig_{\GL_r}^{\mathrm{cl}}=(2g-2)r^2+1,\qquad
\dim \rB_{\GL_r}^{\mathrm{cl}}=(g-1)r^2+1 .
\]
\end{thm}

Let $S$ be the total space of $\Omega_C$
\begin{align*}
    S:=\mathrm{Tot}_C(\Omega_C) \stackrel{\pi}{\to} C.
\end{align*}
Each point $b=(b_i)_{1\leq i\leq r}\in \rB_{\GL_r}^{\mathrm{cl}}$ corresponds to the 
spectral curve (by setting $b_0=1$)
\begin{align}\label{eqn:Cb}
    \cC_b=\left\{\sum_{i=0}^r (-1)^{i} b_{i} \cdot y^{r-i}=0\right\} \subset S
\end{align}
where $y$ is the tautological section of $\pi^* \Omega_C$. 
Its arithmetic genus is given by 
\begin{align}\label{def:pa}
    p_a=(g-1)r^2+1.
\end{align}
By the spectral construction~\cite{BeNaRa}, 
giving a $\GL_r$-Higgs bundle is equivalent to giving a 
compactly supported one-dimensional sheaf on $S$, whose support is given 
by the spectral curve. 
For $b\in \rB_{\GL_r}$, 
the Hitchin fiber 
$h^{-1}(b)$ is identified with the moduli stack of torsion-free 
sheaves on $\cC_b$
with fundamental one-cycle $[\cC_b]$. 

There is a natural action of $\gamma \in H^0(\Omega_C)$ on $\rB_{\GL_r}^{\mathrm{cl}}$, given by the variable change $y\mapsto y+\gamma$ 
in the equation (\ref{eqn:Cb}), and take the coefficients. Explicitly, 
\begin{align}\label{act:gamma}
    (b_i)_{1\leq i\leq r} \mapsto 
     (b_i^{\gamma})_{1\leq i\leq r}, \ 
b_i^\gamma
=
\sum_{j=0}^i
(-1)^{i+j}
\binom{r-j}{i-j}
b_j \gamma^{i-j}.
\end{align}
The above action of $H^0(\Omega_C)$ on $\rB_{\GL_r}^{\mathrm{cl}}$ is a free action. 

A $\GL_r$-Higgs bundle (\ref{def:higgs}) is called \textit{regular} at $p\in C$ if the map 
\begin{align*}\theta|_{p} \colon F|_{p} \to F|_{p} \otimes \Omega_C|_{p} \cong F|_{p}
\end{align*}
is a regular 
matrix. It is \textit{regular} (resp.~\textit{generically regular}) if $\theta|_{p}$ is regular for any 
$p\in C$ (resp.\ for generic $p\in C$). 
We denote by 
\begin{align*}
    \Hig_{\GL_r}(\chi)^{\mathrm{reg}} 
    \subset 
    \Hig_{\GL_r}(\chi)^{\mathrm{greg}} \subset \Hig_{\GL_r}(\chi)
\end{align*}
the open substack of regular (resp.\ generically regular) Higgs bundles. 
Under the spectral correspondence, regular Higgs bundles correspond to 
line bundles on $\cC_b$, and generically regular Higgs bundles 
correspond to rank one torsion-free sheaves on $\cC_b$. 

We denote by $\pi$ the projection 
\begin{align}\label{project:pi}
    \pi \colon S=\mathrm{Tot}_C(\Omega_C) \to C.
\end{align}
For a rank-one sheaf $E\in \Coh^{\heartsuit}(\cC_b)$, we define its degree on $\cC_b$ 
to be 
\begin{align}\label{deg:Cb}
    \deg_{\cC_b}(E):=\chi(E)-\chi(\mathcal{O}_{\cC_b}).
\end{align}
Then we have 
\begin{align}\label{formula:deg}
\deg_{\cC_b}(E)=\deg \pi_{*}E+(r^2-r)(g-1).
\end{align}
Later we will use the following lemma: 
\begin{lemma}\label{lem:add}
Let $\cC_b=C_1 \cup C_2$ where $C_i$ are unions of irreducible components 
without common components (which is automatic if $\cC_b$ is reduced). For a line bundle $\cL$ on $\cC_b$, we have 
\begin{align*}
    \deg_{\cC_b}(\mathcal{L})=\deg_{C_1}(\cL|_{C_1})+\deg_{C_2}(\cL|_{C_2}).
\end{align*}
\end{lemma}
\begin{proof}
    Let $f \colon C_1 \sqcup C_2 \to C$. The lemma follows from 
    the exact sequence 
    \begin{align*}
        0\to \cL \to f_{*}f^* \cL \to Q \to 0
    \end{align*}
    where $Q$ is supported on $C_1 \cap C_2$,
    and it is zero-dimensional as $C_1$ and $C_2$ do not have common 
    irreducible components. Then the lemma follows 
    noting that $\chi(Q)$ is 
    independent of $\cL$. 
\end{proof}
\subsection{Moduli stacks of semistable Higgs bundles}
A $\GL_r$-Higgs bundle $(F, \theta)$ is called \textit{(semi)stable} if 
for any non-zero proper sub-Higgs bundle $(F', \theta') \subset (F, \theta)$ 
we have 
\begin{align}\label{ineq:deg}
    \frac{\deg F'}{\rank F'} <(\leq) \frac{\deg F}{\rank F}. 
\end{align}
We denote by 
\begin{align*}
\Hig_{\GL_r}(\chi)^{\mathrm{st}} \subset 
    \Hig_{\GL_r}(\chi)^{\mathrm{ss}} \subset \Hig_{\GL_r}(\chi)
\end{align*}
the open substack of (semi)stable Higgs bundles. These substacks of (semi)stable Higgs bundles are 
Artin stacks of finite type, in particular they are 
quasi-compact. 

\begin{remark}\label{rmk:ss:spec}
By the spectral construction, 
the (semi)stable Higgs bundles correspond to compactly supported 
pure one-dimensional coherent sheaves on $S$
which are Gieseker (semi)stable~\cite{Hu} 
with respect 
to $\pi^* h$ where $h$ is an ample divisor on $C$. 
We will often use the following characterizations: 
a pure one-dimensional sheaf $E$ on $S$ corresponds to (semi)stable 
Higgs bundle if and only if for any surjection $E\twoheadrightarrow E'$
for a pure one-dimensional sheaf $E'$, we have
\begin{align}\label{ineq:surj}
    \frac{\deg \pi_{*}E}{\operatorname{rank}\pi_{*}E}<(\leq) 
     \frac{\deg \pi_{*}E'}{\operatorname{rank}\pi_{*}E'}. 
\end{align}

Below we often say that $E\in \Coh^{\heartsuit}(\cC_b)$ is (semi)stable if it corresponds to (semi)stable Higgs bundles; equivalently if it satisfies the above inequality (\ref{ineq:surj}) for any surjection $E\twoheadrightarrow E'$ as above. 
\end{remark}

\subsection{Open locus of Hitchin base}
There are derived open subschemes 
\begin{align*}
    \mathrm{B}_{\GL_r}^{\mathrm{sm}} \subsetneq \mathrm{B}_{\GL_r}^{\mathrm{ell}} \subsetneq \mathrm{B}_{\GL_r}^{\mathrm{red}} \subsetneq \mathrm{B}_{\GL_r}
\end{align*}
where $\mathrm{B}_{\GL_r}^{\mathrm{red}}$ (resp.\ $\mathrm{B}_{\GL_r}^{\mathrm{ell}}, \mathrm{B}_{\GL_r}^{\mathrm{sm}}$) corresponds to reduced (resp.\ integral, smooth) spectral curves. 
The stack $\Hig_{\GL_r}(\chi)$ is quasi-compact over $\mathrm{B}_{\GL_r}^{\mathrm{ell}}$, but over 
$\mathrm{B}_{\GL_r}^{\mathrm{red}}$ it is not quasi-compact. We will note the following: 
\begin{lemma}\label{lem:dense}\emph{(\cite[Remark~2.5]{TodaGL2}, \cite[Fact~2.4]{MRVF2})}
    For $b\in \rB_{\GL_r}^{\mathrm{red}}$, the Hitchin fiber 
    $h^{-1}(b)$ contains dense regular locus. 
    In particular, the complement of $h^{-1}(b) \cap \Hig_{\GL_r}^{\mathrm{reg}}$ in $h^{-1}(b)$ is at least of codimension one. 
\end{lemma}

We also denote by 
\begin{align*}
\rB_{\GL_r}^A \subset \rB_{\GL_r}^{\mathrm{red}}
\end{align*}
the open subset corresponding to spectral curves $\cC_b$ with at worst type $A$ singularities; for any singular point $p\in \cC_b$ with $c=\pi(p) \in C$, 
we have an isomorphism 
\begin{align}\label{loc:model}
(\widehat{(\cC_b)}_{p} \to \widehat{C}_c) \cong 
    (\Spec k[[x, y]]/(y^2-x^m) \to \Spec k[[x]])
\end{align}
for some $m\in \mathbb{Z}_{\geq 2}$.

\begin{example}\label{exam:Atype}
(1) For $G=\GL_2$, we have $\rB_{\GL_2}^A=\rB_{\GL_2}^{\mathrm{red}}$. 
Indeed, an equation of the spectral curve for $G=\GL_2$ is 
\begin{align*}
    y^2 -b_1 y +b_2=0
\end{align*}
for $b_i \in H^0(C, \Omega_C^i)$. It is reduced if and only if 
$b_1^2 -4b_2 \neq 0$ in $H^0(C, \Omega_C^2)$.
Therefore, after the change of variables $y\mapsto y-b_1/2$, the 
above equation is formally locally of the form $y^2-x^m=0$ for some $m$. 

(2)  
For each decomposition $r=r_1+\cdots+r_k$, we have the addition map
\begin{align*}
    \rB_{\GL_{r_1}} \times \cdots \times \rB_{\GL_{r_k}}
    \to \rB_{\GL_r}
\end{align*}
given by adding spectral curves as effective divisors. 
Then $\rB_{\GL_r}^{A}$ contains the image of the generic point of the above map. 
Indeed, for a general choice of spectral curves 
$\cC_1, \ldots, \cC_k$, where $\cC_i$ corresponds to 
$b_i \in \rB_{\GL_{r_i}}$, the curves $\cC_i$ are smooth, and 
$\cC_i$ and $\cC_j$ intersect transversely at points which are not ramification points 
of the projections $\cC_i \to C$ and $\cC_j \to C$. 
Then the $\GL_r$-spectral curve $\cC_1+\cdots+\cC_k$ corresponds to a point in 
$\rB_{\GL_r}^{A}$. 

(3) In particular, for a generic choice of 
$\alpha_i \in H^0(C, \Omega_C)$ with $1 \leq i\leq r$, the spectral curve 
\begin{align*}
    \prod_{i=1}^r (y-\alpha_i)=0
\end{align*}
corresponds to a point in $\rB_{\GL_r}^{A}$; namely,
\begin{align*}
    b=(b_i)_{1\leq i\leq r} \in \rB_{\GL_r}^{A}, \qquad
    b_i=\sum_{1\leq k_1<\cdots<k_i\leq r}
    \alpha_{k_1} \cdots \alpha_{k_i}.
\end{align*}
\end{example}

\subsection{Limit categories}
Here we recall limit categories defined in~\cite{PTlim}, whose construction is inspired by and based on 
the works~\cite{SvdB, hls}. 
Let $\mathfrak{M}$ be a quasi-smooth derived stack which is of finite presentation over $k$, and 
its cotangent complex is self-dual
\begin{align}\label{L:cotangent}
    \mathbb{L}_{\mathfrak{M}} \simeq \mathbb{T}_{\mathfrak{M}}.
\end{align}
For $x \in \mathfrak{M}(k)$, 
we set 
\begin{align*}
T_x:=
\cH^{0}(\mathbb{T}_{\mathfrak{M}}|_{x}), \ 
\mathfrak{g}_x:=
\cH^{-1}(\mathbb{T}_{\mathfrak{M}}|_{x}).
\end{align*}
Note that $\mathfrak{g}_x$ is the Lie algebra of 
$G_x:=\mathrm{Aut}(x)$. 

Given a map 
\begin{align}\notag\nu \colon \bgm\to \mathfrak{M}, \ 
\mathrm{pt} \mapsto x
\end{align}
we have the corresponding cocharacter 
$\mathbb{G}_m \to G_x$. 
We let $\mathbb{G}_m$ act on $\mathfrak{g}_x, T_x$
through the above cocharacter, and regard them as elements of $K(\bgm)$. 
The map $\nu$ may not be quasi-smooth, but it extends to a quasi-smooth map
(unique up to 2-isomorphisms)
\begin{align}\label{nu:reg}
\nu^{\mathrm{reg}} \colon \mathfrak{g}_x^{\vee}[-1]/\mathbb{G}_m \to \mathfrak{M}
\end{align}
called \textit{regularization}, see~\cite[Construction~3.2.2]{HalpK32}, 
\cite[Construction-Definition~3.1]{PTlim}.
It satisfies that the map 
\begin{align*}
    \mathbb{L}_{\mathfrak{M}}|_{x} \to 
    \mathbb{L}_{\mathfrak{g}_x^{\vee}[-1]/\mathbb{G}_m}|_{0}=\mathfrak{g}_x[1]
\end{align*}
induces an isomorphism on $\cH^{-1}(-)$. 
We also denote by $\iota$ the natural 
morphism 
\begin{align*}
    \iota \colon \mathfrak{g}_x^{\vee}[-1]/\mathbb{G}_m \to 
    \bgm.
\end{align*}

Let $\delta \in \mathrm{Pic}(\mathfrak{M})_{\mathbb{R}}$. 
Recall the notations about $\mathbb{G}_m$-weights in Subsection~\ref{subsec:Gmwt}.
\begin{defn}\label{def:Lcat}(\cite[Definition~3.3]{PTlim})
When $\mathfrak{M}$ is QCA, the subcategory 
\begin{align}\label{def:L}
    \mathrm{L}(\mathfrak{M})_{\delta} \subset \Coh(\mathfrak{M})
\end{align}
is defined to consist of objects $\cE$ 
such that, for any map
$\nu \colon \bgm \to \mathfrak{M}$ 
with $\nu(0)=x$ and regularization (\ref{nu:reg}),
we have 
\begin{align}\label{wt:nu}
\wt(\iota_{\ast}\nu^{\mathrm{reg}\ast}(\cE))
\subset \left[\frac{1}{2} c_1 (T_x^{<0}),
\frac{1}{2} c_1 (T_x^{>0}) \right]+
\frac{1}{2}c_1(\mathfrak{g}_x)+c_1(\nu^*\delta).
\end{align}
\end{defn}

For perfect complexes, we have the following lemma
\begin{lemma}\emph{(\cite[Lemma~6.1]{TodaGL2})}\label{lem:perfect}
In the setting of Definition~\ref{def:Lcat}, suppose that
$\cE \in \Coh(\mathfrak{M})$ is perfect. Then for a map
$\nu \colon \bgm \to \mathfrak{M}$ with image $x\in \mathfrak{M}(k)$, 
$\mathrm{wt}(\nu^*\cE)$ is contained in a bounded interval $I \subset \mathbb{R}$ 
if and only if 
$\mathrm{wt}(\iota_{*}\nu^{\mathrm{reg}*}\cE)$ is contained in 
$I+[c_1(\mathfrak{g}_x^{<0}), c_1(\mathfrak{g}_x^{>0})]$. 

In particular, $\iota_{*}\nu^{\mathrm{reg}*}\cE$ satisfies condition~(\ref{wt:cond2}) if and only if 
$\nu^* \cE$ satisfies the following condition:
\begin{align}\label{wt:cond2}
    \wt(\nu^* \cE) \subset \left[\frac{1}{2} c_1 (\nu^{\ast}\mathbb{L}_{\mathfrak{M}}^{<0}), 
    \frac{1}{2} c_1 (\nu^{\ast}\mathbb{L}_{\mathfrak{M}}^{>0})
    \right]+c_1(\nu^{\ast}\delta).
\end{align}
\end{lemma}
\begin{remark}\label{rmk:cond:r}
The conditions (\ref{wt:nu}), (\ref{wt:cond2}) are unchanged by the composition $\nu$ with 
$\bgm \to \bgm, t\mapsto t^r$ for $r>0$. Indeed this replacement just multiplies both 
sides in (\ref{wt:nu}), (\ref{wt:cond2}) by $r$. 
\end{remark}

For non-quasi-compact case, the limit category is defined as follows: 
\begin{defn}(\cite[Section~3.7]{PTlim})\label{def:IndL}
For a general (not necessarily quasi-compact) quasi-smooth derived stack 
$\mathfrak{M}$ satisfying (\ref{L:cotangent}), 
we define 
\begin{align*}
    \IndL(\mathfrak{M})_{\delta}:=\lim_{\mathcal{U}\subset \mathfrak{M}} \Ind(\LL(\mathcal{U})_{\delta})
\end{align*}
where the limit is over all the QCA open substacks $\mathcal{U}\subset \mathfrak{M}$,
and with respect to pull-backs of open immersions. 
We define 
\begin{align*}
    \LL(\mathfrak{M})_{\delta} \subset \IndL(\mathfrak{M})_{\delta}
\end{align*}
to be the subcategory of compact objects. 
For a smooth stack $\mathcal{X}$, we 
call $\LL(\Omega_{\X})_{\delta}$ the \textit{limit category}.
\end{defn}

Let $\Omega_{\mathfrak{M}}[-1]$ be the $(-1)$-shifted cotangent of $\mathfrak{M}$, and 
let 
\begin{align}\label{nilp:cone}
    \mathcal{N} \subset \Omega_{\mathfrak{M}}[-1]
\end{align}
be the nilpotent cone; a fiber of $\Omega_{\mathfrak{M}}[-1] \to \mathfrak{M}$
at $x \in \mathfrak{M}(k)$ is the Lie algebra $\mathfrak{g}_x$ of $\Aut(x)$ by the condition (\ref{L:cotangent}),
and the fiber of $\mathcal{N} \to \mathfrak{M}$ at $x$ consists of nilpotent elements of $\mathfrak{g}_x$. 
We also have the subcategories with nilpotent singular supports~\cite{AG}:
\begin{align*}
    \LL_{\mathcal{N}}(\mathfrak{M})_{\delta} \subset \LL(\mathfrak{M})_{\delta}, \ 
    \IndL_{\mathcal{N}}(\mathfrak{M})_{\delta} \subset \IndL(\mathfrak{M})_{\delta}.
\end{align*}

\begin{defn}(\cite[Remark~3.18]{PTlim})\label{def:Ltilde}
We define the following dg-categories 
\begin{align}\label{def:Ltilde0}
    \widetilde{\LL}_{\mathcal{N}}(\mathfrak{M})_{\delta}
    :=\lim_{\mathcal{U}\subset \mathfrak{M}}\LL_{\mathcal{N}}(\mathcal{U})_{\delta}, 
    \  \widetilde{\LL}(\mathfrak{M})_{\delta}
    :=\lim_{\mathcal{U}\subset \mathfrak{M}}\LL(\mathcal{U})_{\delta}.
\end{align}
\end{defn}

In general, we have 
\begin{align}\label{Ltilde:chain}
    \LL_{(\mathcal{N})}(\mathfrak{M})_{\delta} \subset \widetilde{\LL}_{(\mathcal{N})}(\mathfrak{M})_{\delta} \subset \IndL_{(\mathcal{N})}(\mathfrak{M})_{\delta}
\end{align}
and the dg-categories (\ref{def:Ltilde0}) correspond to locally compact objects. 
In general 
they are strictly larger than the subcategories of compact objects, see~\cite[Remark~3.18]{PTlim}.

\subsection{Some variants of limit categories}\label{subsec:variant}
We also use a variant of limit categories of stacks without self-dual 
cotangent complex. 
Suppose that a derived stack $\mathfrak{M}$ with self-dual 
$\mathbb{L}_{\mathfrak{M}}$ is equivalent to  
\begin{align*}
    \mathfrak{M} \simeq \mathfrak{M}_{\circ}\times Y \times k[-1]
    \end{align*}
    for a smooth scheme $Y$. 
Then for any map $\nu \colon \bgm \to \mathfrak{M}_{\circ}$
with image $x\in \mathfrak{M}_{\circ}(k)$, 
we have the decomposition $\mathfrak{g}_x=(\mathfrak{g}_x)_0 \oplus k$, and 
we can take the regularization map of the form 
\begin{align*}
    \nu_{\circ}^{\mathrm{reg}} \colon (\mathfrak{g}_x)_0^{\vee}[-1]/\mathbb{G}_m \to 
    \mathfrak{M}_{\circ}. 
\end{align*}
Then if $\mathfrak{M}_{\circ}$ is QCA, for $\delta \in \Pic(\mathfrak{M}_{\circ})_{\mathbb{Q}}$, 
the subcategory 
\begin{align*}
    \LL(\mathfrak{M}_{\circ})_{\delta} \subset 
     \Coh(\mathfrak{M}_{\circ})_{\delta}
\end{align*}
is defined as in Definition~\ref{def:Lcat}, using the above regularization map 
and 
\begin{align}\notag
    T_x=\mathcal{H}^{0}(\mathbb{T}_{\mathfrak{M}_{\circ}|_{x}}), \ 
    \mathfrak{g}_x=\mathcal{H}^{-1}(\mathbb{T}_{\mathfrak{M}_{\circ}}|_{x}).
\end{align}
Note that these are the same as those defined for $\mathfrak{M}$ at $(x, y)$ for any $y\in Y$. 
Namely an object $\mathcal{E} \in     \Coh(\mathfrak{M}_{\circ})_{\delta}$
lies in $\LL_{(\mathcal{N})}(\mathfrak{M}_{\circ})_{\delta}$
if for any map $\nu$ as above, we have 
\begin{align}\label{cond:regcirc}
\wt(\iota_{\ast}\nu_{\circ}^{\mathrm{reg}\ast}(\cE))
\subset \left[\frac{1}{2} c_1 (T_x^{<0}),
\frac{1}{2} c_1 (T_x^{>0}) \right]+
\frac{1}{2}c_1(\mathfrak{g}_x)+c_1(\nu^*\delta).
\end{align}

In general, we define 
\begin{align*}
    \IndL(\mathfrak{M}_{\circ})_{\delta}=\lim_{\mathcal{U} \subset \mathfrak{M}_{\circ}}\Ind(\LL(\mathcal{U})_{\delta})
\end{align*}
where the limit is over QCA open substacks $\mathcal{U} \subset \mathfrak{M}_{\circ}$ as in 
   Definition~\ref{def:IndL}.
   Also the nilpotent cone $\mathcal{N} \subset \Omega_{\mathfrak{M}}[-1]$ is written as 
   \begin{align*}
       \mathcal{N}=\mathcal{N}_{\circ} \times Y \times \{0\} \subset 
       \Omega_{\mathfrak{M}}[-1]=\Omega_{\mathfrak{M}_{\circ}}\times Y \times \mathbb{A}^1.
   \end{align*}
   The subcategory 
   \begin{align*}
         \IndL_{\mathcal{N}}(\mathfrak{M}_{\circ})_{\delta} \subset 
           \IndL(\mathfrak{M}_{\circ})_{\delta}
   \end{align*}
   is defined to be objects with singular supports contained in $\mathcal{N}_{\circ}$. 
   The subcategories of compact/locally compact objects 
   \begin{align}\label{def:Ltilde:circ}
     \LL_{(\mathcal{N})}(\mathfrak{M}_{\circ})_{\delta} \subset 
     \widetilde{\LL}_{(\mathcal{N})}(\mathfrak{M}_{\circ})_{\delta} \subset 
       \IndL_{(\mathcal{N})}(\mathfrak{M}_{\circ})_{\delta}
   \end{align}
   are also defined in the same way as (\ref{Ltilde:chain}).

\subsection{Dolbeault geometric Langlands conjecture}\label{subsec:dlconj}
We recall the formulation of Dolbeault geometric Langlands conjecture 
proposed in~\cite{PTlim}. We have the decomposition 
\begin{align*}
    \IndCoh(\Hig_G)=\bigoplus_{w\in Z_G^{\vee}} \IndCoh(\Hig_G)_{w}
\end{align*}
where each summand corresponds to objects with $Z_G$-weight $w$.
We have the subcategory 
\begin{align}\label{indl}
    \IndL_{\mathcal{N}}(\Hig_G(\chi))_w \subset \IndCoh_{\mathcal{N}}(\Hig_G(\chi))_w
\end{align}
by setting $\delta=\delta_w \in \Pic(\Hig_G)_{\mathbb{Q}}$ associated with $w$, see~\cite[Section~7.2]{PTlim}. 
In the case of $G=\GL_r$, we have 
\begin{align}\label{deltaw}
    \delta_w=\det(\cF|_{c\times \Hig_{\GL_r}})^{w/r}
\end{align}
where $\cF$ is the universal bundle (\ref{univ:F}) and $c\in C$.

The following is a version of Dolbeault geometric Langlands conjecture in~\cite{PTlim}:
\begin{conj}\emph{(\cite[Conjecture~1.1]{PTlim})}\label{conj:dl}
For $(\chi, w)\in \pi_1(G)\times Z_G^{\vee}$, there is a $\rB_G$-linear equivalence 
\begin{align}\label{equiv:main}
    \IndCoh_{\mathcal{N}}(\Hig_{^{L}G}(w)^{\mathrm{ss}})_{-\chi} \simeq \IndL_{\mathcal{N}}(\Hig_G(\chi))_w.
\end{align}
\end{conj}

\begin{remark}\label{rmk:compact}
Since both sides of (\ref{equiv:main}) are compactly generated by~\cite[Theorem~7.18 and 7.19]{PTlim}, 
    an equivalence (\ref{equiv:main}) is equivalent to an equivalence between 
    compact objects 
    \begin{align}\label{equiv:main2}
         \Coh_{\mathcal{N}}(\Hig_{^{L}G}(w)^{\mathrm{ss}})_{-\chi} \simeq \LL_{\mathcal{N}}(\Hig_G(\chi))_w.
    \end{align}
\end{remark}

\subsection{Quasi-BPS categories}
By applying the definition of limit category to the 
quasi-compact moduli stack $\Hig_G(\chi)^{\mathrm{ss}}$, we obtain 
the \textit{quasi-BPS category} (with nilpotent singular support)
\begin{align*}
    \LL_{\mathcal{N}}(\Hig_G(\chi)^{\mathrm{ss}})_w \subset \Coh_{\mathcal{N}}(\Hig_G(\chi)^{\mathrm{ss}})_w.
\end{align*}
The quasi-BPS category is related to BPS invariants in categorical
Donaldson--\allowbreak Thomas theory; see~\cite{PThiggs, PThiggs2, PTlim}. 

We have the open immersion 
\begin{align*}
    j\colon \Hig_G(\chi)^{\mathrm{ss}}\subset \Hig_G(\chi)
\end{align*}
which induces the pull-back functor 
\begin{align}\label{funct:jast}
j^* \colon \LL_{\mathcal{N}}(\Hig_G(\chi))_w\to \LL_{\mathcal{N}}(\Hig_G(\chi)^{\mathrm{ss}})_{w}.
\end{align}
The following is one of the key properties of limit categories:
\begin{thm}\emph{(\cite[Theorem~7.19]{PTlim})}\label{thm:left}
The functor (\ref{funct:jast}) admits a fully-faithful left adjoint 
\begin{align*}
    j_{!} \colon \LL_{\mathcal{N}}(\Hig_G(\chi)^{\mathrm{ss}})_w\hookrightarrow\LL_{\mathcal{N}}(\Hig_G(\chi))_w.
\end{align*}
    
\end{thm}

From the proof of~\cite[Theorem~7.18 and 7.19]{PTlim}, the image of $j_!$
is described as follows: for any map $\nu \colon \bgm \to \Hig_G(\chi)$
corresponding to a center of a \textit{Harder--Narasimhan stratum}, the 
weight condition (\ref{wt:nu}) is strict on the left. 

In the case of $G=\GL_r$, 
for an element (called a \textit{Harder--Narasimhan type})
\begin{align*}
\mu=(\mu_1, \ldots, \mu_k)\in \mathbb{Q}^k, \ \mu_1>\cdots>\mu_k
\end{align*}
let $\mathcal{S}_{\mu} \subset \Hig_{\GL_r}(\chi)$ the locally 
closed substack consisting of Higgs bundles with Harder-Narasimhan type $\mu$; a $k$-valued point $E$ of $\mathcal{S}_{\mu}$ admits a filtration 
\begin{align*}
    0\subset E_1 \subset \cdots \subset E_k=E, \ A_i=E_i/E_{i-1}, \ 
    \mu_i=\frac{\deg A_i}{\operatorname{rank}A_i}
\end{align*}
such that $A_i$ is semistable. 
Then the direct sum $\oplus_{i=1}^k A_i$ corresponds to a $k$-valued point 
of the center $\mathcal{Z}_{\mu}$ of $\mathcal{S}_{\mu}$, and 
the corresponding 
canonical map $\nu \colon \bgm \to \mathcal{Z}_{\mu}$ with image 
$\oplus_{i=1}^k A_i$ is of the form 
\begin{align}\label{map:nupt}
    \nu\colon \mathrm{pt} \mapsto \bigoplus_{i=1}^k A_i, \
    \mathbb{G}_m \mbox{-weight}=(\nu_1, \ldots, \nu_k), \ 
    \nu_1>\cdots>\nu_k. 
\end{align}
The inequality $\nu_i>\nu_{i+1}$ holds since the extension space $\Ext^1(A_{i+1}, A_i)$ has $\lambda$-weight $\nu_i-\nu_{i+1}$, which is positive as the $\mathbb{G}_m$-action on it via $\lambda$ 
has to converge for $t\to 0$ limit, where $t\in \mathbb{G}_m$. 
Then the image of $j_{!}$ equals to the subcategory such that the 
condition (\ref{wt:nu}) is strict on the left for each map $\nu$ as in (\ref{map:nupt}), composed with the natural map $\mathcal{Z}_{\mu} \to \Hig_{\GL_r}(\chi)$.

\begin{remark}\label{rmk:jshrink}
The above fact on the image of $j_{!}$ 
is implicit in the proof of~\cite[Proposition~8.23]{PTlim} and the definition of the magic window subcategory in~\cite[Definition~8.17]{PTlim} for matrix factorizations, together with the comparison of 
the limit category and magic window for matrix factorizations~\cite[Proposition~3.14]{PTlim} under Koszul equivalence.

In particular for 
a perfect complex $\cE$ on $\mathfrak{M}=\Hig_{\GL_r}(\chi)$, 
it is in the image of $j_!$ if it satisfies the 
following stronger condition than (\ref{wt:cond2})
for any map $\nu \colon \bgm \to \mathfrak{M}$ as in (\ref{map:nupt}):
\begin{align}\notag
    \wt(\nu^* \cE) \subset \left(\frac{1}{2} c_1 (\nu^{\ast}\mathbb{L}_{\mathfrak{M}}^{<0}), 
    \frac{1}{2} c_1 (\nu^{\ast}\mathbb{L}_{\mathfrak{M}}^{>0})
    \right]+c_1(\nu^{\ast}\delta_{w}).
\end{align}
\end{remark}

\begin{remark}\label{rmk:HN!}
More generally, 
there is a total order $\preceq$ on the set of Harder--Narasimhan types
$\mu=(\mu_1,\ldots,\mu_k)$, and a Harder--Narasimhan stratification
\[
  \Hig_{\GL_r}(\chi)=\sqcup_{\mu} \cS_{\mu},
\]
where $\cS_{\mu}$ is the Harder--Narasimhan stratum of type $\mu$.
By setting
\[
  \Hig_{\GL_r}(\chi):=\sqcup_{\mu'\preceq \mu}\cS_{\mu'},
   \quad
 \Hig_{\GL_r}(\chi)_{\prec \mu}:=\sqcup_{\mu'\prec \mu}\cS_{\mu'},
\]
we have a open/closed stratification
\[
    \Hig_{\GL_r}(\chi)_{\preceq \mu}=\Hig_{\GL_r}(\chi)_{\prec \mu} \sqcup \cS_{\mu},
\]
and $\cS_{\mu}$ is a $\Theta$-stratum of $\Hig_{\GL_r}(\chi)_{\preceq \mu}$.

For the open immersion
\begin{align}\label{open:ju}
   j_{\mu}\colon\Hig_{\GL_r}(\chi)_{\prec \mu}\hookrightarrow
  \Hig_{\GL_r}(\chi)_{\preceq \mu},
\end{align}
there is a fully faithful left adjoint to $j_{\mu}^{\ast}$,
by~\cite[Theorems~7.18 and~7.19]{PTlim},
\[
     (j_{\mu})_{!}\colon
     \LL_{\mathcal{N}}(\Hig_{\GL_r}(\chi)_{\prec \mu})_{w}
     \hookrightarrow
     \LL_{\mathcal{N}}(\Hig_{\GL_r}(\chi)_{\preceq \mu})_{w},
\]
whose image is described by the weight condition~\eqref{wt:cond2}, with the
left inequality strict, for any map $\nu$ of Harder--Narasimhan type $\mu$
as in~\eqref{map:nupt}.
\end{remark}

\begin{remark}\label{rmk:Ltilde2}
Recall the category $\widetilde{\LL}(\mathfrak{M})_{\delta}$ from
Definition~\ref{def:Ltilde}. We have an embedding
\[
    \LL_{\mathcal{N}}(\mathfrak{M})_{\delta}
    \subset
    \widetilde{\LL}_{\mathcal{N}}(\mathfrak{M})_{\delta}.
\]
In the case $\mathfrak{M}=\Hig_{\GL_r}(\chi)$, the subcategory
$\LL_{\mathcal{N}}(\mathfrak{M})_{w}$ is identified with the subcategory
obtained as the image of $(j_{\preceq \mu})_{!}$ for the open immersion (see~\cite[(8.65)]{PTlim})
\[
    j_{\preceq \mu}\colon
    \Hig_{\GL_r}(\chi)_{\preceq \mu}\hookrightarrow \Hig_{\GL_r}(\chi).
\]
Equivalently, it consists of objects satisfying the condition~\eqref{wt:cond2}
for all maps $\nu\colon \bgm \to \Hig_{\GL_r}(\chi)$, with the left inequality
strict for maps corresponding to the center of a HN stratum~\eqref{map:nupt} for sufficiently large $\mu$.
\end{remark}

\begin{remark}
    The remarks in Remarks~\ref{rmk:jshrink},~\ref{rmk:HN!},~\ref{rmk:T}, and~\ref{rmk:Ltilde2} also apply to any 
    reductive group $G$, see~\cite[Theorem~7.18 and 7.19]{PTlim}. 
\end{remark}

\subsection{Simplified notation}
Below we use the following simplified notation
in the case of $G=\GL_r$. We take  
 an open subset $\cB \subset \rB_{\GL_r}^{\mathrm{cl}}$, 
 and simply write 
\begin{align}\label{sH}
    \cH:=\Hig_{\GL_r}\times_{\rB_{\GL_r}} \cB.
\end{align}
Note that $\cH \times k[-1]$ is an open substack of $\Hig_{\GL_r}$, hence 
has a self-dual cotangent complex. For $\nu \colon \bgm \to \cH$, its 
regularization map is denoted by $\nu_{\circ}^{\mathrm{reg}}$ as in Subsection~\ref{subsec:variant}.
We also write 
\begin{align*}
    \cH(\chi)^{\mathrm{ss}} \subset \cH(\chi) \subset \cH
\end{align*}
where $\cH(\chi)$ is 
the connected component of $\cH$
corresponding to Higgs bundles $(F, \theta)$ such that $\deg F=\chi$, 
and $\cH(\chi)^{\mathrm{ss}}$ is the semistable part. 
For $T \to \cB$, we often write 
\begin{align*}\cH_T:=\cH\times_{\cB}T.
\end{align*}
In particular for a point $b\in \cB$, we write 
$\cH_b:=\cH\times_{\cB}\{b\}$.

Note that the limit category for $\cH_T(\chi)$
for an \'etale map $T\to \cB$ 
is defined as in Subsection~\ref{subsec:variant}. 
We say that the \textit{Dolbeault geometric Langlands conjecture (DL conjecture for short)}
holds over $\cB$ if there is a $\cB$-linear equivalence
\begin{align}\label{equiv:H}
    \IndCoh_{\mathcal{N}}(\cH(w)^{\mathrm{ss}})_{-\chi} \simeq \IndL_{\mathcal{N}}(\cH(\chi))_w.
\end{align}

\begin{remark}\label{rmk:nilp}
If $\cB \subset \rB_{\GL_r}^{\mathrm{red, cl}}$, and $T\to \cB$ 
is an \'etale map, we have 
\begin{align*}
   \Coh_{\mathcal{N}}(\cH_T(\chi))=\mathrm{Perf}(\cH_T(\chi)), \ 
   \IndCoh_{\mathcal{N}}(\cH_T(\chi))=\QCoh(\cH_T(\chi)).
\end{align*}
The above identities hold since the automorphism groups of $\cH_T(\chi)$ 
are subgroups of tori; in particular the nilpotent cone (\ref{nilp:cone}) is the zero section 
over $\cB$, see~\cite[Lemma~4.1]{TodaGL2}. 
\end{remark}

\begin{remark}\label{rmk:T}
The result of Theorem~\ref{thm:left}
    also applies to
    the limit category 
    \begin{align}\notag
    \LL_{\mathcal{N}}(\cH_T(\chi))_w \subset \IndL_{\mathcal{N}}(\cH_T(\chi))_w
\end{align}
    for an \'etale map $T\to \cB$; 
    for the open immersion 
    \begin{align*}j_T \colon \cH_T(\chi)^{\mathrm{ss}}:=\cH(\chi)^{\mathrm{ss}}\times_{\cB}T\hookrightarrow \cH_T(\chi),
    \end{align*}
    we have the 
    fully-faithful left adjoint of $j_T^*$
    \begin{align*}
(j_T)_{!} \colon \LL_{\mathcal{N}}(\cH_T(\chi)^{\mathrm{ss}})_w \hookrightarrow \LL_{\mathcal{N}}(\cH_T(\chi))_w
    \end{align*}
    and its image is characterized as in Remark~\ref{rmk:jshrink}, i.e.\ the weight condition (\ref{cond:regcirc}) is strict on the left for any 
  map $\nu \colon \bgm \to \cH_T(\chi)$
    whose image in $\cH(\chi)$ corresponds to a HN filtration. 
    It follows from the comparison of limit categories after removing the $k[-1]$-factor, 
    see Lemma~\ref{lem:equiv:k}. 
    
    The same statement as in Remark~\ref{rmk:HN!} applies for 
    the image of
    \begin{align*}
        (j_{T\mu})_{!} \colon \LL_{\mathcal{N}}(\cH_T(\chi)_{\prec \mu})_w \hookrightarrow \LL_{\mathcal{N}}(\cH_T(\chi)_{\preceq \mu})_w
    \end{align*}
    where $j_{T\mu}$ is the pull-back of the open immersion (\ref{open:ju}) to $T$, as well as the description of the subcategory
    as in Remark~\ref{rmk:Ltilde2}
    \begin{align*}
        \LL_{\mathcal{N}}(\cH_T(\chi))_w \subset \widetilde{\LL}_{\mathcal{N}}(\cH_T(\chi))_w.
    \end{align*}
\end{remark}

\subsection{Arinkin sheaf}\label{subsec:Arsheaf}
Below, we set $G=\GL_r$. 
Let $\cB \subset \rB_{\GL_r}^{\mathrm{cl}}$ be an open subset, and define $\cH$ as in (\ref{sH}). 
Consider 
\begin{align}\label{Psharp}
    (\cH\times_{\cB}\cH)^{\sharp}:=
    (\cH\times_{\cB}\cH^{\mathrm{reg}}) \cup (\cH^{\mathrm{reg}}\times_{\cB}\cH).
\end{align}
Then there is a line bundle 
\begin{align*}
    \cP^{\sharp} \to  (\cH\times_{\cB}\cH)^{\sharp}
\end{align*}
whose fiber at $(A_1, A_2)$, where $A_i \in \Coh^{\heartsuit}(\cC_b)$
for $b\in \cB$, is given by the \textit{Deligne pairing}
\begin{align}\label{def:Psharp}
     \cP^{\sharp}|_{(A_1, A_2)}=\det \chi(A_1\otimes A_2)\otimes \det \chi(A_1)^{-1}\otimes \det \chi(A_2)^{-1}\otimes \det \chi(\mathcal{O}_{\cC_b}).
\end{align}
The above formula makes sense since either $A_1$ or $A_2$ is a line bundle on $\cC_b$, 
so that $\chi(A_1\otimes A_2)$ makes sense. 

Arinkin~\cite{Ardual} constructed the maximal Cohen--Macaulay 
extension of $\cP^{\sharp}$ to $\cH\times_{\cB}\cH$
when $\cB=\rB_{\GL_r}^{\mathrm{ell, cl}}$. Outside the locus $\rB_{\GL_r}^{\mathrm{ell, cl}}$, i.e.\ when $\mathrm{B}_{\GL_r}^{\mathrm{ell, cl}} \subsetneq \cB$, it is observed in~\cite[Corollary~3.2.3]{MLi} (which works for any $G=\GL_r$, see~\cite[Remark~2.20]{TodaGL2}) that $\cP^{\sharp}$ can be extended 
to the generically regular Higgs locus. 
\begin{thm}\emph{(\cite{Ardual, MLi})}\label{thm:Pflat}
For an open subset $\cB \subset \rB_{\GL_r}^{\mathrm{cl}}$, 
the line bundle $\cP^{\sharp}$ extends to a maximal 
Cohen--Macaulay sheaf
\begin{align}\label{Arsheaf0}
    \cP\in \Coh^{\heartsuit}((\cH\times_{\cB}\cH)^{\flat}),
\end{align}
where $(\cH\times_{\cB}\cH)^{\flat}$ is defined by 
\begin{align}\label{Pflat}
    (\cH\times_{\cB}\cH)^{\flat}:=
    (\cH\times_{\cB}\cH^{\mathrm{greg}}) \cup (\cH^{\mathrm{greg}}\times_{\cB}\cH).
\end{align}
Moreover it is flat over both sides of $\cH$. 
\end{thm}
The maximal Cohen--Macaulay extension is unique up to isomorphism, which 
follows from the following well-known fact (which we will also use later): 
\begin{lemma}\emph{(\cite[Lemma~2.2]{Ardual}, \cite[Lemma~2.22]{TodaGL2})}\label{lem:MCM-extension}
Let $\mathcal{X}$ be a classical Artin stack of pure dimension, and let
$
M\in \Coh^{\heartsuit}(\mathcal{X})
$
be a maximal Cohen--Macaulay sheaf on $\mathcal{X}$. Let
$\mathcal{Z}\subset \mathcal{X}$ be a closed substack of codimension at
least two, and let
$
j\colon \mathcal{X}\setminus \mathcal{Z}\hookrightarrow \mathcal{X}
$
be the open immersion. Then the natural morphism
$
M \to j^{\heartsuit}_{*}(M|_{\mathcal{X}\setminus \mathcal{Z}})
$
is an isomorphism, where
$
j^{\heartsuit}_{*}:=\mathcal{H}^0(j_*).
$
\end{lemma}

In the case that $\cB\subset \rB_{\GL_r}^{\mathrm{red, cl}}$, 
we have \begin{align*}\cH^{\mathrm{greg}}=\cH, \ 
(\cH\times_{\cB}\cH)^{\flat}=\cH\times_{\cB}\cH.
\end{align*}
In this case, we use 
the same symbol $\cP$ to denote its restriction to 
$\cH^{\mathrm{ss}}\times_{\cB}\cH$. 

For a rank-one torsion-free sheaf $E\in \Coh^{\heartsuit}(\cC_b)$ and 
the corresponding map $\iota_E \colon \mathrm{pt} \to \cH^{\mathrm{greg}}$, 
we write 
\begin{align}\label{def:PE}
\cP_E:=(\iota_E \times \id)^* \cP \in \Coh^{\heartsuit}(\cH_b)
\end{align}
which is a maximal Cohen--Macaulay sheaf. 

When $E=\cL$ is a line bundle on $\cC_b$, the sheaf $\cP_{\cL}$ is a line 
bundle on $\cH_b$. Later we will use the following lemma on the computation 
of its $\mathbb{G}_m$-weight: 
\begin{lemma}\label{lem:PLGm}
Let $\nu \colon \bgm \to \cH_b$ be a map corresponding to 
\begin{align*}
    \mathrm{pt}\mapsto M=M_1 \oplus \cdots \oplus M_k
\end{align*}
for $M_i \in \Coh^{\heartsuit}(\cC_b)$ 
with scheme theoretic support $C_i=\mathrm{Supp}(M_i)$ (which is regarded as an effective divisor on $S$) and 
$\mathbb{G}_m$-weight $(\nu_1, \ldots, \nu_k)$. 
Then the $\mathbb{G}_m$-weight of $\nu^* \cP_{\cL}$ is given by 
\begin{align*}
    \wt(\nu^* \cP_{\cL})=\sum_{i=1}^k l_i \cdot \nu_i.
\end{align*}
Here $l_i:=\deg_{C_i}(\cL|_{C_i})\in \mathbb{Z}$ is defined as in (\ref{deg:Cb}). 
\end{lemma}
\begin{proof}
By (\ref{def:Psharp}), we have 
\begin{align*}
  \nu^* \cP_{\cL}=\bigotimes_{i=1}^k \left(\det \chi(\cL\otimes M_i)\otimes \det \chi(M_i)^{-1}\right) \otimes \det \chi(\cL)^{-1}\otimes \det \chi(\cO_{\cC_b}).
\end{align*}
Then the lemma follows since the $\mathbb{G}_m$-weight of $M_i$ is $\nu_i$ 
and 
\begin{align*}
    \chi(\cL\otimes M_i)-\chi(M_i)=\chi(\cL|_{C_i})-\chi(\cO_{C_i})
\end{align*}
which is $\deg_{C_i}(\cL|_{C_i})$.
\end{proof}

Later we will also use the following calculation of $\mathbb{G}_m$-weights of the cotangent complex of $\cH$: 
\begin{lemma}\label{lem:Gmwt:L}
In the setting of Lemma~\ref{lem:PLGm}, 
we have 
\begin{align*}
    &c_1(\nu^* \mathbb{L}_{\Hig_{\GL_r}}^{>0})=c_1(\nu^* \mathbb{L}_{\cH}^{>0})=c_1(\nu^{*}\mathbb{L}_{\cH_b}^{>0}) \\
     &c_1(\nu^* \mathbb{L}_{\Hig_{\GL_r}}^{<0})=c_1(\nu^* \mathbb{L}_{\cH}^{<0})=c_1(\nu^{*}\mathbb{L}_{\cH_b}^{<0}) 
\end{align*}
and, if $\nu_1>\cdots>\nu_k$, they are equal to 
\begin{align}\label{compute:wg}
(2g-2)\sum_{i<j}r_i r_j (\nu_i-\nu_j), \ 
(2g-2)\sum_{i>j}r_i r_j (\nu_i-\nu_j)
    \end{align}
    respectively. 
    Here $r_i:=\operatorname{rank}(\pi_{*}\cO_{C_i})$, and we have used the same symbol $\nu$ for its compositions 
    $\cH_b \hookrightarrow \cH \hookrightarrow \Hig_{\GL_r}$. 
\end{lemma}
\begin{proof}
The first statement follows since the cotangent complexes of 
$\cH_b$, $\cH$ differ from those of $\Hig_{\GL_r}$ by the cotangent complexes of 
$k[-1]$ and $\cB$, which have trivial $\mathbb{G}_m$-weights. 

As for the second statement, note that we have 
\begin{align*}
    \nu^* \mathbb{L}_{\Hig_{\GL_r}} \simeq
     \nu^* \mathbb{T}_{\Hig_{\GL_r}}=
    \Hom_S(M, M)[1].
\end{align*}
Therefore we have 
\begin{align*}
     \nu^* \mathbb{L}_{\Hig_{\GL_r}}^{>0}=
     \bigoplus_{i<j}\Hom_S(M_j, M_i)[1]
\end{align*}
where $\Hom_S(M_j, M_i)[1]$ has $\mathbb{G}_m$-weight $\nu_i-\nu_j$. 
    Therefore the formula (\ref{compute:wg}) for positive part follows from 
    the following (which is computed using the Riemann-Roch formula)
    \begin{align*}
        -\chi_S(M_j, M_i)=C_i \cdot C_j=r_i r_j(2g-2).
    \end{align*}
    The computation of the negative part is the same. 
\end{proof}

We have the following proposition, which explains the naturality of the weight condition in the definition of the limit category, which is equivalent to the condition (\ref{wt:cond2}) for perfect complexes. Although this result will not be used in the rest of the paper, we include it because its proof serves as a model for the weight estimates in the next section.

\begin{prop}\label{prop:natural}
Let $\mathcal{C}_b \subset S$ be a spectral curve corresponding to
$b \in \rB_{\GL_r}^{\mathrm{cl}}$, which is not necessarily reduced. For a line
bundle $\mathcal{L}$ on $\cC_b$ with $w'=\deg_{\cC_b}(\mathcal{L})$, let
$\cP_{\mathcal{L}}$ be the corresponding line bundle on $\cH_b$. Then
$\mathcal{L}$ corresponds to a semistable Higgs bundle if and only if, for
any map $\nu \colon \bgm \to \cH_b$, we have
\begin{align}\label{nuPL}
\wt(\nu^* \cP_{\mathcal{L}})
\subset
\left[
 \frac{1}{2}c_1(\nu^*\mathbb{L}^{<0}_{\cH_b}),
 \frac{1}{2}c_1(\nu^* \mathbb{L}_{\cH_b}^{>0})
\right]
+c_1(\nu^* \delta_{w'}).
\end{align}
\end{prop}

\begin{proof}
Let $\nu \colon \bgm \to \cH_b$ be a map as in
Lemma~\ref{lem:PLGm}, and set $r_i=\operatorname{rank}(\pi_{*}\cO_{C_i})$, 
$l_i=\deg_{C_i}(\cL|_{C_i})$.  For a subset
$I\subset \{1,\ldots,k\}$, we write
$C_I:=\sum_{i\in I} C_i$ as an effective divisor on $S$, and set
$r_I:=\sum_{i\in I}r_i$ and $l_I:=\sum_{i\in I}l_i$. 
Note that $\wt(\nu^* \cP_{\mathcal{L}})$ consists of the single element
$c_1(\nu^* \cP_{\mathcal{L}})\in \mathbb{Z}$, as $\cP_{\mathcal{L}}$ is a line bundle.
Since
\[
c_1(\nu^* \delta_{w'})
=
\frac{w'}{r}\sum_{i=1}^k r_i \nu_i,
\]
if we set
$\widetilde{l}_i:=l_i-\frac{w'}{r}r_i$ and
$\widetilde{l}_I:=\sum_{i\in I}\widetilde{l}_i$, then
\begin{align}\label{eq:c1P}
c_1(\nu^* \cP_{\mathcal{L}})-c_1(\nu^* \delta_{w'})
=
\sum_{i=1}^k \nu_i \widetilde{l}_i
=
\sum_{i=1}^{k-1}\widetilde{l}_{I_i}(\nu_i-\nu_{i+1}),
\end{align}
where $I_i=\{1,\ldots,i\}$. Also, by Lemma~\ref{lem:Gmwt:L}, we have
\begin{align}\label{eq:Lpos}
    \frac{1}{2}c_1(\nu^* \mathbb{L}_{\cH_b}^{>0})
    =
    (g-1)\sum_{i<j}r_i r_j(\nu_i-\nu_j)
    =
    (g-1)\sum_{i=1}^{k-1}
    r_{I_i}r_{I_i^{\circ}}(\nu_i-\nu_{i+1}),
\end{align}
where $I^{\circ}$ denotes the complement of $I$.

Suppose that $\mathcal{L}$ corresponds to a semistable Higgs bundle.
For any subset $I\subset \{1,\ldots,k\}$, we have the exact sequence
\[
0 \to \mathcal{L}|_{C_{I^{\circ}}}(-C_I) \to \mathcal{L}
\to \mathcal{L}|_{C_I} \to 0 .
\]
The semistability of $\mathcal{L}$ gives (see Remark~\ref{rmk:ss:spec})
\[
\frac{\deg \pi_{*}\mathcal{L}}{\rank \pi_{*}\mathcal{L}}
\leq
\frac{\deg \pi_{*}(\mathcal{L}|_{C_I})}
     {\rank \pi_{*}(\mathcal{L}|_{C_I})},
\]
or equivalently (using (\ref{formula:deg}))
\[
\frac{w'-(r^2-r)(g-1)}{r}
\leq
\frac{l_I-(r_I^2-r_I)(g-1)}{r_I}.
\]
Expanding this inequality, we obtain
\[
    \widetilde{l}_I
    :=
    l_I-\frac{r_I}{r}w'
    \geq
    -(g-1)r_I r_{I^{\circ}}.
\]
Replacing $I$ with $I^{\circ}$ and using
$\widetilde{l}_I+\widetilde{l}_{I^{\circ}}=0$, we also obtain
$\widetilde{l}_I \leq (g-1)r_I r_{I^{\circ}}$. Hence, by
(\ref{eq:c1P}) and (\ref{eq:Lpos}), the condition~(\ref{nuPL}) holds.

Conversely, suppose that $\mathcal{L}$ is not semistable. Then there is
an exact sequence in $\Coh^{\heartsuit}(\cC_b)$
\[
    0\to E_2 \to \mathcal{L} \to E_1 \to 0
\]
such that
\begin{align}\label{ineq:LE}
\frac{\deg \pi_{*}\mathcal{L}}{\rank \pi_{*}\mathcal{L}}>
\frac{\deg \pi_{*}E_1}{\rank \pi_{*}E_1}.
\end{align}
By taking the saturation, we can take $E_1$ to be torsion-free on $\cC_b$. 
Then the induced surjection
$\mathcal{O}_{\cC_b} \twoheadrightarrow E_1 \otimes \mathcal{L}^{-1}$
gives $E_1 \otimes \mathcal{L}^{-1}\cong \cO_{C_1'}$ for a decomposition
$\cC_b=C_1'+C_2'$ as effective divisors on $S$. Writing
$r_i'=\rank(\pi_{*}\cO_{C_i'})$ and
$l_i'=\deg_{C_i'}(\mathcal{L}|_{C_i'})$, the inequality~(\ref{ineq:LE}) is
equivalent to
\[
\frac{w'-(r^2-r)(g-1)}{r}
>
\frac{l_1'-(r_1^{'2}-r_1')(g-1)}{r_1'}.
\]
Expanding this inequality gives
\[
    \widetilde{l}_1'
    :=
    l_1'-\frac{r_1'}{r}w'
    <
    -(g-1)r_1' r_2'.
\]
It follows that, for any map $\nu'\colon \bgm \to \cH_b$ of the form
$\mathrm{pt} \mapsto M_1'\oplus M_2'$ with
$\mathrm{Supp}(M_i')=C_i'$ and $\mathbb{G}_m$-weights $(\nu_1', \nu_2')$
with 
$\nu_1'>\nu_2'$, we have
\begin{align*}
   c_1((\nu')^* \cP_{\mathcal{L}})
   -c_1((\nu')^* \delta_{w'})
   &=
   \widetilde{l}_1'(\nu_1'-\nu_2')  \\
   &<
   -(g-1)r_1' r_2'(\nu_1'-\nu_2') \\
   &=
   \frac{1}{2}c_1((\nu')^*\mathbb{L}_{\cH_b}^{<0}).
\end{align*}
Therefore the condition~(\ref{nuPL}) is not satisfied.
\end{proof}

\subsection{Fourier--Mukai functor via Arinkin sheaf}
Let $\cB \subset \rB_{\GL_r}^{\mathrm{red, cl}}$ be an open subset. 
By the description (\ref{def:Psharp}), 
the Arinkin sheaf (\ref{Arsheaf0}) on 
$\cH(w)\times_{\cB}\cH(\chi)$ has bi-weight 
$(\chi', w')$, 
\begin{align*}
    \chi':=\chi+(r^2-r)(g-1), \ w':=w+(r^2-r)(g-1).
\end{align*}
Here the weight shift $(\chi, w)\mapsto (\chi', w')$ is due to 
the change of degrees on $C$ and $\cC_b$, see the formula (\ref{formula:deg}). 
Therefore it induces the $\cB$-linear functor 
\begin{align}\label{induce:Phi0}
\Phi \colon 
    \QCoh(\cH(w)^{\mathrm{ss}})_{-\chi'} \to \QCoh(\cH(\chi))_{w'}, 
\end{align}
defined by the Fourier--Mukai functor
\begin{align*}
    \Phi(-)=p_{2*}(p_1^{*}(-)\otimes \cP).
\end{align*}
Here $p_i$ are the projections from $\cH(w)^{\mathrm{ss}}\times_{\cB}\cH(\chi)$ onto the corresponding factors. 

By cohomological properness of $\cH(w)^{\mathrm{ss}}\to \cB$,
(which holds as $\cH(w)^{\mathrm{ss}}$ admits a good moduli space which is projective over $\cB$)
together with the flatness of $\cP$ over the second factor, the functor (\ref{induce:Phi0}) preserves coherent sheaves
\begin{align}\label{induce:Phi}
    \Phi \colon   \Coh(\cH(w)^{\mathrm{ss}})_{-\chi'} \to \Coh(\cH(\chi))_{w'}. 
\end{align}

\subsection{The Fourier--Mukai transform for compactified Jacobians of reduced curves}\label{subsec:FMJ}

Let $T$ be a $k$-scheme of finite type, and let $X\to T$ be a flat family of reduced projective curves with at worst planar singularities 
(e.g.\ $X\to T$ is the universal spectral curve $\cC\to \cB$ for $\cB \subset \rB_{\GL_r}^{\mathrm{red,cl}}$).
Here we recall the Fourier--Mukai transform for relative compactified Jacobians of $X$ 
via the Arinkin sheaf constructed in~\cite{MRVF2}, using the language of Gieseker semistable sheaves as in~\cite[Subsection~2.10]{TodaGL2}.

For a $T$-relative $\mathbb{Q}$-ample divisor $H$ on $X$ and $\chi' \in \mathbb{Z}$, let 
\begin{align}\label{map:T}
\overline{\mathcal{J}}_{X/T}(\chi')^H \to T
\end{align}
be the moduli stack of rank-one torsion-free $H$-semistable sheaves on the fibers of $X\to T$ of degree $\chi'$. 
Suppose that $H$ is generic, i.e.\ that $\overline{\mathcal{J}}_{X/T}(\chi')^H$
consists of $H$-stable sheaves. 
Then the good moduli space morphism
\begin{align*}
    \overline{\mathcal{J}}_{X/T}(\chi')^H \to \overline{J}_{X/T}(\chi')^H
\end{align*}
is a $\mathbb{G}_m$-gerbe with Brauer class $\beta$, and we have the decomposition into $\mathbb{G}_m$-weight subcategories
\begin{align*}
    \Coh(\overline{\mathcal{J}}_{X/T}(\chi')^H)
    =
    \bigoplus_{w'\in \mathbb{Z}}
    \Coh(\overline{\mathcal{J}}_{X/T}(\chi')^H)_{w'},
\end{align*}
where each summand is equivalent to the category of $\beta^{w'}$-twisted sheaves.

The construction of the Arinkin sheaf in the previous subsection yields a Cohen--Macaulay sheaf
\begin{align*}
    \cP_{\mathcal{J}/T} \in
    \Coh^{\heartsuit}\bigl(
    \overline{\mathcal{J}}_{X/T}(w')^H
    \times_T
    \overline{\mathcal{J}}_{X/T}(\chi')^H
    \bigr)
\end{align*}
of bi-weight $(\chi',w')$.

\begin{thm}\label{thm:MRV}\emph{(\cite[Theorem~A]{MRVF2}, \cite[Proposition~2.28]{TodaGL2})}
The Fourier--Mukai functor
\begin{align}\label{Phi:J}
\Phi_{\mathcal{J}/T} \colon
\Coh(\overline{\mathcal{J}}_{X/T}(w')^H)_{-\chi'}
\to
\Coh(\overline{\mathcal{J}}_{X/T}(\chi')^H)_{w'}
\end{align}
with kernel $\cP_{\mathcal{J}/T}$ is an equivalence. Moreover, the inverse of 
$\Phi_{\mathcal{J}/T}$ is given by the kernel object
\begin{align}\label{kernel:inv}
\mathbb{D}(\cP_{\mathcal{J}/T})\otimes p_1^*\omega_{\overline{J}/T}[p_a],
\end{align}
where $p_a$ is the arithmetic genus of the fibers of $X\to T$.
\end{thm}

\begin{remark}\label{rmk:dualizing}
    The dualizing sheaf of $\overline{J}_{X/T}(\chi')$ is trivial 
    on the fiber over any $t\in T$ by~\cite[Theorem~A]{MRF3}. 
    Therefore $\omega_{\overline{J}/T}$ in Theorem~\ref{thm:MRV} is locally trivial on $T$.
\end{remark}

\subsection{Hitchin section}\label{subsec:Hitsec}
We set 
\begin{align*}
    \chi_0 := \deg \pi_{*}\mathcal{O}_{\cC_b}=(r^2-r)(1-g).
\end{align*}
We have the following \textit{Hitchin section}
\begin{align}\label{sec:s}
    s \colon \rB_{\GL_r} \to \Hig_{\GL_r}(\chi_0). 
\end{align}
It is induced by the structure sheaf of the universal spectral curve, i.e.\ 
it sends $b\in \rB_{\GL_r}$ to the structure sheaf
$\mathcal{O}_{\cC_b}$. 
It is well-known that $\mathcal{O}_{\cC_b}$ corresponds to a stable 
Higgs bundle, and this fact is implicitly used in~\cite{TodaGL2}. 
Here we give a proof of it. 
\begin{lemma}\label{lem:OCstable}
The sheaf $\mathcal{O}_{\cC_b}\in \Coh^{\heartsuit}(\cC_b)$ corresponds to a stable Higgs bundle. 
In particular, the Hitchin section $s$ factors through
$\Hig_{\GL_r}(\chi_0)^{\mathrm{st}} \subset \Hig_{\GL_r}(\chi_0)$. 
\end{lemma}
\begin{proof}
    Let 
    \begin{align*}
        0\to E_2 \to \mathcal{O}_{\cC_b} \to E_1 \to 0
    \end{align*}
    be an exact sequence in $\Coh^{\heartsuit}(\cC_b)$
    such that $E_i$ are torsion-free sheaves on $\cC_b$;
    by the spectral 
    construction, it corresponds to an exact sequence of Higgs bundles. 
As $E_1$ is a quotient of $\mathcal{O}_{\cC_b}$, it is also 
    a structure sheaf of a spectral curve of a $\GL_{r'}$-Higgs bundle
    where $r'=\operatorname{rank}\pi_{*}E_1<r$. 
    Then we have 
    \begin{align*}
        \frac{\deg\pi_{*}\mathcal{O}_{\cC_b}}{\rank \pi_{*}\mathcal{O}_{\cC_b}}=(r-1)(1-g)<\frac{\deg\pi_{*}E_1}{\rank \pi_{*}E_1}=(r'-1)(1-g).
    \end{align*}
    Since $g\geq 2$ and $r\geq 2$, $\cO_{\cC_b}$ corresponds to a stable Higgs bundle. 
\end{proof}

For an open subset $\cB \subset \rB_{\GL_r}^{\mathrm{cl}}$, 
the Hitchin section $s$ 
induces the section 
\begin{align*}
    s \colon \cB \to \cH(\chi_0)
\end{align*}
which by Lemma~\ref{lem:OCstable} factors through 
\begin{align*}
    s\colon \cB \stackrel{\overline{s}}{\to} \cH(\chi_0)^{\mathrm{ss}} \stackrel{j}{\hookrightarrow} \cH(\chi_0)
\end{align*}
such that $\overline{s}$ is a closed immersion after the $\mathbb{G}_m$-rigidification, see the proof of~\cite[Lemma~4.3]{TodaGL2}.
We recall the following fact, which summarizes the arguments in~\cite[Section~4.2]{TodaGL2}:
\begin{lemma}\emph{(\cite[Section~4.2, Lemma~4.3]{TodaGL2})}\label{lem:factori}
The functor 
\begin{align*}
    s^* \colon \LL_{\mathcal{N}}(\cH(\chi_0))_{w'} \to \Coh(\cB)
\end{align*}
admits a left adjoint
\begin{align}\label{sleft}
   s_{!} \colon \Coh(\cB) \to \LL_{\mathcal{N}}(\cH(\chi_0))_{w'}
\end{align}
given by 
\begin{align}\label{formula:sbar}
    s_{!}=j_{!} \circ \overline{s}_{!} \colon 
    \Coh(\cB) \stackrel{\overline{s}_!}{\to} \LL_{\mathcal{N}}(\cH(\chi_0)^{\mathrm{ss}})_{w'} \stackrel{j_!}{\to}
    \LL_{\mathcal{N}}(\cH(\chi_0))_{w'}.
\end{align}
Moreover 
the 
functor $\overline{s}_{!}$ is given by 
\begin{align}\label{isom:sbar!}
    \overline{s}_!\cong(\overline{s}_{*})_{w'}[-p_a]
\end{align}
where $\overline{s}_{*}$ is the usual $*$-push forward and $(-)_{w'}$ is taking the direct summand of the weight $w'$-part.
\end{lemma}
\begin{remark}\label{rmk:sstar}
An isomorphism (\ref{isom:sbar!}) holds as $\overline{s}$ is a closed immersion 
after the $\mathbb{G}_m$-rigidification. The same formula does not apply to 
$s_{!}=j_{!}\overline{s}_!$, as the image of $s$ in $\cH$ is not closed if $\cB$ is not contained in the elliptic locus. 
\end{remark}

We also recall the following lemma, which is immediate from the definition of the Deligne pairing (\ref{def:Psharp}) and the fact that the image of $s$ lies in the regular part: 
\begin{lemma}\emph{(\cite[Lemma~4.2]{TodaGL2})}\label{lem:ids}
For the morphism 
\begin{align*}
    \id \times s\colon \cH=\cH\times_{\cB} \cB \to \cH\times_{\cB} \cH, 
\end{align*}
there is an isomorphism 
\begin{align*}
    (\id\times s)^* \cP \cong \cO_{\cH}.
\end{align*}
\end{lemma}

\subsection{Whittaker normalization}\label{subsec:whit}
Note that the functor $\Phi$ in (\ref{induce:Phi}) for $\chi'=0$ 
is the functor 
\begin{align*}
    \Phi \colon \Coh(\cH(w)^{\mathrm{ss}})_0 \to \Coh(\cH(\chi_0))_{w'}.
\end{align*}
The following conjecture is the main subject in this paper: 
\begin{conj}\emph{(\cite[Conjecture~4.4]{TodaGL2})}\label{conj:Whit}
    For an open subset $\cB \subset \rB_{\GL_r}^{\mathrm{red, cl}}$, there is an isomorphism 
    \begin{align*}
        \Phi(\mathcal{O}_{\cH(w)^{\mathrm{ss}}}) \cong s_{!}\mathcal{O}_{\cB}.
    \end{align*}
\end{conj}

The above conjecture is relevant to Conjecture~\ref{conj:dl}
as follows: 
\begin{thm}\emph{(\cite[Theorem~4.29]{TodaGL2})}\label{thm:WhitGL}
For $G=\GL_r$ and an open subset $\cB \subset \rB_{\GL_r}^{\mathrm{red, cl}}$, 
suppose that the following conditions hold: 
\begin{enumerate}
    \item Conjecture~\ref{conj:Whit} holds;
    \item For $b\in \cB$, let $x\in \cC_b$ be a singular point and $I_x \subset \widehat{\cO}_{\cC_b, x}$
    the conductor ideal. Then $\widehat{\cO}_{\cC_b, x}/I_x$ is isomorphic to $k[t]/t^m$ for some $m\geq 1$.
\end{enumerate}
Then the functor (\ref{induce:Phi0}) induces an equivalence
\begin{align}\label{equiv:WhitL}
    \Phi \colon \IndCoh_{\mathcal{N}}(\cH(w)^{\mathrm{ss}})_{-\chi'} \stackrel{\sim}{\to} \IndL_{\mathcal{N}}(\cH(\chi))_{w'}.
\end{align}
\end{thm}

\begin{remark}\label{rmk:thm1}
Since we have equivalences (\cite[Remark~2.14]{TodaGL2})
\begin{align*}
 \IndCoh_{\mathcal{N}}(\cH(w)^{\mathrm{ss}})_{-\chi'} &\simeq
    \IndCoh_{\mathcal{N}}(\cH(w)^{\mathrm{ss}})_{-\chi}, \\
 \IndL_{\mathcal{N}}(\cH(\chi))_{w'} &\simeq \IndL_{\mathcal{N}}(\cH(\chi))_{w}
\end{align*}
the equivalence (\ref{equiv:WhitL}) implies that DL conjecture holds over $\cB$. 
\end{remark}

\begin{remark}\label{rmk:thm2}
The condition (ii) in Theorem~\ref{thm:WhitGL} is automatic 
for $b\in \rB_{\GL_r}^{A, \mathrm{cl}}$, see~\cite[Example~4.19]{TodaGL2}. 
\end{remark}

\begin{remark}\label{rmk:thm3}
    If the DL conjecture holds over an open subset 
    $\cB \subset \rB_{\GL_r}^{\mathrm{cl}}$, then the DL conjecture holds 
    over $\cB \times k[-1] \subset \rB_{\GL_r}$ by taking $\otimes \QCoh(k[-1])$, 
    see~\cite[Remark~2.15]{TodaGL2} and Lemma~\ref{lem:equiv:k}. 
    Therefore the DL conjecture over $\rB_{\GL_r}^A \subset \rB_{\GL_r}$
    is reduced to showing Conjecture~\ref{conj:Whit} for $\cB=\rB_{\GL_r}^{A, \mathrm{cl}}$. 
\end{remark}

\subsection{The stack of Hecke correspondences}
For $G=\GL_r$, let
$
    \mathrm{Hecke}_{\GL_r}
$
be the derived stack which 
classifies exact sequences 
\begin{align}\label{E:hecke}
    0\to E_1 \to E_2 \to \mathcal{O}_x \to 0
\end{align}
where each $E_i \in \Coh^{\heartsuit}(S)$ is a compactly supported pure 
one-dimensional sheaf and $x\in S$. 
We have the evaluation morphisms 
\begin{align*}
    \xymatrix{
\mathrm{Hecke}_{\GL_r} \ar[r]^-{\ev_2^{\sharp}} \ar[d]_{(\ev_3^{\sharp}, \ev_1^{\sharp})} & 
\Hig_{\GL_r} \ar[d] \\
\mathfrak{M}(1) \times \Hig_{\GL_r} \ar[r] & \rB_{\GL_r}.
    }
\end{align*}
The stack $\mathfrak{M}_S(1)$ is the derived moduli stack of 
zero-dimensional sheaves on $S$ with length one. It is equivalent to 
\begin{align*}
    \mathfrak{M}(1) \simeq S\times k[-1] \times \bgm.
\end{align*}
The map $\ev_i^{\sharp}$ sends (\ref{E:hecke}) to $E_i$ and $\ev_3^{\sharp}$ 
sends it to $\mathcal{O}_x$. 

We consider the base-change which removes the $k[-1]$-factor of $\mathfrak{M}_S(1)$
\begin{align*}
\mathrm{Hecke}_{\GL_r}':=\mathrm{Hecke}_{\mathrm{GL}_r}\times_{\mathfrak{M}(1)}(S\times \bgm). 
\end{align*}
It admits the following evaluation maps 
\begin{align}\label{dia:hecke'}
    \xymatrix{
\mathrm{Hecke}_{\GL_r}' \ar[r]^-{\ev_2} \ar[d]_{(\ev_3, \ev_1)} & 
\Hig_{\GL_r} \ar[d] \\
S\times \Hig_{\GL_r} \ar[r] & \rB_{\GL_r}.
    }
\end{align}
Here $\ev_1, \ev_2$ are induced from $\ev_1^{\sharp}, \ev_2^{\sharp}$, and $\ev_3$ is the composition of the projection to $S\times \bgm$ with the projection onto $S$. We will use the following lemma: 
\begin{lemma}\emph{(\cite[Lemma~3.1]{TodaGL2})}\label{lem:qsmooth}
    The maps $\ev_1, \ev_2$ are quasi-smooth and proper of relative virtual dimension $1$. The maps $(\ev_3, \ev_2)$, $(\ev_3, \ev_1)$ are also quasi-smooth and proper of relative virtual dimension $-1$. 
\end{lemma}

\section{Weight estimates of Arinkin sheaf}
In this section, we estimate the $\mathbb{G}_m$-weights of Arinkin sheaf (\ref{def:PE})
and show that it lies in the limit category if $E$ is semistable. 
The result of this section is a key step relating Arinkin's Cohen--Macaulay extension 
and the definition of limit categories. 

In what follows, we consider $\GL_r$-Higgs bundles, and use the notation (\ref{sH}) for 
an open subset $\cB \subset \rB_{\GL_r}^{\mathrm{cl}}$ and the base-change 
$\cH=\Hig_{\GL_r}\times_{\rB_{\GL_r}} \cB$. 
\subsection{Resolution of the Arinkin sheaf}
We consider the following open subset of $\rB_{\GL_r}^{\mathrm{cl}}$
\begin{align*}
   \cB= \cB^{A} :=\rB_{\GL_r}^{A, \mathrm{cl}}.
\end{align*}
Below we fix $b\in \cB$ and a torsion-free sheaf $E\in \Coh^{\heartsuit}(\cC_b)$ of rank-one. Let
\begin{align*}
    \cP_E \in \Coh^{\heartsuit}(\cH_b)
\end{align*}
be the corresponding Arinkin sheaf defined in (\ref{def:PE}). We first discuss the weight condition 
for the above object when $E$ is semistable, by constructing its explicit resolution by vector bundles. 

We denote by $\mathrm{Sing}(\cC_b) \subset \cC_b$ the set of singular points. 
Then for $p\in \mathrm{Sing}(\cC_b)$, 
the complete local ring $\widehat{\mathcal{O}}_{\cC_b, p}$ is isomorphic to 
\begin{align*}
    k[[x, y]]/(y^2-x^{m_p})
\end{align*}
for some $m_p \in \mathbb{Z}_{\geq 2}$. 
The conductor ideal at $p$ defines the closed subscheme 
\begin{align}\label{def:Zp}
    Z_p \cong \Spec k[t]/t^{m_p'} \hookrightarrow \cC_b
\end{align}
where $m_p'$ is the round-down of $m_p/2$, i.e.\ $m_p'=m_p/2$ when $m_p$ is even and $m_p'=(m_p-1)/2$ when $m_p$ is odd. 
We denote by 
\begin{align*}
    N \subset \mathrm{Sing}(\cC_b)
\end{align*}
the subset such that $E$ is not a line bundle at $N$. We recall the following 
structure result for rank-one torsion-free sheaves on $\cC_b$ with only type $A$ singularities:
\begin{lemma}\emph{(cf.~\cite[Lemma~5.2]{TodaGL2})}\label{lem:Qp}
There is an exact sequence 
\begin{align}\label{exact:LE}
    0\to \mathcal{L} \to E \to \bigoplus_{p\in N}Q_p \to 0
\end{align}
where $\mathcal{L}$ is a line bundle on $\cC_b$ and $Q_p$ is a zero-dimensional sheaf of length $n_p$, 
where $1\leq n_p \leq m_p'$, which is isomorphic to $\mathcal{O}_{W_p}$ for 
a closed subscheme $W_p \hookrightarrow Z_p$. 

In particular, for any $x\in \cC_b$ we have 
\begin{align}\label{dim:est}
    \dim \Hom(E, \mathcal{O}_x) \leq 2
\end{align}
and $\dim \Hom(E, \cO_x)=1$ if $x\in \cC_b$ is a smooth point.  
\end{lemma}
\begin{proof}
The assertion is local at the singular points and follows from the classification of rank-one torsion-free modules over the complete local rings 
$k[[x,y]]/(y^2-x^m)$. Indeed from~\cite[Lemma~4.26]{TodaGL2}, there is a line bundle $\cL$ on $\cC_b$
and an injection $\cL \hookrightarrow E$, which is an isomorphism at 
the locus where $E$ is a line bundle on $\cC_b$, such that 
$E/\cL \subset g_{*}g^* \cL/\cL$ for the normalization $g \colon \widetilde{\cC}_b \to \cC_b$. 
Then $g_{*}g^* \cL/\cL$ is supported at $\mathrm{Sing}(\cC_b)$, and 
$(g_{*}g^*\cL/\cL)_{p}$ is isomorphic to $k[t]/t^{m_p'}$. The lemma follows as $E/\mathcal{L}$ is a subsheaf of $g_{*}g^* \cL/\cL$. 

From the exact sequence (\ref{exact:LE}), since each $Q_p$ is locally generated by one element as an $\mathcal{O}_{\cC_b}$-module, 
the sheaf $E$ is locally generated at $p\in \cC_b$ 
by two elements as a $\cO_{\cC_b}$-module. Therefore we obtain the dimension estimate (\ref{dim:est}). 
\end{proof}

For $p\in \mathrm{Sing}(\cC_b)$ let $c=\pi(p) \in C$, where $\pi$ is the projection (\ref{project:pi}). 
We have the rank $r$-vector bundle, given by the restriction of the universal rank $r$ vector bundle (\ref{univ:F}) to $\{c\}\times \cH_b$:
\begin{align*}
    \cF_c=\cF|_{c\times \cH_b} \in \Coh^{\heartsuit}(\cH_b) 
\end{align*}
There is a rank two sub-vector bundle (indeed a direct summand)
\begin{align}\label{def:Vp}
    \mathcal{V}_p \subset \mathcal{F}_c
\end{align}
corresponding to the generalized eigenspace of the Higgs field with 
generalized eigenvalue $p \in \Omega_{C}|_{c}$. In other words
for $M\in \Coh^{\heartsuit}(\cC_b)$ corresponding to a point in $\cH_b$, 
we have 
\begin{align*}
    \mathcal{V}_p|_{M}=(M\otimes \mathcal{O}_{f_c})_p \subset \cF_c|_{M}=M\otimes \cO_{f_c}. 
\end{align*}
Here $f_c=\mathbb{A}^1$ is the fiber of $\pi \colon S \to C$ at $c$.

Let $\mathcal{L}$ and $n_p$ be as in Lemma~\ref{lem:Qp}. Note that 
the Arinkin sheaf $\cP_{\mathcal{L}}$ is a line bundle on $\cH_b$. 
We define the following 
vector bundle of rank $2^{\sum_{p\in N}n_p}$ on $\cH_b$
\begin{align}\label{def:VE}
    \mathcal{V}_E:=\mathcal{P}_{\mathcal{L}}\otimes \bigotimes_{p\in N}\mathcal{V}_p^{\otimes n_p}.
\end{align}

We have the following proposition, which generalizes~\cite[Corollary~5.12]{TodaGL2}.
Its proof will be completed in Subsection~\ref{subsec:exAr}. 
Since this is long and technical, we first explain an idea of the proof.
The details of the proof will be given in the latter part of this section, 
and will be finished in Subsection~\ref{subsec:exAr}.
\begin{prop}\label{prop:resol}
The sheaf $\cP_E$ admits a left resolution of the following form for some $k_i \in \mathbb{Z}_{\geq 0}$
    \begin{align}\label{resol:VE}
        \cdots \to \mathcal{V}_E^{\oplus k_1} \to \mathcal{V}_E^{\oplus k_0} \to \mathcal{P}_E \to 0.
    \end{align}
\end{prop}
\begin{proof}
\textit{Idea of the proof.}
For simplicity, we explain the simplest case, that is $\cL=\cO_{\cC_b}$, $N=\{p\}$
and $\cC_b$ has a nodal singularity at $p$ (i.e.\ $m_p=2$), 
so that we have the exact sequence 
\begin{align}\label{exact:OEOp}
    0\to \cO_{\cC_b} \to E \to \cO_{p} \to 0
\end{align}
In this case, we show the existence of an exact sequence in $\Coh^{\heartsuit}(\cH_b)$
\begin{align}\label{exact:Vp}
0\to \cP_E \to \mathcal{V}_E \to \cP_E \to 0.
\end{align}
Indeed, then we have the left resolution 
\begin{align*}
    \cdots \to \mathcal{V}_E \to \mathcal{V}_E \to \cP_E \to 0.
\end{align*}

For a line bundle $F$ on $\cC_b$, by the Deligne pairing (\ref{def:Psharp}) and the exact sequence (\ref{exact:OEOp}),
we have the following identification of the fiber of $\cP_E$ at $F$: 
    \begin{align}\label{left:res}
        \mathcal{P}_E|_{F}=F|_{p}.
    \end{align}
    By the assumption on the singularity of $\cC_b$, 
    we have
    \begin{align*}
        2p:=\pi|_{\cC_b}^{-1}(c)_p \cong \Spec k[t]/(t^2).
    \end{align*}
     Indeed, the fiber over $c$ is obtained by setting $x=0$ in the local model (\ref{loc:model}), hence it is $\Spec k[t]/(t^2)$.
    By the definition of $\mathcal{V}_E$, we have 
    \begin{align*}
        \mathcal{V}_E|_{F}=F|_{2p}.
    \end{align*}
Then from the exact sequence 
\begin{align*}
0\to \cO_p \to \cO_{2p} \to \cO_p \to 0
\end{align*}
there is an exact sequence 
\begin{align}\label{exact:F}
    0\to \cP_E|_{F} \to \mathcal{V}_E|_{F} \to \cP_E|_{F} \to 0.
\end{align}
Namely we have the exact sequence (\ref{exact:Vp}) on $\cH_b^{\mathrm{reg}}$. 

On the other hand, for $x\in \cC_b$, let $E_x=I_x^{\vee}$ be the dual of the ideal sheaf $I_x \subset \cO_{\cC_b}$, which fits into 
an exact sequence 
    \begin{align*}
        0\to \mathcal{O}_{\cC_b} \to E_x \to \mathcal{O}_x \to 0.
    \end{align*}
    The above construction gives a family of objects of $\Coh^{\heartsuit}(\cC_b)$ parametrized by $x\in \cC_b$, 
    and hence there is a map $\cC_b \to \cH_b$ sending $x\in \cC_b$ to $E_x$. By restricting it to $2p \hookrightarrow \cC_b$, we obtain the map $2p \to \cH_b$. By pulling back the Arinkin sheaf $\cP$ on 
    $\cH\times_{\cB} \cH$ to $2p \times_{\cB}\cH$, and pushing-forward to $\cH_b$, we obtain the Cohen--Macaulay sheaf 
    \begin{align*}
        \widetilde{\cP}_E \in \Coh^{\heartsuit}(\cH_b). 
    \end{align*}
    It fits into the exact sequence 
    \begin{align}\label{exact:tilde}
        0\to \cP_E \to \widetilde{\cP}_E \to \cP_E \to 0.
    \end{align}
   From the construction, we have $\widetilde{\cP}_E|_{\cH_b^{\mathrm{reg}}}\cong \mathcal{V}_E|_{\cH_b^{\mathrm{reg}}}$
   and the above exact sequence is identified with (\ref{exact:F}) over $\cH_b^{\mathrm{reg}}$. 

   Then if the complement of $\cH_b^{\mathrm{reg}} \subset \cH_b$ were of codimension at least two, then 
   we have $\widetilde{\cP}_E=\mathcal{V}_E$ by the uniqueness of Cohen--Macaulay extension. In particular, we have the exact sequence (\ref{exact:Vp}) by (\ref{exact:tilde}). 
   However, although $\cH_b^{\mathrm{reg}} \subset \cH_b$ is dense (see~\cite[Remark~2.5]{TodaGL2}), its complement may be of codimension one; instead we take a flat deformation $E_{\Delta}$
of $E$ over a smooth affine $0\in \Delta$ with $\dim \Delta>0$ such that the induced map 
\begin{align*}\Delta \to \cH \to \cB
\end{align*}
sends the generic point to the locus $\cB^{\mathrm{sm}} \subset \cB$ corresponding to the smooth 
spectral curves. Then since $\cH^{\mathrm{reg}}\times_{\cB}\cB^{\mathrm{sm}}=\cH\times_{\cB}\cB^{\mathrm{sm}}$, $\cH_{\Delta}\to \Delta$ is flat, and $\cH_b^{\mathrm{reg}} \subset \cH_b$ is dense, 
the complement of $\cH_{\Delta}^{\mathrm{reg}} \subset \cH_{\Delta}$ is of codimension 
at least two. Here $(-)_{\Delta}$ is the base change $(-)\times_{\cB}\Delta$. 
We then apply the above construction for 
$E_{\Delta}$, restrict it to $0\in \Delta$, and conclude that $\widetilde{\cP}_E=\mathcal{V}_E$.
    \end{proof}

\subsection{The weight estimate of the Arinkin sheaf}
For $b\in \cB^A$, let $E\in \Coh^{\heartsuit}(\cC_b)$ be a rank-one 
torsion-free sheaf. 
In this subsection, by admitting Proposition~\ref{prop:resol}, we show that $\cP_E$ lies in the limit category, 
by investigating the $\mathbb{G}_m$-weights of the vector bundle $\mathcal{V}_E$ defined by (\ref{def:VE}). This result will be used 
in the next section to show that $\Phi(\mathcal{O}_{\cH(w)^{\mathrm{ss}}})$
lies in the image of $!$-push-forward from the limit category of the semistable locus. 

Below we often regard an object in $\Coh^{\heartsuit}(\cH_b)$
as an object in $\Coh^{\heartsuit}(\cH)$ by the $*$-push-forward along 
$i_b \colon \cH_b \hookrightarrow \cH$, e.g.\ we regard $\cP_E, \mathcal{V}_E \in \Coh^{\heartsuit}(\cH)$. 
We also set 
\begin{align*}w' :=\deg_{\cC_b}(E)=\chi(E)-\chi(\mathcal{O}_{\cC_b}).
\end{align*}
Recall the category $\widetilde{\LL}(\cH)_{w'}$ in Definition~\ref{def:Ltilde} (also see (\ref{def:Ltilde:circ})). We have the 
following proposition: 
\begin{prop}\label{prop:VE}
If $E\in \Coh^{\heartsuit}(\cC_b)$ is semistable, we have 
\begin{align}\label{VE:semistable}
\mathcal{V}_E \in \widetilde{\LL}(\cH)_{w'}. 
\end{align}
\end{prop}
\begin{proof}
We take the line bundle $\mathcal{L}$ and $Q_p$ for $p\in N$ with length $n_p$ as in Lemma~\ref{lem:Qp}.
We divide the proof into 3 steps. 
\begin{step}
    Reduction to the weight inequality.
\end{step}
Let $\nu$ be a map 
\begin{align*}
    \nu \colon \bgm \to \cH
\end{align*}
corresponding to the object 
\begin{align*}
M=M_1 \oplus \cdots \oplus M_k, \ M_i \in \Coh^{\heartsuit}(\cC_b)
\end{align*}
such that $r_i:=\operatorname{rank}(\pi_{*}M_i)$
satisfy $r_1+\cdots+r_k=r$, and 
with $\mathbb{G}_m$-weight
\begin{align*}(\nu_1, \ldots, \nu_k), \ 
\nu_1>\cdots>\nu_k.
\end{align*}
Let $\delta_{w'}$ be the $\mathbb{Q}$-line bundle on $\cH$ defined as in (\ref{deltaw})
\begin{align*}
    \delta_{w'}=\det (\cF|_{c\times \cH})^{w'/r}.
\end{align*}
Since $\mathcal{V}_E$ is perfect on $\cH$ (as it is a push-forward of a vector bundle under the regular closed immersion $i_b \colon \cH_b \hookrightarrow \cH$),
it is enough to show the following by Lemma~\ref{lem:perfect} and Lemma~\ref{lem:eta}
\begin{align}\label{wt:V1}
    \mathrm{wt}(\nu^* \mathcal{V}_E) \subset \left[\frac{1}{2}c_1(\nu^* \mathbb{L}_{\cH}^{<0}), 
    \frac{1}{2}c_1(\nu^* \mathbb{L}_{\cH}^{>0})\right]+c_1(\nu^* \delta_{w'}).
\end{align}

Note that we have 
\begin{align*}
    \cC_b=C_1 \cup \cdots \cup C_k, \ C_i:=\mathrm{Supp}(M_i)
\end{align*}
where $C_i$ is a reduced curve, which is not necessarily irreducible. 
We set 
\begin{align*}
    l_i:=\deg_{C_i}(\mathcal{L}|_{C_i})=\chi(\mathcal{L}|_{C_i})-\chi(\mathcal{O}_{C_i}).
\end{align*}
Since we have (using Lemma~\ref{lem:add})
\begin{align*}
    w'=\deg_{\cC_b}(\mathcal{L})+\sum_{p\in N}\mathrm{length}(Q_p)
    =\sum_{i=1}^k l_i+\sum_{p\in N}n_p
\end{align*}
we have the equality 
\begin{align*}
    c_1(\nu^* \delta_{w'})=\frac{1}{r}\left(\sum_{i=1}^k l_i+\sum_{p\in N}n_p  \right)
  \sum_{i=1}^k r_i \nu_i.
\end{align*}
We also have the equality by Lemma~\ref{lem:Gmwt:L}
\begin{align*}
    c_1(\nu^* \mathbb{L}_{\cH}^{<0})=\sum_{i>j} (2g-2)r_i r_j(\nu_i-\nu_j).
\end{align*}

A $\nu$-weight of $\mathcal{V}_E$ is of the form 
\begin{align}\label{wt:VE}
    \sum_{i=1}^k l_i \nu_i+\sum_{p\in N}(a_p\nu_{\alpha(p)}+b_p\nu_{\beta(p)})
\end{align}
where $1\leq \alpha(p) \leq \beta(p) \leq k$ satisfies $p \in C_{\alpha(p)}\cap C_{\beta(p)}$, 
and $a_p, b_p$ satisfy 
\begin{align*}
    a_p+b_p=n_p, \ a_p, b_p \geq 0.
\end{align*}
Note that the pair $(\alpha(p), \beta(p))$ exists uniquely because $p\in \cC_b$ is of type $A$ (i.e.\ there are at most two local branches). Here the first term of (\ref{wt:VE}) is the $\nu$-weight of $\mathcal{P}_{\mathcal{L}}$ by Lemma~\ref{lem:PLGm}.
The second term at each $p\in N$ is a weight of $\mathcal{V}_p^{\otimes n_p}$ since the set of $\nu$-weights of $\mathcal{V}_p$ is $\{\nu_{\alpha(p)}, \nu_{\beta(p)}\}$. 
We also note that, when $\mathcal V_E$ is regarded as an object of $\Coh(\cH)$, it means 
$(i_b)_*\mathcal V_E$ for $i_b\colon \cH_b\hookrightarrow \cH$. Since $i_b$ is obtained by base change from the regular closed immersion 
$\{b\}\hookrightarrow \cB$, the object $i_b^* (i_b)_{*}\mathcal{V}_E$ is generated by $\mathcal{V}_E$ and other Koszul factors, which are direct sums of $\mathcal{V}_E$ up to shift. Hence they do not affect the above weight estimate of $\nu^* \mathcal{V}_E$.

Then the lower bound in (\ref{wt:V1}) is equivalent to 
\begin{align} \label{ubound1}
    &\sum_{i>j}(g-1)r_i r_j (\nu_i-\nu_j)  \\
  \notag &\leq \sum_{i=1}^k l_i \nu_i+\sum_{p\in N}(a_p\nu_{\alpha(p)}+b_p\nu_{\beta(p)})-\frac{1}{r}\left(\sum_{i=1}^k l_i+\sum_{p\in N}n_p  \right)
  \sum_{i=1}^k r_i \nu_i.
\end{align}
\begin{step}
    The inequalities from the semistability of $E$.
\end{step}
Now suppose that $E$ is semistable. We will derive the inequality (\ref{ubound1})
from the semistability of $E$. 

For a subset $I \subset \{1, \ldots, k\}$, 
we write 
\begin{align*}
    C_I:=\bigcup_{i\in I} C_i, \ r_I:=\sum_{i\in I} r_i, \ 
    l_I:=\sum_{i\in I} l_i=\deg_{C_I}(\mathcal{L}|_{C_I}).
\end{align*}
Here the last identity follows from Lemma~\ref{lem:add}.
We also write $I^{\circ}$ for the complement of $I$ in $\{1, \ldots, k\}$. 
By Lemma~\ref{lem:torsion} below, 
we have the exact sequence 
\begin{align*}
    0\to \mathcal{L}|_{C_I} \to E|_{C_I}\to \bigoplus_{p\in N \cap C_I}Q_p \to 0
\end{align*}
such that the torsion part of $E|_{C_I}$ is 
\begin{align*}
    \bigoplus_{p\in N \cap C_I \cap C_{I^{\circ}}}Q_p \subset E|_{C_I}. 
\end{align*}
Let $E|_{C_I}\twoheadrightarrow E|_{C_I}^{\mathrm{free}}$ be the torsion-free quotient of $E|_{C_I}$. Then we have the 
exact sequence 
\begin{align}\label{ex:E1}
    0\to \mathcal{L}|_{C_I} \to E|_{C_I}^{\mathrm{free}} \to \bigoplus_{p\in (N\cap C_I) \setminus C_{I^{\circ}}}Q_p \to 0. 
\end{align}
The kernel $E'$ of $E\twoheadrightarrow E|_{C_I}^{\mathrm{free}}$
fits into the commutative diagram of exact sequences
\begin{equation}\label{com:exE}
\xymatrix@C=4em{
0 \ar[r]
&
\mathcal{L}|_{C_{I^\circ}}(-C_I)
\ar[r]
\ar@{^{(}->}[d]
&
\mathcal{L}
\ar[r]
\ar@{^{(}->}[d]
&
\mathcal{L}|_{C_I}
\ar[r]
\ar@{^{(}->}[d]
&
0
\\
0 \ar[r]
&
E'
\ar[r]
&
E
\ar[r]
&
E|_{C_I}^{\mathrm{free}}
\ar[r]
&
0
}
\end{equation}
By taking the cokernels, we obtain the exact sequence 
\begin{align}\label{ex:E2}
    0\to \mathcal{L}|_{C_{I^{\circ}}}(-C_I)\to E' \to \bigoplus_{p\in N\cap C_{I^{\circ}}}Q_p  \to 0. 
\end{align}
Here we have used that 
\begin{align*}
  N \setminus ((N\cap C_I)\setminus C_{I^{\circ}})=N\cap C_{I^{\circ}}.  
\end{align*}

By the bottom exact sequence (\ref{com:exE})
and the semistability of $E$, we have the inequality 
\begin{align}\label{ineq:rII}
    \frac{\deg \pi_{*}E'}{r_{I^{\circ}}} \leq \frac{\deg \pi_{*}E|_{C_I}^{\mathrm{free}}}{r_I}.
\end{align}
From (\ref{ex:E1}) and the formula (\ref{formula:deg}), we have 
\begin{align*}
\deg \pi_{*}E|_{C_I}^{\mathrm{free}}=l_I-(r_I^2-r_I)(g-1)+\sum_{p\in (N\cap C_I)\setminus C_{I^{\circ}}}n_p
\end{align*}
and from (\ref{ex:E2}) we have 
\begin{align*}
    \deg \pi_{*}E'=l_{I^{\circ}}-(r_{I^{\circ}}^2-r_{I^{\circ}})(g-1)-C_I \cdot C_{I^{\circ}}+\sum_{p\in N \cap C_{I^{\circ}}}n_p.
\end{align*}
By noting that 
$C_I \cdot C_{I^{\circ}}=(2g-2)r_I r_{I^{\circ}}$, 
the inequality (\ref{ineq:rII}) is simplified as 
\begin{align}\label{ineq:lI}
    \frac{1}{r_{I^{\circ}}}\left(\sum_{p\in N \cap C_{I^{\circ}}}n_p  \right)
    -\frac{1}{r_I}\left(\sum_{p\in (N\cap C_I)\setminus C_{I^{\circ}}}n_p  \right) 
    \leq \frac{l_I}{r_I}-\frac{l_{I^{\circ}}}{r_{I^{\circ}}}+(g-1)r.
\end{align}

Let $\widetilde{l}_i$ be defined by 
\begin{align*}
    \widetilde{l}_i:=l_i -\frac{r_i}{r}\sum_{j=1}^k l_j.
\end{align*}
Then we have $\widetilde{l}_1+\cdots +\widetilde{l}_k=0$, and 
the inequalities (\ref{ubound1}), (\ref{ineq:lI}) are unchanged after replacing $l_i$ with $\widetilde{l}_i$. 
Therefore we may assume that 
\begin{align}\label{sum:l:VE}l_1+\cdots+l_k=0.
\end{align}
Then by substituting $l_{I}+l_{I^{\circ}}=0$ to (\ref{ineq:lI}), and multiplying $r_I r_{I^{\circ}}/r$,
we obtain 
\begin{align}\label{lI:bound}
    \frac{r_I}{r}\left(\sum_{p\in N \cap C_{I^{\circ}}}n_p  \right)-\frac{r_{I^{\circ}}}{r}
\left(\sum_{p\in (N\cap C_I)\setminus C_{I^{\circ}}}n_p  \right) -(g-1)r_I r_{I^{\circ}}\leq l_I.
\end{align}
By setting $\mu_i:=\nu_i-\nu_{i+1}>0$, by (\ref{sum:l:VE}) the right-hand side in (\ref{ubound1}) is 
\begin{align*}
\sum_{i=1}^{k-1}(l_1+\cdots+l_i)\mu_i+\sum_{p\in N}(a_p\nu_{\alpha(p)}+b_p\nu_{\beta(p)})-
\frac{1}{r}\sum_{p\in N}n_p 
  \sum_{i=1}^k r_i \nu_i. 
\end{align*}
By setting $I_i=\{1, \ldots, i\}$ and substituting (\ref{lI:bound}), the right-hand side in (\ref{ubound1}) is at least
\begin{align*}
   & -\sum_{i=1}^{k-1} (g-1)r_{I_i}r_{I_i^{\circ}} \mu_i+\sum_{p\in N}(a_p\nu_{\alpha(p)}+b_p\nu_{\beta(p)})-\frac{1}{r}\sum_{p\in N}n_p 
  \sum_{i=1}^k r_i \nu_i \\
&+\sum_{i=1}^{k-1}\left(\frac{r_{I_i}}{r}\left(\sum_{p\in N \cap C_{I_i^{\circ}}}n_p  \right) 
  -
    \frac{r_{I_i^{\circ}}}{r}\left(\sum_{p\in (N \cap C_{I_i})\setminus C_{I_i^{\circ}}}n_p\right)\right)\mu_i.
\end{align*}
Since we have 
\begin{align*}
        \sum_{i<j}(g-1)r_i r_j (\nu_i-\nu_j)= \sum_{i=1}^{k-1} (g-1)r_{I_i}r_{I_i^{\circ}} \mu_i
\end{align*}
the inequality (\ref{ubound1}) follows if the following inequality holds: 
\begin{align}\label{ineq:lambdap}
   & \sum_{p\in N}(a_p\nu_{\alpha(p)}+b_p\nu_{\beta(p)})-\frac{1}{r}\sum_{p\in N}n_p  
  \sum_{i=1}^k r_i \nu_i \\
&\notag \geq \sum_{i=1}^{k-1}\left(-\frac{r_{I_i}}{r}\left(\sum_{p\in N \cap C_{I_i^{\circ}}}n_p  \right) 
  +
    \frac{r_{I_i^{\circ}}}{r}\left(\sum_{p\in (N \cap C_{I_i})\setminus C_{I_i^{\circ}}}n_p\right)\right)\mu_i.
\end{align}

\begin{step}
    Proof of the inequality (\ref{ineq:lambdap}).
\end{step}
In order to prove (\ref{ineq:lambdap}), it is enough to prove this for each $p\in N$, 
so we may assume that $N=\{p\}$.
Recall that $\alpha(p)\leq \beta(p)$, therefore 
$\nu_{\beta(p)} \leq \nu_{\alpha(p)}$. Then the left-hand side in (\ref{ineq:lambdap}) is at least
\begin{align}\label{lambda:jp}
   n_p \nu_{\beta(p)}-\frac{n_p}{r}
\sum_{i=1}^k r_i \nu_i=n_p\left(\sum_{i=\beta(p)}^{k-1}\frac{r_{I_i^{\circ}}}{r}\mu_i -\sum_{i=1}^{\beta(p)-1}\frac{r_{I_i}}{r}\mu_i\right).
\end{align}
On the other hand, we have 
\begin{align*}
    p \in N \cap C_{I_i^{\circ}} \ \mbox{ if and only if } \ \beta(p) > i
\end{align*}
and 
\begin{align*}
    p\in (N \cap C_{I_i})\setminus C_{I_i^{\circ}}  \ \mbox{ if and only if } \ \beta(p) \leq i.
\end{align*}
It follows that (\ref{lambda:jp}) equals the right-hand side in (\ref{ineq:lambdap}), therefore we have the inequality (\ref{ineq:lambdap}). We conclude the lower bound in (\ref{wt:V1}). The upper bound is proved similarly. 
\end{proof}

For later use, we record the following lemma, which follows from the above proof. 
\begin{lemma}\label{lem:lower}
In the proof of Proposition~\ref{prop:VE}, the lower bound in (\ref{wt:V1}) is achieved only if 
the inequality (\ref{ineq:rII}) is an equality for all $I=I_i$ with $1\leq i\leq k-1$.  
\end{lemma}
\begin{proof}
    If (\ref{ineq:rII}) is strict for $I=I_i$, then the inequality (\ref{lI:bound}) is also strict, 
    therefore the inequality (\ref{ineq:lambdap}) is also strict. Therefore (\ref{ubound1}) is strict, and the lower bound in (\ref{wt:V1}) is not achieved. 
\end{proof}

We have used the following lemma: 
\begin{lemma}\label{lem:torsion}
In the setting of Lemma~\ref{lem:Qp}, let $\cC_b=C'\cup C''$ be a decomposition 
where $C', C''$ are unions of irreducible components (which do not have common irreducible components since $\cC_b$ is reduced). Then 
the classical $*$-restriction $E|_{C'} \in \Coh^{\heartsuit}(C')$ fits into 
the exact sequence 
\begin{align}\label{exact:L}
    0\to \mathcal{L}|_{C'} \to E|_{C'} \to \bigoplus_{p\in N\cap C'}Q_p \to 0
\end{align}
such that the torsion part of $E|_{C'}$ is $\oplus_{p\in N\cap C'\cap C''}Q_p$. 
\end{lemma}
\begin{proof}
    The exact sequence (\ref{exact:L}) follows from the long exact sequence 
    \begin{align*}
        \cdots \to \mathrm{Tor}_1^{\mathcal{O}_{\cC_b}}(Q, \mathcal{O}_{C'}) \to 
        \mathcal{L}|_{C'} \to E|_{C'} \to Q|_{C'} \to 0
    \end{align*}
    and the first map is a zero map as $\mathcal{L}|_{C'}$ is a line bundle and 
    $ \mathrm{Tor}_1^{\mathcal{O}_{\cC_b}}(Q, \mathcal{O}_{C'})$ is a zero-dimensional sheaf. 

    As for the second statement, it is enough to show that (\ref{exact:L}) splits formally locally 
    at any $p\in C'\cap C''$. We set 
    \begin{align*}R=\widehat{\cO}_{\cC_b, p}, \ R'=\widehat{\cO}_{C', p}, \ 
    R''=\widehat{\cO}_{C'', p}, \ M=E\otimes_{\mathcal{O}_{\cC_b}}R.
    \end{align*}
    Since $p\in C'\cap C''$ and the singularity is of type $A$, it has two local branches at $p$; hence the normalization of $R$ is $R'\oplus R''$.
    By~\cite[Lemma~4.26]{TodaGL2} and the construction of the exact sequence (\ref{exact:LE}) in~\cite[Lemma~5.2]{TodaGL2}, we have 
    \begin{align*}
        R \subset M \subset R' \oplus R''.
    \end{align*}
    Let $\otimes_R^{\mathrm{cl}}$ be the classical tensor product. 
    Then $R' \to M\otimes_R^{\mathrm{cl}} R'$ is a split injection, with splitting given by 
    \begin{align*}
        M\otimes_R^{\mathrm{cl}} R' \to (R' \oplus R'')\otimes_R^{\mathrm{cl}} R'=R' \oplus (R' \otimes_R^{\mathrm{cl}} R'') \twoheadrightarrow R'.
    \end{align*}
\end{proof}

By Proposition~\ref{prop:resol} and Proposition~\ref{prop:VE}, we obtain the following corollary: 
\begin{cor}\label{cor:PE}
    If $E$ is semistable, we have 
    \begin{align*}
        \cP_E \in \widetilde{\LL}(\cH)_{w'}.
    \end{align*}
\end{cor}
\begin{proof}
The corollary follows from 
Proposition~\ref{prop:resol}, Proposition~\ref{prop:VE} and Lemma~\ref{lem:perfect2}.
Indeed, we have the colimit expression in $\QCoh(\cH)$
\begin{align*}
    \cP_E \cong \operatorname*{colim}_{m\to\infty}(\cdots \to 0 \to \mathcal{V}_E^{\oplus k_m} \to \cdots \to 
    \mathcal{V}_E^{\oplus k_0} \to 0 \to \cdots)
\end{align*}
by Proposition~\ref{prop:resol}, and each term satisfies the condition (\ref{wt:cond2})
by Proposition~\ref{prop:VE}. Therefore by Lemma~\ref{lem:perfect2}, 
the object $\cP_E$ satisfies the condition (\ref{wt:nu}). 
\end{proof}

\subsection{Construction of the double cover}\label{subsec:proof}
The rest of this section is devoted to a complete proof of Proposition~\ref{prop:resol}.
We first construct a double cover on an \'etale neighborhood of $C\times \cB$, 
which will be used to construct deformations of Arinkin sheaves. 

For a torsion-free sheaf $E\in \Coh^{\heartsuit}(\cC_b)$ of 
rank-one, recall the exact sequence (\ref{exact:LE}). 
We take $p\in N$ such that $\chi(Q_p)>0$ and a surjection 
$Q_p \twoheadrightarrow \cO_p$.
Then there is an exact sequence 
\begin{align}\label{exact:E1}
    0\to E' \to E \to \cO_p \to 0
\end{align}
where $E'$ fits into an exact sequence 
\begin{align}\label{exact:E2}
    0\to \cL \to E' \to \bigoplus_{q\in N}Q'_q\to 0
\end{align}
such that $Q'_q=Q_q$ for $q\neq p$
and $Q_p'$ is isomorphic to the structure 
sheaf of a closed subscheme of $Z_p$ with $\chi(Q_p')=\chi(Q_p)-1$.

For the above $p\in N$, let $c=\pi(p) \in C$. 
Let $\cC \to \cB$ be the universal spectral curve, and consider the projection 
\begin{align*}
    \pi \colon \cC \to C\times \cB.
\end{align*}
It is a finite flat morphism such that its fiber at $(c, b)$ decomposes into open and closed subschemes
\begin{align}\label{decompose:cb}
    \pi^{-1}(c, b)=\pi|_{\cC_b}^{-1}(c)_p \sqcup (\cdots).
\end{align}
Here $\pi|_{\cC_b}^{-1}(c)_p$ is isomorphic to 
\begin{align*}
    \pi|_{\cC_b}^{-1}(c)_p \cong \Spec k[t]/t^2
\end{align*}
by the type $A$ assumption on $p \in \cC_b$.
Therefore there is an \'{e}tale neighborhood 
\begin{align}\notag
(D, z) \to (C\times \cB, (c, b)) 
\end{align}
such that the Cartesian square 
\begin{align}\label{etale:D}
\xymatrix{
(\cD, q) \ar[r] \ar[d] \diasquare & (D, z) \ar[d] \\
(\cC, p) \ar[r] &(C\times \cB, (c, b))
}
\end{align}
decomposes into the open and closed subsets
\begin{align*}
\cD=\cD_p \sqcup (\cdots)
\end{align*}
which restricts to the decomposition (\ref{decompose:cb}) at $z$. 
Here $q\in \cD$ is determined by $(z, p, c, b)$ and it lies in $\cD_p$.
The composition  
\begin{align}\label{map:D}
\pi_D \colon 
    \cD_p \hookrightarrow \cD  \to D 
\end{align}
is a finite flat double cover with covering involution $\sigma$. 

\subsection{Deformations via Hecke correspondences}
In this subsection, we construct flat families of Cohen--Macaulay 
sheaves 
\begin{align*}
    \cP_{\Delta 1}, \cP_{\Delta 2}, \cP_{\Delta 2}^{\sigma}\in \Coh^{\heartsuit}(\cH_{\Delta})
\end{align*}
for a pointed smooth affine scheme $0\in \Delta$ over $\cB$, which 
restricts to $\cP_{E'}$, $\cP_E$, $\cP_E$ respectively. 
These sheaves will be used to construct an exact sequence relating 
$\cP_{E'}$, $\cP_E$ and $\mathcal{V}_p$. Below we will construct such deformations using the stack of Hecke correspondences. 

Let $\mathrm{Hecke}_{\GL}'$ be the (modified) stack of Hecke correspondences 
in the diagram (\ref{dia:hecke'}). 
We define $\cHecke$ to be 
\begin{align*}
    \cHecke:=\mathrm{Hecke}_{\GL_r}'\times_{\rB_{\GL_r}} \cB.
\end{align*}
We have the diagram of evaluation morphisms, given by the base-change of 
the diagram (\ref{dia:hecke'})
\begin{align*}
    \xymatrix{
\cHecke \ar[r]^{\ev_2} \ar[d]_{(\ev_3, \ev_1)} & \cH \ar[d] \\
S\times \cH \ar[r] & \cB.
    }
\end{align*}
The left vertical arrow factors through the closed 
immersion $\cC\times_{\cB}\cH \hookrightarrow S\times \cH$; 
we also use the same notation $\ev_3$ for the composition 
\begin{align*}
   \ev_3 \colon \cHecke \stackrel{(\ev_3, \ev_1)}{\to} \cC\times_{\cB}\cH 
   \to \cC
\end{align*}
where the last arrow is the projection. 

\begin{lemma}\label{lem:heckeclass}
The stack $\cHecke$ is classical, and the evaluation maps 
\begin{align*}
    \ev_i \colon \cHecke \to \cH
\end{align*}
for $i=1, 2$ are flat and surjective of relative dimension one. 
\end{lemma}

\begin{proof}
    By Lemma~\ref{lem:qsmooth}, the relative virtual dimension of $\ev_i$ is one. 
    On the other hand, let $p\in \cH(k)$ be a $k$-valued point, 
    and let $E\in \Coh^{\heartsuit}(\cC_b)$ be the corresponding 
    torsion-free sheaf. Consider the restriction of $\ev_3$
    \begin{align*}
     \ev_3' \colon    (\ev_2^{-1}(p))^{\mathrm{cl}} \to \cC_b. 
    \end{align*}
    Then, for each point $x\in \cC_b$, we have 
    \begin{align*}
        ((\ev_3')^{-1}(x))^{\mathrm{cl}}
        =\mathbb{P}(\Hom(E, \mathcal{O}_x)).
    \end{align*}
    By Lemma~\ref{lem:Qp}, this is either $\mathbb{P}^1$ or a point; 
    the former case occurs precisely when $E$ is not a line bundle at $x$. 
    It follows that $(\ev_2^{-1}(p))^{\mathrm{cl}}$ is one-dimensional, 
    which equals the relative virtual dimension of $\ev_2$. 
    Therefore we have
    \begin{align*}
        \dim \cHecke^{\mathrm{cl}}=\dim \cH+1=\operatorname{vdim}\cHecke.
    \end{align*}
    Hence $\cHecke$ is classical, and $\ev_2$ is flat and surjective of relative dimension one. 

    As for the statement for $\ev_1$, 
    let $\psi$ be the isomorphism given by the shifted derived dual on $S$:
    \begin{align*}
        \psi \colon \cH \stackrel{\cong}{\to} \cH, \quad E \mapsto E^{\vee},
    \end{align*}
    where $E^{\vee}=\mathbb{D}_S(E)[1]$. 
    Then we have the commutative diagram 
    \begin{align}\label{dia:heckedual}
        \xymatrix{
        \cHecke \ar[r]^-{\ev_2} \ar[d]^-{\psi}_-{\cong}  
        & \cH \ar[d]^-{\psi }_-{\cong} \\
        \cHecke  \ar[r]^-{\ev_1} & \cH.
        }
    \end{align}
    Here the left vertical arrow is given by taking the shifted derived dual over $S$:
    \begin{align*}
        (0\to E_1 \to E_2 \to \mathcal{O}_x \to 0) 
        \mapsto 
        (0\to E_2^{\vee} \to E_1^{\vee} \to \mathcal{O}_x \to 0).
    \end{align*}
    By the above diagram and the fact that $\ev_2$ is flat and surjective of relative dimension one, 
    it follows that $\ev_1$ is flat and surjective of relative dimension one. 
\end{proof}

We then consider the fiber product 
\begin{align}\label{dia:heckeD}
    \xymatrix{
    \cHecke_{\cD_p} \ar[r] \ar[d] \diasquare & \cHecke \ar[d]^-{\ev_3} \\
    \cD_p \ar[r] & \cC.
    }
\end{align}
Here the bottom arrow is the composition $\cD_p \subset \cD \to \cC$, which is \'{e}tale 
by the construction. 
Since $q\in \cD_p$ is a lift of $p\in \cC_b$, the exact sequence (\ref{exact:E1}) 
together with $q\in \cD_p$ 
corresponds to a map 
\begin{align}\label{map:tau0}
    \tau_0 \colon \Spec k \to \cHecke_{\cD_p}.
\end{align}
The following lemma is a generalization of~\cite[Lemma~5.3]{TodaGL2}.
\begin{lemma}\label{lem:Delta}
For each $\ell \geq 1$, there is a pointed smooth and connected affine variety 
    $(\Delta, 0)$ for $0\in \Delta$ with $\dim \Delta=\ell$ and an extension of the map (\ref{map:tau0})
    \begin{align}\label{map:Delta}
        \tau \colon \Delta \to \cHecke_{\cD_p}, \ \tau|_{0}=\tau_0
    \end{align}
    such that the composition 
    \begin{align}\label{map:DeltaB}
\Delta \to \cHecke_{\cD_p} \to \cHecke \to \cB
    \end{align}
    sends the generic point of $\Delta$ to $\cB^{\mathrm{sm}} \subset \cB$, the open 
    subset corresponding to smooth spectral curves. 
\end{lemma}
\begin{proof}
    It is enough to show that each irreducible component of $\cHecke_{\cD_p}$ which contains the image of $\tau_0$
    is dominant over $\cB$
    under the map $\cHecke_{\cD_p} \to \cB$, following the first part of the proof of~\cite[Lemma~5.3]{TodaGL2}. Indeed suppose that the above statement holds. Let $W\to \cHecke_{\cD_p}$ be an atlas with a point $p'\in W$ 
which maps to $\tau_0$, and take an irreducible component $p' \in W' \subset W$. 
We take a resolution of singularities $W'' \to W'$ and a lift $p'' \in W''$ of $p'$. Then we take a smooth affine open neighborhood 
$p''\in W'''$; by the construction $W''' \to \cB$ is dominant. A desired $(\Delta, 0)$ is 
constructed by taking a generic hypersurface $0=p'' \in \Delta \subset W'''$
if $\ell<\dim W'''$, or if $\ell \geq \dim W'''$ take $\Delta=W''' \times \mathbb{A}^{\ell-\dim W'''}$, $0=(p'', 0)$ and construct a map $\Delta \to \cHecke_{\cD_p}$ by 
composing with the projection $\Delta \to W'''$.

We have the factorization 
\begin{align}\label{factor:Hecke}
    \cHecke_{\cD_p} \to \cHecke \stackrel{\ev_2}{\to} \cH \stackrel{h}{\to} \cB.
\end{align}
The first map $\cHecke_{\cD_p} \to \cHecke$ is \'{e}tale, as it is obtained by 
an \'{e}tale base change from the diagram~\eqref{dia:heckeD}.
The second map $\ev_2$ is flat and surjective by Lemma~\ref{lem:heckeclass}.
The last map $h$ is also flat and surjective by Theorem~\ref{thm:higgs:basic}. 
It follows that the composition~\eqref{factor:Hecke} is flat. 
Since $\cB$ is irreducible, every irreducible component of 
$\cHecke_{\cD_p}$ dominates $\cB$ under a flat morphism. 
Therefore, for each irreducible component 
$\mathcal{Z} \subset \cHecke_{\cD_p}$, the morphism 
$\mathcal{Z} \to \cB$ is dominant.

\end{proof}

Below, we use the notation $(-)_{\Delta}$ for the base change $\times_{\cB}\Delta$, where $\Delta \to \cB$ is the map (\ref{map:DeltaB}). 
For example 
\begin{align*}
    \cC_{\Delta}:=\cC\times_{\cB} \Delta, \ \cD_{\Delta}:=\cD\times_{\cB} \Delta, \ \cH_{\Delta}:=\cH\times_{\cB}\Delta.
\end{align*}
The map (\ref{map:Delta}) corresponds to an exact sequence 
in $\Coh^{\heartsuit}(\cC_{\Delta})$
\begin{align}\label{exact:EQ}
    0\to E_{\Delta}'\to E_{\Delta} \to R_{\Delta}\to 0
\end{align}
such that each term is flat over $\Delta$, and the sequence restricts to the 
sequence (\ref{exact:E1}) at $0\in \Delta$, together 
with a lift of the support of $R_{\Delta}$ to $\cD_{p}$. 

More precisely, the sheaf $R_{\Delta}$ in (\ref{exact:EQ}) is a $\Delta$-flat family of 
skyscraper sheaves of points on the fibers of $\cC_{\Delta} \to \Delta$, 
hence its support determines a section 
$\eta \colon \Delta \hookrightarrow \cC_{\Delta}$
of the projection $\cC_{\Delta} \to \Delta$. 
A lift of $\eta$ is a commutative diagram 
\begin{align}\label{dia:eta}
    \xymatrix{
    & (\cD_p)_{\Delta} \ar[d] \ar[dr] \\
    \Delta \ar@<-0.3ex>@{^{(}->}[ru]^-{\eta_D} \inclusion^-{\eta} & \cC_{\Delta} \ar[r] & \Delta. 
    }
\end{align}
Below by shrinking $\Delta$ if necessary, we may assume that 
\begin{align*}
    R_{\Delta} \cong \eta_{*}\mathcal{O}_{\Delta} \in \Coh^{\heartsuit}(\cC_{\Delta}).
\end{align*}

Let $\tau_i$ for $i=1, 2$ be the composition 
\begin{align}\label{def:taui}
\tau_i \colon
    \Delta \stackrel{\tau}{\to} \cHecke_{\cD_p} \to \cHecke  \stackrel{\ev_i}{\to}\cH.
\end{align}
The above maps define 
\begin{align*}
    \cP_{\Delta i}:=(\tau_i \times \id)^* \cP\in 
    \Coh^{\heartsuit}(\cH_{\Delta}), \ i=1, 2.
\end{align*}
Here $\cP$ is the Cohen--Macaulay sheaf in Theorem~\ref{thm:Pflat}. 
Since $\cP$ is flat over both factors of $\cH$, 
the sheaves $\cP_{\Delta i}$ are 
Cohen--Macaulay sheaves on $\cH_{\Delta}$ flat over $\Delta$, 
such that $\cP_{\Delta 2}|_{0}\cong \cP_E$ and 
$\cP_{\Delta 1}|_{0}\cong \cP_{E'}$. 

Recall that $\sigma$ is the covering involution on $\cD_p$ of the map (\ref{map:D}).
It induces an involution on $(\cD_p)_{\Delta}$. 
By applying $\sigma$ to the diagram (\ref{dia:eta}), we obtain the following 
commutative diagram 
\begin{align}\label{dia:eta:sigma}
    \xymatrix{
    & (\cD_p)_{\Delta} \ar[d] \ar[dr] \\
    \Delta \ar@<-0.3ex>@{^{(}->}[ru]^-{\eta_D^{\sigma}} \inclusion^-{\eta^{\sigma}} & \cC_{\Delta} \ar[r] & \Delta
    }
\end{align}
where $\eta_D^{\sigma} :=\sigma \circ \eta_D$ and $\eta^{\sigma}$ is defined by the above 
commutative diagram. We also define 
\begin{align*}
    R_{\Delta}^{\sigma}:=\eta^{\sigma}_{*}\cO_{\Delta} \in \Coh^{\heartsuit}(\cC_{\Delta}).
\end{align*}

\begin{lemma}\label{lem:sigma} After shrinking $\Delta$ and replacing the map (\ref{map:Delta}) if necessary, there is an exact sequence 
    \begin{align}\label{ext:sigma}
        0\to E_{\Delta}' \to E_{\Delta}^{\sigma} \to R_{\Delta}^{\sigma} \to 0
    \end{align}
    such that $E_{\Delta}^{\sigma}$ is a $\Delta$-flat family of rank-one torsion-free sheaves on the fibers of 
    $\cC_{\Delta}\to \Delta$. Moreover it is isomorphic to (\ref{exact:EQ}) at $0\in \Delta$ under the 
    identifications 
    \begin{align}\label{identify:sigma}R_{\Delta}^{\sigma}|_{0} \cong \mathcal{O}_{p} \cong 
    R_{\Delta}|_{0}.
\end{align}
\end{lemma}
\begin{proof}
    We define $\widetilde{\mathcal{H}ecke}$ to be the fiber product
    \begin{align*}
    \xymatrix{
       \widetilde{\mathcal{H}ecke} \ar[r] \ar[d] \ar@{}[dr]|{\square} & \mathcal{H}ecke \ar[d]^-{\ev_1} \\
       \mathcal{H}ecke \ar[r]_-{\ev_1} & \mathcal{H}.
    }
    \end{align*}
Since $\cHecke, \cH$ are classical and $\ev_1$ is flat by Lemma~\ref{lem:heckeclass}, the stack $\widetilde{\cHecke}$ is classical. 
    We then define $\widetilde{\cHecke}_{\cD_p}$ to be the fiber product 
    \begin{align}\label{hecke:dp}
        \xymatrix{
        \widetilde{\cHecke}_{\cD_p}\ar[r] \ar[d] \diasquare & \cD_p\times_{\cB} \cD_p \ar[d] \\
        \widetilde{\cHecke} \ar[r]^{(\ev_3, \ev_3)} & \cC \times_{\cB} \cC.
        }
    \end{align}
    It consists of diagrams in $\Coh^{\heartsuit}(\cC_b)$ for $b\in \mathcal{B}$
    \begin{align}\label{dia:xy}
    \xymatrix{
    0\ar[r] & E' \ar[r] \ar@{=}[d] & E \ar[r] & \mathcal{O}_x \ar[r] & 0 \\
    0\ar[r] & E' \ar[r] & \widetilde{E} \ar[r] & \mathcal{O}_y \ar[r] & 0
    }
    \end{align}
    where $x, y \in \cC_b$, each horizontal row is exact, together with their lifts $x', y' \in \cD_p$.
    Since the left vertical arrow in (\ref{hecke:dp}) is \'etale, the stack 
    $\widetilde{\cHecke}_{\cD_p}$ is also classical, and the map 
    \begin{align*}
        \widetilde{\cHecke}_{\mathcal{D}_p} \to \cH
    \end{align*}
    is flat of relative dimension two. 
    Let
    \begin{align*}
        \mathcal{Z} \subset \widetilde{\mathcal{H}ecke}_{\cD_p}
    \end{align*}
    be the classical closed substack consisting of \eqref{dia:xy} such that $y'=\sigma(x')$. It is enough to show that each irreducible component $\mathcal{Z}' \subset \mathcal{Z}$ is dominant over $\mathcal{B}$.

We show that $\mathcal{Z}'$ is dominant over $\cB$
following a similar argument of the proof of~\cite[Lemma~5.5]{TodaGL2}.
Consider the composition 
\begin{align}\label{map:ZH}
    \mathcal{Z} \subset \widetilde{\cHecke}_{\cD_p} \to \cHecke \to \cH. 
\end{align}
Then over the regular locus $\cH^{\mathrm{reg}} \subset \cH$, the fibers 
of the above map over $b\in \cB$ consist of $(x', y') \in \cD_p$, which lie over $b$, such that $y'=\sigma(x')$;
this follows since for a line bundle $E'$ on $\cC_b$ the diagram (\ref{dia:xy}) 
is uniquely determined by $x, y \in \cC_b$ up to isomorphisms, which 
are in turn determined by their lifts $x', y' \in \cD_p$. 
Therefore the map (\ref{map:ZH}) over $\cH^{\mathrm{reg}}$ is isomorphic to 
$\cD_p \times_{\cB}\cH^{\mathrm{reg}}$. 
Each irreducible component of $\cD_p \times_{\cB}\cH^{\mathrm{reg}}$
is dominant over $\cB$, as we have the factorization 
\begin{align*}
    \cD_p \times_{\cB}\cH^{\mathrm{reg}} \to \cC \times_{\cB}\cH^{\mathrm{reg}} \to \cB,
\end{align*}
the first map is \'etale, and the second map is dominant 
on each irreducible component. 
Therefore if an irreducible component $\mathcal{Z}' \subset \mathcal{Z}$ 
is not dominant over $\cB$, then it lies in the complement of $\cH^{\mathrm{reg}}$ in $\cH$. Since $\cH^{\mathrm{reg}} \subset \cH$ 
is an isomorphism in codimension one by Lemma~\ref{lem:codimension}, 
this implies that 
\begin{align}\label{codim:Z}
    \dim \mathcal{Z}' \leq \dim (\cH \setminus \cH^{\mathrm{reg}})+2 \leq \dim \cH.
\end{align}

However, $\mathcal{Z}'$ is an irreducible component of the classical 
stack $\widetilde{\mathcal{Z}}$, defined by the (classical) Cartesian square 
\begin{align*}
\xymatrix{
    \widetilde{\mathcal{Z}} \inclusion \ar[d] & \widetilde{\cHecke}_{\mathcal{D}_p} \ar[d] \\
    D \inclusion^-{\Delta} & D\times_{\cB} D
    }
\end{align*}
The right vertical arrow is the top arrow in (\ref{hecke:dp}), which 
records the points $(x', y')$ of a lift of $(x, y)$ in the diagram (\ref{dia:xy}), composed with the double covering map $\cD_p \to D$ in the diagram (\ref{etale:D}). 
We have decompositions 
\begin{align*}
    \widetilde{\mathcal{Z}}=\widetilde{\mathcal{Z}}_1 \cup \widetilde{\mathcal{Z}}_2
\end{align*}
where each $\widetilde{\mathcal{Z}}_i$ is a union of irreducible components, 
 $\widetilde{\mathcal{Z}}_1$ corresponds to $y'=x'$, 
and $\widetilde{\mathcal{Z}}_2$ corresponds to $y'=\sigma(x')$. 
We have $\mathcal{Z}=\widetilde{\mathcal{Z}}_2$, and $\mathcal{Z}'$ is one of the irreducible components of $\widetilde{\mathcal{Z}}_2$. 

However, since $D \to C\times \cB$ is \'etale and $C$ is smooth, 
the map $D \to \cB$ is smooth and 
$\Delta(D) \subset D\times_{\cB}D$ is a Cartier divisor. 
As a pull-back of the Cartier divisor, each irreducible component of 
$\widetilde{\mathcal{Z}}$ should have codimension at most one, which 
contradicts with (\ref{codim:Z}). 
\end{proof}

The sheaf $E_{\Delta}^{\sigma}$ in (\ref{ext:sigma}) defines a morphism 
\begin{align}\label{def:tausig}
    \tau_2^{\sigma} \colon \Delta \to \mathcal{H}. 
\end{align}
We define 
\begin{align*}
    \cP_{\Delta 2}^{\sigma}:=(\tau_2^{\sigma}\times \id)^* \cP\in \Coh^{\heartsuit}(\cH_{\Delta}).
\end{align*}
Similarly to $\mathcal{P}_{\Delta i}$, the sheaf $\mathcal{P}_{\Delta 2}^{\sigma}$ is a Cohen--Macaulay sheaf on 
$\mathcal{H}_{\Delta}$ flat over $\Delta$ such that $\mathcal{P}_{\Delta 2}^{\sigma}|_{0} \cong \mathcal{P}_E$. 
Moreover we have the following lemma: 
\begin{lemma}\emph{(cf.~\cite[Lemma~5.6]{TodaGL2})}\label{lem:inde}
The sheaf $\cP_{\Delta 2}^{\sigma}$ is 
independent of a choice of an extension (\ref{ext:sigma}) up to an isomorphism. 
\end{lemma}
\begin{proof}
By the construction, 
the sheaf $\mathcal{P}_{\Delta 2}^{\sigma}$ is a Cohen--Macaulay extension of a line bundle on 
$\Delta \times_{\cB} \mathcal{H}^{\mathrm{reg}}$, which is determined by the K-theory class of 
$[E_{\Delta}^{\sigma}]$. 
Since the complement of $\Delta \times_{\cB} \mathcal{H}^{\mathrm{reg}}$ in $\mathcal{H}_{\Delta}$ has 
codimension at least two by Lemma~\ref{lem:codimension} below, the lemma follows from Lemma~\ref{lem:MCM-extension}. 
\end{proof}

\begin{lemma}\emph{(cf.~\cite[Lemma~5.7]{TodaGL2})}\label{lem:codimension}
Let $T$ be an irreducible variety and $T\to \cB$ be a dominant map. 
The open immersion 
\begin{align}\label{jreg}
    j^{\mathrm{reg}}\colon \cH^{\mathrm{reg}}_{T}:=T\times_{\cB}\cH^{\mathrm{reg}}
    \hookrightarrow \cH_{T}     
\end{align}
    is an isomorphism in codimension one. 
\end{lemma}
\begin{proof}
Since $T$ is irreducible and $T \to \cB$ sends the generic point of $T$ to $\cB^{\mathrm{sm}}$, 
there is a closed subset $Z\subset T$ with 
$\operatorname{codim}_{T} Z \geq 1$ such that 
$j^{\mathrm{reg}}$ is an isomorphism over $T \setminus Z$. 
Therefore $\cH_{T} \setminus \cH_{T}^{\mathrm{reg}}$ is contained in 
$\cH_{T}\times_{T} Z$. Since $\cH_{T} \to T$ is flat by 
Theorem~\ref{thm:higgs:basic}, the closed substack 
$\cH_{T}\times_{T} Z$ has codimension at least one in 
$\cH_{T}$. Moreover, $\cH_{T}\times_{T} Z \to Z$ is flat, and 
each fiber contains the dense regular locus by Lemma~\ref{lem:dense}. Hence 
$\cH_{T} \setminus \cH_{T}^{\mathrm{reg}}$ has codimension at least 
one in $\cH_{T}\times_{T}Z$. Combining the two codimension estimates gives
$\operatorname{codim}_{\cH_{T}}
(\cH_{T} \setminus \cH_{T}^{\mathrm{reg}})\geq 2$.
\end{proof}

\subsection{Construction of an exact sequence of deformed sheaves}
We next construct a rank two vector bundle 
\begin{align*}
    \mathcal{V}_{\Delta} \in \Coh^{\heartsuit}(\cH_{\Delta})
\end{align*}
which gives a deformation of $\mathcal{V}_p$ given in (\ref{def:Vp}), and 
show the existence of an exact sequence 
relating $\cP_{\Delta}$, $\cP_{\Delta 2}^{\sigma}$ and $\mathcal{V}_{\Delta}$. 

We have the following commutative diagram 
\begin{align}\label{dia:delta1}
    \xymatrix{
& \Delta \ar[d]_-{\eta_D'} \ar[rd]^-{\eta_D''} \ar[ld]_-{\eta_D} & &\\
(\cD_{p})_{\Delta} \inclusion^-{i} \ar[rd]_-{g} & \cD_{\Delta} \ar[r]^-{\pi}\ar[d]_-{g'} \diasquare & D_{\Delta}  \ar[d]_-{g''} \ar[rd] & \\
& \cC_{\Delta} \ar[r]^-{\pi} & C\times \Delta \ar[r]^-{\pi'} &
C.       }
\end{align}
Here the bottom arrows are projections. 
By taking $\times_{\cB}\cH$, it also induces the commutative diagram 
\begin{align}\label{dia:delta2}
    \xymatrix{
& \cH_{\Delta} \ar[d]_-{\eta_D'} \ar[rd]^-{\eta_D''} \ar[ld]_-{\eta_D} & &\\
(\cD_{p})_{\Delta}\times_{\Delta}\cH_{\Delta} \inclusion^-{i} \ar[rd]_-{g} & \cD_{\Delta}\times_{\Delta}\cH_{\Delta} \ar[r]^-{\pi}\ar[d]_-{g'} \diasquare & D_{\Delta}\times_{\Delta}\cH_{\Delta}  \ar[d]_-{g''} \ar[rd] & \\
& \cC_{\Delta}\times_{\Delta}\cH_{\Delta} \ar[r]^-{\pi} & C\times \cH_{\Delta} \ar[r]^-{\pi'} &
C \times \cH.       }
\end{align}
Here we have used the same symbols as in (\ref{dia:delta1}) by abuse of notation. 

We have the vector bundle of rank $r$ on $\cH_{\Delta}$
\begin{align*}
  \mathcal{F}_{\Delta}
  :=
  (\eta_D'')^*(g'')^*(\pi')^*\cF
\end{align*}
given by the pull-back of the universal vector bundle $\cF$ on $C \times \cH$ to $\cH_{\Delta}$.

\begin{lemma}\label{lem:Vdelta}
There is a rank two sub-vector bundle (indeed a direct summand) 
\begin{align*}
  \mathcal{V}_{\Delta} \subset \mathcal{F}_{\Delta}
\end{align*}
which restricts to $\mathcal{V}_p \subset \cF_c$ at $0 \in \Delta$.
\end{lemma}

\begin{proof}
Let
\begin{align*}
  \cE_{\Delta} \in \Coh^{\heartsuit}(\cC_{\Delta} \times_{\Delta} \cH_{\Delta})
\end{align*}
be the pull-back of the universal one-dimensional sheaf $\cE$ in
$\Coh(\cC \times_{\cB} \cH)$ given by the spectral construction.
By its construction, we have
\begin{align*}
  \pi_*\cE_{\Delta} \cong (\pi')^*\cF.
\end{align*}
Therefore we have
\begin{align*}
  \cF_{\Delta}
  &\cong (\eta_D'')^*(g'')^*\pi_*\cE_{\Delta} \\
  &\cong (\eta_D'')^*\pi_* (g')^*\cE_{\Delta} \hookleftarrow
  (\eta_D'')^*\pi_{*} i_* g^* \cE_{\Delta}.
\end{align*}
Here the last arrow is a split injection given by the decomposition into the open and closed
subsets $\cD = \cD_p \sqcup (\cdots)$.
The direct summand 
\begin{align}\label{def:Vdelta}
  \mathcal{V}_{\Delta}
  :=
  (\eta_D'')^*\pi_{*} i_* g^*\cE_{\Delta},
\end{align}
is a rank two vector bundle since the map (\ref{map:D}) is a double 
cover, which restricts to $\mathcal{V}_p$ at $0\in \Delta$ by the construction. Therefore we obtain the lemma.
\end{proof}

We have the following proposition: 
\begin{prop}\label{prop:exact}
There is an exact sequence of vector bundles on $\Delta\times_{\cB}\cH^{\mathrm{reg}}$
\begin{align}\label{exact:jF}
    0\to j^{\mathrm{reg}*} \cP_{\Delta 2}^{\sigma} \to j^{\mathrm{reg}*} \mathcal{V}_{\Delta}
    \otimes j^{\mathrm{reg}*} \cP_{\Delta 1}\to j^{\mathrm{reg}*} \cP_{\Delta 2}\to 0.
\end{align}
Here $j^{\mathrm{reg}}$ is the open immersion 
\begin{align*}
    j^{\mathrm{reg}}\colon \cH^{\mathrm{reg}}_{\Delta}:=\Delta\times_{\cB}\cH^{\mathrm{reg}}
    \hookrightarrow \cH_{\Delta}. 
\end{align*}
\end{prop}
\begin{proof}
   We have the following diagram
   \begin{align*}
    \xymatrix{
& \cH_{\Delta}^{\mathrm{reg}}  \ar[d]_-{\eta_D}  & \\
(\cD_p)_{\Delta} \ar[d]_-{g} \diasquare &
(\cD_{p})_{\Delta}\times_{\Delta}\cH_{\Delta}^{\mathrm{reg}}
\ar[r]^-{q_{2}} \ar[l]_-{q_{1}} \ar[d]_-{g} &
\cH_{\Delta}^{\mathrm{reg}} \\
\cC_{\Delta} & \cC_{\Delta}\times_{\Delta}\cH_{\Delta}^{\mathrm{reg}} \ar[l]_-{p_1} \ar[ru]_-{p_2}
&
}
\end{align*}
Here the bottom arrows are projections. Let 
\begin{align}\label{E:pull}
\cE_{\cD}^{\mathrm{reg}}\in \Coh^{\heartsuit}((\cD_p)_{\Delta}\times_{\Delta}\cH_{\Delta}^{\mathrm{reg}}), \
\cE^{\mathrm{reg}}_{\Delta}\in \Coh^{\heartsuit}(\cC_{\Delta}\times_{\Delta}\cH_{\Delta}^{\mathrm{reg}})
\end{align}
be the pull-back of the universal one-dimensional sheaf $\mathcal{E} \in \Coh^{\heartsuit}(\cC\times_{\cB}\cH)$. Note that (\ref{E:pull}) are line bundles by the definition of $\cH^{\mathrm{reg}}$. 
Then by the construction and using (\ref{exact:EQ}), we have an isomorphism 
of line bundles on $\cH_{\Delta}^{\mathrm{reg}}$
\begin{align}\notag
    j^{\mathrm{reg}*}\cP_{\Delta 2}&=\det p_{2*}(p_1^*E_{\Delta}\otimes 
   \cE^{\mathrm{reg}}_{\Delta})\otimes (\det p_{2*}p_1^*E_{\Delta})^{-1} \otimes (\det p_{2*}\cE^{\mathrm{reg}}_{\Delta})^{-1}\otimes \det p_{2*}\mathcal{O}\\
   \notag &\cong j^{\mathrm{reg}*}\cP_{\Delta 1}\otimes \det p_{2*}(p_1^* R_{\Delta}\otimes \cE^{\mathrm{reg}}_{\Delta}).
\end{align}
We have isomorphisms 
\begin{align*}
    p_{2*}(p_1^* R_{\Delta}\otimes \cE^{\mathrm{reg}}_{\Delta})&\cong 
    p_{2*}(g_{*}\eta_{D*}\mathcal{O}_{\cH_{\Delta}^{\mathrm{reg}}}\otimes \cE^{\mathrm{reg}}_{\Delta}) \\
    &\cong p_{2*}g_{*}(\eta_{D*}\mathcal{O}_{\cH_{\Delta}^{\mathrm{reg}}}\otimes \cE^{\mathrm{reg}}_{D}) \\
    &\cong \eta_D^* \cE^{\mathrm{reg}}_{D}.
\end{align*}
Therefore we have 
\begin{align}\label{isom:P}
      j^{\mathrm{reg}*}\cP_{\Delta 2} \cong j^{\mathrm{reg}*}\mathcal{P}_{\Delta 1} \otimes \eta_D^* \cE^{\mathrm{reg}}_{D}.
\end{align}
In a similar way, we also have an isomorphism 
\begin{align}\label{isom:P2}
      j^{\mathrm{reg}*}\cP_{\Delta 2}^{\sigma} \cong j^{\mathrm{reg}*}\mathcal{P}_{\Delta 1} \otimes (\sigma \circ \eta_D)^* \cE^{\mathrm{reg}}_{D}.
\end{align}

Below we use the same notation as in (\ref{dia:delta2}) for the 
restricted diagram 
\begin{align}\label{dia:delta3}
    \xymatrix{
& \cH_{\Delta}^{\mathrm{reg}} \ar[d]_-{\eta_D'} \ar[rd]^-{\eta_D''} \ar[ld]_-{\eta_D} & &\\
(\cD_{p})_{\Delta}\times_{\Delta}\cH_{\Delta}^{\mathrm{reg}} \inclusion^-{i} \ar[rd]_-{g} & \cD_{\Delta}\times_{\Delta}\cH_{\Delta}^{\mathrm{reg}} \ar[r]^-{\pi}\ar[d]_-{g'} \diasquare & D_{\Delta}\times_{\Delta}\cH_{\Delta}^{\mathrm{reg}}  \ar[d]_-{g''} \ar[rd] & \\
& \cC_{\Delta}\times_{\Delta}\cH_{\Delta}^{\mathrm{reg}} \ar[r]^-{\pi} & C\times \cH_{\Delta}^{\mathrm{reg}} \ar[r]^-{\pi'} &
C \times \cH^{\mathrm{reg}}.       }
\end{align}
The composition of the maps in the diagram (\ref{dia:delta1})
\begin{align*}
    \pi_p \colon (\cD_p)_{\Delta} \stackrel{i}{\hookrightarrow} \cD_{\Delta} \stackrel{\pi}{\to} D_{\Delta}
\end{align*}
is a double cover with covering involution~$\sigma$. 
In the notation of (\ref{dia:delta1}), we have the decomposition
into the irreducible components 
\begin{align*}
    \pi_p^{-1}\eta_{D}''(\Delta)=\eta_D(\Delta)\cup (\sigma\circ \eta_D)(\Delta)
\end{align*}
whose intersection is a Cartier divisor on $\Delta$. Therefore by shrinking 
$\Delta$ if necessary, we have 
the exact sequence in $\Coh^{\heartsuit}((\cD_p)_{\Delta})$
\begin{align*}
    0\to \sigma_{*}\eta_{D*}\mathcal{O}_{\Delta} \to \pi_p^* \pi_{p*}\eta_{D*}\mathcal{O}_{\Delta} \to \eta_{D*}\mathcal{O}_{\Delta} \to 0.
\end{align*}
By pulling it back to $(\cD_p)_{\Delta}\times_{\Delta}\cH_{\Delta}^{\mathrm{reg}}$
and applying $\otimes \cE^{\mathrm{reg}}_{D}$, 
we obtain the exact sequence in $\Coh^{\heartsuit}((\cD_p)_{\Delta}\times_{\Delta}\cH_{\Delta}^{\mathrm{reg}})$
\begin{align*}
     0\to q_1^*\sigma_{*}\eta_{D*}\mathcal{O}_{\Delta} \otimes \cE^{\mathrm{reg}}_{D}\to q_1^*\pi_p^* \pi_{p*}\eta_{D*}\mathcal{O}_{\Delta}\otimes \cE^{\mathrm{reg}}_{D} \to q_1^*\eta_{D*}\mathcal{O}_{\Delta}\otimes \cE^{\mathrm{reg}}_{D} \to 0. 
\end{align*}
Then applying $\pi_{p*}$ for the composition 
\begin{align*}
    \pi_p \colon (\cD_p)_{\Delta}\times_{\Delta}\cH_{\Delta}^{\mathrm{reg}}
    \stackrel{i}{\hookrightarrow} \cD_{\Delta}\times_{\Delta}\cH_{\Delta}^{\mathrm{reg}} \stackrel{\pi}{\to} D_{\Delta}\times_{\Delta}\cH_{\Delta}^{\mathrm{reg}},
\end{align*}
we obtain 
the exact sequence in $\Coh^{\heartsuit}(D_{\Delta}\times_{\Delta}\cH_{\Delta}^{\mathrm{reg}})$
\begin{align*}
0\to {}&
\pi_{p*}(q_1^*\sigma_{*}\eta_{D*}\mathcal{O}_{\Delta}
\otimes \cE^{\mathrm{reg}}_{D}) \\
\to {}&
\pi_{p*}(q_1^*\pi_p^* \pi_{p*}\eta_{D*}\mathcal{O}_{\Delta}
\otimes \cE^{\mathrm{reg}}_{D}) \\
\to {}&
\pi_{p*}(q_1^*\eta_{D*}\mathcal{O}_{\Delta}
\otimes \cE^{\mathrm{reg}}_{D}) \to 0. 
\end{align*}

Note that by the definition of $\mathcal{V}_{\Delta}$ in (\ref{def:Vdelta}), we have 
\begin{align*}
    j^{\mathrm{reg}*}\mathcal{V}_{\Delta}=(\eta_D'')^* \pi_{p*}\cE^{\mathrm{reg}}_{D}.
\end{align*}
Therefore we have 
\begin{align*}
    \pi_{p*}(q_1^*\pi_p^* \pi_{p*}\eta_{D*}\mathcal{O}_{\Delta}\otimes \cE^{\mathrm{reg}}_{D})  &\cong \pi_{p*}(\pi_p^* \pi_{p*}\eta_{D*}\mathcal{O}_{\cH_{\Delta}^{\mathrm{reg}}}\otimes \cE^{\mathrm{reg}}_{D}) \\
    &\cong \pi_{p*}\eta_{D*}\mathcal{O}_{\cH_{\Delta}^{\mathrm{reg}}}\otimes \pi_{p*}\cE^{\mathrm{reg}}_{D}\\
    &\cong \eta_{D*}''(\eta_D'')^*\pi_{p*}\cE^{\mathrm{reg}}_{D} \\
    &\cong \eta_{D*}''j^{\mathrm{reg}*}\mathcal{V}_{\Delta}.
\end{align*}
Similarly, we have 
\begin{align*}
   &\pi_{p*}(q_1^*\eta_{D*}\mathcal{O}_{\Delta}\otimes \cE^{\mathrm{reg}}_{D}) \cong \eta_{D*}''\eta_D^* \cE^{\mathrm{reg}}_{D}, \\
   & \pi_{p*}(q_1^*\sigma_*\eta_{D*}\mathcal{O}_{\Delta}\otimes \cE^{\mathrm{reg}}_{D}) \cong \eta_{D*}''(\sigma \circ \eta_D)^* \cE^{\mathrm{reg}}_{D}.
\end{align*} 

It follows that we obtain the exact sequence 
\begin{align*}
    0\to \eta^{''}_{D*}(\sigma \circ \eta_D)^* \cE^{\mathrm{reg}}_{D} \to \eta^{''}_{D*}j^{\mathrm{reg}*}\mathcal{V}_{\Delta}\to \eta^{''}_{D*}\eta_D^{*}\cE^{\mathrm{reg}}_{D} \to 0.
\end{align*}
Since $\eta_D''$ is a closed immersion (as it is a section of the projection),
we obtain the exact sequence in $\Coh^{\heartsuit}(\cH_{\Delta}^{\mathrm{reg}})$
\begin{align*}
     0\to (\sigma \circ \eta_D)^* \cE^{\mathrm{reg}}_{D} \to j^{\mathrm{reg}*}\mathcal{V}_{\Delta}\to \eta_D^{*}\cE^{\mathrm{reg}}_{D}\to 0.
\end{align*}
By applying $j^{\mathrm{reg}*}\cP_{\Delta 1} \otimes$ and using (\ref{isom:P}), (\ref{isom:P2}), we
obtain the desired exact sequence (\ref{exact:jF}). 
\end{proof}

As a corollary of the above proposition, we have the following: 
\begin{cor}\label{cor:CM}
    The exact sequence (\ref{exact:jF}) extends to a left exact sequence of maximal Cohen--Macaulay 
    sheaves on $\cH_{\Delta}$
    \begin{align}\label{extend:CM}
        0\to \cP_{\Delta 2}^{\sigma}\to \mathcal{V}_{\Delta}\otimes \cP_{\Delta 1}
        \to \cP_{\Delta 2}.
    \end{align}
\end{cor}
\begin{proof}
    Let $j_{*}^{\mathrm{reg}\heartsuit}$ be the functor 
    \begin{align*}
    j_{*}^{\mathrm{reg}\heartsuit}:=\cH^0(j_{*}^{\mathrm{reg}}) \colon 
    \Coh^{\heartsuit}(\cH_{\Delta}^{\mathrm{reg}}) \to 
     \Coh^{\heartsuit}(\cH_{\Delta}).
    \end{align*}
    The above functor is left exact. 
    By Lemma~\ref{lem:codimension}, the open immersion $\cH_{\Delta}^{\mathrm{reg}} \subset 
    \cH_{\Delta}$ is an isomorphism in codimension one. Then as $\cP_{\Delta i}, \cP_{\Delta}^{\sigma}$ are 
    maximal Cohen--Macaulay sheaves and $\mathcal{V}_{\Delta}$ is a vector bundle, by Lemma~\ref{lem:MCM-extension} we have 
    \begin{align*}j_{*}^{\mathrm{reg}\heartsuit}j^{\mathrm{reg}*} \cP_{\Delta i} \cong \cP_{\Delta i}, \ j_{*}^{\mathrm{reg}\heartsuit}j^{\mathrm{reg}*} \cP_{\Delta 2}^{\sigma} \cong \cP_{\Delta 2}^{\sigma}, \ 
    j_{*}^{\mathrm{reg}\heartsuit}j^{\mathrm{reg}*} \mathcal{V}_{\Delta} \cong \mathcal{V}_{\Delta}.
    \end{align*}
    Therefore by applying $j_{*}^{\mathrm{reg}\heartsuit}$ to (\ref{exact:jF}), 
    we obtain the left exact sequence (\ref{extend:CM}). 
\end{proof}

\subsection{Exact sequence of Arinkin sheaves}\label{subsec:exAr}
We keep the situation of Corollary~\ref{cor:CM}.
By restricting \eqref{extend:CM} to $0\in \Delta$, we obtain in
$\Coh^{\heartsuit}(\cH_b)$ the sequence
\begin{align}\label{extend:CM2}
0\to \cP_E \to \mathcal{V}_p \otimes \cP_{E'} \to \cP_E
\end{align}
which at this stage is not yet known to be either left exact or right exact. 
In this subsection, we show that the sequence (\ref{extend:CM}) is right exact, which in particular implies that (\ref{extend:CM2}) is an 
exact sequence. 
Then by an induction argument, we will complete the proof of Proposition~\ref{prop:resol}.

We first show this in the case that $p\in \cC_b$ is a nodal singularity: 
\begin{lemma}\label{lem:nodal}
If $p\in \cC_b$ is a nodal singularity, the 
sequence (\ref{extend:CM}) is also right exact. In particular, the sequence (\ref{extend:CM2}) is an exact sequence. 
\end{lemma}
\begin{proof}
    The assumption on $\cC_b$ implies that 
    $m_p=2$ and 
    the sheaf $E'$ on $\cC_b$ is a line bundle at $p$. Let $\cC_{\Delta}'\subset \cC_{\Delta}$ be an open neighborhood of $p \in \cC_b=\cC_{\Delta}\times_{\Delta}0$ such that 
    $E_t':=E_{\Delta}'|_{\cC_{\Delta}\times_{\Delta}t}$ is a line bundle on $\cC_{\Delta}'\times_{\Delta} t$
    for all $t\in \Delta$. 
    Then there is a morphism 
    \begin{align}\label{mor:circ}
\cC_{\Delta}' \to \cH_{\Delta}
    \end{align}
    sending $x\in \cC_{\Delta}'\times_{\Delta}t$ 
    to $E_x \in \Coh^{\heartsuit}(\cC_t)$ which fits into 
    the unique non-trivial extension (up to isomorphism)
    \begin{align*}
        0\to E_t' \to E_x \to \mathcal{O}_x \to 0. 
    \end{align*}
In the diagram (\ref{dia:eta}), note that $\eta(0)=p\in \cC_{\Delta}'$. 
Therefore by shrinking $\Delta$ if necessary we may assume that 
    $\eta \colon \Delta \to \cC_{\Delta}$ factors through 
    $\eta \colon \Delta \to \cC_{\Delta}'$.
    Then from the diagram (\ref{dia:eta}), 
    the map $\eta_D$ factors as 
    \begin{align*}
        \eta_D \colon \Delta \to (\cD_p)_{\Delta}':=(\cD_p)_{\Delta}\times_{\cC_{\Delta}}\cC_{\Delta}'.
        \end{align*}
The above situation is summarized in the following diagram 
    \begin{align}\label{dia:Dpprime}
        \xymatrix{
&(\cD_p)_{\Delta}' \inclusion \dinclusion \diasquare & (\cD_p)_{\Delta} \dinclusion  & \\
&\cD_{\Delta}' \inclusion \ar[d] \diasquare & \cD_{\Delta} \ar[r] \ar[d] \diasquare & D_{\Delta} \ar[d] \\
\Delta \ar[r]^-{\eta} \ar[ruu]^-{\eta_D} &\cC_{\Delta}' \inclusion & \cC_{\Delta} \ar[r] & C\times \Delta.  
        }
    \end{align}
 We define 
    $\widetilde{\Delta}$ to be the Cartesian square 
    \begin{align}\label{dia:tdelta}
        \xymatrix{
\widetilde{\Delta} \ar[d]_-{\widetilde{\pi}_p} \ar[rr] \ar@{}[rrd]|\square & & (\cD_p)_{\Delta}' \ar[d]_-{\pi_p} \\
\Delta \ar[r]^-{\eta_D} & (\cD_p)_{\Delta}' \ar[r]^-{\pi_p} & D_{\Delta}.
        }
    \end{align}
The map $\widetilde{\pi}_p \colon \widetilde{\Delta} \to \Delta$ is a double cover 
with involution $\sigma$ induced by the covering involution of 
$(\cD_p)_{\Delta} \to D_{\Delta}$. 
The map $\widetilde{\pi}_p$ has two sections 
\begin{align*}
    \widetilde{\eta}_D:=(\eta_D, \eta_D) \colon \Delta \hookrightarrow \widetilde{\Delta}, \ 
     \widetilde{\eta}_D^{\sigma}:=(\eta_D, \eta_D^{\sigma}) \colon \Delta \hookrightarrow \widetilde{\Delta} 
\end{align*}
such that 
\begin{align*}\widetilde{\Delta}=\widetilde{\eta}_D(\Delta)\cup \widetilde{\eta}_D^{\sigma}(\Delta).\end{align*}
Let $\Delta_1=\widetilde{\eta}_D(\Delta)$
and $\Delta_2=\widetilde{\eta}^{\sigma}_D(\Delta)$. 
Then $\Delta_1 \cap \Delta_2$ is a divisor in $\Delta_i$. 
By shrinking 
$\Delta$ if necessary, we may assume that $\mathcal{O}_{\Delta_i}(\Delta_1 \cap \Delta_2)$ is trivial, so that we have the exact sequence 
\begin{align}\label{exact:sigma}
    0\to \mathcal{O}_{\Delta_2} \to \mathcal{O}_{\widetilde{\Delta}}
    \to \mathcal{O}_{\Delta_1} \to 0. 
\end{align}

   By composing (\ref{mor:circ}) with the top horizontal arrow in (\ref{dia:tdelta}) and the left vertical arrow in (\ref{dia:Dpprime}), 
   we obtain the map 
   \begin{align*}
       \widetilde{\tau} \colon \widetilde{\Delta} \to (\cD_p)'_{\Delta} \to \cC_{\Delta}' \to \cH_{\Delta}. 
   \end{align*}
   It yields the maximal Cohen--Macaulay sheaf 
   \begin{align*}
       \cP_{\widetilde{\Delta}}:=(\widetilde{\tau}\times \id)^* \cP\in \Coh^{\heartsuit}(\cH_{\widetilde{\Delta}})
       \end{align*}
       Here $\widetilde{\tau}\times \id$ is 
       \begin{align*}
           \widetilde{\tau}\times \id \colon \cH_{\widetilde{\Delta}}=\widetilde{\Delta}\times_{\Delta}\cH_{\Delta} \to 
           \cH_{\Delta}\times_{\Delta} \cH_{\Delta} \to \cH\times_{\cB} \cH.
       \end{align*}
By the constructions, we have the commutative diagrams 
\begin{align*}
    \xymatrix{
\widetilde{\Delta} \ar[r]^-{\widetilde{\tau}} & \cH_{\Delta}, \\
\Delta \ar[u]^-{\widetilde{\eta}_D} \ar[ur]_-{\tau_2} &
    } \quad 
      \xymatrix{
\widetilde{\Delta} \ar[r]^-{\widetilde{\tau}} & \cH_{\Delta}. \\
\Delta \ar[u]^-{\widetilde{\eta}_D^{\sigma}} \ar[ur]_-{\tau_2^{\sigma}} &
    }
\end{align*}
Here $\tau_2, \tau_2^{\sigma}$ are given in (\ref{def:taui}), (\ref{def:tausig}) respectively. 
Therefore the exact sequence (\ref{exact:sigma}) induces the exact
       sequence in $\Coh^{\heartsuit}(\cH_{\Delta})$
       \begin{align}\label{exact:Pdelta}
        0\to \cP_{\Delta 2}^{\sigma} \to \widetilde{\pi}_{p*}\cP_{\widetilde{\Delta}} \to \cP_{\Delta 2} \to 0. 
       \end{align}
       Here we have used the same notation $\widetilde{\pi}_p$ for the 
       base-change map 
       $\cH_{\widetilde{\Delta}} \to \cH_{\Delta}$ induced by 
       $\widetilde{\pi}_p \colon \widetilde{\Delta} \to \Delta$.
       
Unraveling the construction, the exact sequence (\ref{exact:Pdelta}) restricted to 
$\cH_{\Delta}^{\mathrm{reg}}$ is identified with (\ref{exact:jF}). 
Since $\cH_{\Delta}^{\mathrm{reg}} \subset \cH_{\Delta}$ is an isomorphism in codimension one by Lemma~\ref{lem:codimension}, and both $\widetilde{\pi}_{p*}\cP_{\widetilde{\Delta}}$ and $\mathcal{V}_{\Delta}\otimes \cP_{\Delta 1}$ are maximal Cohen--Macaulay sheaves, by Lemma~\ref{lem:MCM-extension} we 
conclude that \begin{align*}
    \widetilde{\pi}_{p*}\cP_{\widetilde{\Delta}}\cong 
   \mathcal{V}_{\Delta}\otimes \cP_{\Delta 1}. 
\end{align*}
Therefore, from the exact sequence (\ref{exact:Pdelta}), we obtain the lemma. 
       
\end{proof}

Using the above lemma for the nodal case, we prove that (\ref{extend:CM2}) is exact 
in general:
\begin{prop}\label{prop:exact2}
The sequence (\ref{extend:CM2}) is an exact sequence
\begin{align}\notag
0\to \cP_E \to \mathcal{V}_p\otimes \cP_{E'} \to \cP_E \to 0.
\end{align}
\end{prop}
\begin{proof}
The proof is the same as~\cite[Proposition~5.11]{TodaGL2}, but we include it for completeness.
Let $\cB^{\mathrm{sg}} \subset \cB$ be the closed subset corresponding 
to the singular spectral curves. It is of codimension one, and its 
generic point corresponds to those with at worst nodal singularities, see Lemma~\ref{lem:nodeB} below. 
Then we can take $\Delta$ so that $\dim \Delta =2$ and 
the image of the composition 
\begin{align*}\Delta \setminus \{0\} \hookrightarrow \Delta \to \cB
\end{align*}
corresponds to the locus of spectral curves which are either smooth or have at worst nodal singularities. Then, from Lemma~\ref{lem:nodal}, the sequence (\ref{extend:CM}) is 
exact on $\Delta \setminus \{0\}$. 

Assume that the sequence (\ref{extend:CM}) is not exact.
Let $\mathbb{P}_{\Delta}$ be the complex (\ref{extend:CM})
regarded as an object in $\Coh(\cH_{\Delta})$, where 
$\cP_{\Delta 2}$ is located in degree zero. 
Since (\ref{extend:CM}) is left exact, the object $\mathbb{P}_{\Delta}$ is a non-zero object
in $\Coh^{\heartsuit}(\cH_{\Delta})$ and supported 
over $0\in \Delta$. 
Let $i \colon \Delta' \subset \Delta$ be a generic smooth divisor 
such that $0\in \Delta'$.
We use the same notation $i$ for the induced map $i \colon \cH_{\Delta'} \hookrightarrow \cH_{\Delta}$. 
Then as $\mathbb{P}_{\Delta}$ is supported over $0\in \Delta'$, 
it follows that 
\begin{align*}
0\neq \cH^{-1}(i^* \mathbb{P}_{\Delta}) \in \Coh^{\heartsuit}(\cH_{\Delta'}).
\end{align*}

On the other hand, 
since each term of $\mathbb{P}_{\Delta}$ is flat over $\Delta$, 
we have $i^* \mathbb{P}_{\Delta}=\mathbb{P}_{\Delta'}$ and it 
lies in $\Coh^{\heartsuit}(\cH_{\Delta'})$ by
the left exactness of
the sequence 
   \begin{align}\notag
        0\to \cP_{\Delta' 2}^{\sigma}\to \mathcal{V}_{\Delta'}\otimes \cP_{\Delta' 1}
        \to \cP_{\Delta' 2}
    \end{align}
    which is obtained as in 
(\ref{extend:CM}) by replacing $\Delta$ with $\Delta'$. 
It contradicts $\cH^{-1}(i^*\mathbb{P}_{\Delta})\neq 0$,
hence the sequence (\ref{extend:CM}) is exact. 
Then by taking the restriction of (\ref{extend:CM}) to $0\in \Delta$, 
we conclude that the sequence (\ref{extend:CM2}) is exact. 
\end{proof}

We have used the following lemma: 
\begin{lemma}\label{lem:nodeB}
Let $\rB_{\GL_r}^{\mathrm{node}} \subset \rB_{\GL_r}^{\mathrm{cl}}$ be the open subset 
corresponding to spectral curves with at worst nodal singularities. 
Then
\[
\operatorname{codim}(\rB_{\GL_r}^{\mathrm{cl}} \setminus \rB_{\GL_r}^{\mathrm{node}}) \geq 2 .
\]
\end{lemma}

\begin{proof}
The case $g\geq 3$ follows from~\cite[Corollary~1.5 and Remark~1.7]{KouvidakisPantev1995}. 
Indeed, it is proven in loc. cit. that the locus 
$\rB_{\GL_r}^{\mathrm{sg}} \subset \rB_{\GL_r}^{\mathrm{cl}}$ corresponding to singular 
spectral curves is an irreducible hypersurface, and that there exists 
$b\in \rB_{\GL_r}^{\mathrm{sg}}$ such that $\cC_b$ is irreducible with only one node. 
Therefore the generic point of $\rB_{\GL_r}^{\mathrm{sg}}$ corresponds to an 
irreducible curve with only one nodal singularity. 
It follows that the complement of $\rB_{\GL_r}^{\mathrm{node}}$ in 
$\rB_{\GL_r}^{\mathrm{cl}}$ has codimension at least two.

The case $g=2$ has to be treated separately, which we prove in Subsection~\ref{subsec:g=2}.

\end{proof}

\textit{Proof of Proposition~\ref{prop:resol}}
\begin{proof}
Let $Q=\bigoplus_{p\in N}Q_p$ be the zero-dimensional sheaf as in the exact sequence (\ref{exact:LE}). 
We prove the proposition by induction on $\chi(Q)\geq 0$. 
The base case, $\chi(Q)=0$, is obvious since $E=\mathcal{L}$ 
and $\cP_E=\mathcal{V}_{E}$ in this case. 

Suppose that $\chi(Q)>0$. 
    By Proposition~\ref{prop:exact2}, 
    there is an exact sequence of the form 
    \begin{align*}
        \cdots\to \mathcal{V}_{p}\otimes \cP_{E'} \to \mathcal{V}_p\otimes \cP_{E'} \to 
        \mathcal{V}_p\otimes \cP_{E'} \to \cP_E \to 0.
    \end{align*}
    Since $\chi(Q')=\chi(Q)-1$ in the sequence (\ref{exact:E2}), 
    we can apply the induction hypothesis for $\cP_{E'}$. Then the claim 
    for $\cP_E$ holds by the above exact sequence together with 
    $\mathcal{V}_p\otimes \mathcal{V}_{E'}\cong \mathcal{V}_E$. 
\end{proof}

\section[Whittaker normalization over the type A-locus]{Whittaker normalization over the type \texorpdfstring{$A$}{A}-locus}
In this section, we prove Conjecture~\ref{conj:Whit} over the type $A$-locus 
of the Hitchin base. Below we keep the setting of the previous section. 
We may further assume that $w> 0$ by the periodicity 
of Higgs moduli stacks of degree/weight by $r$, e.g.\ we may replace $w$ with $w+kr$ for $k\gg 0$ if necessary, 
see~\cite[Remark~2.14]{TodaGL2}. 
Since the proof is long, 
we first give an idea of a proof in Subsection~\ref{subsec:idea}. 
\subsection{Idea of the proof}\label{subsec:idea}
We first prove the following statement (see Theorem~\ref{thm:PhiO}):
\begin{align}\label{idea:j!}
\Phi(\cO_{\cH(w)^{\mathrm{ss}}}) \in j_{!}\LL(\cH(\chi_0)^{\mathrm{ss}})_{w'}.
\end{align}
Here $j \colon \cH(\chi_0)^{\mathrm{ss}} \hookrightarrow \cH(\chi_0)$ is the open immersion. 
This is a necessary condition for Conjecture~\ref{conj:Whit} to hold.
We will derive~\eqref{idea:j!} from the weight estimates for the Arinkin sheaf obtained in 
the previous section. 

One of the difficulties in proving~\eqref{idea:j!} is that $\cH(w)^{\mathrm{ss}}$ is an Artin stack, 
rather than a scheme. Suppose that there were an open substack
$\cH^{\circ} \subset \cH(w)^{\mathrm{ss}}$
satisfying the following properties:
\begin{enumerate}
    \item $\cH^\circ$ is ``almost'' a scheme; more precisely, it is a $\mathbb{G}_m$-gerbe over a 
    scheme which is proper over $\cB$;
    \item for all $y\in \cH^{\circ}$, we have 
    \begin{align}\label{idea:Phi}
        \Phi(\cO_y) \in  j_{!}\LL(\cH(\chi_0)^{\mathrm{ss}})_{w'};
    \end{align}
    \item the natural map
    \[
        \Phi(\cO_{\cH(w)^{\mathrm{ss}}})\to \Phi(\cO_{\cH^{\circ}})
    \]
    is an isomorphism. 
\end{enumerate}
If (i), (ii), and (iii) hold, then intuitively we would like to conclude that
\begin{align*}
   \Phi(\cO_{\cH(w)^{\mathrm{ss}}})
   \cong \Phi(\cO_{\cH^{\circ}}) 
   \mathrel{\text{``=``}} 
   \int_{y\in \cH^{\circ}}\Phi(\mathcal{O}_y)
   \in j_{!}\LL(\cH(\chi_0)^{\mathrm{ss}})_{w'}.
\end{align*}

Ideally, we would like to construct $\cH^{\circ} \subset \cH(w)^{\mathrm{ss}}$ by deforming a 
stability condition so that the ``semistable = stable'' condition holds, and then define
$\cH^{\circ}$ to be the corresponding semistable locus. The condition~\eqref{idea:Phi} would then
follow by modifying the proof of Corollary~\ref{cor:PE}, imposing the additional condition that
$y$ corresponds to a semistable sheaf with respect to the deformed stability condition. 

However, there is essentially only one stability condition for Higgs bundles on curves, since the 
Picard number of $S=\mathrm{Tot}_C(\Omega_C)$ is one. Instead, we take \'etale coverings of the 
Hitchin base in order to allow more deformations of stability conditions. Namely, we take an
\'etale base change $T\to \cB$, and find a suitable perturbation of a stability condition in order
to construct the desired open substack
$\cH_T^{\circ}$ of $\cH(w)^{\mathrm{ss}}\times_{\cB}T$. We also construct 
$\cH_T^{\circ}$ step by step for each Harder--Narasimhan stratum of $\cH(\chi_0)$, rather than as
a single semistable locus as in (ii). 

Once~\eqref{idea:j!} is proved, the claim is reduced to showing that
(see Theorem~\ref{thm:Whit:intro})
\begin{align}\label{isom:sbar}
    \Phi(\cO_{\cH(w)^{\mathrm{ss}}})|_{\cH(\chi_0)^{\mathrm{ss}}}
    \cong \overline{s}_! \mathcal{O}_{\cB},
\end{align}
where $\overline{s}$ is the Hitchin section for the semistable locus 
$\cH(\chi_0)^{\mathrm{ss}} \to \cB$. The point is that, since $\overline{s}$ is ``almost'' 
a closed immersion---more precisely, a closed immersion after $\mathbb{G}_m$-rigidification---the object 
$\overline{s}_!(\cO_{\cB})$ is much easier to understand than $s_{!}\cO_{\cB}$; it is 
``almost'' the usual $*$-pushforward of $\cO_{\cB}$ as stated in Lemma~\ref{lem:factori}.
Then the isomorphism~\eqref{isom:sbar} is proved by combining the perturbation of stability 
conditions after the above \'etale base change with the derived equivalence for 
stable sheaves with respect to a generic stability condition in Theorem~\ref{thm:MRV}.

\subsection[Maps from BGm: spectral and automorphic side]{Maps from \texorpdfstring{$\bgm$}{BGm}: spectral and automorphic side}\label{subsec:setting}
In this subsection, we fix the notation needed to implement the strategy outlined in the previous subsection.
In what follows, we fix $b\in \cB$ and take maps from $\bgm$ to both $\cH(w)^{\mathrm{ss}}$ and $\cH(\chi)$
(or their \'{e}tale base changes) which map to $b$. In order to avoid the confusion, 
we say that a construction in $\cH(w)^{\mathrm{ss}}$ is on the \textit{spectral side} and
a construction in $\cH(\chi)$ is on the \textit{automorphic side}. 

\subsubsection*{Maps corresponding to HN filtration (automorphic side)}
We take a map 
\begin{align}\label{nu:F12}
    \nu \colon \bgm \to \cH(\chi)
\end{align}
over $b\in \cB$, 
corresponding to the center of a Harder--Narasimhan stratum of $\cH(\chi)$, 
see Remark~\ref{rmk:jshrink}. 
Namely the map $\nu$ corresponds to a direct sum 
\begin{align*}
    M=M_1\oplus M_2 \oplus \cdots \oplus M_k
\end{align*}
such that $M_i\in \Coh^{\heartsuit}(\cC_b)$ is semistable, and by setting 
\begin{align}\label{Ci}
    C_i=\mathrm{Supp}(M_i), \ r_i=\operatorname{rank}(\pi_{*}M_i), \ \chi_i=\deg \pi_{*}M_i
\end{align}
we have the Harder--Narasimhan inequalities
\begin{align}\label{slope:rchi}
    \frac{\chi_1}{r_1}>\cdots>\frac{\chi_k}{r_k}, \ (r_1, \chi_1)+\cdots+(r_k, \chi_k)=(r, \chi).
\end{align}
Moreover the $\mathbb{G}_m$-weight on $M$ is given by 
\begin{align*}(\nu_1, \nu_2, \cdots, \nu_k), \ \nu_1>\nu_2>\cdots>\nu_k. 
\end{align*}
Note that we have 
\begin{align*}
    \cC_b=C_1 \cup \cdots \cup C_k. 
\end{align*}

\subsubsection*{Perturbation of an ample divisor after \'{e}tale base change}
For a map $T\to \cB$, we denote by $(-)_T$ the base change for $T\to \cB$, i.e.\ 
\begin{align*}\cC_T:=\cC\times_{\cB}T, \ 
\cH_T:=\cH\times_{\cB} T.
\end{align*}
The pull-back of the Arinkin sheaf (\ref{Arsheaf0}) is denoted by 
\begin{align}\label{P:pullT}
    \cP_T \in \Coh^{\heartsuit}(\cH_T(w)^{\mathrm{ss}} \times_T \cH_T(\chi))
\end{align}
and the associated Fourier--Mukai functor is denoted by 
\begin{align*}
    \Phi_T \colon \Coh(\cH_T(w)^{\mathrm{ss}})_{-\chi'} \to \Coh(\cH_T(\chi))_{w'}.
\end{align*}

Below we fix $b\in \cB$, and let 
\begin{align*}
    \cC_b=\cC_b^{(1)}\cup \cdots \cup \cC_b^{(m)}
\end{align*}
be the decomposition into irreducible components. 
We will use the following lemma: 
\begin{lemma}\label{lem:etale}
There is an \'etale neighborhood $(T, 0) \to (\cB, b)$
and Cartier divisors $H^{(1)}, \ldots, H^{(m)}$
on $\cC_T$ such that $H^{(i)} \cdot \cC_b^{(j)}=\delta_{ij}$.
Here the fiber of $\cC_T \to T$ at $0$ is identified with $\cC_b$. 
\end{lemma}
\begin{proof}
Let $U\hookrightarrow \cC$ be a general hypersurface such that 
it intersects $\cC_b$ at smooth points of $\cC_b$. 
For each $1\leq i\leq m$, choose $t^{(i)} \in U\cap \cC_b^{(i)}$, 
and an open neighborhood $t^{(i)} \in U^{(i)} \subset U$ such that 
$U^{(i)} \to \cB$ is \'etale. 
Then we have the section 
\begin{align}\label{sec:U}
    U^{(i)}\hookrightarrow \cC_{U^{(i)}}
\end{align}
which intersects at $\cC_{U^{(i)}}\times_{U^{(i)}}\{t^{(i)}\}=\cC_b$
only at $t^{(i)} \in \cC_b$. 
We set 
\begin{align*}
    T=U^{(1)}\times_{\cB} \cdots \times_{\cB}U^{(m)}, \ 0=(t^{(1)}, \ldots, t^{(m)})
\end{align*}
and $H^{(i)} \subset \cC_T$ the pull-back of the section (\ref{sec:U}) for each $i$. By the construction, it satisfies $H^{(i)} \cdot \cC_b^{(j)}=\delta_{ij}$. 
\end{proof}

We take an \'{e}tale neighborhood 
\begin{align*}(T, 0) \to (\cB, b)\end{align*} 
as in Lemma~\ref{lem:etale}.
Let $\nu \colon \bgm \to \cH(\chi)$ be
a map as (\ref{nu:F12}), 
with $(C_i, r_i, \chi_i)$ as in (\ref{Ci}). 
Then by Lemma~\ref{lem:etale}, there is 
a $\mathbb{Q}$-ample divisor $h_T^{\varepsilon}$ 
on $\cC_T$ 
for $0<\varepsilon \ll 1$ satisfying 
the following: for any irreducible component $\cC_b^{(a)}$ of $\cC_b$
with $\cC_b^{(a)} \subset C_i$ 
and $r^{(a)}=\operatorname{rank}(\pi_{*}\cO_{\cC_b^{(a)}})$, 
we have 
\begin{align}\label{div:hT}
    h_T^{\varepsilon} \cdot \cC_b^{(a)}=r^{(a)}\left(1+ \left(\frac{\chi_i}{r_i}-\frac{\chi}{r}\right)\varepsilon \right). 
\end{align}
The ample divisor $h_T^{\varepsilon}$ is a small perturbation of $\pi_T^* h$ where $h$ is a degree 
one ample divisor on $C$ and $\pi_T$ is the projection 
$\pi_T \colon \cC_T \to C$. Note that we have to take an \'{e}tale cover for the existence of 
such a $\mathbb{Q}$-ample divisor. 
By the construction, the map $\nu$ lifts to a map over $0\in T$, which 
we denote by the same symbol $\nu$
\begin{align*}
    \nu \colon \bgm \to \cH_T(\chi).
\end{align*}

The $\mathbb{Q}$-ample divisor $h_T^{\varepsilon}$ defines the refinement of 
stability conditions on $\Coh^{\heartsuit}(\cC_t)$ for $\cC_t=\cC_{T}\times_T \{t\}$ with $t\in T$. 
Namely let $\mu^{\varepsilon}(E)$ for $E\in \Coh^{\heartsuit}(\cC_t)$ be defined by 
\begin{align}\label{slope:mue} 
    \mu^{\varepsilon}(E):=\frac{\deg \pi_{*}E}{h_T^{\varepsilon}\cdot [E]} \in \mathbb{Q} \cup \{\infty\}
\end{align}
where $[E]$ is the fundamental one-cycle associated with $E$. 
Then $E$ is $h_T^{\varepsilon}$-(semi)stable 
if we have the inequality 
\begin{align}\label{ineq:muep}
 \mu^{\varepsilon}(E') \leq (<)  \mu^{\varepsilon}(E)
\end{align}
for any non-trivial subsheaf $0\neq E' \subsetneq E$ in $\Coh^{\heartsuit}(\cC_t)$. (Here, recall that $E'$ corresponds to sub-Higgs bundle under the spectral 
correspondence.)
Note that for $\varepsilon=0$, the above stability condition is the same
as the usual stability condition for Higgs bundles. 

\subsubsection*{The HN stratification for $h_T^{\varepsilon}$-stability (spectral side)}
For $0<\varepsilon \ll 1$, the above $h_T^{\varepsilon}$ refines the 
stability condition for $\varepsilon=0$. Therefore we have the corresponding
Harder--Narasimhan stratification with respect to $h_T^{\varepsilon}$, which is 
a 
$\Theta$-stratification in~\cite{Halpinstab}:
\begin{align}\label{Theta:eps}
    \cH_T(w)^{\mathrm{ss}}=\cH_T(w)^{\varepsilon \text{-ss}}\sqcup \mathcal{S}_1 \sqcup \cdots \sqcup \mathcal{S}_N 
\end{align}
with center $\cS_i \to \mathcal{Z}_i$. We regard $\cS_i$, $\mathcal{Z}_i$ as derived stacks where the derived structures are given by the descriptions as mapping stacks from $\Theta$ or $\bgm$ as in Subsection~\ref{subsec:notation0}, see~\cite{Halpinstab, HalpK32}.
We have the corresponding diagram 
\begin{align}\label{dia:S}
\xymatrix{
\cS_i \inclusion \ar[d] 
& \cH_T(w)^{\mathrm{ss}}. \\
\mathcal{Z}_i \ar[ru] \ar@/^12pt/[u]
}
\end{align}

Let $\lambda \colon \bgm \to \mathcal{Z}_1$ be the canonical 
map over $0\in T$. By the identification $\cC_T\times_T \{0\}=\cC_b$, it corresponds to a direct sum 
\begin{align}\label{map:lambda}
    \lambda \colon \mathrm{pt} \mapsto E_1 \oplus \cdots \oplus E_a, \ E_j \in \Coh^{\heartsuit}(\cC_b)
\end{align}
where each $E_j$ is $h_T^{\varepsilon}$-semistable such that 
\begin{align}\label{slope:cond}
 \mu^{\varepsilon}(E_1)>\cdots>\mu^{\varepsilon}(E_a), \ 
 \mu^0(E_1)=\cdots=\mu^0(E_a). 
 \end{align}
 Moreover the $\mathbb{G}_m$-weight of $E_1 \oplus \cdots \oplus E_a$ is 
 of the form 
\begin{align*}(\lambda_1, \ldots, \lambda_a), \ 
\lambda_1>\cdots>\lambda_a.
\end{align*}

Below we write 
\begin{align*}
D_j=\mathrm{Supp}(E_j), \ R_j=\operatorname{rank}(\pi_{*}E_j), \ 
C_{ij}=C_i \cap D_j, \ r_{ij}=\operatorname{rank}(\pi_{*}\mathcal{O}_{\cC_{ij}}).
\end{align*}
Note that 
\begin{align*}
    r_i=r_{i1}+\cdots+r_{ia}, \ R_j=r_{1j}+\cdots+r_{kj}
\end{align*}
and that $\deg \pi_{*}E_j=R_j \cdot w/r$ by the second condition of (\ref{slope:cond}). 
\begin{lemma}\label{slope:equiv}
   The condition (\ref{slope:cond}) implies that 
\begin{align}\label{slope:cond2}
\frac{1}{R_j}\sum_{i=1}^k \frac{r_{ij}}{r_i}\chi_i < 
\frac{1}{R_{j'}}\sum_{i=1}^k \frac{r_{ij'}}{r_i}\chi_i, \ j'>j.
\end{align} 
\end{lemma}
\begin{proof}
The lemma follows from the following computation, together with the $w>0$ assumption: 
\begin{align*}
    \mu^{\varepsilon}(E_j)=\frac{w}{r} \cdot \frac{R_j}{\sum_{i=1}^k r_{ij}\left(1+ \left(\frac{\chi_i}{r_i}-\frac{\chi}{r}\right)\varepsilon \right)}.
\end{align*}
\end{proof}

We will also use the following lemma: 
\begin{lemma}\label{lem:SZqsm}
    The stacks $\cS_i$, $\mathcal{Z}_i$ are quasi-smooth, 
    and the maps $\cS_i \to \mathcal{Z}_i$, $\cS_i \hookrightarrow \cH_T(w)^{\mathrm{ss}}$ are quasi-smooth. 
\end{lemma}
\begin{proof}
    The statements of the lemma except the quasi-smoothness of $\cS_i \hookrightarrow \cH_T(w)^{\mathrm{ss}}$ follow from~\cite[Lemma~3.1.4]{HalpK32}.
    We prove the quasi-smoothness of $\cS_i \hookrightarrow \cH_T(w)^{\mathrm{ss}}$. For simplicity, we only prove the case $i=N$, i.e. for the closed stratum. 
    For a map $\lambda$ as in (\ref{map:lambda}), by~\cite[Lemma~1.3.2]{HalpK32}, we have 
    \begin{align}\label{lambda:pull}
        \lambda^* \mathbb{L}_{\cS_N/\cH_T(w)}=\widetilde{\lambda}^* \mathbb{L}_{\cH_T(w)}^{>0}[1]=  \bigoplus_{j<j'}\Hom(E_j, E_{j'})[2].
    \end{align}
    Here $\widetilde{\lambda}$ is the composition of $\lambda$ with 
    $\mathcal{Z}_N \to \cH_T(w)^{\mathrm{ss}}$. 
   We have $\Hom(E_j, E_{j'})=0$ for $j\neq j'$
   as $D_j$ and $D_{j'}$ do not have common 
   irreducible components and $E_j$ are pure one-dimensional sheaves on $D_j$. 
   By the Serre duality, we also have $\Hom^2(E_j, E_{j'})=0$. 
   Therefore (\ref{lambda:pull}) has only degree $-1$-cohomology, hence 
   $\cS_N \hookrightarrow \cH_T(w)^{\mathrm{ss}}$ is quasi-smooth
   at the image of $\lambda$. Since this holds for any $\lambda$, and 
   any open neighborhood of the image of $\mathcal{Z}_N \to \cS_N$ is 
   $\cS_N$ itself, we conclude that $\cS_N \hookrightarrow \cH_T(w)^{\mathrm{ss}}$ is quasi-smooth. The same argument applies 
   to all other stratum $\cS_i \hookrightarrow \cH_T(w)^{\mathrm{ss}}$. 
\end{proof}

We took maps from $\bgm$ both to $\cH(\chi)$ and $\cH(w)^{\mathrm{ss}}$. The notation
is summarized in Table~\ref{tab:notation}. 
\begin{table}[t]
\centering
\caption{Notation on the automorphic and spectral sides}
\label{tab:notation}
\begin{tabular}{l|c|c}
 & automorphic side & spectral side
 \\
\hline
moduli stack 
& $\mathcal{H}(\chi)$
& $\mathcal{H}(w)^{\mathrm{ss}}$
 \\
maps from $\bgm$
& $\nu$
& $\lambda$
 \\
 index
& $1\leq i\leq k$
& $1\leq j\leq a$
 \\
$\mathbb{G}_m$-weights
& $\nu_i$
& $\lambda_j$
 \\
components
& $M_i$
& $E_j$
 \\
ranks
& $r_i$
& $R_j$
 \\
degrees
& $\chi_i$
& $R_j \cdot w/r$
 \\
supports
& $C_i$
& $D_j$
\end{tabular}
\end{table}

\subsection{Proof of cohomology vanishings}\label{subsec:pfvanish}
In this subsection, we give a certain local cohomology vanishing which 
will be used to show the isomorphism in (iii) in Subsection~\ref{subsec:idea}. 
Below, we follow the setting in the previous subsection. 

 For $A \in \LL(\cH(w)^{\mathrm{ss}})_{-\chi'}$, its $*$-pull-back to $\cH_b$ is denoted by 
    \begin{align*}A_b:=A|_{\cH_b} \in \Coh(\cH_b(w)^{\mathrm{ss}})_{-\chi'}.
    \end{align*}
    We regard objects in $\Coh(\cH_b(w)^{\mathrm{ss}})$ as an object in $\Coh(\cH_T(w)^{\mathrm{ss}})$
    by the $*$-push-forward along the closed immersion \begin{align*}\cH_b(w)^{\mathrm{ss}}=\cH_T(w)^{\mathrm{ss}}\times_T \{0\}\hookrightarrow \cH_T(w)^{\mathrm{ss}},
    \end{align*}
    For example, $A_b \in \Coh(\cH_T(w)^{\mathrm{ss}})$.

    For simplicity, we write $\cS=\cS_N$ and $\mathcal{Z}=\mathcal{Z}_N$
    for the HN stratification (\ref{Theta:eps}). 
Then we have
\begin{align}
    \label{show:zero}
    \cHom_{\cH_T(w)^{\mathrm{ss}}}(\mathcal{O}_{\cS}, A_b \otimes \mathcal{V}_M)
    \cong (A_b \otimes \mathcal{V}_M)|_{\cS}\otimes \det \mathbb{L}_{\cS/\cH_T(w)^{\mathrm{ss}}}[-d].
\end{align}
Here $d$ is the relative dimension of $\cS$ in $\cH_T(w)^{\mathrm{ss}}$. 
The above isomorphism holds by the adjunction $(i_{S*}, i_{S}^{!})$
for the closed immersion $i_S \colon \cS \hookrightarrow \cH_T(w)^{\mathrm{ss}}$, where, as $i_S$ is quasi-smooth by Lemma~\ref{lem:SZqsm}
 we have (see Subsection~\ref{subsec:notation0} or~\cite[Proposition~7.3.8]{MR3136100})
 \begin{align*}
     i_S^!(-)=i_S^*(-) \otimes \det \mathbb{L}_{\cS/\cH_T(w)^{\mathrm{ss}}}[-d].
 \end{align*}

We prove the following lemma: 
\begin{lemma}\label{lem:nweight}
    The object (\ref{show:zero}) has negative $\lambda$-weights on $\mathcal{Z}$ along the diagram (\ref{dia:S}). 
\end{lemma}
\begin{proof} We divide the proof into 2 steps. 
\begin{sstep}
    Reduction to the combinatorial inequality. 
\end{sstep}
Since $A\in \LL(\cH(w)^{\mathrm{ss}})_{-\chi'}$, the $\lambda$-weights
of $A_b|_{\mathcal{S}}$ are contained in 
\begin{align}\label{wt:A}
    \left[\frac{1}{2}c_1(\lambda^*\mathbb{L}_{\cH_T(w)^{\mathrm{ss}}}^{<0}), 
    \frac{1}{2}c_1(\lambda^*\mathbb{L}_{\cH_T(w)^{\mathrm{ss}}}^{>0})
    \right]+c_1(\lambda^* \delta_{-\chi'}).
\end{align}
Therefore 
by noting that $\chi'=\chi+r(r-1)(g-1)$ and Lemma~\ref{lem:Gmwt:L}, 
a $\lambda$-weight of $A_b|_{\mathcal{S}}$
is at most
\begin{align}\label{wt:1}
    (g-1)\sum_{1\leq j<j' \leq a}R_j R_{j'}(\lambda_j-\lambda_{j'})-\frac{\chi}{r}\sum_{j=1}^a R_j \lambda_j-(r-1)(g-1)\sum_{j=1}^a R_j\lambda_j.
\end{align}
The $\lambda$-weight of 
$\det \mathbb{L}_{\cS/\cH_T(w)^{\mathrm{ss}}}$
is \begin{align}\label{wt:2}
c_1(\lambda^*\mathbb{L}_{\cH_T(w)^{\mathrm{ss}}}^{<0})=(2g-2)\sum_{1\leq j'<j \leq a}R_j R_{j'}(\lambda_j-\lambda_{j'}).
\end{align}

We will study $\lambda$-weights of $\mathcal{V}_M|_{\cS}$. 
Let 
\begin{align*}
    N_i \subset \mathrm{Sing}(C_i)
\end{align*}
be the subset on which $M_i$ is not locally free. 
By Lemma~\ref{lem:Qp}, we have the exact sequence 
\begin{align}\label{ex:LM1}
    0\to L_i \to M_i \to \bigoplus_{p\in N_i}Q_p \to 0
\end{align}
where $L_i$ is a line bundle on $C_i$ and $Q_p \cong \mathcal{O}_{W_p}$ for 
a closed subscheme $W_p \hookrightarrow Z_p$ with length $1\leq n_p \leq m_p/2$, where $Z_p$ is given in (\ref{def:Zp}). Moreover we have the exact sequence 
\begin{align}\label{ex:LM2}
    0\to L \to \bigoplus_{i=1}^k L_i \to \bigoplus_{p\in C_i \cap C_{i'}, i<i'}Q_p \to 0
\end{align}
where $L$ is a line bundle on $\cC_b$ obtained by gluing $L_i$ at each intersection 
points of $C_i \cap C_{i'}$; it satisfies that 
$L|_{C_i}\cong L_i$ and $Q_p \cong \mathcal{O}_{Z_p}$ 
for each $p\in C_i \cap C_{i'}$. 
Then we have 
\begin{align*}
    \mathcal{V}_M \cong \mathcal{P}_L \otimes \bigotimes_{p\in N_i, 1\leq i\leq k}\mathcal{V}_p^{\otimes n_p} \otimes 
    \bigotimes_{p\in C_i \cap C_{i'}, i<i'}\mathcal{V}_p^{\otimes m_p/2}.
\end{align*}
Here note that $m_p$ is even if $p$ is an intersection of irreducible components. By Lemma~\ref{lem:PLGm}, 
a $\lambda$-weight of 
$\mathcal{P}_L$ is given by 
\begin{align}\label{wt:3}
\sum_{j=1}^a l_j \lambda_j, \ l_j:=\deg_{D_j}(L|_{D_j})
\end{align}
Since $D_j=C_{1j}+\cdots+C_{kj}$ and $L|_{C_{ij}}\cong L_i|_{C_{ij}}$, 
by Lemma~\ref{lem:add} we have 
\begin{align}\label{lj}
l_j=\deg_{C_{1j}}(L_1|_{C_{1j}})+\deg_{C_{2j}}(L_2|_{C_{2j}})+\cdots+ \deg_{C_{kj}}(L_k|_{C_{kj}}).
\end{align}

A $\lambda$-weight of $\bigotimes_{p\in N_i, 1\leq i\leq k}\mathcal{V}_p^{\otimes n_p}$
is of the form 
\begin{align}\label{wt:4}
    \sum_{i=1}^k \sum_{p\in N_i} (a_p\lambda_{\alpha(p)}+b_p\lambda_{\beta(p)}), 
\end{align}
where $p \in C_{i\alpha(p)}\cap C_{i\beta(p)}$ for $1\leq \alpha(p) \leq \beta(p) \leq a$, 
and $a_p+b_p=n_p$ with $a_p, b_p\geq 0$. 

Since we have 
\begin{align*}
    C_{i} \cap C_{i'}=\bigcup_{1\leq j, j' \leq a} C_{ij} \cap C_{i'j'},
\end{align*}
and (noting that $m_p/2$ is the intersection multiplicity at $p\in C_{ij} \cap C_{i' j'}$)
\begin{align*}
    C_{ij} \cdot C_{i' j'}=\sum_{p\in C_{ij} \cap C_{i' j'}}\frac{m_p}{2}=r_{ij} r_{i' j'}(2g-2), 
\end{align*}
a $\lambda$-weight of $\bigotimes_{p\in C_i \cap C_{i'}, i<i'}\mathcal{V}_p^{\otimes m_p/2}$
is of the form 
\begin{align}\label{wt:5}
    \sum_{1\leq i<i' \leq k} \sum_{1\leq j, j' \leq a}(c_{i i' j j'}\lambda_j+
    d_{i i' j j'}\lambda_{j'})
\end{align}
where $c_{i i' j j'},  d_{i i' j j'}$ satisfy 
\begin{align}\label{cond:bc}
    c_{i i' j j'}\geq 0, \ d_{i i' j j'} \geq 0, \  c_{i i' j j'}+d_{i i' j j'}=r_{ij} r_{i' j'}(2g-2).
\end{align}
By (\ref{wt:1}), (\ref{wt:2}), (\ref{wt:3}), (\ref{wt:4}), (\ref{wt:5}), 
a $\lambda$-weight of (\ref{show:zero}) is at most 
\begin{align}\label{wt:final}
&(g-1)\sum_{1\leq j<j' \leq a}R_j R_{j'}(\lambda_j-\lambda_{j'})-\frac{\chi}{r}\sum_{j=1}^a R_j \lambda_j
 -(r-1)(g-1)\sum_{j=1}^a R_j \lambda_j 
\\
&\notag +(2g-2)\sum_{1\leq j'<j \leq a}R_j R_{j'}(\lambda_j-\lambda_{j'})
+\sum_{j=1}^a l_j \lambda_j \\
&\notag + \sum_{i=1}^k \sum_{p\in N_i} (a_p\lambda_{\alpha(p)}+b_p\lambda_{\beta(p)})
+ \sum_{1\leq i<i' \leq k} \sum_{1\leq j, j' \leq a}(c_{i i' j j'}\lambda_j+
    d_{i i' j j'}\lambda_{j'}).
\end{align}
\begin{sstep}
A $\lambda$-weight (\ref{wt:final}) is negative. 
\end{sstep}
Since $\deg \pi_{*}M=\chi$, by the formula (\ref{formula:deg}) we have 
\begin{align*}
    \deg_{\cC_b}(M)=\chi'=\chi+r(r-1)(g-1).
\end{align*}
By (\ref{ex:LM1}), (\ref{ex:LM2}), we have 
\begin{align*}
    \deg_{\cC_b}(M)
    =\deg_{\cC_b}(L)+\sum_{i=1}^k \sum_{p\in N_i}n_p
    +\sum_{1\leq i<i'\leq k}\sum_{1\leq j, j' \leq a}r_{ij} r_{i'j'}(2g-2).
\end{align*}
Using Lemma~\ref{lem:add}, it follows that 
\begin{align}\label{sum:l:lambda}
    &l_1+\cdots+l_a+\sum_{i=1}^k \sum_{p\in N_i}n_p +\sum_{1\leq i<i'\leq k}\sum_{1\leq j, j' \leq a}r_{ij} r_{i'j'}(2g-2) \\
    &\notag =\chi+r(r-1)(g-1).
\end{align}
By (\ref{sum:l:lambda}), the weight (\ref{wt:final}) is of the form 
\begin{align*}
    \sum_{j=1}^a \alpha_j \lambda_j, \ \alpha_1+\cdots+\alpha_a=0.
\end{align*}
Therefore it is written as 
\begin{align*}
    \sum_{m=1}^{a-1} \beta_m (\lambda_m-\lambda_{m+1}), \ \beta_m=\sum_{j=1}^m \alpha_j.
\end{align*}
Since $\lambda_m-\lambda_{m+1}>0$, it is enough to show that $\beta_m<0$. 

For a subset $J\subset \{1, \ldots, a\}$, 
we write 
\begin{align*}r_{iJ}:=\sum_{j\in J}r_{ij}, \ l_{J}:=\sum_{j\in J}l_j, \ 
R_J:=\sum_{j\in J} R_j.
\end{align*}
Let $J_m=\{1, \ldots, m\} \subset \{1, \ldots, a\}$, and $J_m^{\circ}$ its complement. 
A direct computation shows that 
\begin{align*}
    \beta_m&=-(g-1)R_{J_m}R_{J_m^{\circ}}-\frac{\chi}{r}R_{J_m}-(r-1)(g-1)R_{J_m}+l_{J_m} \\
    &+\sum_{i=1}^k\left(\sum_{p\in N_i, \alpha(p)\leq m}a_p+\sum_{p\in N_i, \beta(p)\leq m}b_p\right)+B_{J_m}.
\end{align*}
Here $B_{J_m}$ is given by 
\begin{align*}
    B_{J_m}=\sum_{i<i'}\left(\sum_{j\leq m, j'\leq a}c_{i i' j j'}+\sum_{j\leq a, j'\leq m}d_{i i' j j'}\right).
\end{align*}
By the condition (\ref{cond:bc}), 
we have 
\begin{align*}
    B_{J_m} \leq (2g-2)\sum_{i<i'}(r_{iJ_m}r_{i'J_m}+r_{iJ_m}r_{i' J_m^{\circ}}+r_{i J_m^{\circ}}r_{i' J_m}).
\end{align*}

For a subset $J\subset \{1, \ldots, a\}$, we have the 
exact sequence of the form 
\begin{align*}
    0\to M_i' \to M_i \to M_i|_{C_{iJ}}^{\mathrm{free}} \to 0.
\end{align*}
Here $C_{iJ}:=\cup_{j\in J} C_{ij}$. 
By the semistability of $M_i$, we have 
\begin{align}\label{wt:M}
    \frac{\deg \pi_{*}M_i'}{r_{iJ^{\circ}}} \leq \frac{\chi_i}{r_i}.
\end{align}
Similarly to (\ref{ex:E2}), we have the 
exact sequence 
\begin{align*}
    0\to L_i|_{C_{iJ^{\circ}}}(-C_{iJ}) \to M_i' \to\bigoplus_{p\in N_i \cap C_{iJ^{\circ}}}Q_p\to 0.
\end{align*}
Together with (\ref{wt:M}), we obtain the inequality
\begin{align*}
    \frac{1}{r_{iJ^{\circ}}}\left(\deg \pi_{*}(L_{i}|_{C_{iJ^{\circ}}}(-C_{iJ}))+\sum_{p\in N_i \cap C_{iJ^{\circ}}}n_p  \right)\leq\frac{\chi_i}{r_i}.
\end{align*}
By multiplying $r_{iJ^{\circ}}$ and replacing $J$ with $J^{\circ}$ and taking the sum over $1\leq i \leq k$, 
from (\ref{lj}) we have 
\begin{align}\label{wt:lJ}
    l_J \leq \sum_{i=1}^k \left(\frac{\chi_i}{r_i}r_{iJ}+r_{iJ}(r_{iJ}-1)(g-1)+r_{iJ}r_{iJ^{\circ}}(2g-2)-\sum_{p\in N_i \cap C_{iJ}}n_p   \right).
\end{align}
It follows that 
\begin{align*}
    \beta_m \leq & -(g-1)R_{J_m}R_{J_m^{\circ}}-\frac{\chi}{r}R_{J_m}-(r-1)(g-1)R_{J_m} \\
    &+\sum_{i=1}^k \left(\frac{\chi_i}{r_i}r_{iJ_m}+r_{iJ_m}(r_{iJ_m}-1)(g-1)+r_{iJ_m}r_{iJ_m^{\circ}}(2g-2)-\sum_{p\in N_i \cap C_{iJ_m}}n_p  \right) \\
    &+\sum_{i=1}^k \left(\sum_{p\in N_i, \alpha(p)\leq m}a_p+\sum_{p\in N_i, \beta(p)\leq m}b_p\right) \\
    &+(2g-2)\sum_{i<i'}(r_{iJ_m}r_{i'J_m}+r_{iJ_m}r_{i' J_m^{\circ}}+r_{i J_m^{\circ}}r_{i' J_m}).
\end{align*}
In the right-hand side, a straightforward computation shows that the terms which are divisible by $(g-1)$
vanish: 
\begin{align*}
    (g-1)&(-R_{J_m}R_{J_m^{\circ}}-(r-1)R_{J_m}+\sum_{i=1}^k(r_{iJ_m}(r_{iJ_m}-1)+2r_{iJ_m}r_{iJ_m^{\circ}}) \\
    &+2\sum_{i<i'}(r_{iJ_m}r_{i'J_m}+r_{iJ_m}r_{i'J_m^{\circ}}+r_{iJ_m^{\circ}}r_{i'J_m}))=0.
\end{align*}
Therefore we obtain 
\begin{align}\label{ineq:beta}
    \beta_m &\leq -\frac{\chi}{r}R_{J_m}+\sum_{i=1}^k \left(\frac{\chi_i}{r_i}r_{iJ_m}-
    \sum_{p\in N_i \cap C_{iJ_m}}n_p
  +\sum_{p\in N_i, \alpha(p)\leq m}a_p+\sum_{p\in N_i, \beta(p)\leq m}b_p  \right) \\
    &\notag \leq -\frac{\chi}{r}R_{J_m}+\sum_{i=1}^k \frac{\chi_i}{r_i}r_{iJ_m}.
\end{align}
Here the last inequality follows since $p\in N_i$ with $\alpha(p) \leq m$ implies 
$p\in C_{iJ_m}$, $\alpha(p)\leq \beta(p)$ and $a_p+b_p=n_p$.

By (\ref{slope:cond2}), we have 
\begin{align}\label{ineq:q}
    q_1<\cdots<q_a, \ q_j:=\frac{1}{R_j}\sum_{i=1}^k \frac{r_{ij}}{r_i}\chi_i.
\end{align}
It follows that 
\begin{align*}
    \frac{1}{R_{J_m}}\sum_{i=1}^k \frac{\chi_i}{r_i}r_{i J_m}=\frac{\sum_{j\leq m}R_j q_j}{\sum_{j\leq m}R_j}<\frac{\sum_{j\leq a}R_j q_j}{\sum_{j\leq a}R_j}=\frac{\chi}{r}.
\end{align*}
Therefore $\beta_m<0$ follows. 
\end{proof}

Using the above lemma, we prove the following proposition: 
\begin{prop}\label{prop:vanishing}
We have the vanishing 
\begin{align}\label{vanish:Gamma}
    \Gamma(\cH_T(w)^{\mathrm{ss}}, \Gamma_{\cS}(A_b \otimes \mathcal{V}_M))=0. 
\end{align}
\end{prop}
\begin{proof}
There is a sequence of closed immersions 
\begin{align*}
   \cS=\cS^{(1)} \hookrightarrow \cS^{(2)} \hookrightarrow \cS^{(3)} \hookrightarrow 
    \cdots \hookrightarrow \cH_T(w)^{\mathrm{ss}}
\end{align*}
where each $\cS^{(n)}$ is the $n$-th infinitesimal thickening of $\cS$ in $\cH_T(w)^{\mathrm{ss}}$, 
whose structure sheaves fit into distinguished triangles 
\begin{align*}
    \mathrm{Sym}^{n}(\mathbb{L}_{\cS/\cH_T(w)^{\mathrm{ss}}}[-1]) \to \mathcal{O}_{\cS^{(n+1)}} \to \mathcal{O}_{\cS^{(n)}}.
\end{align*}
Then $\cHom_{\cH_T(w)^{\mathrm{ss}}}(\mathcal{O}_{\cS^{(n)}}, A_b \otimes \mathcal{V}_M)$ is filtered by 
\begin{align}\label{filt:negative}
(A_b \otimes \mathcal{V}_M)|_{\cS} \otimes \det \mathbb{L}_{\cS/\cH_T(w)^{\mathrm{ss}}}[-d]\otimes\mathrm{Sym}^i(\mathbb{L}_{\cS/\cH_T(w)^{\mathrm{ss}}}[-1])^{\vee}  
\end{align}
for $0\leq i\leq n-1$. The first two factors have negative 
$\lambda$-weights by Lemma~\ref{lem:nweight}. Also since $\cS \subset \cH_T(w)^{\mathrm{ss}}$ is a $\Theta$-stratum, the object 
$\mathbb{L}_{\cS/\cH_T(w)^{\mathrm{ss}}}$ has positive $\lambda$-weights, 
see~\cite[Lemma~1.3.2]{HalpK32}. Therefore the last factor of (\ref{filt:negative}) has also negative $\lambda$-weights, and 
we conclude that (\ref{filt:negative}) has negative 
$\lambda$-weights. 

Now for the good moduli space map 
$r_T \colon \cH_T(w)^{\mathrm{ss}} \to \mathrm{H}_T(w)^{\mathrm{ss}}$, we have 
the commutative diagram 
\begin{align*}
    \xymatrix{
    \cH_T(w)^{\mathrm{ss}} \ar[rr]^-{r_T} & &\mathrm{H}_T(w)^{\mathrm{ss}} \\
    \cS \uinclusion \ar[r]^-{g_T} & \mathcal{Z} \ar[r]^-{g_T'} & Z. \ar[u]
    }
\end{align*}
Here the top and the right bottom horizontal arrows are good moduli space morphisms.  
Since (\ref{filt:negative}) has negative $\lambda$-weights, 
its push-forward $g_{T*}$ has also negative $\lambda$-weights; this follows from~\cite[Lemma~1.5.6]{HalpK32} which says that $g_{T*}$ does not increase the maximum $\lambda$-weights. 
Since the push-forward $g_{T*}'$ to the good moduli space $Z$ extracts the weight-zero
part along the stabilizer $\mathbb G_m$ on $\mathcal{Z}$, an object 
$E \in \QCoh(\mathcal{Z})$ with strictly negative $\lambda$-weights satisfies $g_{T*}'E=0$. 
Therefore from the above commutative diagram, we have 
\begin{align}\label{vanish:rt}
    r_{T*}\cHom_{\cH_T(w)^{\mathrm{ss}}}(\mathcal{O}_{\cS^{(n)}}, A_b \otimes \mathcal{V}_M)=0. 
\end{align}

Note that we have the isomorphisms in $\QCoh(\cH_T(w)^{\mathrm{ss}})$
\begin{align*}
\Gamma_{\cS}(A_b\otimes \mathcal{V}_M)&\cong 
\operatorname*{colim}_{\cS \subset \cS'}
    \cHom_{\cH_T(w)^{\mathrm{ss}}}(\mathcal{O}_{\cS'}, A_b \otimes \mathcal{V}_M) \\
    &\cong \operatorname*{colim}_{n\geq 0}\cHom_{\cH_T(w)^{\mathrm{ss}}}(\mathcal{O}_{\cS^{(n)}}, A_b \otimes \mathcal{V}_M).
    \end{align*}
    Since $r_{T*}$ is continuous on $\QCoh$, by (\ref{vanish:rt}) we obtain the vanishing 
    \begin{align*}
        r_{T*}\Gamma_{\cS}(A_b \otimes \mathcal{V}_M)=0.
    \end{align*}
Therefore by taking the global section over $\mathrm{H}_T(w)^{\mathrm{ss}}$, we obtain the vanishing (\ref{vanish:Gamma}). 
\end{proof}

Recall the $\Theta$-stratification (\ref{Theta:eps}). 
We consider the following closed substack 
\begin{align*}
    \cS_{\geq k}:=\cS_k \sqcup \cdots \sqcup \cS_N \subset \cH_T(w)^{\mathrm{ss}}.
\end{align*}
Similarly to Proposition~\ref{prop:vanishing}, we have the following 
proposition: 
\begin{prop}\label{prop:vanish2}
We have the vanishing 
    \begin{align}\label{vanish:Gamma2.5}
\Gamma(\cH_T(w)^{\mathrm{ss}}, 
     \Gamma_{\cS_{\geq k}}(A_b\otimes \mathcal{V}_M))=0.
\end{align}
\end{prop}
\begin{proof}
We replace the argument of Proposition~\ref{prop:vanishing} to 
the $\Theta$-stratum 
\begin{align*}
    \cS_k \subset \cH_T(w)^{\mathrm{ss}}_{\leq k}:=\cH_T(w)^{\mathrm{ss}}\setminus \cS_{\geq k+1}
\end{align*}
to conclude that 
\begin{align}\label{vanish:Gamma2}
  \Gamma(\cH_T(w)^{\mathrm{ss}}_{\leq k}, \Gamma_{\cS_k}((A_b \otimes \mathcal{V}_M)|_{\cH_T(w)_{\leq k}^{\mathrm{ss}}}))=0. 
\end{align}
We have the distinguished triangle in $\QCoh(\cH_T(w)^{\mathrm{ss}})$
\begin{align*}
    \Gamma_{\cS_{\geq k+1}}(A_b\otimes \mathcal{V}_M) \to  \Gamma_{\cS_{\geq k}}(A_b\otimes \mathcal{V}_M)
    \to (j_{\leq k})_{*}\Gamma_{\cS_k}((A_b \otimes \mathcal{V}_M)|_{\cH_T(w)^{\mathrm{ss}}_{>k}}).
\end{align*}
Here $j_{\leq k} \colon \cH_T(w)^{\mathrm{ss}}_{\leq k} \hookrightarrow \cH_T(w)^{\mathrm{ss}}$ is the open immersion. It follows that 
we have the isomorphism 
\begin{align*}
    \Gamma(\cH_T(w)^{\mathrm{ss}},  \Gamma_{\cS_{\geq k+1}}(A_b\otimes \mathcal{V}_M)) \stackrel{\cong}{\to} 
    \Gamma(\cH_T(w)^{\mathrm{ss}}, \Gamma_{\cS_{\geq k}}(A_b\otimes \mathcal{V}_M)).
\end{align*}
Therefore by (\ref{vanish:Gamma}) and the induction on $k$, we conclude (\ref{vanish:Gamma2.5}). 
\end{proof}

\subsection{The case of the structure sheaf}\label{subsec:stO}
We apply the above argument to the case of 
\begin{align*}
    A=\mathcal{O}_{\cH(w)^{\mathrm{ss}}} \in \LL(\cH(w)^{\mathrm{ss}})_0.
\end{align*}
Here we generalize the construction in Subsection~\ref{subsec:setting}, and further perturb the 
$\mathbb{Q}$-ample divisor $h_T^{\varepsilon}$. 

\subsubsection*{Further perturbation of $h_T^{\varepsilon}$}
Let $\nu$ be a map 
\begin{align}\label{map:v2}
    \nu \colon \bgm \to \cH(\chi_0)
\end{align}
over $b\in \cB$, 
as in (\ref{nu:F12}), 
corresponding to the direct sum 
\begin{align*}M=M_1\oplus M_2 \oplus \cdots \oplus M_k
\end{align*}
where each $M_i$ is semistable satisfying (\ref{Ci}), 
with the same notation $(C_i, r_i, \chi_i)$ as in (\ref{Ci}), 
but with a weaker condition than (\ref{slope:rchi})
\begin{align*}
    \frac{\chi_1}{r_1} \geq \cdots \geq \frac{\chi_k}{r_k}, \ (r_1, \chi_1)+\cdots+(r_k, \chi_k)=(r, \chi_0).
\end{align*}

We take an \'{e}tale neighborhood $(T, 0) \to (\cB, b)$ as in Lemma~\ref{lem:etale} 
and an ample $\mathbb{Q}$-divisor $h_T^{\varepsilon}$ on $\cC_T$
as in (\ref{div:hT}). 
We choose 
\begin{align}\label{choice:alpha}
    (\alpha^{(a)})_{1\leq a\leq m} \in \mathbb{Q}^m, \quad 
    \sum_{a=1}^m r^{(a)} \cdot \alpha^{(a)}=0.
\end{align}
We take a further perturbation $h_T^{\varepsilon \dag}$ of $h_T^{\varepsilon}$
satisfying
\begin{align}\label{hTeps}
    h_T^{\varepsilon \dag}\cdot \cC_b^{(a)}=r^{(a)}\left(1+\left(\frac{\chi_i}{r_i}-\frac{\chi}{r}  \right)\varepsilon +\alpha^{(a)}\varepsilon^2   \right).
\end{align}
for any irreducible component $\cC_b^{(a)} \subset \cC_b$
with $\cC_b^{(a)} \subset C_i$. 
Similarly to (\ref{ineq:muep}), the slope function 
\begin{align*}
    \mu^{\varepsilon \dag}(E)=\frac{\deg \pi_{*}E}{h_T^{\varepsilon \dag} \cdot [E]} \in \mathbb{Q} \cup \{\infty\}
\end{align*}
defines the $h_{T}^{\varepsilon \dag}$-stability on $\Coh^{\heartsuit}(\cC_t)$ for $t\in T$. 

Since the $h_T^{\varepsilon \dag}$-stability refines the $h_T^{\varepsilon}$-stability, 
we have the open immersions 
\begin{align*}
    \cH_T(w)^{\varepsilon \dag \text{-ss}} \subset \cH_T(w)^{\varepsilon \text{-ss}} \subset \cH_T(w)^{\mathrm{ss}}. 
\end{align*}
For each $1\leq a\leq m$, we denote by $i(a)$ the unique $1\leq i(a) \leq k$
such that $\cC_b^{(a)} \subset C_{i(a)}$. 
We have the following lemma: 
\begin{lemma}\label{lem:alphagen}
Suppose that for any non-empty proper subset $K\subset \{1, \ldots, m\}$, 
we have 
\begin{align}\label{cond:K}
    \sum_{a\in K}r^{(a)}\left( \frac{\chi_{i(a)}}{r_{i(a)}}-\frac{\chi}{r} \right)=0 \quad \Rightarrow \quad \sum_{a\in K} r^{(a)} \alpha^{(a)} \neq 0.
\end{align}
Then there are no strictly $h_T^{\varepsilon\dag}$-semistable sheaves, i.e.\ 
\begin{align}\label{ss=st}
    \cH_T(w)^{\varepsilon \dag \text{-ss}}=\cH_T(w)^{\varepsilon \dag \text{-st}}.
\end{align}
\end{lemma}
\begin{proof}
Suppose that there is a strictly $h_T^{\varepsilon\dag}$-semistable sheaf. 
Then there is a non-empty proper subset $K\subset \{1, \ldots, m\}$ and 
$w'\in \mathbb{Z}$ with an equality 
\begin{align}\label{equal:ra}
    \frac{w'}{\sum_{a\in K}r^{(a)}\left(1+\left(\frac{\chi_i}{r_i}-\frac{\chi}{r}  \right)\varepsilon +\alpha^{(a)}\varepsilon^2\right)}
    =\frac{w}{r}
\end{align}
for $0<\varepsilon \ll 1$. By setting $\varepsilon \to 0$, we obtain 
\begin{align*}
    \frac{w'}{\sum_{a\in K}r^{(a)}}
    =\frac{w}{r}.
\end{align*}
By substituting into (\ref{equal:ra}), and using the $w>0$ assumption, 
we obtain 
\begin{align*}
     \sum_{a\in K}r^{(a)}\left( \frac{\chi_{i(a)}}{r_{i(a)}}-\frac{\chi}{r} \right)+\varepsilon\cdot \sum_{a\in K} r^{(a)} \alpha^{(a)} =0
\end{align*}
for $0<\varepsilon \ll 1$. This does not happen precisely when (\ref{cond:K}) 
is satisfied. 
\end{proof}

Note that the condition (\ref{cond:K}) is satisfied for a generic choice of 
  $(\alpha^{(a)})_{1\leq a\leq m}$ in (\ref{choice:alpha}). 
In this case, the identity (\ref{ss=st}) holds. 
In particular, the stack $\cH_T(w)^{\varepsilon \dag \text{-ss}}$
is a $\mathbb{G}_m$-gerbe 
\begin{align}\label{Gmgerb}
    \cH_T(w)^{\varepsilon \dag \text{-ss}} \to H_T(w)^{\varepsilon \dag \text{-ss}}
\end{align}
over a scheme $H_T(w)^{\varepsilon \dag \text{-ss}}$, which is projective over $T$, 
\begin{align*}
    H_T(w)^{\varepsilon \dag \text{-ss}}\to T.
\end{align*}
Indeed the above map is a $T$-relative fine compactified Jacobian of 
$\cC_T\to T$ with generic polarization $h_T^{\vd}$ as in Subsection~\ref{subsec:FMJ}. 

We denote by 
\begin{align*}
    \cH_T(w)^{\mathrm{ss}}=
    \cH_T(w)^{\varepsilon\dag\text{-ss}} \sqcup 
    \cS_1^{\dag} \sqcup \cdots \sqcup \cS_{N'}^{\dag} 
\end{align*}
the Harder--Narasimhan stratification with respect to the $h_T^{\varepsilon \dag}$-stability, 
with center $\mathcal{S}_i^{\dag} \to \mathcal{Z}_i^{\dag}$. We have the following analogue of Proposition~\ref{prop:vanish2}: 
\begin{lemma}\label{lem:vanish2dag}
We have the vanishing 
\begin{align}\label{vanish:gammaO}
\Gamma(\cH_T(w)^{\mathrm{ss}}, \Gamma_{\cS_{\geq k}^{\dag}}(\mathcal{V}_M))=0.
\end{align}
\end{lemma}
\begin{proof}
We write $\cS^{\dag}=\cS_{N'}^{\dag}$ and $\mathcal{Z}^{\dag}=\mathcal{Z}_{N'}^{\dag}$. 
   Let $\lambda \colon \bgm \to \mathcal{Z}^{\dag}$ be a map corresponding to the HN filtration with respect to the $h_T^{\vd}$-stability; it is 
   given by 
   \begin{align*}
       \lambda \colon \mathrm{pt} \mapsto E_1 \oplus \cdots \oplus E_a
   \end{align*}
   with $\mathbb{G}_m$-weights $(\lambda_1, \ldots, \lambda_a)$
   for $\lambda_1>\cdots>\lambda_a$, 
    each $E_i$ is $h_T^{\varepsilon\dag}$-semistable such that 
   \begin{align*}
       \mu^{\varepsilon \dag}(E_1)>\cdots>\mu^{\varepsilon \dag}(E_a), \ \mu^0(E_1)=\cdots=\mu^0(E_a)
   \end{align*}
   for $0<\varepsilon \ll 1$. 
   Then we have 
   \begin{align}\label{ineq:weak}
       \mu^{\varepsilon}(E_1) \geq \cdots \geq \mu^{\varepsilon}(E_a).
   \end{align}

   We now apply the argument of the proof of Lemma~\ref{lem:nweight}, by replacing 
   the inequality (\ref{slope:cond}) with the weaker one (\ref{ineq:weak}). 
   Using the notation in Lemma~\ref{lem:nweight}, we have 
   the inequality slightly weaker than (\ref{ineq:q})
   \begin{align}\label{ineq:q2}
   q_1 \leq \cdots \leq q_a. 
   \end{align}
   
   On the other hand, for $A_b=\mathcal{O}_{\cH_b(w)^{\mathrm{ss}}}$, 
   its restriction to $\cS^{\dag}$ has $\lambda$-weight zero. 
   Therefore we obtain the upper bound of $\beta_m$ in (\ref{ineq:beta})
   without the contributions from (\ref{wt:1}). 
   Since we have 
   \begin{align*}
       \sum_{j<j'}R_j R_{j'}(\lambda_j-\lambda_{j'})=\sum_{m=1}^{a-1}R_{J_m}R_{J_m^{\circ}}(\lambda_m-\lambda_{m+1})
   \end{align*}
   it follows that we have the inequality slightly better than (\ref{ineq:beta})
   \begin{align*}
   \beta_m &\leq 
    -\frac{\chi}{r}R_{J_m}+\sum_{i=1}^k \frac{\chi_i}{r_i}r_{iJ_m}
    -(g-1)R_{J_m}R_{J_m^{\circ}} \\&\leq -(g-1)R_{J_m}R_{J_m^{\circ}}.
    \end{align*}
    Here the second inequality follows from (\ref{ineq:q2}). 
    
    Since $g\geq 2$, we have $\beta_m\leq 0$ for all $1\leq m \leq a-1$, 
    and $\beta_m<0$ for some $1\leq m\leq a-1$.
    Therefore the $\lambda$-weights of $\mathcal{V}_M \otimes \det \mathbb{L}_{\cS^{\dag}/\cH_T(w)^{\mathrm{ss}}}$ are negative 
    and we obtain the vanishing (\ref{vanish:gammaO}) for $k=1$ as in the proof of Proposition~\ref{prop:vanishing}. Then the lemma follows as in 
    the argument of Proposition~\ref{prop:vanish2}. 
\end{proof}
Recall the functor $\Phi$ from (\ref{induce:Phi})
\begin{align*}
    \Phi \colon \Coh(\cH(w)^{\mathrm{ss}})_0 \to \Coh(\cH(\chi_0))_{w'}.
\end{align*}
By using the above cohomology vanishing result, we prove the following proposition: 
\begin{thm}\label{thm:PhiO}
We have 
\begin{align*}
\Phi(\mathcal{O}_{\cH(w)^{\mathrm{ss}}})\in j_{!}\LL(\cH(\chi_0)^{\mathrm{ss}})_{w'}.
\end{align*}   
Here $j_{!}$ is the fully-faithful left adjoint of the pull-back $j^*$
for the open immersion $j \colon \cH(\chi_0)^{\mathrm{ss}} \hookrightarrow \cH(\chi_0)$
\begin{align*}j_{!} \colon \LL(\cH(\chi_0)^{\mathrm{ss}})_{w'} \hookrightarrow \LL(\cH(\chi_0))_{w'}.
\end{align*}
\end{thm}
\begin{proof}
We divide the proof into 2 steps.
\begin{ssstep}
The result on the semistable locus. 
\end{ssstep}
   We first show that 
   \begin{align}\label{Phi:rest}
\Phi(\mathcal{O}_{\cH(w)^{\mathrm{ss}}})|_{\cH(\chi_0)^{\mathrm{ss}}} \in \LL(\cH(\chi_0)^{\mathrm{ss}})_{w'}.
   \end{align}
Let $\nu$ be a map 
\begin{align*}
    \nu \colon \bgm \to \cH_T(\chi_0)^{\mathrm{ss}}, \ 
    \mathrm{pt} \mapsto M=M_1 \oplus \cdots \oplus M_k
\end{align*}
over $b\in \cB$ as in (\ref{map:v2}); since $M$ is semistable, 
each $M_i\in \Coh^{\heartsuit}(\cC_b)$ is semistable with 
\begin{align}\label{chir}
    \frac{\chi_1}{r_1}=\cdots=\frac{\chi_k}{r_k}, 
\end{align}
where $r_i=\operatorname{rank}(\pi_{*}M_i)$ and $\chi_i=\deg \pi_{*}M_i$. 
We take an \'{e}tale neighborhood $(T, 0)\to (\cB, b)$ and an ample $\mathbb{Q}$-divisor 
$h_T^{\varepsilon\dag}$ as in (\ref{hTeps}). 
Note that by the condition (\ref{chir}), the coefficient of $\varepsilon$ vanishes in (\ref{hTeps}), and 
$h_T^{\varepsilon\dag}$ can be taken to be independent of $\nu$ over $b$; namely 
for a choice of $(\alpha^{(a)})_{1\leq a\leq m}$ as in (\ref{choice:alpha})
satisfying the condition in Lemma~\ref{lem:alphagen}, 
we take $h_T^{\varepsilon\dag}$ such that 
\begin{align}\label{hT:indep}
    h_T^{\varepsilon\dag} \cdot \cC_b^{(a)}=r^{(a)}(1+\alpha^{(a)}\varepsilon^2).
\end{align}

Let $\Phi_T^{\varepsilon\dag}$ be the functor 
\begin{align*}
  \Phi_T^{\varepsilon\dag}\colon \Coh(\cH_T(w)^{\varepsilon\dag\text{-ss}})_0 \to 
  \Coh(\cH_T(\chi_0)^{\mathrm{ss}})_{w'}
\end{align*}
defined by the kernel object 
\begin{align}\label{PTvd}
\cP_T^{\vd}:=
    \mathcal{P}_T|_{\cH_T(w)^{\varepsilon\dag\text{-ss}}\times_T \cH_T(\chi_0)^{\text{ss}}} \in \Coh(\cH_T(w)^{\varepsilon\dag\text{-ss}}\times_T \cH_T(\chi_0)^{\text{ss}}).
\end{align}
Here $\cP_T$ is the pull-back of the Arinkin sheaf (\ref{P:pullT}). 
From the distinguished triangle 
\begin{align*}
\Gamma_{\cS^{\dag}_{\leq N'}}
\bigl(\mathcal{P}_T|_{\cH_T(w)^{\text{ss}}\times_T \cH_T(\chi_0)^{\text{ss}}}\bigr)
\to {}& \mathcal{P}_T|_{\cH_T(w)^{\text{ss}}\times_T \cH_T(\chi_0)^{\text{ss}}} \\
\to {}& \mathcal{P}_T|_{\cH_T(w)^{\varepsilon\dag\text{-ss}}\times_T
\cH_T(\chi_0)^{\text{ss}}},
\end{align*}
after pushing forward to $\cH_T(\chi_0)^{\mathrm{ss}}$, we obtain, in
\[
\Coh(\cH_T(\chi_0)^{\mathrm{ss}}),
\]
a distinguished triangle of the form
\begin{align*}
K\to {}&
\Phi_T(\mathcal{O}_{\cH_T(w)^{\mathrm{ss}}})|_{\cH_T(\chi_0)^{\text{ss}}} \\
\to {}&
\Phi_T^{\varepsilon\dag}(\mathcal{O}_{\cH_T(w)^{\varepsilon\dag\text{-ss}}}). 
\end{align*}
We have 
\begin{align*}
    \nu^* K=\Gamma(\cH_T(w)^{\mathrm{ss}}, \Gamma_{\cS_{\leq N'}^{\dag}}(\cP_M))
\end{align*}
which vanishes by Proposition~\ref{prop:resol} and Lemma~\ref{lem:vanish2dag}. 
Indeed by Proposition~3.2, $\cP_M$ is represented as a filtered colimit of a left
resolution whose terms are finite direct sums of $\mathcal{V}_M$. Since
the local cohomology functor and the global section functor are continuous, Lemma~\ref{lem:vanish2dag} implies the same vanishing for
$\cP_M$. 

Therefore by shrinking $T$ if necessary, we have the isomorphism 
\begin{align}\label{PhiTdag1}
\Phi_T(\mathcal{O}_{\cH_T(w)^{\mathrm{ss}}})|_{\cH_T(\chi_0)^{\text{ss}}} \stackrel{\cong}{\to} 
     \Phi_T^{\varepsilon\dag}(\mathcal{O}_{\cH_T(w)^{\varepsilon\dag\text{-ss}}}).
\end{align}
It is enough to show that 
\begin{align}\label{show:dag}
  \Phi_T^{\varepsilon\dag}(\mathcal{O}_{\cH_T(w)^{\varepsilon\dag\text{-ss}}}) \in \LL(\cH_T(\chi_0)^{\mathrm{ss}})_{w'}.   
\end{align}

By our choice of $h_T^{\varepsilon\dag}$, the stack 
$\cH_T(w)^{\varepsilon\dag \text{-ss}}$ is a $\mathbb{G}_m$-gerbe over 
a quasi-projective scheme $H_T(w)^{\varepsilon\dag \text{-ss}}$, see (\ref{Gmgerb}). 
We have the equivalence 
\begin{align}\label{equiv:wt0}
\Coh(H_T(w)^{\varepsilon\dag \text{-ss}}) \stackrel{\sim}{\to}
    \Coh(\cH_T(w)^{\varepsilon\dag \text{-ss}})_0. 
\end{align}
We identify $\Phi_T^{\varepsilon \dag}$ as a functor through the above equivalence
\begin{align*}
    \Phi_T^{\varepsilon \dag} \colon 
    \Coh(H_T(w)^{\varepsilon\dag\text{-ss}}) \to 
  \Coh(\cH_T(\chi_0)^{\mathrm{ss}})_{w'}
\end{align*}
whose kernel object 
\begin{align*}
    \overline{\cP}_T^{\vd} \in \Coh(H_T(w)^{\varepsilon\dag\text{-ss}}\times_T \cH_T(\chi_0)^{\text{ss}})
\end{align*}
is identified with the weight zero-part 
of the object (\ref{PTvd}) under the equivalence (\ref{equiv:wt0}).  

For any $y\in H_T(w)^{\varepsilon\dag\text{-ss}}$ and the 
corresponding sheaf 
$E_y \in \Coh^{\heartsuit}(\cC_t)$, 
we have 
\begin{align}\label{Phi(Oy)}
    \Phi_T^{\varepsilon\dag}(\mathcal{O}_y)=\mathcal{P}_{E_y}|_{\cH_T(\chi_0)^{\mathrm{ss}}} \in \LL(\cH_T(\chi_0)^{\mathrm{ss}})_{w'}
\end{align}
by Corollary~\ref{cor:PE}. 
Let $i_T$ be the closed immersion 
\begin{align*}
    i_T \colon H_T(w)^{\mathrm{ss}}\times_T \cH_T(\chi_0) \hookrightarrow 
H_T(w)^{\mathrm{ss}}\times \cH_T(\chi_0).
\end{align*}
Since $H_T(w)^{\vdss}$ is a quasi-projective scheme and 
the category $\LL(\cH_T(\chi_0)^{\mathrm{ss}})$ is a semiorthogonal summand of 
$\Coh(\cH_T(\chi_0)^{\mathrm{ss}})$, 
 we have 
\begin{align}\label{PTed}
  i_{T*}  \overline{\cP}_T^{\vd} \in \Coh(H_T(w)^{\vdss}) \otimes \LL(\cH_T(\chi_0)^{\mathrm{ss}})_{w'}. 
\end{align}
Indeed let 
\begin{align*}\cP' \in 
\Coh(H_T(w)^{\vd\text{-ss}})\otimes \Coh(\cH_T(\chi_0)^{\mathrm{ss}})_{w'}
\end{align*}
be a semiorthogonal factor of $i_{T*}\overline{\cP}_T^{\vd}$, 
not contained in the right-hand side in (\ref{PTed}). 
Then the condition \eqref{Phi(Oy)} implies that 
\[
\cP'|_{y\times \cH_T(\chi_0)^{\mathrm{ss}}}=0
\]
for any $y\in H_T(w)^{\vd\text{-ss}}$. Hence $\cP'=0$. 
Therefore (\ref{PTed}) holds
(we also refer to the argument
of~\cite[Lemma~6.2]{PThiggs}). 
Therefore 
the functor $\Phi_T^{\varepsilon\dag}$ 
restricts to the functor 
\begin{align}\label{rest:L}
\Phi_T^{\varepsilon\dag} \colon 
    \Coh(H_T(w)^{\varepsilon\dag\text{-ss}}) \to 
  \LL(\cH_T(\chi_0)^{\mathrm{ss}})_{w'}. 
\end{align}
In particular the condition (\ref{show:dag}) holds. 
\begin{ssstep}
    The result on Harder--Narasimhan strata.
\end{ssstep}
We next consider Harder--Narasimhan stratification of $\cH(\chi_0)$ with respect to the usual stability condition (\ref{ineq:deg}). 
Let 
\begin{align*}
    \cH(\chi_0)_{\preceq \mu} \subset \cH(\chi_0)
\end{align*}
be an open substack 
consisting of a finite number of Harder--Narasimhan strata 
such that 
\begin{align*}
    \mathcal{S}_{\mu}=\cH(\chi_0)_{\preceq \mu}\setminus \cH(\chi_0)_{\prec \mu}
\end{align*}
is a Harder--Narasimhan stratum with type $(\mu_1, \ldots, \mu_k)$, see Remark~\ref{rmk:jshrink}. 
We have the following open immersions 
\begin{align*}
j_{\preceq \mu}\colon 
    \cH(\chi_0)^{\mathrm{ss}} \stackrel{j_{\prec \mu}}{\hookrightarrow}
\cH(\chi_0)_{\prec\mu} \stackrel{j_{\mu}}{\hookrightarrow}
\cH(\chi_0)_{\preceq \mu}.
\end{align*}
We consider the functor 
\begin{align}\label{PhiTmu}
    \Phi_{\preceq \mu}:=\Phi|_{\cH(\chi_0)_{\preceq \mu}} \colon 
    \Coh(\cH(w)^{\mathrm{ss}})_0 \to \Coh(\cH(\chi_0)_{\preceq \mu})_{w'}.
\end{align}
Below by induction on $\mu$, we show that 
\begin{align}\label{induct:mu}
\Phi_{\preceq \mu}(\mathcal{O}_{\cH(w)^{\mathrm{ss}}})
\in (j_{\preceq \mu})_{!}\LL(\cH(\chi_0)^{\mathrm{ss}})_{w'}.
\end{align}
The base case of the induction is (\ref{Phi:rest}). 

Suppose by the induction hypothesis that 
\begin{align}\label{induct:mu2}
    \Phi_{\prec \mu}(\mathcal{O}_{\cH(w)^{\mathrm{ss}}})
\in (j_{\prec \mu})_{!}\LL(\cH(\chi_0)^{\mathrm{ss}})_{w'}.
\end{align}
Let $\nu$ be a map 
\begin{align*}
    \nu \colon \bgm \to \cH(\chi_0)_{\preceq \mu}, \ \mathrm{pt} \mapsto 
    M=M_1 \oplus \cdots \oplus M_k
\end{align*}
such that each $M_i$ is semistable
satisfying 
\begin{align}\label{ineq:chir}
r_i=\operatorname{rank}\pi_{*}M_i, \ 
\chi_i=\deg \pi_{*}M_i, \ 
\frac{\chi_i}{r_i}=\mu_i, 
\end{align}
i.e. it is of Harder--Narasimhan type $\mu$. 
We take an \'{e}tale neighborhood $(T, 0)\to (\cB, b)$ and an ample $\mathbb{Q}$-divisor $h_T^{\varepsilon\dag}$ as in (\ref{hTeps});
namely it satisfies 
\begin{align}\label{def:hTedag}
    h_T^{\varepsilon \dag}\cdot \cC_b^{(a)}=r^{(a)}\left(1+\left(\frac{\chi_i}{r_i}-\frac{\chi}{r}  \right)\varepsilon +\alpha^{(a)}\varepsilon^2   \right)
\end{align}
for $(\alpha^{(a)})_{1\leq a\leq m}$ as in (\ref{choice:alpha})
satisfying the condition in Lemma~\ref{lem:alphagen}, which only depends on 
the Harder--Narasimhan type (\ref{ineq:chir}). 
The base change of the functor (\ref{PhiTmu}) gives 
\begin{align*}
    \Phi_{T\preceq \mu} \colon \Coh(\cH_T(w)^{\mathrm{ss}})_0 \to 
    \Coh(\cH_T(\chi_0)_{\preceq \mu})_{w'}.
\end{align*}
Let $\Phi_{T\preceq \mu}^{\varepsilon\dag}$ be the functor 
\begin{align}\label{funct:edag2}
   \Phi_{T\preceq \mu}^{\varepsilon\dag} \colon \Coh(\cH_T(w)^{\varepsilon\dag\text{-ss}})_0 \to 
   \Coh(\cH_T(\chi_0)_{\preceq \mu})_{w'}
\end{align}
defined by the kernel object 
\begin{align}\label{PTprecmu}
\cP_{T\preceq \mu}^{\vd}:=
    \mathcal{P}_T|_{\cH_T(w)^{\varepsilon\dag\text{-ss}}\times_T \cH_T(\chi_0)_{\preceq \mu}}\in \Coh(\cH_T(w)^{\varepsilon\dag\text{-ss}}\times_T \cH_T(\chi_0)_{\preceq \mu}).
\end{align}
Then as in the argument proving the isomorphism (\ref{PhiTdag1}), 
by Proposition~\ref{prop:resol} and Lemma~\ref{lem:vanish2dag}, 
the natural map 
\begin{align}\label{nat:isom}
    \Phi_{T \preceq \mu}(\mathcal{O}_{\cH_T(w)^{\mathrm{ss}}})\to \Phi_{T\preceq \mu}^{\varepsilon\dag}(\mathcal{O}_{\cH_T(w)^{\varepsilon\dag\text{-ss}}})
\end{align}
is an isomorphism (after shrinking $T$ if necessary).

Below we show that 
\begin{align}\label{ets:mu}
     \Phi_{T\preceq \mu}^{\varepsilon\dag}(\mathcal{O}_{\cH_T(w)^{\varepsilon\dag\text{-ss}}})\in 
     (j_{T\mu})_{!}\LL(\cH_T(\chi_0)_{\prec \mu})_{w'}
     \subset \LL(\cH_T(\chi_0)_{\preceq \mu})_{w'}.
\end{align}
In fact as (\ref{nat:isom}) is an isomorphism, the condition (\ref{ets:mu}) together with 
the induction hypothesis (\ref{induct:mu2}) implies 
(\ref{induct:mu}), see Remark~\ref{rmk:jshrink}. 

As in (\ref{Gmgerb}), the stack $\cH_T(w)^{\varepsilon\dag\text{-ss}}$ is a $\mathbb{G}_m$-gerbe over 
$H_T(w)^{\varepsilon\dag\text{-ss}}$, and the functor (\ref{funct:edag2}) is identified with 
\begin{align}\label{funct:edag3}
   \Phi_{T\preceq \mu}^{\varepsilon\dag} \colon \Coh(H_T(w)^{\varepsilon\dag\text{-ss}}) \to 
   \Coh(\cH_T(\chi_0)_{\preceq \mu})_{w'}
\end{align}
with kernel object 
\begin{align*}
    \overline{\cP}_{T\preceq \mu}^{\vd} \in \Coh(H_T(w)^{\vd\text{-ss}}\times_T \cH_T(\chi_0)_{\preceq \mu})
\end{align*}
the weight zero part of (\ref{PTprecmu}) for the factor 
$\cH_T(w)^{\vd\text{-ss}}$.
For $y\in H_T(w)^{\varepsilon\dag\text{-ss}}$, let $E_y\in \Coh^{\heartsuit}(\cC_t)$ be the corresponding sheaf. 
Then we have 
\begin{align*}
     \Phi_{T\preceq \mu}^{\varepsilon\dag}(\mathcal{O}_y)=\mathcal{P}_{E_y}|_{\cH_T(\chi_0)_{\preceq \mu}} \in \LL(\cH_T(\chi_0)_{\preceq \mu})_{w'}
\end{align*}
by Corollary~\ref{cor:PE}.

Moreover, suppose that $(\nu_1, \ldots, \nu_k)$ 
are the $\mathbb{G}_m$-weights of $M$, and suppose that they satisfy $\nu_1>\cdots>\nu_k$. Then by Lemma~\ref{lem:lower},  
the lower bound of the condition 
\begin{align*}
    \wt(\iota_{*}\nu_{\circ}^{\mathrm{reg}*}\mathcal{P}_{E_y}) \subset\left[\frac{1}{2}c_1(\nu^* \mathbb{L}_{\cH_T}^{<0}), \frac{1}{2}c_1(\nu^* \mathbb{L}_{\cH_T}^{>0})  \right]+c_1(\nu^* \delta_{w'})
\end{align*}
can be achieved only if there is an exact sequence 
\begin{align*}
    0\to E_y'' \to E_y \to E_y' \to 0
\end{align*}
such that $E_y''$ is supported on $C_{I_i^{\circ}}$ and $E_y'$ is supported on $C_{I_i}$, and 
\begin{align*}
    \frac{\deg (\pi_{*}E_{y}'')}{\rank (\pi_{*}E_y'')}= \frac{\deg (\pi_{*}E_{y}')}{\rank (\pi_{*}E_y')}.
\end{align*}
Here $C_i=\mathrm{Supp}(M_i)$ and $I_i=\{1, \ldots, i\}$. 
This implies that 
\begin{align*}
     \frac{\deg (\pi_{*}E_{y}'')}{h_T^{\varepsilon\dag}(C_{I_i^{\circ}})}> \frac{\deg (\pi_{*}E_{y}')}{h_T^{\varepsilon\dag}(C_{I_i})}
\end{align*}
by the inequalities $\mu_1>\cdots>\mu_k$ and the definition of $h_T^{\vd}$ in (\ref{def:hTedag}). 
The above inequality contradicts that $E_y$ is $h_T^{\varepsilon\dag}$-semistable,
therefore we have 
\begin{align*}
    \wt(\iota_{*}\nu_{\circ}^{\mathrm{reg}*}\mathcal{P}_{E_y}) \subset\left(\frac{1}{2}c_1(\nu^* \mathbb{L}_{\cH_T}^{<0}), \frac{1}{2}c_1(\nu^* \mathbb{L}_{\cH_T}^{>0})  \right]+c_1(\nu^* \delta_{w'})
\end{align*}
when $\nu_1>\cdots>\nu_k$. 

This implies that (see Remarks~\ref{rmk:HN!} and~\ref{rmk:T})
\begin{align*}
      \Phi_{T\preceq \mu}^{\varepsilon\dag}(\mathcal{O}_y)=\mathcal{P}_{E_y}
      \in  (j_{T\mu})_{!}\LL(\cH_T(\chi_0)_{\prec \mu})_{w'}.
\end{align*}
Then by the same argument proving (\ref{PTed}), we have 
\begin{align*}
    (i_{T\preceq \mu})_{*}\overline{\cP}_{T\preceq \mu}^{\vd} \in 
    \Coh(H_T(w)^{\vd\text{-ss}})\otimes (j_{T_{\mu}})_{!}\LL(\cH_T(\chi_0)_{\prec \mu})_{w'}.
\end{align*}
Here $i_{T\preceq \mu}$ is the closed immersion 
\begin{align*}
  i_{T\preceq \mu} \colon     H_T(w)^{\vd\text{-ss}}\times_T \cH_T(\chi_0)_{\preceq \mu} \hookrightarrow 
H_T(w)^{\vd\text{-ss}}\times \cH_T(\chi_0)_{\preceq \mu}.
\end{align*}
Therefore as in (\ref{rest:L}), 
the functor (\ref{funct:edag3})
restricts to the functor 
\begin{align*}
       \Phi_{T\preceq \mu}^{\varepsilon\dag} \colon \Coh(H_T(w)^{\varepsilon\dag\text{-ss}}) \to 
  (j_{T\mu})_{!} \LL(\cH_T(\chi_0)_{\prec \mu})_{w'}.
   \end{align*}
   In particular 
the condition (\ref{ets:mu}) holds. 
\end{proof}

\subsection{Whittaker normalization}
Recall the Hitchin section from Subsection~\ref{subsec:whit}
\begin{align*}
    s \colon \cB \to \cH(\chi_0). 
\end{align*}
Using the results in the previous subsections, we prove Conjecture~\ref{conj:Whit}
over $\cB=\cB^A$: 
\begin{thm}\label{thm:whit}
For $\cB=\cB^A$, there 
is an isomorphism in $\LL_{\mathcal{N}}(\cH(\chi_0))_{w'}$
\begin{align*}
    \Phi(\mathcal{O}_{\cH(w)^{\mathrm{ss}}})
    \cong s_{!}\mathcal{O}_{\cB}.
\end{align*}
\end{thm}

\begin{proof}
Recall the factorization of $s$ in Lemma~\ref{lem:factori}
\begin{align*}
s \colon \cB
\stackrel{\overline{s}}{\to}
\cH(\chi_0)^{\mathrm{ss}}
\stackrel{j}{\hookrightarrow}
\cH(\chi_0).
\end{align*}
Then by Lemma~\ref{lem:ids}, we have the isomorphism 
\begin{align*}
    \overline{s}^*(\Phi(\cO_{\cH(w)^{\mathrm{ss}}})|_{\cH(\chi_0)^{\mathrm{ss}}})\cong h_{*}\cO_{\cH(w)^{\mathrm{ss}}}
\end{align*}
where $h \colon \cH(w)^{\mathrm{ss}} \to \cB$ is the Hitchin map. 
By the adjunction of the natural map $\cO_{\cB} \to h_{*}\cO_{\cH(w)^{\mathrm{ss}}}$, 
there is a natural morphism
\begin{align}\label{map:sdag}
    \overline{s}_{!}\mathcal{O}_{\cB}
    \to
    \Phi(\mathcal{O}_{\cH(w)^{\mathrm{ss}}})
    |_{\cH(\chi_0)^{\mathrm{ss}}}.
\end{align}
Here $\overline{s}_!$ is given by the formula (\ref{formula:sbar}).  
By Theorem~\ref{thm:PhiO}, it is enough to show that the above map is an isomorphism.

Let $(T, 0) \to (\cB, b)$ be an \'{e}tale map as in Step~1 of Theorem~\ref{thm:PhiO}, 
and take a perturbation 
$h_T^{\vd}$ of $h_T$ as in Subsection~\ref{subsec:stO};
namely it satisfies 
\begin{align}\label{h:generic}
    h_T^{\varepsilon\dag} \cdot \cC_b^{(a)}=r^{(a)}(1+\alpha^{(a)}\varepsilon^2)
\end{align}
for $(\alpha^{(a)})_{1\leq a\leq m}$ as in (\ref{choice:alpha}) satisfying 
the condition in Lemma~\ref{lem:alphagen}.

Let $\Phi_T^{\vd}$ be the functor
\begin{align}\label{PhiTe}
\Phi_T^{\vd} \colon
\Coh(\cH_T(w)^{\vdss})_0
\to
\Coh(\cH_T(\chi_0)^{\vdss})_{w'}
\end{align}
defined by the kernel object
\begin{align*}
    \cP_T^{\vd\dag}
    :=
    \mathcal{P}_T
    |_{\cH_T(w)^{\vdss}
    \times_T
    \cH_T(\chi_0)^{\vdss}}
    \in
    \Coh(
    \cH_T(w)^{\vdss}
    \times_T
    \cH_T(\chi_0)^{\vdss}
    ).
\end{align*}
By the genericity of $h_T^{\vd}$ (i.e.\ semistability is equivalent to stability as in (\ref{ss=st})), 
by Theorem~\ref{thm:MRV} 
the functor \eqref{PhiTe} is an equivalence 
with inverse given by (after shrinking $T$ if necessary, see Remark~\ref{rmk:dualizing})
\begin{align}\label{kernel:inv2}
(\cP_T^{\vd\dag})^{\vee}[p_a]
\in
\Coh(
\cH_T(w)^{\vdss}
\times_T
\cH_T(\chi_0)^{\vdss}
).
\end{align}
Here $p_a$ is the arithmetic genus of the spectral curve (\ref{def:pa}). 

We denote by
$\overline{s}_T \colon T \to \cH_T(\chi_0)^{\mathrm{ss}}$
the base change of $\overline{s}$.
It factors through the map $\overline{s}_T^{\vd}$:
\begin{align*}
    \overline{s}_T \colon
    T
    \stackrel{\overline{s}_T^{\vd}}{\to}
    \cH_T(\chi_0)^{\vdss}
    \hookrightarrow
    \cH_T(\chi_0)^{\mathrm{ss}}.
\end{align*}
The composition 
\begin{align}\label{compose:vd}
T\stackrel{\overline{s}_T^{\vd}}{\to} \cH_T(\chi_0)^{\vd\text{-ss}} \to H_T(\chi_0)^{\vd\text{-ss}} 
\end{align}
is a closed immersion since it is a section of the projective morphism 
$H_T(\chi_0)^{\vd\text{-ss}}  \to T$, given by the universality of the good moduli space morphism. Since the second map in (\ref{compose:vd}) is a $\mathbb{G}_m$-gerbe by the genericity of $h_T^{\vd}$, similarly to Lemma~\ref{lem:factori}
we have the isomorphism
\begin{align}\label{isom:vd}
    (\overline{s}_{T}^{\vd})_{!}\cong ((\overline{s}_{T}^{\vd})_{*})_{w'}[-p_a]
\end{align}
as functors 
\begin{align*}
    \Coh(T) \to \LL(\cH_T(\chi_0)^{\vd\text{-ss}})_{w'} =
    \Coh(\cH_T(\chi_0)^{\vd\text{-ss}})_{w'}.
\end{align*}
On the other hand, by Lemma~\ref{lem:ids} we have 
\begin{align}\label{isom:vd2}
    (\id \times \overline{s}_T^{\vd})^* \cP_T^{\vd\dag}\cong 
    \cO_{\cH_T(w)^{\vd\text{-ss}}}.
\end{align}
Therefore (\ref{isom:vd}), (\ref{isom:vd2}) together with the description of the kernel object (\ref{kernel:inv2})
of $(\Phi_T^{\vd})^{-1}$
show the isomorphism in $\Coh(\cH_T(w)^{\vd\text{-ss}})_0$
\begin{align*}
 (\Phi_T^{\vd})^{-1}
 \bigl(
 (\overline{s}_T^{\vd})_{!}\mathcal{O}_T
 \bigr)
 \cong
 \mathcal{O}_{\cH_T(w)^{\vdss}}.
\end{align*}
As $\Phi_T^{\vd}$ is an equivalence, we conclude the isomorphism 
in $\Coh(\cH_T(\chi_0)^{\vd\text{-ss}})_{w'}$
\begin{align*}
    \Phi_T^{\vd}
    (\mathcal{O}_{\cH_T(w)^{\vdss}})
    \cong
    (\overline{s}_T^{\vd})_{!}\mathcal{O}_T.
\end{align*}

An argument proving the isomorphism \eqref{PhiTdag1} also shows an isomorphism
\begin{align*}
\Phi_T(\mathcal{O}_{\cH_T(w)^{\mathrm{ss}}})
|_{\cH_T(\chi_0)^{\vdss}}
\stackrel{\cong}{\to}
\Phi_T^{\vd}
(\mathcal{O}_{\cH_T(w)^{\vdss}}).
\end{align*}
It follows that the natural map
\begin{align}\label{sT:natural}
    (\overline{s}_T)_{!}\mathcal{O}_T
    \to
    \Phi_T(\mathcal{O}_{\cH_T(w)^{\mathrm{ss}}})
\end{align}
is an isomorphism on the open substack
\begin{align*}
   \widetilde{\cH}_T(\chi_0)^{\mathrm{ss}} :=
    \bigcup_{h_T^{\vd}}
    \cH_T(\chi_0)^{\vdss}
    \subset
    \cH_T(\chi_0)^{\mathrm{ss}}.
\end{align*}
Here the union is taken over the $h_T^{\vd}$-semistable loci
for all perturbations $h_T^{\vd}$ as in (\ref{h:generic}), satisfying the condition of Lemma~\ref{lem:alphagen}.
Its complement $\mathcal{Z}$ is a closed substack which is disjoint
from the closed substack $\overline{s}_T(T)$. Therefore, we have the 
distinguished triangle
\begin{align*}
R \to  (\overline{s}_T)_{!}\mathcal{O}_T \to 
    \Phi_T(\mathcal{O}_{\cH_T(w)^{\mathrm{ss}}})
\end{align*}
where $R$ is supported on the complement of $\widetilde{\cH}_T(\chi_0)^{\mathrm{ss}}$. 
By Lemma~\ref{lem:consv} below, we have $R=0$, hence
\eqref{sT:natural} is an isomorphism.
Since this holds in an \'{e}tale neighborhood of any $b \in \cB$,
the map \eqref{map:sdag} is an isomorphism.
\end{proof}
\begin{lemma}\label{lem:consv}
    The restriction functor
    \begin{align*}
        \LL(\cH_T(\chi_0)^{\mathrm{ss}})_{w'}
        \to
        \LL(\widetilde{\cH}_T(\chi_0)^{\mathrm{ss}})_{w'}
    \end{align*}
    is conservative.
\end{lemma}
\begin{proof}
   Suppose that there is a non-zero object
\[
A \in \LL(\cH_T(\chi_0)^{\mathrm{ss}})_{w'}
\]
satisfying $A|_{\widetilde{\cH}_T(\chi_0)^{\mathrm{ss}}}=0$.
Let $h_{T}^{\vd}$ be taken as in \eqref{h:generic} for
$(\alpha^{(a)})_{1\leq a\leq m}$ satisfying the condition in
Lemma~\ref{lem:alphagen}. Let
    \begin{align*}
        \cH_T(\chi_0)^{\mathrm{ss}}=\cH_T(\chi_0)^{\vdss}\sqcup \cS_1^{\dag} \sqcup \cdots \sqcup \cS_{N}^{\dag} 
    \end{align*}
    the Harder--Narasimhan stratification with respect to $h_T^{\vd}$, such that 
\begin{align*}
       \cH_T(\chi_0)^{\mathrm{ss}}_{\preceq i}= \cH_T(\chi_0)^{\vdss}\sqcup \cS_1^{\dag} \sqcup \cdots \sqcup \cS_{i}^{\dag} 
\end{align*}
 is open in $\cH_T(\chi_0)^{\mathrm{ss}}$.    
 Then as $A|_{\cH_T(\chi_0)^{\vdss}}=0$, there is $i$ such that 
 \begin{align}\label{A:vanish}
     A|_{\cH_T(\chi_0)_{\prec i}^{\mathrm{ss}}}=0, \ A|_{\cH_T(\chi_0)_{\preceq i}^{\mathrm{ss}}}\neq 0.
 \end{align}

 Since $\cS_i^{\dag}$ is a $\Theta$-stratum of $\cH_T(\chi_0)_{\preceq i}$, 
 the same argument of~\cite[Proposition~8.23]{PTlim} shows that we have the 
 semiorthogonal decomposition 
 \begin{align*}
     \LL(\cH_T(\chi_0)_{\preceq i}^{\mathrm{ss}})=\langle \LL(\mathcal{Z}_i^{\dag})_{\delta_{w'}}, (j_i)_{!}\LL(\cH_T(\chi_0)^{\mathrm{ss}}_{\prec i})_{w'} \rangle.
 \end{align*}
 Here $j_i$ is the open immersion 
 \begin{align*}
     j_i \colon \cH_T(\chi_0)^{\mathrm{ss}}_{\prec i} \hookrightarrow \cH_T(\chi_0)^{\mathrm{ss}}_{\preceq i}
 \end{align*}
 and $\cS_i^{\dag} \to \mathcal{Z}_i^{\dag}$ is the center of the $\Theta$-stratum $\cS_i^{\dag}$. 
 The first summand is given by 
$\tau_{i*}f_i^*$ for the diagram 
\begin{align*}
        \mathcal{Z}_i^{\dag} \stackrel{f_i}{\leftarrow} \cS_i^{\dag} \stackrel{\tau_i}{\hookrightarrow} \cH_T(\chi_0)_{\preceq i}^{\mathrm{ss}}.
    \end{align*}
 Then from (\ref{A:vanish}), we have 
 \begin{align*}
     A|_{\cH_T(\chi_0)_{\preceq i}^{\mathrm{ss}}}\cong 
     \tau_{i*}f_i^* B_i
 \end{align*}
 for some non-zero $B_i \in \LL(\mathcal{Z}_i^{\dag})_{\delta_{w'}}$. 
 
    By Lemma~\ref{lemma:nonempty} below, for each $z\in \mathcal{Z}_i^{\dag}$, there is another $h_T^{\varepsilon'\dag}$ determined by generic $(\alpha^{'(a)})_{1\leq a\leq m}$
   and $z' \in \cS_i^{\dag} \cap \cH_T(\chi_0)^{\varepsilon'\dag\text{-ss}}$ which maps to $z$
   under the natural map $\cS_i^{\dag} \to \mathcal{Z}_i^{\dag}$. 
   Therefore by taking $z$ to be a point in the support of $B_i$, we conclude 
   $A|_{\cH_T(\chi_0)^{\varepsilon'\dag\text{-ss}}}\neq 0$, 
   which contradicts the assumption that $A|_{\widetilde{\cH}_T(\chi_0)^{\mathrm{ss}}}=0$. 
\end{proof}
\begin{lemma}\label{lemma:nonempty}
For each $z\in \mathcal{Z}_i^{\dag}$, there is another
$h_T^{\varepsilon'\dag}$ determined by generic $(\alpha^{'(a)})_{1\leq a\leq m}$ as in Lemma~\ref{lem:alphagen}, 
and $z' \in \cS_i^{\dag}$
which corresponds to an $h_T^{\varepsilon'\dag}$-stable sheaf on $\cC_b$ 
and maps to $z$
under the natural map $\cS_i^{\dag} \to \mathcal{Z}_i^{\dag}$. 
\end{lemma}
\begin{proof}
For a rank-one torsion-free sheaf $E \in \Coh^{\heartsuit}(\cC_b)$, recall from 
Lemma~\ref{lem:Qp} that there is an exact sequence 
\begin{align*}
    0\to \mathcal{L} \to E \to 
    \bigoplus_{p \in \mathrm{Sing}(\cC_b)}\mathcal{O}_{W_p} \to 0
\end{align*}
where $\mathcal{L}$ is a line bundle and $W_p \hookrightarrow Z_p$ is a possibly empty closed subscheme
with length $0\leq n_p \leq m_p/2$. We set $n_p':=m_p'-n_p$,
where $m_p'$ is the round-down of $m_p/2$. Note that $n_p'=0$ if and only if 
$E|_{\widehat{\cC}_{b, p}}$ is obtained as a push-forward from its normalization. 

A point $z\in \mathcal{Z}_i^{\dag}$ corresponds to a direct sum
\begin{align*}
    A_1 \oplus \cdots \oplus A_{\ell} \in \Coh^{\heartsuit}(\cC_b)
\end{align*}
for $\ell\geq 2$ such that each $A_{\alpha}$ is
$h_T^{\varepsilon\dag}$-semistable with 
\begin{align*}
    \mu^{\varepsilon \dag}(A_1)>\cdots>\mu^{\varepsilon \dag}(A_{\ell}), \ 
    \mu^0(A_1)=\cdots=\mu^0(A_{\ell}).
\end{align*}
Let $D_{\alpha}:=\mathrm{Supp}(A_{\alpha})$. 
Note that each $A_{\alpha}$ is a line bundle on $D_{\alpha}$ at 
$D_{\alpha} \cap D_{\beta}$ for $\beta\neq \alpha$, as $D_{\alpha}$ is smooth at 
$D_{\alpha} \cap D_{\beta}$. 
We then take $E\in \Coh^{\heartsuit}(\cC_b)$ with $h_T^{\varepsilon\dag}$-HN filtration, using Lemma~\ref{lem:line} below,
\begin{align*}
    E_1 \subset E_2 \subset \cdots \subset E_{\ell}=E
\end{align*}
such that $E_{\alpha}/E_{\alpha-1}\cong A_{\alpha}$ and $E$ is a line bundle at any 
$p \in D_{\alpha} \cap D_{\beta}$ 
for $1\leq \alpha<\beta \leq \ell$; this is possible since each $A_{\alpha}$ is a line bundle at any 
$p\in D_{\alpha} \cap D_{\beta}$ 
for $\beta\neq \alpha$ by the definition of type $A$ singularities of $\cC_b$. 
Note that $E$ corresponds to a $k$-valued point of $\cS_i^{\dag}$ and 
satisfies $\deg \pi_{*}E=\chi_0$. 

Below we show that $E$ is stable with respect to 
$h_T^{\varepsilon'\dag}$ determined by generic $(\alpha^{'(a)})_{1\leq a\leq m}$ as in Lemma~\ref{lem:alphagen}. 
Let 
\begin{align*}
    \cC_b=\cC_b^{(1)} \cup \cdots \cup \cC_b^{(m)}, \quad 
    r^{(a)}=\operatorname{rank}(\pi_{*}\cO_{\cC_b^{(a)}})
\end{align*}
be the decomposition into irreducible components. 
For a subset $K\subset \{1, \ldots, m\}$, we write 
\[
\cC_b^{(K)}=\bigcup_{a\in K} \cC_b^{(a)}, 
\qquad 
K^{\circ}:=\{1,\ldots,m\}\setminus K.
\]
By our choice of the extensions, \(E\) is locally free at
the intersection points between components belonging to different HN factors;
equivalently the gluing length \(n'_p\) is positive at such points. 
Since \(\ell\geq 2\), any non-empty proper subset \(K\subset \{1, \ldots, m\}\) separates two irreducible
components belonging to different HN factors. As any two distinct irreducible
components of \(\cC_b\) meet for \(g\geq 2\), the intersection
\(\cC_b^{(K)}\cap \cC_b^{(K^{\circ})}\) contains a point at which \(E\) is locally free by
construction. Thus
\begin{align}\label{sum:positive}
\sum_{p\in \cC_b^{(K)}\cap \cC_b^{(K^{\circ})}} n'_p>0.
\end{align}

We then set 
\begin{align*}
    \chi^{(a)}:=\deg \pi_{*}(E|_{\cC_b^{(a)}}^{\mathrm{free}}) \in \mathbb{Z}. 
\end{align*}
Here $E|_{\cC_b^{(a)}}^{\mathrm{free}}$ is the torsion-free quotient of $E|_{\cC_b^{(a)}}$. 
Let $r^{(K)}:=\sum_{a\in K}r^{(a)}$. 
Then we have 
\begin{align*}
    \deg \pi_{*}(E|_{\cC_b^{(K)}}^{\mathrm{free}})
    =
    \sum_{a\in K} \chi^{(a)} 
    -\sum_{\substack{a,a' \in K\\ a<a'}}
    \sum_{p\in \cC_b^{(a)} \cap \cC_b^{(a')}}n_p'.
\end{align*}
By the exact sequence 
\begin{align*}
    0\to E' \to E \to E|_{\cC_b^{(K)}}^{\mathrm{free}} \to 0
\end{align*}
we have, by the semistability of $E$,
\begin{align}\label{ineq:chi}
    \frac{\chi_0}{r} \leq 
    \frac{1}{r^{(K)}}\left(
    \sum_{a\in K} \chi^{(a)} 
    -\sum_{\substack{a,a' \in K\\ a<a'}}
    \sum_{p\in \cC_b^{(a)} \cap \cC_b^{(a')}}n_p'  
    \right).
\end{align}
By the exact sequence 
\begin{align*}
    0\to E'' \to E \to (E|_{\cC_b^{(K^\circ)}})^{\mathrm{free}} \to 0
\end{align*}
and noting that 
\begin{align}\notag
    \deg \pi_{*}E''
    =
    \deg \pi_{*}(E|_{\cC_b^{(K)}}^{\mathrm{free}})
    -\sum_{p\in \cC_b^{(K)} \cap \cC_b^{(K^{\circ})}}n_p',
\end{align}
we obtain the inequality 
\begin{align}\label{ineq:chi2}
   \frac{1}{r^{(K)}}\left(
   \sum_{a\in K} \chi^{(a)} 
   -\sum_{\substack{a,a' \in K\\ a<a'}}
   \sum_{p\in \cC_b^{(a)} \cap \cC_b^{(a')}}n_p'  
   \right)
   -\frac{1}{r^{(K)}}\sum_{p\in \cC_b^{(K)} \cap \cC_b^{(K^{\circ})}}n_p'
   \leq \frac{\chi_0}{r}.
\end{align}

Moreover, by $\deg \pi_{*}E=\chi_0$, we have 
\begin{align}\label{equal:chi}
    \sum_{a=1}^m \chi^{(a)}
    -\sum_{a<a'}\sum_{p\in \cC_b^{(a)} \cap \cC_b^{(a')}}n_p'
    =\chi_0.
\end{align}
By setting 
\begin{align*}
    \widetilde{\chi}^{(a)}
    =
    \chi^{(a)}-\frac{r^{(a)}}{r}\cdot \chi_0,
\end{align*}
the inequalities (\ref{ineq:chi}), (\ref{ineq:chi2}) imply 
\begin{align}\label{ineq:chi4}
    0
    \leq 
    \sum_{a\in K}\widetilde{\chi}^{(a)}
    -\sum_{\substack{a,a' \in K\\ a<a'}}
    \sum_{p\in \cC_b^{(a)}\cap \cC_b^{(a')}}n_p'
    \leq 
    \sum_{p\in \cC_b^{(K)} \cap \cC_b^{(K^{\circ})}}n_p', 
\end{align}
and the equality (\ref{equal:chi}) implies 
\begin{align}\label{cond:chitilde}
    \sum_{a=1}^m \widetilde{\chi}^{(a)}
    =
    \sum_{a<a'}\sum_{p\in \cC_b^{(a)} \cap \cC_b^{(a')}}n_p'.
\end{align}

Let $\nabla \subset \mathbb{R}^m$ be the subset satisfying the above conditions for
$(\widetilde{\chi}^{(1)}, \ldots, \widetilde{\chi}^{(m)})$. Then it is an $(m-1)$-dimensional 
polytope by the conditions (\ref{sum:positive}), (\ref{cond:chitilde}). Indeed it contains 
the interior point $(t^{(1)}, \ldots, t^{(m)}) \in \nabla$ by setting 
\begin{align*}
    t^{(a)}=\frac{1}{2}\sum_{a'\neq a}
    \sum_{p\in \cC_b^{(a)} \cap \cC_b^{(a')}} n_p',
\end{align*}
where the inequalities in (\ref{ineq:chi4}) are strict by replacing $\widetilde{\chi}^{(a)}$ with $t^{(a)}$. 

Suppose that $(\widetilde{\chi}^{(1)}, \ldots, \widetilde{\chi}^{(m)})$ lies in the boundary of $\nabla$. Then by a small perturbation 
\begin{align*}
r^{(a)} \mapsto r^{(a)}(1+\varepsilon^{(a)}), \quad
\lvert \varepsilon^{(a)} \rvert \ll 1, \quad
\sum_{a=1}^m r^{(a)}\varepsilon^{(a)}=0,
\end{align*}
the point $(\widetilde{\chi}^{(1)}, \ldots, \widetilde{\chi}^{(m)})$ changes as  
\begin{align}\label{change:chi}
 (\widetilde{\chi}^{(1)}, \ldots, \widetilde{\chi}^{(m)}) \mapsto   
 \left(\widetilde{\chi}^{(1)}-\frac{r^{(1)} \varepsilon^{(1)}}{r}\chi_0, \ldots,
 \widetilde{\chi}^{(m)}-\frac{r^{(m)} \varepsilon^{(m)}}{r}\chi_0\right).
\end{align}
A choice of $(\varepsilon^{(1)}, \ldots, \varepsilon^{(m)})$ is $(m-1)$-dimensional.
Therefore for a generic choice of $(\alpha^{'(a)})_{1\leq a\leq m}$ 
satisfying the condition of Lemma~\ref{lem:alphagen}, by setting
$\varepsilon^{(a)}=\alpha^{'(a)} \cdot \varepsilon^2$ for $\lvert \varepsilon \rvert \ll 1$, 
the right-hand side of (\ref{change:chi}) 
lies in the interior of $\nabla$.
This means that for a perturbation $h_T^{\varepsilon'\dag}$ of $h_T$ 
with $h_T^{\varepsilon'\dag} \cdot \cC_b^{(a)}=r^{(a)}(1+\alpha^{'(a)}\varepsilon^2)$, the sheaf $E$ is $h_T^{\varepsilon'\dag}$-stable.
\end{proof}

We have used the following lemma: 
\begin{lemma}\label{lem:line}
Let $\cC$ be a reduced planar curve, and $\cC=C_1 \cup C_2$
where $C_i$ are unions of irreducible components. 
Let $E_i \in \Coh^{\heartsuit}(C_i)$ be rank-one torsion-free 
sheaves which are line bundles on $C_i$ at $C_1 \cap C_2$. Then 
there is an exact sequence in $\Coh^{\heartsuit}(\cC)$
\begin{align*}
    0\to E_1 \to E \to E_2 \to 0
\end{align*}
such that $E$ is a line bundle on $\cC$ at $C_1 \cap C_2$. 
\end{lemma}
\begin{proof}
Put $Z:=C_1\cap C_2$.  Since $E_i$ is a line bundle along $Z$, it is
enough to construct an extension which is locally, at every point of
$Z$, the standard extension
\[
0\to I_{C_2/\cC}\to \mathcal O_{\cC}\to \mathcal O_{C_2}\to 0 .
\]
Indeed, as $\cC$ is reduced and planar, 
we have an isomorphism $I_{C_2/C, p} \cong \mathcal{O}_{C_1, p}$. 
Thus the above standard sequence at $p$ is an extension of
$\mathcal O_{C_2,p}$ by $\mathcal O_{C_1,p}$ whose middle term is
$\mathcal O_{\cC,p}$.

Using local trivializations of $E_i$ along $Z$, these standard local
extensions give a section of
$
\mathcal Ext^1_{\cC}(E_2,E_1)$, 
which is supported on the finite scheme $Z$.  Since $\cC$ is a curve,
the local-to-global Ext sequence gives a surjection
\[
\Ext^1_{\cC}(E_2,E_1)
\twoheadrightarrow
H^0\bigl(\cC,\mathcal Ext^1_{\cC}(E_2,E_1)\bigr).
\]
Hence the above local extension classes lift to a global extension
\[
0\to E_1\to E\to E_2\to 0 .
\]
By construction, after localizing at any $p\in Z$, this extension is
isomorphic to the standard one.  Therefore
$E_p\cong \mathcal O_{\cC,p}$ for all $p\in Z$, so $E$ is a line bundle
along $C_1\cap C_2$.
\end{proof}
By Theorem~\ref{thm:WhitGL} (together with Remarks~\ref{rmk:thm1},~\ref{rmk:thm2}, and~\ref{rmk:thm3}) and Theorem~\ref{thm:Whit:intro}, we obtain the following: 
\begin{thm}\label{thm:GL}
    For $G=\GL_r$, the DL conjecture holds over $\rB_{\GL_r}^A \subset \rB_{\GL_r}$.
\end{thm}

\section[The Dolbeault geometric Langlands conjecture for SLr/PGLr]{The Dolbeault geometric Langlands conjecture for \texorpdfstring{$\mathrm{SL}_r/\mathrm{PGL}_r$}{SLr/PGLr}}
In this section, we prove Conjecture~\ref{conj:intro} for the dual pairs
\[
({}^{L}G,G)=(\mathrm{SL}_r,\mathrm{PGL}_r),\qquad
(\mathrm{PGL}_r,\mathrm{SL}_r),
\]
over $\rB_{G}^{A}\subset \rB_{G}$, corresponding to spectral curves with at worst
type $A$ singularities.
Here $\rB_{G}$ is the Hitchin base for $\SL_r$ or $\PGL_r$.

\subsection[Moduli stacks of SLr/PGLr-Higgs bundles]{Moduli stacks of \texorpdfstring{$\mathrm{SL}_r/\mathrm{PGL}_r$}{SLr/PGLr}-Higgs bundles}
Here we construct moduli stacks of $\SL_r/\PGL_r$-Higgs bundles from the moduli 
stack of $\GL_r$-Higgs bundles. 
For $G\in \{\SL_r, \PGL_r\}$, the Hitchin base $\rB_G$ is given by 
\begin{align*}
    \rB_{G}=\bigoplus_{i=2}^r \Gamma(\Omega_C^i) \hookrightarrow 
    \rB_{\GL_r}=\bigoplus_{i=1}^r \Gamma(\Omega_C^i)
\end{align*}
such that the $H^0(\Omega_C)$-action (\ref{act:gamma}) gives the isomorphism 
\begin{align}\label{act:isom}
    \rB_{G}\times H^0(\Omega_C) \stackrel{\cong}{\to} \rB_{\GL_r}^{\mathrm{cl}}.
\end{align}
The moduli stack of Higgs bundles $(F, \theta)$ with $\tr\theta=0$
is given by the Cartesian square 
\begin{align*}
    \xymatrix{
\Hig_{\GL_r}^{\tr =0} \inclusion \ar[d]\diasquare  & \Hig_{\GL_r} \ar[d] \\
\rB_{G} \inclusion & \rB_{\GL_r}. 
    }
\end{align*}
Let $\mathrm{Pic}(C)=\Bun_{\mathbb{G}_m}$ and 
consider the map 
\begin{align*}
\det \colon 
    \Hig_{\GL_r}^{\tr=0} \to \Pic(C), \ (F, \theta)\mapsto \det F.
\end{align*}
The moduli stack of $\SL_r$-Higgs bundles is given by the Cartesian square 
\begin{align*}
    \xymatrix{
\Hig_{\SL_r} \ar[r] \ar[d] \diasquare & \Hig_{\GL_r}^{\tr=0} \ar[d] \\
\mathrm{pt} \ar[r]& \mathrm{Pic}(C). 
    }   
\end{align*}
Here the bottom map corresponds to $\mathcal{O}_C$. 

The stack $\Pic(C)$ acts on $\Hig_{\GL_r}^{\tr=0}$ by 
$L \cdot (F, \theta)=(F\otimes L, \theta\otimes \id)$. 
The moduli stack of $\PGL_r$-Higgs bundles is given by the quotient 
\begin{align*}
    \Hig_{\PGL_r}=\Hig_{\GL_r}^{\tr=0}/\Pic(C).
\end{align*}

Since we have 
\begin{align*}
    \pi_1(\SL_r)=Z_{\PGL_r}^{\vee}=1, \ 
    \pi_1(\PGL_r)=Z_{\SL_r}^{\vee}=\mathbb{Z}/r\mathbb{Z}, 
\end{align*}
the Conjecture~\ref{conj:intro} for $\overline{\chi} \in \mathbb{Z}/r\mathbb{Z}$
is a $\rB_{G}$-linear equivalence 
\begin{align*}
    \IndCoh_{\mathcal{N}}(\Hig_{\SL_r}^{\mathrm{ss}})_{-\overline{\chi}}
    \stackrel{\sim}{\to} \IndL_{\mathcal{N}}(\Hig_{\PGL_r}(\overline{\chi})).
\end{align*}
Here we omitted the notation $(0)$ for the degree on the left-hand side, 
and $_{0}$ for the weight on the right-hand side. 

Below we fix an open subset 
\begin{align*}
    \cB_{\circ} \subset \rB_{G}, \ \cB:=\cB_{\circ} \times H^0(\Omega_C) \hookrightarrow \rB_{\GL_r}^{\mathrm{cl}},
\end{align*}
where the latter is given by the isomorphism (\ref{act:isom}). 
We consider the following stacks 
\begin{align*}
    &\cH:=\Hig_{\GL_r}\times_{\rB_{\GL_r}}\cB, \ 
    \cH_{\circ}:=\Hig_{\GL_r}^{\tr=0}\times_{\rB_{G}}\cB_{\circ}, \\
    &\cH_{\rS}:=\Hig_{\SL_r}\times_{\rB_{G}}\cB_{\circ}, \ 
    \cH_{\mathrm{P}}:=\Hig_{\PGL_r}\times_{\rB_{G}}\cB_{\circ}.
\end{align*}
\begin{remark}\label{rmk:Acirc}
Later we will consider $\cB=\cB^A$. It is of the form 
\begin{align}\label{BAcirc}
\cB_{\circ}^A \times H^0(\Omega_C) \stackrel{\cong}{\to} \cB^A
\end{align}
for an open subset $\cB_{\circ}^A \subset \cB_{\circ}$ under the isomorphism (\ref{act:isom}), 
    since the type $A$-locus $\cB^A$ is unchanged by the action of $H^0(\Omega_C)$.
\end{remark}

Note that we have the isomorphism 
\begin{align}\label{isom:tr}
    \cH_{\circ}\times H^0(\Omega_C) \stackrel{\cong}{\to} \cH, \ 
    ((F, \theta), \gamma) \mapsto (F, \theta+\gamma \cdot \id_F).
\end{align}

Let $w_0:=(r-r^2)(g-1)=\deg \pi_{*}\cO_{\cC_b}$. We also consider the norm map 
\begin{align}\label{def:norm}
    \mathrm{Nm} \colon \cH_{\circ}(w_0)\to \Pic^0(C)
\end{align}
given by 
\begin{align*}
    E \mapsto \det \pi_{*}E \otimes (\det \pi_{*}\mathcal{O}_{\cC_b})^{-1}
\end{align*}
where $E\in \Coh^{\heartsuit}(\cC_b)$. 
We define $\cH_{\rS}(w_0)$ by the Cartesian square 
\begin{align*}
    \xymatrix{
\cH_{\rS}(w_0) \ar[r]^-{i} \ar[d] \diasquare & \cH_{\circ}(w_0) \ar[d]^-{\mathrm{Nm}} \\
\mathrm{pt} \ar[r] & \Pic^0(C).
    }
\end{align*}
Here the bottom map corresponds to $\mathcal{O}_C$. 
Note that there is an isomorphism 
\begin{align}\label{isom:norm}
\cH_{\rS} \stackrel{\cong}{\to}\cH_{\rS}(w_0), \ E \mapsto E\otimes \pi^{*}\Omega_C^{-\frac{1}{2}(r-1)}
\end{align}
for $E\in \Coh^{\heartsuit}(\cC_b)$ and a fixed choice of 
$\Omega_C^{-\frac{1}{2}(r-1)}\in \Pic(C)$. Here 
we note that $\Omega_C^{-\frac{1}{2}(r-1)}$ gives a $r$-th root of 
$\det \pi_{*}\cO_{\cC_b}$. 
\begin{remark}\label{rmk:perfSP}
If $\cB_{\circ} \subset \rB_{G}^{\mathrm{red}}:=\rB_{\GL_r}^{\mathrm{red}} \cap \rB_{G}$, 
then similarly to Remark~\ref{rmk:nilp} we have 
\begin{align*}
    \Coh_{\mathcal{N}}(\cH_{\rS}(w_0))=\mathrm{Perf}(\cH_{\rS}(w_0)), \ 
    \Coh_{\mathcal{N}}(\cH_{\rP})=\mathrm{Perf}(\cH_{\rP}).
\end{align*}
   Indeed the automorphism groups of $\cH_{\rS}(w_0)$, $\cH_{\rP}$ are 
   subgroups (or quotients) of tori, hence abelian. Therefore the nilpotent cone is the 
   zero section, as in Remark~\ref{rmk:nilp}. 
\end{remark}
\subsection[Arinkin sheaf for SLr/PGLr-Higgs stacks]{Arinkin sheaf for \texorpdfstring{$\SL_r/\PGL_r$}{SLr/PGLr}-Higgs stacks}
In this subsection, we descend the Arinkin sheaf from $\GL_r$ to 
$\SL_r/\PGL_r$-Higgs moduli stacks, and construct the relevant functor. 

Let $\cB \subset \rB_{\GL_r}^{\mathrm{red, cl}}$ be an open subset. 
In this case, recall that we have the Arinkin sheaf (\ref{Arsheaf0})
\begin{align*}
    \cP \in \Coh^{\heartsuit}(\cH\times_{\cB}\cH).
\end{align*}
By pulling it back by the base change 
$\cB_{\circ} \to \cB$, we obtain the Cohen--Macaulay sheaf 
\begin{align*}
    \cP_{\circ} \in \Coh^{\heartsuit}(\cH_{\circ}\times_{\cB_{\circ}}\cH_{\circ}). 
\end{align*}

We consider the following Cartesian square 
\begin{align}\label{dia:alpha}
    \xymatrix{(\cH_{\circ}(w_0)\times_{\cB_{\circ}}\cH_{\circ}) \times \Pic(C) \ar[r]^-{\alpha} \ar[d]_-{p_{12}} \diasquare 
    & \cH_{\circ}(w_0)\times_{\cB_{\circ}}\cH_{\circ} \ar[d]_-{\id\times p} \\
    \cH_{\circ}(w_0)\times_{\cB_{\circ}}\cH_{\circ} \ar[r]^-{\id \times p} & \cH_{\circ}(w_0)\times_{\cB_{\circ}}\cH_{\rP}.
    }
\end{align}
Here $\alpha$ is the action map on the second factor of $\cH_{\circ}(w_0)\times_{\cB_{\circ}}\cH_{\circ}$,  and $p_{ij}$ is the projection from $(\cH_{\circ}(w_0)\times_{\cB_{\circ}}\cH_{\circ}) \times \Pic(C)$
    onto the corresponding factor.
    
Let $\cL$ be the line bundle 
    \begin{align*}
        \cL=(\mathrm{Nm}\times\id)^* \cP_C \to \cH_{\circ}(w_0)\times \Pic(C).
    \end{align*}
    Here $\mathrm{Nm}$ is the norm map (\ref{def:norm}) and $\cP_C \to \Pic(C) \times \Pic(C)$ is the line bundle whose fiber 
    at $(A, B)$ is given by the Deligne pairing 
    \begin{align*}
        \det \chi(A\otimes B) \otimes \det \chi(A)^{-1}\otimes \det \chi(B)^{-1}\otimes \det \chi(\mathcal{O}_{C}).
    \end{align*}
\begin{lemma}\label{lbundle:pull}
    There is a natural isomorphism 
    \begin{align*}
        \alpha^* \cP_{\circ} \cong p_{12}^* \cP_{\circ} \otimes p_{13}^* \cL.
    \end{align*}
\end{lemma}
\begin{proof}
   Since both sides are maximal Cohen--Macaulay sheaves, it is enough 
   to show the isomorphism over the open substack 
   \begin{align*}((\cH_{\circ}(w_0)^{\mathrm{reg}}\times_{\cB_{\circ}}\cH_{\circ})
   \cup (\cH_{\circ}(w_0)\times_{\cB_{\circ}}\cH_{\circ}^{\mathrm{reg}}))\times \Pic(C)
   \end{align*}
   since its complement has codimension at least two. 
We take a $k$-valued point $(A, B, M)$ of the above substack, namely $A, B \in \Coh^{\heartsuit}(\cC_b)$ such that either one of them is a line bundle, $M$ is a line bundle on $C$.
   By the multiplicative property of Deligne pairing~\cite[Section~6.2]{Deligne1987Determinant}, we have 
   \begin{align*}
       \alpha^* \cP_{\circ}|_{(A, B, M)} \cong 
       \cP_{\circ}|_{(A, B\otimes \pi^* M)} 
       \cong \cP_{\circ}|_{(A, B)}\otimes \cP_{\circ}|_{(A, \pi^*M)}.
   \end{align*}
   We have the isomorphisms
   \begin{align*}
    &\cP_{\circ}|_{(A, \pi^*M)} \\
    &\cong   
    \det \chi(\pi_{*}A\otimes M)\otimes \det \chi(\pi_{*}A)^{-1}\otimes \det \chi(M\otimes \pi_{*}\mathcal{O}_{\cC_b})^{-1}\otimes 
    \det \chi(\pi_{*}\mathcal{O}_{\cC_b}) \\
    &\cong \cP_C|_{(\det \pi_{*}A, M)}\otimes \cP_{C}|_{(\det \pi_{*}\mathcal{O}_{\cC_b}, M)}^{-1} \\
    &\cong \cP_{C}|_{(\mathrm{Nm}(A), M)}.
   \end{align*}
\end{proof}

By pulling the sheaf $\cP_{\circ}$ back via
$i \colon\cH_{\rS}(w_0) \to \cH_{\circ}$, we obtain 
\begin{align*}
   (i\times \id)^* \cP_{\circ}\in \Coh^{\heartsuit}(\cH_{\rS}(w_0)\times_{\cB_{\circ}}\cH_{\circ}).
\end{align*}
By Lemma~\ref{lbundle:pull},
there is a natural isomorphism 
\begin{align}\label{isom:Palpha}
    \alpha^* (i\times \id)^* \cP_{\circ} \cong p_{12}^* (i\times \id)^* \cP_{\circ}.
\end{align}
Here the notation is as in the diagram (\ref{dia:alpha})
\begin{align*}
    \xymatrix{(\cH_{\rS}(w_0)\times_{\cB_{\circ}}\cH_{\circ}) \times \Pic(C) \ar[r]^-{\alpha} \ar[d]_-{p_{12}} \diasquare
    & \cH_{\rS}(w_0)\times_{\cB_{\circ}}\cH_{\circ} \ar[d]_-{\id\times p} \\
    \cH_{\rS}(w_0)\times_{\cB_{\circ}}\cH_{\circ} \ar[r]^-{\id \times p} & \cH_{\rS}(w_0)\times_{\cB_{\circ}}\cH_{\rP}.
    }
\end{align*}

The isomorphism (\ref{isom:Palpha}) also satisfies the 
cocycle condition by the associativity of the multiplicative property of the Deligne pairing.
It 
gives the $\Pic(C)$-equivariant structure on 
$(i\times \id)^* \cP_{\circ}$, therefore it descends to 
a Cohen--Macaulay sheaf 
\begin{align}\label{AR:SP}
\cP_{\rS/\rP} \in \Coh(\cH_{\rS}(w_0)\times_{\cB_{\circ}}\cH_{\rP})
\end{align}
with an isomorphism (cf.~\cite[Proposition~4.4]{GS})
\begin{align}\label{isom:PSP}
    (i\times \id)^* \cP_{\circ}\cong (\id\times p)^* \cP_{\rS/\rP}.
\end{align}
Here we used the notation in the following diagram 
\begin{align}\label{dia:SP}
    \xymatrix{
    \cH_{\rS}(w_0)\times_{\cB_{\circ}}\cH_{\circ} \ar[r]^-{i\times \id} 
    \ar[d]_-{\id \times p} \diasquare & \cH_{\circ}(w_0)\times_{\cB_{\circ}}\cH_{\circ} \ar[d]_-{\id \times p}  \\
    \cH_{\rS}(w_0)\times_{\cB_{\circ}}\cH_{\rP} \ar[r]^-{i\times \id}  & \cH_{\circ}(w_0)\times_{\cB_{\circ}}\cH_{\rP}.
    }
\end{align}
For $\chi \in \mathbb{Z}$, let $\overline{\chi}\in \mathbb{Z}/r\mathbb{Z}$ be its class. 
We have the induced Fourier--Mukai functor with kernel object $\cP_{\mathrm{S}/\mathrm{P}}$
\begin{align}\label{funct:PhiSP}
    \Phi_{\rS/\rP} \colon \Coh(\cH_{\rS}(w_0)^{\mathrm{ss}})_{-\overline{\chi}} \to 
    \Coh(\cH_{\rP}(\overline{\chi}))_0.
\end{align}

\subsection{Compatibilities of the functors}
In this subsection, we show that the functor (\ref{funct:PhiSP}) 
is compatible with other functors which compare with the moduli 
stacks of $\GL_r$-Higgs bundles. 

The isomorphism (\ref{isom:PSP}) implies that the following diagram commutes 
\begin{align}\label{commute:Phi}
    \xymatrix{
\Coh(\cH_{\rS}(w_0)^{\mathrm{ss}})_{-\overline{\chi}} \ar[r]^-{\Phi_{\rS/\rP}} \ar[d]_-{i_{*}} & \Coh(\cH_{\rP}(\overline{\chi}))_0 \ar[d]_-{p^*} \\
\Coh(\cH_{\circ}(w_0)^{\mathrm{ss}})_{-\chi'} \ar[r]^-{\Phi} & \Coh(\cH_{\circ}(\chi))_0.
    }
\end{align}
Here by abuse of notation, we denoted by $i_{*}$ the weight $-\chi$-component of $i_{*}$, 
which lives in $\Coh(\cH_{\circ}(w_0)^{\mathrm{ss}})_{-\chi'}$. 
Indeed we have the factorization of $i$ by the diagram

\begin{align}\label{factor:i}
    \xymatrix{
\cH_{\rS}(w_0)^{\mathrm{ss}} \ar[r]_-{\widetilde{i}} \ar@/^15pt/[rr]^-{i} \ar[d] & \widetilde{\cH}_{\circ}(w_0)^{\mathrm{ss}} \ar[d] \inclusion_-{0} & \cH_{\circ}(w_0)^{\mathrm{ss}} \ar[d]_-{\mathrm{Nm}} \\
\mathrm{pt} \ar[r]& \bgm \inclusion^-{0}  & \mathrm{Pic}^0(C).
    }
\end{align}
Here
$0 \colon \bgm \hookrightarrow \Pic^0(C)$ is the closed immersion corresponding to $\cO_C$. 
For each Higgs bundle $E$ corresponding to a $k$-valued point of $\cH_{\circ}(w_0)$, the induced map 
\begin{align*}
    \mathbb{G}_m \subset \mathrm{Aut}(E) \stackrel{\mathrm{Nm}}{\to}
   \mathbb{G}_m
\end{align*}
is given by $t\mapsto t^r$, whose kernel is $\mu_r$. 
Therefore the $\mathbb{G}_m$-rigidification of $\widetilde{\cH}_{\circ}(w_0)$
and the $\mu_r$-rigidification of $\widetilde{\cH}_{\circ}(w_0)$ are identified. 
We have the decompositions 
\begin{align*}
&\Coh(\cH_{\rS}(w_0)^{\mathrm{ss}} )=\bigoplus_{n\in \mathbb{Z}/r\mathbb{Z}}
\Coh(\cH_{\rS}(w_0)^{\mathrm{ss}} )_{n}, \\ 
&\Coh(\widetilde{\cH}_{\circ}(w_0)^{\mathrm{ss}} )=\bigoplus_{m\in \mathbb{Z}}\Coh(\widetilde{\cH}_{\circ}(w_0)^{\mathrm{ss}})_m,
\end{align*}
and the map $\widetilde{i}$ induces the equivalence 
\begin{align*}
    \widetilde{i}^* \colon \Coh(\widetilde{\cH}_{\circ}(w_0)^{\mathrm{ss}})_m
    \stackrel{\sim}{\to}\Coh(\cH_{\rS}(w_0)^{\mathrm{ss}} )_{n}
\end{align*}
if $\overline{m}=n$ in $\mathbb{Z}/r\mathbb{Z}$. 
Then the functor $i_{*}$ is given by the composition 
\begin{align}\label{compose:0i}
    i_{*} \colon \Coh(\cH_{\rS}(w_0)^{\mathrm{ss}})_{-\overline{\chi}} \stackrel{\sim}{\to} \Coh(\widetilde{\cH}_{\circ}(w_0)^{\mathrm{ss}})_{-\chi'}
    \stackrel{0_{*}}{\to} \Coh(\cH_{\circ}(w_0)^{\mathrm{ss}})_{-\chi'}.
\end{align}
Also  
similarly by $p^*$ the degree $\chi$-component of $p^*$ 
\begin{align*}
    p^* \colon \Coh(\cH_{\mathrm{P}}(\overline{\chi}))_{0} \hookrightarrow 
    \Coh(\cH_{\mathrm{P}})_{0} \stackrel{p^*}{\to}
    \Coh(\cH_{\circ})_0 \twoheadrightarrow 
    \Coh(\cH_{\circ}(\chi))_0.
\end{align*}
\begin{remark}
    In what follows, several functors such as $p^*,p_*,p_!,p^!, i_{*}, i^*$ will appear. They will always denote the functors restricted to the indicated connected and weight components.
\end{remark}

The commutative diagram (\ref{commute:Phi})
induces the 
natural transformation 
\begin{align}\label{nat:trans}
     p_{!}\circ \Phi \Rightarrow \Phi_{\rS/\rP} \circ i^*
\end{align}
in the following diagram (which does not necessarily commute at this stage)
\begin{align}\label{dia:!}
    \xymatrix{
\Coh(\cH_{\rS}(w_0)^{\mathrm{ss}})_{-\overline{\chi}} \ar[r]^-{\Phi_{\rS/\rP}} & \Coh(\cH_{\rP}(\overline{\chi}))_0 \\
\Coh(\cH_{\circ}(w_0)^{\mathrm{ss}})_{-\chi'} \ar[r]^-{\Phi} \ar[u]_-{i^{*}} 
& \Coh(\cH_{\circ}(\chi))_0 \ar[u]_-{p_{!}}.
    }
\end{align}
Here $p_{!}$ is the left adjoint of $p^*$ in the diagram (\ref{commute:Phi}). 
Since the map $p$ (on the component $\cH_{\circ}(\chi)$) is a smooth fibration 
with fiber $\Pic^0(C)$ (which is a trivial $\mathbb{G}_m$-gerbe over the Jacobian of $C$), we have $p_{!}\cong p_{*}[g]$. 
The natural transform (\ref{nat:trans}) is defined 
by the adjunction and the commutative diagram (\ref{commute:Phi})
\begin{align*}
    \Phi \Rightarrow \Phi \circ i_{*} \circ i^* \cong p^* \circ \Phi_{\rS/\rP} \circ i^*. 
\end{align*}
The natural transform (\ref{nat:trans}) for each $\chi$ 
corresponds to the morphism 
\begin{align}\label{gamma}
\gamma \colon
    (\id \times p)_{!}\cP_{\circ} \to (i\times \id)_{*}\cP_{\rS/\rP}
\end{align}
in $\QCoh(\cH_{\circ}(w_0)^{\mathrm{ss}}\times_{\cB_{\circ}}\cH_{\mathrm{P}})$. 
\begin{prop}\label{prop:gisom}
The morphism $\gamma$ in (\ref{gamma}) is an isomorphism. 
\end{prop}
\begin{proof}
  We consider the map 
\begin{align}\label{idp*}
(\id \times p)^* \gamma \colon 
(\id \times p)^* (\id \times p)_{!} \cP_{\circ}
\to 
(\id \times p)^* (i\times \id)_{*}\cP_{\rS/\rP}.
\end{align}
We have the isomorphisms (see (\ref{dia:alpha}) for the notation of the morphisms)
\begin{align*}
(\id \times p)^* (\id\times p)_{!}\cP_{\circ}
&\cong (p_{12})_{!}\alpha^* \cP_{\circ} \\
&\cong (p_{12})_{!}(p_{12}^* \cP_{\circ}\otimes p_{13}^* \mathcal{L}) \\
&\cong \cP_{\circ}\otimes (p_{12})_{!}p_{13}^* \mathcal{L}.
\end{align*}
By the following Cartesian diagrams
\begin{align*}
\xymatrix{
(\cH_{\circ}(w_0)^{\mathrm{ss}}\times_{\cB_{\circ}}\cH_{\circ})\times \Pic(C)
\ar[r]^-{p_{13}} \ar[d]_-{p_{12}} \diasquare
&
\cH_{\circ}(w_0)^{\mathrm{ss}}\times \Pic(C)
\ar[r]^-{\mathrm{Nm}\times \id} \ar[d]_-{p_1} \diasquare
&
\Pic(C) \times \Pic(C) \ar[d]_-{q}
\\
\cH_{\circ}(w_0)^{\mathrm{ss}}\times_{\cB_{\circ}}\cH_{\circ}
\ar[r]^-{r}
&
\cH_{\circ}(w_0)^{\mathrm{ss}}
\ar[r]^-{\mathrm{Nm}}
&
\Pic(C)
}
\end{align*}
where $q$ and $r$ are the first projections, 
we have the isomorphisms
\begin{align*}
(p_{12})_{!}p_{13}^* \mathcal{L}
&\cong r^* (p_1)_{!}\mathcal{L} \\
&\cong r^* (p_1)_{!}(\mathrm{Nm}\times \id)^* \cP_C \\
&\cong r^* \mathrm{Nm}^* q_{!}\cP_C.
\end{align*}
By Lemma~\ref{lem:Pc} below, we have $q_{!}\cP_C \cong i_{*}\mathcal{O}_{\mathrm{pt}}$, and 
\begin{align*}
    (p_{12})_{!}p_{13}^* \mathcal{L} \cong r^* \mathrm{Nm}^* i_{*}\mathcal{O}_{\mathrm{pt}} 
    \cong (i\times \id)_{*}\mathcal{O}_{\cH_{\rS}(w_0)^{\mathrm{ss}}\times_{\cB_{\circ}}\cH_{\circ}}.
\end{align*}
It follows that 
\begin{align*}
    (\id \times p)^* (\id \times p)_{!}\cP_{\circ} &\cong \cP_{\circ}\otimes 
    (i\times \id)_{*}\mathcal{O}_{\cH_{\rS}(w_0)^{\mathrm{ss}}\times_{\cB_{\circ}}\cH_{\circ}} 
    \\
    &\cong (i\times \id)_{*}(i\times \id)^* \cP_{\circ}.
\end{align*}

On the other hand, by the Cartesian square (\ref{dia:SP}) and the isomorphism (\ref{isom:PSP}), we have 
\begin{align*}
    (\id \times p)^* (i\times \id)_{*}\cP_{\rS/\rP} &\cong (i\times \id)_{*}(\id \times p)^* \cP_{\rS/\rP} \\
    &\cong (i\times \id)_{*}(i\times \id)^* \cP_{\circ}.
\end{align*}
Therefore both sides in (\ref{idp*}) are canonically identified with 
$(i\times \id)_{*}(i\times \id)^*\cP_{\circ}$, and the map (\ref{idp*}) 
corresponds to the identity map under the above identifications. 
In particular the map (\ref{idp*}) is an isomorphism. 

Since $\id\times p$ is a smooth surjective morphism, the functor 
\begin{align*}
    (\id \times p)^* \colon \Coh(\cH_{\circ}(w_0)^{\mathrm{ss}}\times_{\cB_{\circ}}\cH_{\rP}) \to \Coh(\cH_{\circ}(w_0)^{\mathrm{ss}}\times_{\cB_{\circ}}\cH_{\circ}) 
\end{align*}
is conservative. Therefore, 
we conclude that the map $\gamma$ is an isomorphism.  
\end{proof}

We have used the following lemma: 
\begin{lemma}\label{lem:Pc}
    There is an isomorphism in $\QCoh(\Pic^0(C))$
    \begin{align*}
        q_{!}\cP_C \cong i_{*}\mathcal{O}_{\mathrm{pt}}.
    \end{align*}
    Here by abuse of notation $i$ is the composition 
    \begin{align*}
i \colon \mathrm{pt}
\stackrel{\eta}{\to}
\bgm
\stackrel{0}{\hookrightarrow}
\Pic^0(C),
\end{align*}
and $0$ is the closed immersion corresponds to $\mathcal{O}_C$. 
\end{lemma}
\begin{proof}
    Note that $q_{!}=q_{*}[g]$.
We have the decomposition into connected components
\begin{align*}
\Pic(C)=\bigsqcup_{a\in \mathbb{Z}}\Pic^a(C),
\end{align*}
where $\Pic^a(C)$ corresponds to degree $a$ line bundles on $C$.
Moreover, $\Pic^a(C)\times \Pic^b(C)$ is a $\mathbb{G}_m^2$-gerbe over the product of the Jacobian of $C$. The line bundle
\begin{align*}
\cP_C|_{\Pic^a(C)\times \Pic^b(C)}
\to
\Pic^a(C) \times \Pic^b(C)
\end{align*}
has $\mathbb{G}_m^2$-weight $(b,a)$; therefore
\begin{align*}
q_{*}\left(\cP_C|_{\Pic^a(C)\times \Pic^b(C)}\right)=0
\end{align*}
unless $a=0$. For $a=0$, we have
\begin{align*}
(q_{*}\cP_C)|_{\Pic^0(C)}
=
\bigoplus_{b\in \mathbb{Z}}
q_{*}\left(\cP_C|_{\Pic^0(C)\times \Pic^b(C)}\right).
\end{align*}
The line bundle $\cP_C^{\vee}|_{\Pic^0(C)\times \Pic^b(C)}$ induces the Fourier--Mukai equivalence (as a special case of Theorem~\ref{thm:MRV})
\begin{align*}
\Phi_C \colon 
\Coh(\Pic^0(C))_{-b}
\stackrel{\sim}{\to}
\Coh(\Pic^b(C))_0
\end{align*}
such that
\begin{align*}
\Phi_C((\mathcal{O}_0)_{-b})
\cong
\mathcal{O}_{\Pic^b(C)}.
\end{align*}
Here, setting 
$
0 \colon \bgm \hookrightarrow \Pic^0(C)
$
to be the closed immersion corresponding to $\mathcal{O}_C$, the sheaf $(\mathcal{O}_0)_b$ is defined by
\begin{align*}
(\mathcal{O}_0)_b
=
0_{*}\mathcal{O}_{\bgm}(b)
\in
\Coh(\Pic^0(C))_b.
\end{align*}
It follows that
\begin{align*}
\Phi_C^{-1}(\mathcal{O}_{\Pic^b(C)})
\cong
(\mathcal{O}_0)_{-b}.
\end{align*}
The kernel object of $\Phi_C^{-1}$ is given by $\cP_C[g]$; therefore
\begin{align*}
(q_{!}\cP_C)|_{\Pic^0(C)}
\cong
\bigoplus_{b\in \mathbb{Z}}(\mathcal{O}_0)_b
=
i_{*}\mathcal{O}_{\mathrm{pt}}.
\end{align*}
The last isomorphism follows from
\begin{align*}
\eta_{*}\mathcal{O}_{\mathrm{pt}}
=
k[t^{\pm 1}]
=
\bigoplus_{b\in \mathbb{Z}}\mathcal{O}_{\bgm}(b).
\end{align*}
Therefore we conclude that
$
q_! \cP_C \cong i_{*}\mathcal{O}_{\mathrm{pt}}.
$
\end{proof}

As a corollary of Proposition~\ref{prop:gisom}, we obtain the following: 
\begin{cor}\label{cor:dialeft}
    The diagram (\ref{dia:!}) is commutative. 
\end{cor}
\begin{proof}
By Proposition~\ref{prop:gisom}, the natural transform (\ref{nat:trans}) is an isomorphism, hence the corollary holds. 
\end{proof}

Now we take $\cB=\cB^A$ and $\cB_{\circ}=\cB_{\circ}^A$, see Remark~\ref{rmk:Acirc}. Then by Theorem~\ref{thm:GL}, the functor $\Phi$ 
induces the equivalence 
\begin{align}\label{equiv:HGL}
    \Phi \colon \IndCoh_{\mathcal{N}}(\cH(w_0)^{\mathrm{ss}})_{-\chi'} \stackrel{\sim}{\to}
    \IndL_{\mathcal{N}}(\cH(\chi))_{0}.
\end{align}
By the $\cB$-linearity of the equivalence (\ref{equiv:HGL}), 
by applying 
\begin{align*}
    (-)\otimes_{\QCoh(\cB)}\QCoh(\cB_{\circ})=(-)\otimes_{\QCoh(H^0(\Omega_C))}\otimes \QCoh(\mathrm{pt})
\end{align*}
where $\mathrm{pt} \hookrightarrow H^0(\Omega_C)$ corresponds to the origin, 
and using Lemma~\ref{lem:resty}, we obtain the equivalence 
\begin{align*}
    \Phi \colon \IndCoh_{\mathcal{N}}(\cH_{\circ}(w_0)^{\mathrm{ss}})_{-\chi'} \stackrel{\sim}{\to}
    \IndL_{\mathcal{N}}(\cH_{\circ}(\chi))_{0}.
\end{align*}
By taking the subcategory of compact objects, it restricts
to the equivalence 
\begin{align}\label{PhiL}
      \Phi \colon \Coh_{\mathcal{N}}(\cH_{\circ}(w_0)^{\mathrm{ss}})_{-\chi'} \stackrel{\sim}{\to}
    \LL_{\mathcal{N}}(\cH_{\circ}(\chi))_{0}.
\end{align}
Here see Subsection~\ref{subsec:variant} for the limit categories for $\cH_{\circ}(\chi)$. We have the following lemma: 
\begin{lemma}\label{lem:pL}
An object $A \in \Coh(\cH_{\mathrm{P}}(\overline{\chi}))_0$ lies in $\LL_{\mathcal{N}}(\cH_{\mathrm{P}}(\overline{\chi}))_0$
if and only if $p^* A \in \LL_{\mathcal{N}}(\cH_{\circ}(\chi))_0$.
\end{lemma}
\begin{proof}
Let $\widetilde{\nu} \colon \bgm \to \cH_{\circ}(\chi)$ be a map and consider the composition
\begin{align*}
\nu \colon \bgm \xrightarrow{\widetilde{\nu}} \cH_{\circ}(\chi)
\xrightarrow{p} \cH_{\mathrm{P}}(\overline{\chi}).
\end{align*}
Then, since $p$ is a smooth map with fiber $\Pic(C)$, we have
\begin{align}\label{equal:tilde}
c_1\left(\widetilde{\nu}^* \mathbb{L}_{\cH_{\circ}(\chi)}^{<0}\right)
=
c_1\left(\nu^* \mathbb{L}_{\cH_{\mathrm{P}}(\overline{\chi})}^{<0}\right), \ 
c_1\left(\widetilde{\nu}^* \mathbb{L}_{\cH_{\circ}(\chi)}^{>0}\right)
=
c_1\left(\nu^* \mathbb{L}_{\cH_{\mathrm{P}}(\overline{\chi})}^{>0}\right).
\end{align}
Moreover the nilpotent singular support condition is preserved and detected by the smooth surjective morphism $p$.
It follows that
\begin{align}\label{tilde:iff}
A \in \widetilde{\LL}_{\mathcal{N}}(\cH_{\mathrm{P}}(\overline{\chi}))_0
\quad \text{if and only if} \quad
p^*A \in \widetilde{\LL}_{\mathcal{N}}(\cH_{\circ}(\chi))_0.
\end{align}
Moreover, the map $p$ is compatible with the Harder--Narasimhan stratification; namely, if
\begin{align*}
\cH_{\mathrm{P}}(\overline{\chi})
=
\sqcup_{\mu}\cH_{\mathrm{P}}(\overline{\chi})_{\mu}
\end{align*}
is the HN stratification, then the stratification
\begin{align*}
\cH_{\circ}(\chi)
=
\sqcup_{\mu}p^{-1}\bigl(\cH_{\mathrm{P}}(\overline{\chi})_{\mu}\bigr)
\end{align*}
is the HN stratification for $\cH(\chi)$ restricted to $\cH_{\circ}(\chi)$.
Therefore we can replace $\widetilde{\LL}$ with $\LL$ in (\ref{tilde:iff}), 
by the description of compact objects~\cite[(8.65)]{PTlim}, see Remark~\ref{rmk:Ltilde2}. 
Therefore the lemma holds. 
\end{proof}

By (\ref{commute:Phi}), (\ref{PhiL}) and Lemma~\ref{lem:pL}, 
we have the commutative diagram 
\begin{align}\label{commute:PhiL}
    \xymatrix{
\Coh_{\mathcal{N}}(\cH_{\rS}(w_0)^{\mathrm{ss}})_{-\overline{\chi}} \ar[r]^-{\Phi_{\rS/\rP}} \ar[d]_-{i_{*}} & \LL_{\mathcal{N}}(\cH_{\rP}(\overline{\chi}))_0 \ar[d]_-{p^*} \\
\Coh_{\mathcal{N}}(\cH_{\circ}(w_0)^{\mathrm{ss}})_{-\chi'} \ar[r]^-{\Phi} & \LL_{\mathcal{N}}(\cH_{\circ}(\chi))_0.
    }
\end{align}
Here $\overline{\chi}=\overline{\chi'}$ in $\mathbb{Z}/r\mathbb{Z}$, $i_{*}$ sends $\Coh_{\mathcal{N}}(-)$ to $\Coh_{\mathcal{N}}(-)$
since $0 \colon \bgm \hookrightarrow \Pic(C)$ is a proper quasi-smooth map. 
It extends to the commutative diagram of continuous functors 
\begin{align}\label{commute:PhiL2}
    \xymatrix{
\IndCoh_{\mathcal{N}}(\cH_{\rS}(w_0)^{\mathrm{ss}})_{-\overline{\chi}} \ar[r]^-{\Phi_{\rS/\rP}} \ar[d]_-{i_{*}} & \IndL_{\mathcal{N}}(\cH_{\rP}(\overline{\chi}))_0 \ar[d]_-{p^*} \\
\IndCoh_{\mathcal{N}}(\cH_{\circ}(w_0)^{\mathrm{ss}})_{-\chi'} \ar[r]^-{\Phi} & \IndL_{\mathcal{N}}(\cH_{\circ}(\chi))_0.
    }
\end{align}

In the next lemma, we show that the functor $p_{!}$ in (\ref{dia:!}) preserves the limit categories: 
\begin{lemma}\label{lem:p!L}
For $A\in \LL_{\mathcal{N}}(\cH_{\circ}(\chi))_0$, we have 
$p_{!}A \in \LL_{\mathcal{N}}(\cH_{\mathrm{P}}(\overline{\chi}))_0$.
\end{lemma}
\begin{proof}
For a map $\nu \colon \bgm \to \cH_{\mathrm{P}}(\overline{\chi})$, we consider the fiber 
product 
\begin{align*}
    \xymatrix{
\cH_{\circ}(\chi)\times_{\cH_{\mathrm{P}}(\overline{\chi})}\bgm \ar[r]^-{\widetilde{\nu}} \ar[d] \diasquare & 
\cH_{\circ}(\chi) \ar[d]_-{p} \\
\bgm \ar[r]^-{\nu} & \cH_{\mathrm{P}}(\overline{\chi}).
    }
\end{align*}
Let $x=\nu(\mathrm{pt})$. 
The fiber of $p$ at $x$ is $\Pic^0(C)$, on which $\mathbb{G}_m$ acts 
through $\nu$. The above $\mathbb{G}_m$-action is trivial since 
an abelian variety does not contain a rational curve; therefore 
\begin{align}\label{dia:nupull}
    \cH_{\circ}(\chi)\times_{\cH_{\mathrm{P}}(\overline{\chi})}\bgm\cong \Pic^0(C)\times \bgm
\end{align}
and the left vertical arrow in (\ref{dia:nupull}) is identified with the projection. 
Therefore for $A \in \Coh_{\mathcal{N}}(\cH_{\circ}(\chi))$, if 
we have 
\begin{align*}
    \widetilde{\nu}^* A \in \bigoplus_{w\in I}\Coh(\Pic^0(C)\times \bgm)_{w}
\end{align*}
for an interval $I\subset \mathbb{R}$, then by base change we have 
$\wt(\nu^*p_{!}A) \subset I$. 
It follows if $A \in \LL_{\mathcal{N}}(\cH_{\circ}(\chi))_0$, then 
$p_{!}A \in \widetilde{\LL}_{\mathcal{N}}(\cH_{\rP}(\overline{\chi}))_0$. 
Since the map $p$ is compatible with Harder--Narasimhan stratification as in the proof of Lemma~\ref{lem:pL}, the lemma holds. 
\end{proof}

By Lemma~\ref{lem:p!L} and Corollary~\ref{cor:dialeft}, we have the commutative diagram
\begin{align}\label{dia:pL}
\xymatrix{
\Coh_{\mathcal{N}}(\cH_{\rS}(w_0)^{\mathrm{ss}})_{-\overline{\chi}}
\ar[r]^-{\Phi_{\rS/\rP}}
&
\LL_{\mathcal{N}}(\cH_{\mathrm{P}}(\overline{\chi}))_0
\\
\Coh_{\mathcal{N}}(\cH_{\circ}(w_0)^{\mathrm{ss}})_{-\chi'}
\ar[r]^-{\Phi}
\ar[u]^-{i^*}
&
\LL_{\mathcal{N}}(\cH_{\circ}(\chi))_0
\ar[u]_-{p_{!}} .
}
\end{align}
It extends to the commutative diagram of cocomplete dg-categories with continuous functors preserving compact objects
\begin{align}\label{dia:pL2}
\xymatrix{
\IndCoh_{\mathcal{N}}(\cH_{\rS}(w_0)^{\mathrm{ss}})_{-\overline{\chi}}
\ar[r]^-{\Phi_{\rS/\rP}}
&
\IndL_{\mathcal{N}}(\cH_{\mathrm{P}}(\overline{\chi}))_0
\\
\IndCoh_{\mathcal{N}}(\cH_{\circ}(w_0)^{\mathrm{ss}})_{-\chi'}
\ar[r]^-{\Phi}
\ar[u]^-{i^*}
&
\IndL_{\mathcal{N}}(\cH_{\circ}(\chi))_0
\ar[u]_-{p_{!}} .
}
\end{align}
\subsection[Proof of Theorem 1.2 for SLr/PGLr]{Proof of Theorem~\ref{intro:mainthm} for \texorpdfstring{$\SL_r/\PGL_r$}{SLr/PGLr}}
Now we have the following result: 
\begin{prop}\label{thm:SP}
The functor 
\begin{align}\label{funct:SP}
    \Phi_{\rS/\rP} \colon \IndCoh_{\mathcal{N}}(\cH_{\rS}(w_0)^{\mathrm{ss}})_{-\overline{\chi}}
    \to \IndL_{\mathcal{N}}(\cH_{\mathrm{P}}(\overline{\chi}))_0
\end{align} 
is an equivalence. 
\end{prop}
\begin{proof}
Let
\begin{align*}
\Phi_{\rS/\rP}^R \colon
\IndL_{\mathcal{N}}(\cH_{\mathrm{P}}(\overline{\chi}))_0
\to
\IndCoh_{\mathcal{N}}(\cH_{\rS}(w_0)^{\mathrm{ss}})_{-\overline{\chi}}
\end{align*}
be the right adjoint of the functor $\Phi_{\rS/\rP}$ in (\ref{funct:SP}), which exists and is continuous by the adjoint functor theorem. By taking right adjoints in \eqref{dia:pL2}, we have the commutative diagram
\begin{align}\label{dia:pR}
\xymatrix{
\IndCoh_{\mathcal{N}}(\cH_{\rS}(w_0)^{\mathrm{ss}})_{-\overline{\chi}} \ar[d]_-{i_{*}}
&
\IndL_{\mathcal{N}}(\cH_{\mathrm{P}}(\overline{\chi}))_0
\ar[l]_-{\Phi_{\rS/\rP}^R} \ar[d]_-{p^*}
\\
\IndCoh_{\mathcal{N}}(\cH_{\circ}(w_0)^{\mathrm{ss}})_{-\chi'}
&
\IndL_{\mathcal{N}}(\cH_{\circ}(\chi))_0
\ar[l]_-{\Phi^R}.
}
\end{align}
Consider the following natural transform 
\begin{align}\label{nat:R}
    \id \Rightarrow \Phi_{\rS/\rP}^R \circ \Phi_{\rS/\rP}.
\end{align}
By applying $i_{*}$, we obtain 
\begin{align}\label{apply:i}
i_{*} \Rightarrow i_{*}\circ\Phi_{\rS/\rP}^R \circ \Phi_{\rS/\rP}.
\end{align}
By the commutative diagrams (\ref{commute:PhiL2}) and (\ref{dia:pR}), 
we have 
\begin{align*}
    i_{*}\circ \Phi_{\rS/\rP}^R \circ \Phi_{\rS/\rP} &\cong \Phi^R \circ p^* \circ \Phi_{\rS/\rP} \\
    &\cong \Phi^R \circ \Phi \circ i_{*} \\
    &\cong i_{*}
\end{align*}
where the last isomorphism follows from the fact that $\Phi$ is an equivalence (\ref{equiv:HGL}). 
Under the above isomorphisms, the natural transform (\ref{apply:i}) is 
$\id \colon i_{*} \Rightarrow i_{*}$, therefore 
(\ref{apply:i}) is an isomorphism. Then (\ref{nat:R}) is an isomorphism 
since $i_{*}$ is conservative; the conservativity of $i_{*}$ follows 
from the factorization (\ref{compose:0i}) and the map $0$ in (\ref{factor:i}) 
is a closed immersion. 

Similarly we have the natural transform 
\begin{align}\label{nat:L}
    \Phi_{\rS/\rP} \circ \Phi_{\rS/\rP}^R \Rightarrow \id.
\end{align}
By applying $p^*$ and using the commutative diagrams (\ref{commute:PhiL2}), (\ref{dia:pR}), 
we obtain 
\begin{align*}
  p^*\cong  p^*  \circ \Phi_{\rS/\rP} \circ \Phi_{\rS/\rP}^R \Rightarrow p^*.
\end{align*}
Since $p$ is a smooth morphism with fiber $\Pic^0(C)$, the functor $p^*$ is 
conservative. Therefore (\ref{nat:L}) is an isomorphism, and 
the functor (\ref{funct:SP}) is an equivalence. 
\end{proof}

By the above proposition, we conclude Theorem~\ref{intro:mainthm} for 
$(^{L}G, G)=(\SL_r, \PGL_r)$:
\begin{thm}\label{thm:SL}
The DL conjecture holds for $(^{L}G, G)=(\SL_r, \PGL_r)$ 
over $\rB_{G}^A:=\rB_{\GL_r}^A \cap \rB_{G}$. 
\end{thm}
\begin{proof}
    The result follows from Proposition~\ref{thm:SP} and an isomorphism (\ref{isom:norm}). 
\end{proof}

\subsection[Arinkin sheaf for PGLr/SLr-Higgs bundles]{Arinkin sheaf for \texorpdfstring{$\PGL_r/\SL_r$}{PGLr/SLr}-Higgs bundles}
We next prove Theorem~\ref{intro:mainthm} for $(^{L}G, G)=(\PGL_r, \SL_r)$. 
Recall the Cohen--Macaulay sheaf $\cP_{\rS/\rP}$ in \eqref{AR:SP}.
By switching the factors, we obtain
\begin{align*}
\cP_{\rP/\rS} \in \Coh(\cH_{\mathrm{P}} \times_{\cB_{\circ}} \cH_{\rS}(\chi_0))
\end{align*}
for $\chi_0=(r-r^2)(g-1)$,
and an isomorphism
\begin{align}\label{isom:descend2}
(p \times \id)^* \cP_{\rP/\rS}
\cong
(\id \times i)^* \cP_{\circ}.
\end{align}
Here the notation is as in the diagram
\begin{align*}
\xymatrix{
\cH_{\circ} \times_{\cB_{\circ}} \cH_{\rS}(\chi_0)
\ar[r]^-{\id \times i}
\ar[d]_-{p \times \id} \diasquare
&
\cH_{\circ} \times_{\cB_{\circ}} \cH_{\circ}(\chi_0)
\ar[d]^-{p \times \id}
\\
\cH_{\mathrm{P}} \times_{\cB_{\circ}} \cH_{\rS}(\chi_0)
\ar[r]_-{\id \times i}
&
\cH_{\mathrm{P}} \times_{\cB_{\circ}} \cH_{\circ}(\chi_0).
}
\end{align*}
For $w \in \mathbb{Z}$ with class
$\overline{w} \in \mathbb{Z}/r\mathbb{Z}$, the sheaf
$\cP_{\rP/\rS}$ defines the Fourier--Mukai functor
\begin{align}\label{Phi:PS}
\Phi_{\rP/\rS} \colon
\Coh(\cH_{\mathrm{P}}(\overline{w})^{\mathrm{ss}})_0
\to
\Coh(\cH_{\rS}(\chi_0))_{\overline{w}}.
\end{align}
By the isomorphism \eqref{isom:descend2}, the following diagram commutes:
\begin{align}\label{comm:PS}
\xymatrix{
\Coh(\cH_{\mathrm{P}}(\overline{w})^{\mathrm{ss}})_0
\ar[r]^-{\Phi_{\rP/\rS}}
&
\Coh(\cH_{\rS}(\chi_0))_{\overline{w}}
\\
\Coh(\cH_{\circ}(w)^{\mathrm{ss}})_0
\ar[r]^-{\Phi}
\ar[u]^-{p_*}
&
\Coh(\cH_{\circ}(\chi_0))_{w'}
\ar[u]_-{i^*}.
}
\end{align}

On the other hand, the isomorphism (\ref{gamma}) implies that the 
following diagram commutes: 
\begin{align}\label{comm:PS2}
\xymatrix{
\Coh(\cH_{\mathrm{P}}(\overline{w})^{\mathrm{ss}})_0
\ar[r]^-{\Phi_{\rP/\rS}} \ar[d]_-{p^!}
&
\Coh(\cH_{\rS}(\chi_0))_{\overline{w}} \ar[d]_-{i_{*}}
\\
\Coh(\cH_{\circ}(w)^{\mathrm{ss}})_0
\ar[r]^-{\Phi}
&
\Coh(\cH_{\circ}(\chi_0))_{w'}.
}
\end{align}
Here $p^!=p^{*}[g]$ is the right adjoint of $p_{*}$ in the diagram (\ref{comm:PS}). 

The following lemmas are analogues of
Lemma~\ref{lem:pL}, Lemma~\ref{lem:p!L}. 
\begin{lemma}\label{lem:i!}
An object $A \in \Coh_{\mathcal{N}}(\cH_{\rS}(\chi_0))_{\overline{w}}$ lies in 
$\LL_{\mathcal{N}}(\cH_{\rS}(\chi_0))_{\overline{w}}$ if and only if 
$i_{*}A \in \LL_{\mathcal{N}}(\cH_{\circ}(\chi_0))_{w'}$. 
\end{lemma}
\begin{proof}
Recall from (\ref{factor:i}) that $i$ factors as 
\begin{align}\label{factor:i2}
   \cH_{\rS}(\chi_0) \stackrel{\widetilde{i}}{\to} \widetilde{\cH}_{\circ}(\chi_0) \stackrel{0}{\to} \cH_{\circ}(\chi_0), 
\end{align}
and from 
(\ref{compose:0i}) that $i_{*}$ factors as 
\begin{align}\label{factor:i*}
\Coh(\cH_{\rS}(\chi_0))_{\overline{w}} 
\stackrel{\sim}{\to} 
\Coh(\widetilde{\cH}_{\circ}(\chi_0))_{w'} 
\stackrel{0_{*}}{\to}
\Coh(\cH_{\circ}(\chi_0))_{w'}.
\end{align}
The above functors preserve perfect complexes, as 
the first map $\widetilde{i}$ in~\eqref{factor:i2} is smooth, and the second 
map $0$ in~\eqref{factor:i2} is a quasi-smooth closed immersion. 
By Remark~\ref{rmk:perfSP}, they preserve nilpotent singular supports. 

We first show that the first equivalence restricts to an equivalence
\begin{align}\label{equiv:rest}
\LL_{\mathcal{N}}(\cH_{\rS}(\chi_0))_{\overline{w}} 
\stackrel{\sim}{\to} 
\LL_{\mathcal{N}}(\widetilde{\cH}_{\circ}(\chi_0))_{w'}.
\end{align}
Let $\nu \colon \bgm \to \cH_{\rS}(\chi_0)$ be a test map, and consider the composition 
\begin{align*}
\widetilde{\nu} \colon 
\bgm \stackrel{\nu}{\to} \cH_{\rS}(\chi_0) 
\stackrel{\widetilde{i}}{\to} 
\widetilde{\cH}_{\circ}(\chi_0).
\end{align*}
Then we have 
\begin{align*}
c_1\bigl(
\widetilde{\nu}^* 
\mathbb{L}^{>0}_{\widetilde{\cH}_{\circ}(\chi_0)}
\bigr)
=
c_1\bigl(
\nu^* 
\mathbb{L}_{\cH_{\rS}(\chi_0)}^{>0}
\bigr), \ c_1\bigl(
\widetilde{\nu}^* 
\mathbb{L}^{<0}_{\widetilde{\cH}_{\circ}(\chi_0)}
\bigr)
=
c_1\bigl(
\nu^* 
\mathbb{L}_{\cH_{\rS}(\chi_0)}^{<0}
\bigr).
\end{align*}
Indeed, we have a distinguished triangle
\begin{align*}
\widetilde{i}^* \mathbb{L}_{\widetilde{\cH}_{\circ}(\chi_0)}
\to 
\mathbb{L}_{\cH_{\rS}(\chi_0)}
\to 
\mathbb{L}_{\widetilde{i}}.
\end{align*}
Moreover, $\mathbb{L}_{\widetilde{i}}$ has no non-trivial $\nu$-weights.

Conversely, let 
$\widetilde{\nu} \colon \bgm \to \widetilde{\cH}_{\circ}(\chi_0)$ 
be a test map whose image is a point $x$. Let $m\in \mathbb{Z}$ be the weight 
of the composition
\begin{align*}
\mathbb{G}_m \to \Aut(x) 
\stackrel{\mathrm{Nm}}{\to} 
\mathbb{G}_m
\end{align*}
where the first map is induced by $\widetilde{\nu}$. 
We modify $\widetilde{\nu}$ by setting 
\begin{align*}
  \widetilde{\nu}' \colon \bgm \to \widetilde{\cH}_{\circ}(\chi_0), \ 
  \mathrm{pt} \mapsto x, \ \mathbb{G}_m \stackrel{\widetilde{\nu}(t^r)\cdot t^{-m}}{\to}\mathrm{Aut}(x).
\end{align*}
Then $\widetilde\nu(t^r)$ has norm weight $rm$, while the scalar automorphism
$t^{-m}$ has norm weight $-rm$. Hence
$\widetilde\nu'$ has trivial norm weight and factors through
$\cH_{\rS}(\chi_0)$:
\begin{align*}
\widetilde{\nu}' \colon 
  \bgm \stackrel{\nu'}{\to} \cH_{\rS}(\chi_0) 
  \stackrel{\widetilde{i}}{\to} 
  \widetilde{\cH}_{\circ}(\chi_0).
\end{align*}
Since 
$\mathbb{G}_m \subset \Aut(x)$ acts trivially on the cotangent complex at $x$, 
we have 
\begin{align*}
c_1\bigl(
(\widetilde{\nu}')^* 
\mathbb{L}^{>0}_{\widetilde{\cH}_{\circ}(\chi_0)}
\bigr) 
=
c_1\bigl(
(\nu')^* 
\mathbb{L}_{\cH_{\rS}(\chi_0)}^{>0}
\bigr)=r\cdot c_1\bigl(
\widetilde{\nu}^{*}
\mathbb{L}^{>0}_{\widetilde{\cH}_{\circ}(\chi_0)}
\bigr),
\end{align*}
and similarly for the negative weight part. 
Therefore, the test maps from $\bgm$ and their associated $c_1$-values are 
identified under the first equivalence in (\ref{factor:i*}), up to the modification of the test 
map by composing with $\mathbb{G}_m \to \mathbb{G}_m, t\mapsto t^r$, and multiplying $c_1$ with $r$. 
The latter modification does not affect the condition (\ref{wt:cond2}), see Remark~\ref{rmk:cond:r}. 
Hence this equivalence restricts 
to the equivalence~\eqref{equiv:rest}.

We next show that an object 
$A\in \Coh_{\mathcal{N}}(\widetilde{\cH}_{\circ}(\chi_0))_{w'}$ lies in 
$\LL_{\mathcal{N}}(\widetilde{\cH}_{\circ}(\chi_0))_{w'}$ if and only if 
$0_{*}A$ lies in $\LL_{\mathcal{N}}(\cH_{\circ}(\chi_0))_{w'}$. 
Indeed, this follows from the fact that $0^*0_{*}A$
is given by iterated extensions of 
\begin{align*}
    A \otimes \bigwedge^i H^1(\mathcal{O}_C)^{\vee}[i],
    \quad 0\leq i\leq g.
\end{align*}
Therefore, for any map 
$\nu \colon \bgm \to \widetilde{\cH}_{\circ}(\chi_0)$, 
the relevant weight condition for $A$ with respect to $\nu$ agrees with 
the weight condition for $0_{*}A$ with respect to 
$0\circ \nu \colon \bgm \to \cH_{\circ}(\chi_0)$.

\end{proof}

\begin{lemma}\label{lem:i!2}
For an object $A\in \LL_{\mathcal{N}}(\cH_{\circ}(\chi_0))_{w'}$, 
we have $i^* A \in \LL_{\mathcal{N}}(\cH_{\rS}(\chi_0))_{\overline{w}}$.
\end{lemma}
\begin{proof}
Let $\nu \colon \bgm \to \cH_{\rS}(\chi_0)$ be a map, and consider the composition 
\begin{align*}
    \nu' \colon 
    \bgm \stackrel{\nu}{\to} \cH_{\rS}(\chi_0) 
    \stackrel{i}{\to} \cH_{\circ}(\chi_0).
\end{align*}
Then we have 
\begin{align*}
    c_1\bigl((\nu')^* \mathbb{L}_{\cH_{\circ}(\chi_0)}^{>0}\bigr)
    =
    c_1\bigl(\nu^*\mathbb{L}_{\cH_{\rS}(\chi_0)}^{>0}\bigr).
\end{align*}
Indeed, the difference is given by the $\nu$-weights of the pullback of 
the relative cotangent complex $\mathbb{L}_{\mathrm{pt}/\mathrm{Pic}^0(C)}$ 
in the diagram~\eqref{factor:i}, which are trivial. Therefore, the lemma follows.
\end{proof}

Using the above lemmas, we have the following proposition: 
\begin{prop}\label{prop:restL}
The functor (\ref{Phi:PS})
   restricts to the functor 
   \begin{align}\label{funct:PhiPSL}
   \Phi_{\rP/\rS} \colon
\Coh_{\mathcal{N}}(\cH_{\mathrm{P}}(\overline{w})^{\mathrm{ss}})_0
\to
\LL_{\mathcal{N}}(\cH_{\rS}(\chi_0))_{\overline{w}}
\end{align}
such that 
we have commutative diagrams 
\begin{align}\label{comm:PSL}
\xymatrix{
\Coh_{\mathcal{N}}(\cH_{\mathrm{P}}(\overline{w})^{\mathrm{ss}})_0
\ar[r]^-{\Phi_{\rP/\rS}}
&
\LL_{\mathcal{N}}(\cH_{\rS}(\chi_0))_{\overline{w}}
\\
\Coh_{\mathcal{N}}(\cH_{\circ}(w)^{\mathrm{ss}})_0
\ar[r]^-{\Phi}_-{\sim}
\ar[u]^-{p_*}
&
\LL_{\mathcal{N}}(\cH_{\circ}(\chi_0))_{w'}
\ar[u]_-{i^*},
}
\end{align}
\begin{align}\label{comm:PS2L}
\xymatrix{
\Coh_{\mathcal{N}}(\cH_{\mathrm{P}}(\overline{w})^{\mathrm{ss}})_0
\ar[r]^-{\Phi_{\rP/\rS}} \ar[d]_-{p^!}
&
\LL_{\mathcal{N}}(\cH_{\rS}(\chi_0))_{\overline{w}} \ar[d]_-{i_{*}}
\\
\Coh_{\mathcal{N}}(\cH_{\circ}(w)^{\mathrm{ss}})_0
\ar[r]^-{\Phi}_-{\sim}
&
\LL_{\mathcal{N}}(\cH_{\circ}(\chi_0))_{w'}.
}
\end{align}
\end{prop}
\begin{proof}
Since the functor $\Phi$ gives an equivalence by Theorem~\ref{thm:GL}
\begin{align*}
    \Phi \colon \Coh_{\mathcal{N}}(\cH_{\circ}(w)^{\mathrm{ss}})_0 \stackrel{\sim}{\to} \LL_{\mathcal{N}}(\cH_{\circ}(\chi_0))_{w'},
\end{align*}
the proposition follows from Lemma~\ref{lem:i!}, Lemma~\ref{lem:i!2} and 
 the commutative diagrams (\ref{comm:PS}), (\ref{comm:PS2}), 
\end{proof}

\subsection[Proof of Theorem 1.2 for PGLr/SLr]{Proof of Theorem~\ref{intro:mainthm} for \texorpdfstring{$\PGL_r/\SL_r$}{PGLr/SLr}}
We show that the functor $\Phi_{\rP/\rS}$ is an equivalence: 
\begin{prop}\label{prop:PGL}
The functor
\begin{align}\label{PhiPS}
    \Phi_{\rP/\rS} \colon \IndCoh_{\mathcal{N}}(\cH_{\mathrm{P}}(\overline{w})^{\mathrm{ss}})_0 \to \IndL_{\mathcal{N}}(\cH_{\rS}(\chi_0))_{\overline{w}}
\end{align}
induced by (\ref{funct:PhiPSL})
is an equivalence. 
\end{prop}
\begin{proof}
    The commutative diagram (\ref{comm:PS2L}) extends to 
    the commutative diagram (with continuous functors)
  \begin{align}\label{comm:PS3L}
\xymatrix{
\IndCoh_{\mathcal{N}}(\cH_{\mathrm{P}}(\overline{w})^{\mathrm{ss}})_0
\ar[r]^-{\Phi_{\rP/\rS}} \ar[d]_-{p^!}
&
\IndL_{\mathcal{N}}(\cH_{\rS}(\chi_0))_{\overline{w}} \ar[d]_-{i_{*}}
\\
\IndCoh_{\mathcal{N}}(\cH_{\circ}(w)^{\mathrm{ss}})_0
\ar[r]^-{\Phi}_-{\sim}
&
\IndL_{\mathcal{N}}(\cH_{\circ}(\chi_0))_{w'}.
}
\end{align}
By taking the ind-completion of (\ref{comm:PSL}) and the right adjoints, 
we have the commutative diagram 
\begin{align}\label{comm:PS4L}
\xymatrix{
\IndCoh_{\mathcal{N}}(\cH_{\mathrm{P}}(\overline{w})^{\mathrm{ss}})_0
\ar[d]_-{p^!}
&
\IndL_{\mathcal{N}}(\cH_{\rS}(\chi_0))_{\overline{w}} \ar[d]_-{i_{*}} 
\ar[l]_-{\Phi_{\rP/\rS}^R} 
\\
\IndCoh_{\mathcal{N}}(\cH_{\circ}(w)^{\mathrm{ss}})_0
&
\IndL_{\mathcal{N}}(\cH_{\circ}(\chi_0))_{w'}. \ar[l]_-{\Phi^R}^-{\sim}
}
\end{align}
We have the 
the natural transform 
\begin{align}\label{nat:P}
\id \Rightarrow \Phi_{\rP/\rS}^R\circ \Phi_{\rP/\rS}.
\end{align}
By applying $p^!$, we obtain 
\begin{align}\label{p!:id}
    p^! \Rightarrow p^!\circ \Phi_{\rP/\rS}^R\circ \Phi_{\rP/\rS}.
\end{align}
By the diagrams (\ref{comm:PS3L}), (\ref{comm:PS4L}), we have the isomorphisms
\begin{align*}
     p^!\circ \Phi_{\rP/\rS}^R \circ \Phi_{\rP/\rS}&\cong \Phi^R \circ i_{*} \circ \Phi_{\rP/\rS} \\
     &\cong \Phi^R \circ \Phi \circ p^! 
\end{align*}
and it is naturally isomorphic to $p^!$ as $\Phi$ is an equivalence. Under the above isomorphisms, 
the natural transform (\ref{p!:id}) is $\id \colon p^! \Rightarrow p^!$, hence an isomorphism. 

Similarly the natural transform 
\begin{align}\label{nat:P2}
    \Phi_{\rP/\rS}\circ \Phi_{\rP/\rS}^R \Rightarrow \id
\end{align}
applied with $i_{*}$ is an isomorphism 
\begin{align*}
    i_{*}\circ \Phi_{\rP/\rS} \circ \Phi_{\rP/\rS}^R \cong i_{*} \Rightarrow i_{*}.
\end{align*}
Therefore (\ref{nat:P}), (\ref{nat:P2}) are isomorphisms since $p^!$, $i_{*}$ are conservative. 
\end{proof}

By combining the above results, we obtain the following: 
\begin{thm}\label{thm:PGL}
Conjecture~\ref{conj:intro} holds for $(^{L}G, G)=(\PGL_r, \SL_r)$ and 
over $\rB_{G}^A$. 
\end{thm}
\begin{proof}
    The theorem follows from Proposition~\ref{prop:PGL}
    and the isomorphism (\ref{isom:norm}). 
\end{proof}

\section{Some auxiliary results}
\subsection{Some lemmas on variants of limit categories}
We consider the setting of Subsection~\ref{subsec:variant}; 
namely $\mathfrak{M}$ is a quasi-smooth derived stack with self-dual 
cotangent complex, and 
\begin{align*}
    \mathfrak{M}=\mathfrak{M}_{\circ}\times Y\times k[-1]
\end{align*}
for a smooth affine scheme $Y$. For $\delta \in \mathrm{Pic}(\mathfrak{M}_{\circ})_{\mathbb{Q}}$, we use the same symbol $\delta$ for
its pull-back to $\mathfrak{M}$. 

Suppose that $Y=\mathrm{pt}$ and $\mathfrak{M}=\mathfrak{M}_{\circ}\times k[-1]$. In this case, giving a map $\nu\colon \bgm \to \mathfrak{M}_{\circ}$
is equivalent to giving $\nu \colon \bgm \to \mathfrak{M}$, which 
we denote the same symbol $\nu$. Consider the following natural 
maps 
\begin{align*}
    \mathfrak{M}_{\circ} \stackrel{i}{\hookrightarrow} 
    \mathfrak{M} \stackrel{\eta}{\to} \mathfrak{M}_{\circ}.
\end{align*}
\begin{lemma}\label{lem:eta}
For $\nu \colon \bgm \to \mathfrak{M}_{\circ}$ with image $x\in \mathfrak{M}_{\circ}(k)$, let $\nu_{\circ}^{\mathrm{reg}}$, 
$\nu^{\mathrm{reg}}$ be the regularization maps for $\mathfrak{M}_{\circ}$, 
$\mathfrak{M}$, respectively as in Subsection~\ref{subsec:variant}. Then 
for $\mathcal{E} \in \Coh(\mathfrak{M}_{\circ})$, we have 
\begin{align*}
    \wt(\iota_{*}\nu_{\circ}^{\mathrm{reg}*} \cE)=\wt(\iota_{*}\nu^{\mathrm{reg}*}i_{*}\cE)
    =\wt(\iota_{*}\nu^{\mathrm{reg}*}\eta^* \cE).
\end{align*}
In particular for $\delta \in \mathrm{Pic}(\mathfrak{M})_{\mathbb{R}}$
with its restriction $\delta\in \mathrm{Pic}(\mathfrak{M}_{\circ})_{\mathbb{R}}$, 
we have $\cE \in \widetilde{\LL}(\mathfrak{M}_{\circ})_{\delta}$ if and only 
if $i_{*}\cE \in \widetilde{\LL}(\mathfrak{M})_{\delta}$, if and only if 
$\eta^*\cE \in \widetilde{\LL}(\mathfrak{M})_{\delta}$.
\end{lemma}
\begin{proof}
    We can take regularization maps which are related by 
    \begin{align*}
        \nu^{\mathrm{reg}}=\nu_{\circ}^{\mathrm{reg}}\times \id \colon 
        (\mathfrak{g}_x)_{0}^{\vee}[-1]/\mathbb{G}_m \times k[-1] \to 
        \mathfrak{M}_{\circ}\times k[-1].
    \end{align*}
    Then we have 
    \begin{align*}
        \iota_{*}\nu^{\mathrm{reg}*}i_{*}\cE\cong\iota_{*}\nu_{\circ}^{\mathrm{reg}*} \cE, \ 
         \iota_{*}\nu^{\mathrm{reg}*}\eta^*\cE\cong\iota_{*}\nu_{\circ}^{\mathrm{reg}*} \cE
         \otimes (k\oplus k[1]).
    \end{align*}
    The lemma follows from the above isomorphisms. 
\end{proof}

\begin{lemma}\label{lem:eta2}
In the setting of Lemma~\ref{lem:eta}, for $\cE \in \Coh(\mathfrak{M})$ we have 
\begin{align*}
\wt(\iota_{*}\nu_{\circ}^{\mathrm{reg}*}\eta_{*}\cE)=\wt(\iota_{*}\nu^{\mathrm{reg}*}\cE).
\end{align*}
\end{lemma}
\begin{proof}
    The lemma follows by the following commutative diagram, where the right 
    square is Cartesian: 
    \begin{align*}
        \xymatrix{
& \ar[ld]_-{\iota} (\mathfrak{g}_x)^{\vee}[-1]/\mathbb{G}_m \ar[r]^-{\nu^{\mathrm{reg}}} \ar[d] \diasquare & \mathfrak{M} \ar[d]^-{\eta} \\
\bgm & \ar[l]_-{\iota} (\mathfrak{g}_x)_0^{\vee}[-1]/\mathbb{G}_m 
\ar[r]^-{\nu_{\circ}^{\mathrm{reg}}}  & \mathfrak{M}_{\circ}.
        }
    \end{align*}
\end{proof}

Later we will use the following lemma: 
\begin{lemma}\emph{(cf.~\cite[Lemma~6.2]{TodaGL2})}\label{lem:perfect2}
In the setting of Definition~\ref{def:Lcat}, suppose that
$\cE \in \Coh(\mathfrak{M}_{\circ})$ admits an expression as a colimit in
$\QCoh(\mathfrak{M}_{\circ})$
\begin{align*}
    \cE = \operatorname*{colim}_{i} A_i
\end{align*}
such that each $A_i$ is perfect. For a map $\nu \colon \bgm \to \mathfrak{M}_{\circ}$ with image $x\in \mathfrak{M}_{\circ}(k)$, 
suppose that there is a bounded interval $I \subset \mathbb{R}$ such that 
each $\wt(\nu^* A_i)$ is contained in $I$. Then 
$\wt(\iota_{*}\nu_\circ^{\mathrm{reg}*}\cE)$ is contained in $I+[c_1(\mathfrak{g}_x^{<0}), c_1(\mathfrak{g}_x^{>0})]$.

In particular if each $A_i$ satisfies the condition
\begin{align}\label{wt:cond2.5}
    \wt(\nu^* A_i) \subset \left[\frac{1}{2} c_1 (\nu^{\ast}\mathbb{L}_{\mathfrak{M}_{\circ}}^{<0}), 
    \frac{1}{2} c_1 (\nu^{\ast}\mathbb{L}_{\mathfrak{M}_{\circ}}^{>0})
    \right]+c_1(\nu^{\ast}\delta).
\end{align}
then $\cE$ satisfies the condition (\ref{cond:regcirc}). 
\end{lemma}
\begin{proof}
    The same statement is proved in~\cite[Lemma~6.2]{TodaGL2}
    for $\mathfrak{M}$, and the version for $\mathfrak{M}_{\circ}$
    is obtained by combining with Lemma~\ref{lem:eta}. 
    Indeed we have 
    \begin{align*}
   \eta^*\cE = \operatorname*{colim}_{i} \eta^*A_i
    \end{align*}
    in $\QCoh(\mathfrak{M})$ and $\eta^* A_i$ is perfect. Hence applying~\cite[Lemma~6.2]{TodaGL2}, 
    we see that $\wt(\iota_{*}\nu^{\mathrm{reg}*}\eta^* \cE)$
    is contained in $I+[c_1(\mathfrak{g}_x^{<0}), c_1(\mathfrak{g}_x^{>0})]$. 
    Since we have $\wt(\iota_{*}\nu^{\mathrm{reg}*}\eta^* \cE)=\wt(\iota_{*}\nu_\circ^{\mathrm{reg}*}\cE)$ by Lemma~\ref{lem:eta}, 
    the lemma follows. 
\end{proof}

We also have the following lemma (which was implicitly used in~\cite{TodaGL2}):

\begin{lemma}\label{lem:equiv:k}
There is an equivalence
\begin{align}\label{equiv:k-1}
    \IndL_{\mathcal{N}}(\mathfrak{M}_{\circ})_{\delta}
    \otimes \QCoh(k[-1])
    \stackrel{\sim}{\to}
    \IndL_{\mathcal{N}}(\mathfrak{M})_{\delta}.
\end{align}
In particular, we have the equivalence 
\begin{align*}
      \IndL_{\mathcal{N}}(\mathfrak{M}_{\circ})_{\delta}
    \stackrel{\sim}{\to}
    \IndL_{\mathcal{N}}(\mathfrak{M})_{\delta}\otimes_{\QCoh(k[-1])}\QCoh(\mathrm{pt}).
\end{align*}
\end{lemma}

\begin{proof}
By~\cite[Corollary~4.2.3]{MR3037900}, we have an equivalence
\begin{align}\label{equiv:k0}
    \IndCoh(\mathfrak{M}_{\circ}) \otimes \QCoh(k[-1])
    \stackrel{\sim}{\to}
    \IndCoh(\mathfrak{M}).
\end{align}
This equivalence restricts to an equivalence
\begin{align}\label{equiv:k2}
    \IndCoh_{\mathcal{N}}(\mathfrak{M}_{\circ}) \otimes \QCoh(k[-1])
    \stackrel{\sim}{\to}
    \IndCoh_{\mathcal{N}}(\mathfrak{M}).
\end{align}
Indeed, this follows from the scheme case in~\cite[Lemma~4.6.4]{AG}, together
with descent for singular supports in~\cite[Section~8.3]{AG}.
We show that the equivalence \eqref{equiv:k2} restricts to an equivalence
\eqref{equiv:k-1}.

We first prove this when $\mathfrak{M}$ is QCA. By Lemma~\ref{lem:eta},
the equivalence restricts to a fully faithful functor
\begin{align}\label{equiv:kl}
    \IndL_{\mathcal{N}}(\mathfrak{M}_{\circ})_{\delta}
    \otimes \QCoh(k[-1])
    \hookrightarrow
    \IndL_{\mathcal{N}}(\mathfrak{M})_{\delta}.
\end{align}
On the other hand, for an object $M \in \QCoh(k[-1])$, there is a functorial
bar resolution for $R=\mathcal{O}_{k[-1]}$:
\begin{align*}
    \cdots \to M \otimes R^{\otimes n} \to \cdots
    \to M \otimes R \to M.
\end{align*}
Applying this resolution to the left-hand side of \eqref{equiv:k0}, and then
using the equivalence \eqref{equiv:k0}, we see that any object
$A \in \IndCoh_{\mathcal{N}}(\mathfrak{M})$ is expressed as a colimit
\begin{align*}
\operatorname*{colim}_{n\to \infty}
    \left(
    0 \to \eta^* \eta_* A \otimes R^{\otimes n}
    \to \cdots
    \to \eta^* \eta_* A \otimes R
    \to \eta^* \eta_* A
    \right)
    \stackrel{\sim}{\to} A.
\end{align*}
Here $\eta^*$ is defined for ind-coherent sheaves since $\eta$ is
quasi-smooth; see Subsection~\ref{subsec:notation0}.
By Lemma~\ref{lem:eta2}, if
$A \in \IndL_{\mathcal{N}}(\mathfrak{M})_{\delta}$, then
$\eta_* A \in \IndL_{\mathcal{N}}(\mathfrak{M}_{\circ})_{\delta}$.
Therefore \eqref{equiv:kl} is essentially surjective, and hence an equivalence.

In general, when $\mathfrak{M}$ is not QCA, the equivalence \eqref{equiv:k-1}
follows from the QCA case, since tensoring with $\QCoh(k[-1])$ commutes with
limits, as $\QCoh(k[-1])$ is dualizable; see~\cite[4.1.3(iii)]{MR3037900}.
\end{proof}

We next consider a general case, that is \begin{align*}\mathfrak{M}=\mathfrak{M}_{\circ}'\times k[-1], \ \mathfrak{M}_{\circ}'=\mathfrak{M}_{\circ}\times Y
\end{align*}
for a 
smooth scheme $Y$. 
Then by~\cite[Corollary~4.2.3]{MR3037900} and \cite[Lemma~4.6.4, Section~8.3]{AG}, 
we have the equivalence 
\begin{align}\label{equiv:Y1}
    \IndCoh_{\mathcal{N}}(\mathfrak{M}_0)\otimes \QCoh(Y) \stackrel{\sim}{\to} \IndCoh_{\mathcal{N}}(\mathfrak{M}_{\circ}'). 
\end{align}
We fix $y\in Y(k)$ and 
consider the symmetric monoidal functor 
\begin{align*}i_y^* \colon \QCoh(Y) \to \QCoh(\mathrm{pt})
\end{align*}
where $i_y \colon \mathrm{pt} \hookrightarrow Y$ is the closed 
immersion corresponding to $y$. By applying $(-)\otimes_{\QCoh(Y)}\QCoh(\mathrm{pt})$, we obtain 
the equivalence 
\begin{align}\label{equiv:Y2}
\IndCoh_{\mathcal{N}}(\mathfrak{M}_{\circ})\stackrel{\sim}{\to}
\IndCoh_{\mathcal{N}}(\mathfrak{M}_{\circ}')\otimes_{\QCoh(Y)}\QCoh(\mathrm{pt}).
\end{align}
Below we take $\delta \in \mathrm{Pic}(\mathfrak{M}_{\circ})_{\mathbb{R}}$, 
and use the same symbol $\delta$ for its pull-back to $\mathfrak{M}_{\circ}'$. 
\begin{lemma}\label{lem:resty}
The equivalence \eqref{equiv:Y2} restricts to an equivalence
\begin{align}\label{equiv:Y3}
\IndL_{\mathcal{N}}(\mathfrak{M}_{\circ})_{\delta}
\stackrel{\sim}{\to}
\IndL_{\mathcal{N}}(\mathfrak{M}_{\circ}')_{\delta}
\otimes_{\QCoh(Y)}
\QCoh(\mathrm{pt}).
\end{align}
\end{lemma}

\begin{proof}
It is clear that the pullback along the smooth projection
$\mathfrak{M}_{\circ}'\to \mathfrak{M}_{\circ}$
induces the functor 
\begin{align*}
    \IndL_{\mathcal{N}}(\mathfrak{M}_{\circ})_{\delta} \to 
    \IndL_{\mathcal{N}}(\mathfrak{M}_{\circ}')_{\delta}.
\end{align*}
Therefore, the functor in \eqref{equiv:Y2}
restricts to a fully faithful functor
\begin{align}\label{equiv:LNQ}
\IndL_{\mathcal{N}}(\mathfrak{M}_{\circ})_{\delta}
\hookrightarrow
\IndL_{\mathcal{N}}(\mathfrak{M}_{\circ}')_{\delta}
\otimes_{\QCoh(Y)}
\QCoh(\mathrm{pt}).
\end{align}

On the other hand, the right-hand side of \eqref{equiv:Y2} is computed
as the colimit of the bar construction:
\begin{align}\label{bar:Y:colim}
\operatorname*{colim}_{[n]\in \Delta^{\mathrm{op}}}
\left(
\IndCoh_{\mathcal{N}}(\mathfrak{M}_{\circ}')
\otimes
\QCoh(Y)^{\otimes n}
\otimes
\QCoh(\mathrm{pt})
\right).
\end{align}

An inverse functor to \eqref{equiv:Y2} is induced termwise on the bar
construction by the functors
\begin{align*}
\IndCoh_{\mathcal{N}}(\mathfrak{M}_{\circ}')
\otimes
\QCoh(Y)^{\otimes n}
\otimes
\QCoh(\mathrm{pt})
\longrightarrow
\IndCoh_{\mathcal{N}}(\mathfrak{M}_{\circ}),
\end{align*}
obtained by restricting to the fiber
$\mathfrak{M}_{\circ}'\times_Y \{y\}$ and applying $(i_y^*)^{\otimes n}$
to the $\QCoh(Y)$-factors, where $i_y\colon \mathrm{pt}\to Y$ is the
corresponding point. By the definition of limit categories, these functors
restrict to
\begin{align*}
\IndL_{\mathcal{N}}(\mathfrak{M}_{\circ}')_{\delta}
\otimes
\QCoh(Y)^{\otimes n}
\otimes
\QCoh(\mathrm{pt})
\longrightarrow
\IndL_{\mathcal{N}}(\mathfrak{M}_{\circ})_{\delta}.
\end{align*}
Taking the colimit, we obtain an inverse to
\eqref{equiv:LNQ}. Hence \eqref{equiv:LNQ} is an equivalence.
\end{proof}

\subsection[Proof of Lemma 3.17 for g=2]{Proof of Lemma~\ref{lem:nodeB} for \texorpdfstring{$g=2$}{g=2}}\label{subsec:g=2}
\begin{proof}
We prove Lemma~\ref{lem:nodeB} in the case $g=2$. The case $r=1$ is trivial, so we assume 
$r\geq 2$. 

We first treat the case $r\geq 4$. 
Let $z=(c,\lambda)\in S$, where $S=\mathrm{Tot}_C(\Omega_C)$, 
$c\in C$, and $\lambda\in \Omega_C|_c$. 
It is enough to show that the locus 
\begin{align}\label{nonnode:z}
    \rB_{\GL_r}^{\mathrm{non-node}, z} \subset \rB_{\GL_r}^{\mathrm{cl}}
\end{align}
having non-nodal singularity at $z$ is of codimension four; 
then the claim follows as a choice of $z$ is two-dimensional. 
By applying the $H^0(\Omega_C)$-action \eqref{act:gamma}, and using the fact 
that $\Omega_C$ is globally generated, 
and that this action does not change the singularity type of $\cC_b$, 
we may assume that $\lambda=0$. 

We choose trivializations
\[
\widehat{\cO}_{C, c} \cong k[[x]], \ 
    \widehat{\Omega}_{C,c}\cong k[[x]]y.
\]
For $b\in \rB_{\GL_r}^{\mathrm{cl}}$
the defining equation \eqref{eqn:Cb} of $\cC_b$ determines an element of 
$\widehat{\cO}_{S,z}$, hence an element
\[
    \widehat{F}_b \in J_z:=\widehat{\cO}_{S,z}/m_z^3\cong k[[x, y]]/(x, y)^3,
\]
where $m_z\subset \widehat{\cO}_{S,z}$ is the maximal ideal. 
Thus we obtain a map
\[
    j_z \colon \rB_{\GL_r}^{\mathrm{cl}} \to J_z, \ b\mapsto \widehat{F}_b.
\]
We show that this map is surjective when $r\geq 4$. Noting that $z=(c, 0)$, the map $j_z$ is 
described as
\begin{align}\label{surj:jz}
    j_z \colon \rB_{\GL_r}^{\mathrm{cl}}
    \to k\oplus kx\oplus ky\oplus kx^2\oplus kxy\oplus ky^2
\end{align}
and is given by
\[
    (b_i)_{1\leq i\leq r}
    \mapsto
    \bigl(
    \widehat{b}_r(0),
    \widehat{b}_{r-1}(0),
    \widehat{b}_r'(0),
    \widehat{b}_{r-2}(0),
    \widehat{b}_{r-1}'(0),
    \widehat{b}_{r}''(0)
    \bigr).
\]
Here $\widehat{b}_i$ denotes the image of $b_i$ under the map
\[
H^0(C,\Omega_C^i)\to \widehat{\Omega}_{C,c}^i\cong k[[x]].
\]
It is enough to show that
\begin{align}\label{map:Omega}
    H^0(C,\Omega_C^{r-j})
    \to
    \Omega_C^{r-j}\otimes \cO_C/m_c^{3-j}
\end{align}
is surjective for $0\leq j\leq 2$. This follows from
\begin{align}\label{vanish:Omega}
    H^1(C,\Omega_C^{r-j}(-(3-j)c))
    =
    H^0(C,\Omega_C^{j-r+1}((3-j)c))^{\vee}
    =
    0.
\end{align}
The last vanishing holds because
\begin{align}\label{ineq:r}
    \deg \Omega_C^{j-r+1}((3-j)c)=j+5-2r<0
\end{align}
for $0\leq j\leq 2$ and $r\geq 4$. 

We write an element of \(J_z^2\) as
\[
c_{00}+c_{10}x+c_{01}y+c_{20}x^2+c_{11}xy+c_{02}y^2 .
\]
The spectral curve $\cC_b$ is singular at \(z\) precisely when
\[
c_{00}=c_{10}=c_{01}=0.
\]
At such a point, the singularity is an ordinary node precisely when the quadratic part
\[
c_{20}x^2+c_{11}xy+c_{02}y^2
\]
is nondegenerate. Hence the condition that \(z\) is a non-nodal singular point is
\begin{align}\label{eqn:c}
c_{00}=c_{10}=c_{01}=0,
\qquad
c_{11}^2-4c_{20}c_{02}=0.
\end{align}
Thus this condition has codimension four in \(J_z\).
Since \eqref{surj:jz} is surjective, the locus 
(\ref{nonnode:z}) has codimension four. 

Next suppose that $r=3$. Then the inequality \eqref{ineq:r} fails for 
$j=1,2$, but the map \eqref{map:Omega} is still surjective for $j=2$, since 
$\Omega_C$ is globally generated. For $j=1$, the map \eqref{map:Omega} fails 
to be surjective precisely when
$
\Omega_C=\mathcal{O}_C(2c),
$
i.e. when $c\in C$ is a Weierstrass point. 
In this case, the image of the map \eqref{map:Omega} has codimension one, and 
the image of $j_z$ is
\[
    \operatorname{Im}(j_z)=W:=\{c_{11}=0\}\subset J_z .
\]
The subspace $W$ is five-dimensional, and the locus in $W$ satisfying 
\eqref{eqn:c} is
\[
c_{00}=c_{10}=c_{01}=0,
\qquad
c_{20}c_{02}=0.
\]
This locus has dimension one, hence codimension four in $W$. Therefore the 
locus (\ref{nonnode:z}) again has 
codimension four. 

It remains to treat the case \(r=2\). By applying the 
$H^0(\Omega_C)$-action \eqref{act:gamma}, it is enough to show that the locus 
of $b_2\in H^0(C,\Omega_C^2)$ such that the spectral curve
\begin{align}\label{curve:cb}
    \cC_b=\{\lambda^2=b_2\}
\end{align}
is non-nodal has codimension at least two. 
Since \(g=2\), the curve \(C\) is hyperelliptic. Let
$
g \colon C \to \mathbb{P}^1
$
be the hyperelliptic double cover. Then
$
\Omega_C=g^*\mathcal{O}_{\mathbb{P}^1}(1),
$
and we have an isomorphism
\[
H^0(\mathbb{P}^1,\mathcal{O}_{\mathbb{P}^1}(2))
\stackrel{\cong}{\longrightarrow}
H^0(C,\Omega_C^2).
\]
Thus we can write $b_2=g^*q$ for some 
$q\in H^0(\mathbb{P}^1,\mathcal{O}_{\mathbb{P}^1}(2))$. 
The curve \eqref{curve:cb} has a non-nodal singularity only if $q$ has a 
double zero at a branch point of $g\colon C\to \mathbb{P}^1$. 
This is a codimension two condition. 
This completes the proof.
    \end{proof}

\newcommand{\etalchar}[1]{$^{#1}$}
\providecommand{\bysame}{\leavevmode\hbox to3em{\hrulefill}\thinspace}
\providecommand{\MR}{\relax\ifhmode\unskip\space\fi MR }
\providecommand{\MRhref}[2]{%
  \href{http://www.ams.org/mathscinet-getitem?mr=#1}{#2}
}
\providecommand{\href}[2]{#2}


\vspace{5mm}

{\small\textsc{Yukinobu Toda: Kavli Institute for the Physics and Mathematics of the Universe (WPI), University of Tokyo, 5-1-5 Kashiwanoha, Kashiwa, 277-8583, Japan.}\par}
{\small\textsc{Inamori Research Institute for Science, 620 Suiginya-cho, Shimogyo-ku, Kyoto 600-8411, Japan.}\par}
{\small\textit{E-mail address:} \texttt{yukinobu.toda@ipmu.jp}\par}

\end{document}